\documentclass[numbers,webpdf,imanum]{ima-authoring-template}%

\usepackage{graphicx,subfigure,color}
\usepackage{cases}
\usepackage{exscale}
\usepackage{relsize}
\usepackage{hyperref}
\usepackage{mathtools}
\numberwithin{figure}{section}
\numberwithin{equation}{section}
\usepackage{url}
\usepackage{epsfig}
\graphicspath{{figs/}}
\theoremstyle{thmstyletwo}%
\newtheorem{theorem}{Theorem}[section]
\newtheorem{lemma}[theorem]{Lemma}

\newtheorem{remark}[theorem]{Remark}%

\begin{document}

\DOI{DOI HERE}
\copyrightyear{2024}
\vol{00}
\pubyear{2024}
\access{Advance Access Publication Date: Day Month Year}
\appnotes{Paper}
\copyrightstatement{Published by Oxford University Press on behalf of the Institute of Mathematics and its Applications. All rights reserved.}
\firstpage{1}


\title[Exponential time differencing methods for the matrix-valued Allen--Cahn equation]{On the maximum bound principle and energy dissipation of exponential time differencing methods for the matrix-valued Allen--Cahn equation}

\author{Yaru Liu
\address{\orgdiv{School of Mathematical Sciences}, \orgname{University of Electronic Science and Technology of China}, \orgaddress{\state{Sichuan}, \country{China}}\&\orgdiv{College of Mathematics and System Science}, \orgname{Xinjiang University}, \orgaddress{ \state{Urumqi}, \country{China}}}}
\author{Chaoyu Quan
\address{\orgdiv{School of Science and Engineering}, \orgname{The Chinese University of Hong Kong, Shenzhen}, \orgaddress{\state{Guangdong}, \country{China}}}}
\author{Dong Wang
\address{\orgdiv{School of Science and Engineering}, \orgname{The Chinese University of Hong Kong, Shenzhen}\& Shenzhen International Center for Industrial and Applied Mathematics, Shenzhen Research Institute of Big Data, \orgaddress{\state{Guangdong}, \country{China}}} }

\authormark{Y. Liu, C. Quan, and D. Wang}


\received{Date}{0}{Year}
\revised{Date}{0}{Year}
\accepted{Date}{0}{Year}


\abstract{This work delves into the exponential time differencing (ETD) schemes for the matrix-valued Allen--Cahn equation. In fact,  the maximum bound principle (MBP) for the first- and second-order ETD schemes is presented in a prior publication [SIAM Review, 63(2), 2021], assuming a symmetric initial matrix field.
Noteworthy is our novel contribution, demonstrating that the first- and second-order ETD schemes for the matrix-valued Allen--Cahn equation—both being linear schemes—unconditionally preserve the MBP, even in instances of nonsymmetric initial conditions. Furthermore, we prove that these two ETD schemes preserve the energy dissipation law unconditionally for the matrix-valued Allen--Cahn equation, and  their convergence analysis is also provided. 
Some numerical examples are presented to verify our theoretical results and to simulate the evolution of corresponding matrix fields.}
\keywords{Matrix-valued Allen--Cahn equation; Exponential time differencing schemes; Maximum bound principle; Energy dissipation law.}

\allowdisplaybreaks
\sloppy
\maketitle

%

\section{Introduction}

Smooth matrix field values have been widely used in various fields, such as inverse problems in image analysis \cite{batard2014covariant,rosman2014augmented}, directional field synthesis problems arising in geometry processing and computer graphics \cite{vaxman2016directional}, and the study of polycrystals \cite{elsey2013simple,elsey2014fast}.

In this paper, we consider the matrix-valued Allen--Cahn equation \cite{osting2020diffusion} with the periodic boundary condition:
   \begin{equation}\label{eq:mac}
     \begin{aligned}
       \begin{cases}
          U_t=\varepsilon^2\Delta U+f(U),\quad t\in(0,T],~\mathbf x \in \Omega,\\
          U|_{t=0}=U_0,
        \end{cases}
     \end{aligned}
   \end{equation}
where $U(t,\mathbf x )\in \mathbb{R}^{m\times m}$ is a real matrix-valued field with $m\geq 2$, $U_0$ is the initial condition, $\Omega$ is taken to be $[-\frac 12,\frac 12]^m$ for simplicity, $T>0$ is a finite time, $\varepsilon>0$ represents the interface width of two phases, and the nonlinear term $f(U)$ is defined as $$f(U)=U-UU^\text{T}U.$$
Obviously, \eqref{eq:mac} is the $L^2$ gradient flow for the energy 
\begin{equation}
    \begin{aligned}\label{energy}
        E(U)&=\int_{\Omega}\left(\frac{\varepsilon^2}{2}\| \nabla 
        U\|_F^2+\left\langle F(U),I\right\rangle_F\right)d\mathbf x ,
    \end{aligned}
\end{equation}
where $\|\nabla U\|_F^2=\sum_{i,j=1}^m|\nabla U_{ij}|^2$, $F(U)=\frac{1}{4}(U^\text{T}U-I)^2$, $I$ is the $m\times m$ identity matrix, $\|\cdot\|_F$ denotes the Frobenius norm, and $\langle\cdot,\cdot\rangle_F$ represents the Frobenius inner product. 

The energy \eqref{energy} can be considered as a model problem for crystallography, where one considers a field that takes values in $SO(3)$ modulo the symmetry group of the crystal. This is also a model problem for the three-dimensional cross field design problem \cite{viertel2019approach,golovaty2021variational}. This problem is related to problems in rigid motion planning, where one tries to find a time-dependent trajectory, $u: [t,s] \rightarrow T$, where the function takes values in a set that describes admissible rigid motions, such as $T = SO(3)$ \cite{sciavicco2001modelling}. The similar idea can also be extended to a wide range of target-valued harmonic maps, which include applications in target-valued image denoising \cite{osting2020diffusion2}, optimal partition problems \cite{wang2022efficient,wang2019diffusion} and so on.

Cross fields are finding wide-spread use in mesh generation and computer graphics \cite{golovaty2021variational}. For example, in two dimensions (or on surfaces in three dimensions), quad meshes can be obtained by finding proper parametrization based on a cross field \cite{li2012all}. How to find a harmonic orthogonal matrix-valued field plays a crucial role in looking for the harmonic cross field within a fixed arbitrary domain. The matrix-valued Allen--Cahn equation was firstly introduced in \cite{osting2020diffusion} to look for the orthogonal matrix-valued harmonic maps in periodic domains or closed surfaces. They further studied the asymptotic dynamics of the interface raised from matrix-valued Allen--Cahn equation at different time scales and performed several numerical experiments at different time scales \cite{wang2019interface}. Here, the interface is implicitly determined by regions where the determinant of the matrix is positive or negative. In the first time scale, the interface moves along its normal direction by its mean curvature, which is consistent with scale Allen--Cahn equations. However, in the second scale, the motion of the interface is found to be dependent on the matrix-valued field. 
In \cite{wang2019interface}, Wang, Osting, and Wang formally derived that the motion law of the interface is determined by the surface diffusion of the matrix-valued field. It then draws a lot of attention from both theoretical and numerical points of view. 
Li, Quan, and Xu investigated the matrix-valued Allen--Cahn equation using a Strang operator splitting method \cite{li2022stability2}. They established a sharp maximum principle and modified energy dissipation under mild splitting step constraints. 
In the work \cite{du2021maximum}, Du, Ju, Li et al. studied first- and second-order exponential time differencing (ETD) schemes for the matrix-valued Allen--Cahn equation satisfying the maximum bound principle (MBP). Notably, their analysis requires the initial matrix field to be symmetric, a condition that might be not met in a general situation. 
Similarly, Sun and Zhou delved into a fourth-order integrating factor Runge--Kutta (IFRK) scheme for the matrix-valued Allen--Cahn equation \cite{sun2023maximum}, that also requires the initial matrix field to be symmetric.
Recently, Fei, Lin, Wang, et al. rigorously analyzed the sharp interface limit of the matrix-valued Allen--Cahn equation and proved that the sharp interface system is a two-phase flow system: the interface evolves according to motion by mean curvature \cite{fei2023matrix}.

The matrix-valued Allen--Cahn is a generalization of the  scalar Allen--Cahn equation, that is also a phase-field model \cite{chen2002phase,kim2012phase,penrose1990thermodynamically}. It is well-known that the Allen--Cahn equation preserves both the MBP and the energy dissipation law. The MBP means that if the absolute value of the initial data is bounded by some constant, then the solution of the equation at any moment is controlled by that constant. 
The energy dissipation property requires that the numerical method can guarantee that the total energy of the system tends to decrease during the evolution process to ensure that the numerical solution agrees with the behavior of the actual physical system. In recent decades, numerous researchers have delved into the exploration of numerical schemes for solving the phase-field models. A variety of well-developed numerical schemes are available to satisfy the MBP or the energy dissipation law, such as implicit-explicit (IMEX) stabilization methods \cite{chen1998applications,tang2016implicit,li2020second,fu2022unconditionally}, invariant energy quadratization (IEQ) method \cite{yang2016linear}, scalar auxiliary variable (SAV) methods \cite{shen2018scalar,akrivis2019energy}, operator splitting methods \cite{cheng2015fast,li2017convergence,li2022stability}, IFRK methods \cite{ju2021maximum,li2021stabilized}, ETD methods \cite{du2019maximum,li2020arbitrarily,li2021unconditionally,fu2022energy,fu2024higher} and so on.
In particular, for the first- and second-order ETD methods for the scalar Allen--Cahn equation, it has been proved that the MBP is preserved unconditionally \cite{du2021maximum}, as well as the energy dissipation law \cite{fu2022energy}. 
Moreover, the exponential cut-off method studied by Li, Yang, and Zhou \cite{li2020arbitrarily} use the $k$th-order multistep exponential integrator in time. At every time level, the extra values violating the MBP are eliminated by a cut-off operation, so that the numerical solutions at nodal points satisfy the MBP.

In this work, we consider first- and second-order ETD schemes for the matrix-valued Allen--Cahn equation and prove the MBP preservation of them in the Frobenius norm, based on a more detailed analysis of the nonlinear term.
That is to say, if the initial condition $U_0(\boldsymbol x)\in \mathbb{R}^{m\times m}$ and $\|U_0(\boldsymbol x)\|_F\leq \sqrt{m}$ for any $\boldsymbol x \in \overline{\Omega}$, then the solution to \eqref{eq:mac} satisfies
\begin{align*}
    \max _{\mathbf x  \in \overline{\Omega}}\|U(t,\mathbf x )\|_F\leq \sqrt{m},\quad t\in(0,T].
\end{align*}
Compared to the MBP results in \cite{du2021maximum}, our analysis has removed the symmetric assumption of the initial condition. More precisely speaking, the authors in \cite{du2021maximum} use the fact that real symmetric matrices are diagonalizable to prove the MBP of the ETD schemes of matrix-valued Allen--Cahn equation with respect to the matrix 2-norm, i.e., denote by $| \cdot |_2$ the matrix 2-norm and by $\mathbb{R}^{m\times m}_s$ the set of all real symmetric $m$-by-$m$ matrices. They show that if the initial condition $U_0(\boldsymbol x)\in \mathbb{R}^{m\times m}_s$ and $|U_0(\boldsymbol x) |_2\leq 1$ for any $\boldsymbol x \in \overline{\Omega}$, then the solution to \eqref{eq:mac} satisfies
\begin{align*}
    \max _{\mathbf x  \in \overline{\Omega}} |U(t,\mathbf x )|_2\leq 1,\quad t\in(0,T].
\end{align*} 
Inspired by \cite{fu2022energy}, we prove for the first time that the first- and second-order ETD schemes unconditionally satisfy the energy dissipation law for the matrix-valued Allen--Cahn equation. In addition, we provide error estimates of these two methods with coefficients  independent of $\varepsilon$.
However, we shall mention that by variable transformation, the matrix-valued Allen--Cahn equation \eqref{eq:mac} can be also written as
\begin{equation}\label{eq:MAC_ref}
          U_t=\Delta U+\frac1{\varepsilon^2} f(U),\quad t\in(0,T],~\mathbf x \in \Omega,
\end{equation}
where the parameter $\varepsilon$ appears in front of $f(U)$.
It turns out that the estimation coefficient will depend on $\varepsilon$. For example, one can refer to \cite{akrivis2022error} for detailed error estimate of an BDF method for the scalar Allen--Cahn equation. 

The rest of this paper is organized as follows. In Section \ref{section 2}, we prove the basic properties of the nonlinear operator under the Frobenius norm, and show that the governing equation has a unique solution and satisfy the MBP under these properties. In Section \ref{section 3}, we give the first- and second-order time-discrete schemes of the ETD, and show that the schemes satisfy the discrete MBP and the original energy dissipation unconditionally. Also, we prove the convergence of the ETD schemes. In Section \ref{section 4}, we perform some numerical experiments to verify the theoretical results.

\section{Maximum bound principle, existence, and uniqueness of solution}\label{section 2}

In this section, we first rewrite \eqref{eq:mac} into the following form:
\begin{equation}\label{eq:mac_ka}
      \begin{aligned}
	U_t+\kappa U=\mathcal{L} U+\mathcal{N}[U], \quad t>0,~\mathbf x \in \Omega,
       \end{aligned}
\end{equation}
where $\kappa>0$ is some constant, $\mathcal{L}=\varepsilon^2\Delta$ is the linear elliptic operator, and $\mathcal{N}=\kappa\mathcal{I}+f$ is a nonlinear operator with the identity operator $\mathcal I$. We then show that the matrix-valued Allen--Cahn equation has a unique solution and satisfies the maximum bound principle.
 
\begin{lemma}\label{Nkappa}
For any $V\in \mathbb{R}^{m\times m}$ with $\| V\|_F\leq \sqrt{m}$, if $\kappa\geq \max\left\{\frac{3}{2}m-1,2\right\}$, then we have 
\begin{align}
\| \mathcal{N}[V] \|_F\leq\kappa \sqrt m.
\end{align}
\end{lemma}
\begin{proof} For any $V\in \mathbb{R}^{m\times m}$, the singular value decomposition (SVD) of $V$ is written as $V=P\Sigma Q$, where $P$ and $Q$ are orthogonal matrices, and $\Sigma$ is the diagonal matrix of singular values $\sigma_i^2$ with $i=1,\ldots,m$. Thus, we have
\begin{equation}
      \begin{aligned}
	\| \mathcal{N}[V]\|_F^2&=\| \kappa V+V-VV^\text{T}V\|_F^2\\
        &=\sum_{i=1}^m(\kappa \sigma_i+\sigma_i-\sigma_i^3)^2\\
	&=\sum_{i=1}^m(\kappa +1)^2\sigma_i^2-2(\kappa+1)\sigma_i^4+\sigma_i^6.
      \end{aligned}
\end{equation}

Since $\| V\|_F\leq \sqrt{m}$, we have $\sum_{i=1}^m \sigma_i^2\leq m$ and $\sum_{i=1}^m \frac{1}{m}\sigma_i^2\leq 1$. 
Let $g(x)=(\kappa+1)^2x-2(\kappa+1)x^2+x^3$, $x\in [0,m]$.
We can compute that $g'(x)=(\kappa+1)^2-4(\kappa+1)x+3x^2$ and $g''(x)=-4(\kappa+1)+6x$. 
If 
\begin{equation}\label{ineq:g''g'}
    g''(x)\leq 0 \quad\forall x\in[0,m]\quad\mbox{and}\quad g'(x)\geq 0 \quad \forall x\in[0,1],
\end{equation}
then we can use the Jensen's inequality to obtain
\begin{equation}
      \begin{aligned}
	\sum_{i=1}^m\frac{1}{m}g(\sigma_i^2)\leq g\left(\sum_{i=1}^m\frac{1} 
           {m}\sigma_i^2\right)\leq g(1)=\kappa^2.
       \end{aligned}
\end{equation}
The constraint \eqref{ineq:g''g'} is equivalent to
\begin{equation}
    \kappa\geq \frac32 m -1 \quad\mbox{and}\quad \kappa\geq 2,
\end{equation}
i.e., $\kappa\geq \max\left\{\frac32m-1,2\right\}$.
Under this constraint, we then have 
\begin{align*}
    \| \mathcal{N}[V]\|_F^2 = \sum_{i=1}^m g(\sigma_i^2)\leq \kappa^2m.
\end{align*}\end{proof}

\begin{lemma}\label{NNkappa}
For any $V_1,V_2\in \mathbb{R}^{m\times m}$ with $\| V_1\|_F,~\| V_2\|_F\leq \sqrt m$, we have 
    \begin{align}
	\| \mathcal{N}[V_1]-\mathcal{N}[V_2]  \|_F\leq (\kappa+1+5m)\| V_1-V_2  \|_F.
    \end{align}
\end{lemma}
\begin{proof} 
Recalling the definition of the nonlinear operator $\mathcal{N}$: $\mathbb{R}^{m\times m}\to \mathbb{R}^{m\times m}$, for any $V_1,V_2\in \mathbb{R}^{m\times m}$, we have
\begin{equation}
    \begin{aligned}
	\| \mathcal{N}[V_1]-\mathcal{N}[V_2]\|_F
		= & \|\kappa V_1+V_1-V_1V_1^\text{T}V_1-(\kappa V_2+V_2-V_2V_2^\text{T}V_2)\|_F\\
		= &\|(\kappa+1) (V_1-V_2)-(V_1V_1^\text{T}V_1-V_2V_2^\text{T}V_2)\|_F\\
		\leq&(\kappa+1)\| V_1-V_2  \|_F+\| V_1V_1^\text{T}V_1-V_2V_2^\text{T}V_2  \|_F\\
		\leq & (\kappa+1+\| V_1\|_F^2+\| V_2\|_F^2+3\| V_1\|_F\| V_2\|_F)\| V_1-V_2  \|_F,
    \end{aligned}
\end{equation}
where we use the following estimate
\begin{equation}
    \begin{aligned}
    &\| V_1V_1^\text{T}V_1-V_2V_2^\text{T}V_2  \|_F\\
    =&\| (V_1-V_2)(V_1^\text{T}V_1+V_2^\text{T}V_2)-V_1V_2^\text{T}V_2+V_2V_1^\text{T}V_1  \|_F\\
    =&\| (V_1-V_2)(V_1^\text{T}V_1+V_2^\text{T}V_2)+V_1V_2^\text{T}(V_1-V_2)+(V_2V_1^\text{T}-V_1V_2^\text{T})V_1  \|_F\\
    =&\| (V_1-V_2)(V_1^\text{T}V_1+V_2^\text{T}V_2-V_2^\text{T}V_1)+V_1V_2^\text{T}(V_1-V_2)+V_2(V_1^\text{T}-V_2^\text{T})V_1  \|_F\\
    \leq&\|V_1-V_2\|_F\|V_1^\text{T}V_1+V_2^\text{T}V_2-V_2^\text{T}V_1\|_F+\|V_1\|_F\|V_2^\text{T}\|_F\|V_1-V_2\|_F+\|V_2\|_F\|V_1^\text{T}-V_2^\text{T}\|_F\|V_1\|_F\\
    \leq&\|V_1-V_2\|_F(\|V_1^\text{T}V_1\|_F+\|V_2^\text{T}V_2\|_F+\|V_2^\text{T}\|_F\|V_1\|_F+\|V_1\|_F\|V_2^\text{T}\|_F+\|V_2\|_F\|V_1\|_F)\\
    \leq&\| V_1-V_2\|_F(\| V_1\|_F^2+\| V_2\|_F^2+3\| V_1\|_F\| V_2\|_F).
    \end{aligned}
\end{equation}
Since $\| V_1\|_F,~\| V_2\|_F\leq \sqrt m$, we have
\begin{align*}
  \| \mathcal{N}[V_1]-\mathcal{N}[V_2]  \|_F
  \leq(\kappa+1+5m)\| V_1-V_2  \|_F.
\end{align*}
\end{proof}

Different from the symmetric Banach space discussed in \cite{du2021maximum}, we now define a general (possibly non-symmetric) Banach space $\mathcal{X}=C\left(\overline{\Omega} ; \mathbb{R}^{m \times m}\right)$, the space of continuous $\mathbb{R}^{m \times m}$-valued functions defined on $\overline{\Omega}$ equipped with the supremum norm
$$
\|W\|_{\mathcal X}=\max _{\mathbf x  \in \overline{\Omega}}\|W(\mathbf x )\|_F \quad \forall W \in \mathcal{X}.
$$
Furthermore, we denote 
$$\mathcal{X}_m=\{W(\mathbf x )\in \mathcal X~|~ \|W(\mathbf x )\|_F\leq \sqrt{m}\}.$$


\begin{lemma}\label{lambda-Delta}
For any $\lambda>0$ and $W\in C^2_{\mbox{per}}\left(\overline{\Omega} ; \mathbb{R}^{m \times m}\right)\coloneqq \{W\in C^2\left(\overline{\Omega}; \mathbb{R}^{m \times m}\right), W \mbox{is periodic}\}$, it holds that
\begin{equation}
    \begin{aligned}\label{lambda}
        \|(\lambda\mathcal{I}-\varepsilon^2\Delta)W(\mathbf x )\|_{\mathcal{X}}\geq \lambda \|W(\mathbf x )\|_{\mathcal{X}}
    \end{aligned}
\end{equation}
and the solution operator of Laplace equation with periodic boundary condition, i.e., $\{e^{t\varepsilon^2}\Delta\}_{t\geq0}$ on $\mathcal{X}$, satisfies the contraction property that $\forall t\geq 0$, 
\begin{equation}\label{eq:2.11}
  \|e^{t\varepsilon^2 \Delta} W(\mathbf{x})\|_{\mathcal X}\leq\|W(\mathbf{x})\|_{\mathcal{X}}.
\end{equation}
\end{lemma}
\begin{proof}
For any $W(\mathbf x )\in \mathcal{X}$, there exists $\mathbf x _0\in \overline{\Omega}$ with the periodic boundary condition such that
\begin{equation}\label{eq:2.12}
    \|W(\mathbf x _0)\|_F=\sqrt{\sum_{i,j=1}^{m}|W_{ij}(\mathbf x _0)|^2} =\|W(\mathbf x )\|_{\mathcal{X}}=\max\limits_{\mathbf x \in\overline{\Omega}}\sqrt{\sum_{i,j=1}^{m}|W_{ij}(\mathbf x )|^2}.
\end{equation}
Since $\mathbf x_0$ is the maximum point of $\|W(\mathbf x)\|_F^2$, it holds that $\Delta \|W(\mathbf x _0)\|_F^2\leq 0$ and consequently, 
\begin{align*}
      2\sum_{i,j=1}^{m}W_{ij}(\mathbf x _0)\Delta W_{ij}(\mathbf x _0)&\leq2\sum_{i,j=1}^{m}W_{ij}(\mathbf x _0)\Delta W_{ij}(\mathbf x _0)+2\sum_{i,j=1}^{m}|\nabla W_{ij}(\mathbf x _0)|^2\\
      &=\Delta \sum_{i,j=1}^{m}|W_{ij}(\mathbf x _0)|^2\leq 0.
\end{align*}
Then, for any $\lambda>0$, we have
\begin{equation}
\begin{aligned}
\lambda^2\|W(\mathbf x _0)\|_F^2
=&\sum_{i,j=1}^m\lambda^2|W_{ij}(\mathbf x _0)|^2
\leq\sum_{i,j=1}^m\lambda^2|W_{ij}(\mathbf x _0)|^2-\lambda\varepsilon^2 W_{ij}(\mathbf x _0)\Delta W_{ij}(\mathbf x _0)\\
=&\sum_{i,j=1}^m\lambda W_{ij}(\mathbf x _0)\left((\lambda\mathcal{I}-\varepsilon^2\Delta)W_{ij}(\mathbf x _0)\right)\\
\leq&\frac{1}{2}\sum_{i,j=1}^{m}\lambda^2|W_{ij}(\mathbf x _0)|^2+\frac{1}{2}\sum_{i,j=1}^{m}|(\lambda\mathcal{I}-\varepsilon^2\Delta)W_{ij}(\mathbf x _0)|^2\\
=&\frac{1}{2}\lambda^2\|W(\mathbf x _0)\|_F^2+\frac{1}{2}\|(\lambda\mathcal{I}-\varepsilon^2\Delta)W(\mathbf x _0)\|_F^2
\end{aligned}
\end{equation}
which leads to 
\begin{equation}
\lambda\|W(\mathbf x )\|_{\mathcal{X}}=\lambda\|W(\mathbf x _0)\|_F\leq \|(\lambda\mathcal{I}-\varepsilon^2\Delta)W(\mathbf x _0)\|_F\leq\|(\lambda\mathcal{I}-\varepsilon^2\Delta)W(\mathbf x )\|_{\mathcal{X}},
\end{equation}
for any $W(\mathbf x )\in \mathcal{X}$.

As the proof of \cite[Lemma 2.1, Property (ii)]{du2021maximum}, one can use the Lumer-Phillips theorem to deduce the contraction property of $e^{t\varepsilon^2\Delta}$ from \eqref{lambda}.
\end{proof}

\begin{remark}
    Here we provide an alternative proof of the contraction property for the operate $e^{t\varepsilon^2\Delta}$. Since $C^2\left(\overline{\Omega} ; \mathbb{R}^{m \times m}\right)$ is dense in $\mathcal{X}$, we only consider the case where $W\in C^2\left(\overline{\Omega}; \mathbb{R}^{m \times m}\right)$. For the simplicity, we consider the following equation
\begin{equation}
   \begin{aligned}\label{eq:02.15}
	 \begin{cases}
	\partial_t W(t,\mathbf x )=\Delta W(t,\mathbf x)\quad t\in(0,T],~\mathbf x \in \Omega,\\
 W(0,\mathbf x)=W^0(\mathbf x),
      \end{cases}
   \end{aligned}
\end{equation}
with the periodic boundary condition. We can obtain the solution of \eqref{eq:02.15} as
\begin{equation}
\begin{aligned}
	W(t,\mathbf x )=e^{t\Delta}W^0(\mathbf x ),\quad t\in(0,T], ~\mathbf x \in \Omega.
	\end{aligned}
\end{equation}
Multiplying both sides of the first equation in \eqref{eq:02.15} by $W^\text{T}(t,\mathbf x)$ and taking the trace, we have
   \begin{align*}
	\text{tr}(W^\text{T}(t,\mathbf x)\partial_t W(t,\mathbf x ))=\text{tr}(W^\text{T}(t,\mathbf x)\Delta W(t,\mathbf x)),
   \end{align*}
which yields
\begin{equation}
   \begin{aligned}\label{eq:02.16}
&\frac 12\partial_t\Big(\sum_{i,j=1}^m W_{ij}^2(t,\mathbf x)\Big)=\sum_{i,j=1}^m W_{ij}(t,\mathbf x )\Delta W_{ij}(t,\mathbf x)\\
&=-\sum_{i,j=1}^m |\nabla W_{ij}(t,\mathbf x )|^2+\frac 12 \Delta\Big(\sum_{i,j=1}^m W_{ij}^2(t,\mathbf x )\Big)\leq\frac 12 \Delta\Big(\sum_{i,j=1}^m W_{ij}^2(t,\mathbf x )\Big).
\end{aligned}
\end{equation}
By the maximum principle of heat equation, we can obtain
\begin{equation}
\begin{aligned}
	\|W(t,\mathbf x)\|_\mathcal{X}=\|e^{t\Delta}W^0(\mathbf x )\|_\mathcal{X}\leq \|W^0(\mathbf x)\|_\mathcal{X},\quad t\in(0,T],~\mathbf x \in \Omega.
	\end{aligned}
\end{equation}
\end{remark}

	
\begin{theorem}[Maximum bound principle, existence, and uniqueness of solution to \eqref{eq:mac_ka}]\label{unique u} Suppose that the matrix-valued Allen--Cahn equation \eqref{eq:mac_ka} is equipped with the periodic boundary condition and the initial value $U_0\in \mathcal X_m$.
If $\kappa \geq \max\left\{\frac32m-1,2\right\}$, then the equation (\ref{eq:mac_ka}) has a unique solution $U(t,\mathbf x)\in C([0,T]; \mathcal{X}_m)$, implying $\|U(t)\|_{F}\leq\sqrt{m}$ for any $t\in[0,T]$.
\end{theorem}	
\begin{proof} For a fixed $t_1>0$ and a given $V\in C([0,t_1];\mathcal{X}_m)$, we define $W:=W(t,\mathbf x )$ to be a solution of the following linear problem:
\begin{equation}\label{eq:2.10}
   \begin{aligned}
	 \begin{cases}
	W_t+\kappa W=\mathcal{L} W+\mathcal{N}[V],\quad t\in(0,t_1],~\mathbf x \in \Omega,\\
	W(\mathbf{x},0)=U_0(\mathbf{x}),
      \end{cases}
   \end{aligned}
\end{equation}
with the periodic boundary condition. Setting $\mathcal{N}[V]=0$ and $U_0=0$ in \eqref{eq:2.10} will led to $W(t)=e^{t(\mathcal{L}-\kappa)}U_0=0$ for each $t\in[0,t_1]$. Thus, $W$ is the unique solution defined on $[0,t_1]\times\Omega$. By Duhamel principle, 
\begin{equation}\label{eq:solution}
      \begin{aligned}
	W(t,\mathbf x )=e^{t(\mathcal{L}-\kappa)}U_0(\mathbf x )+\int_{0}^{t}e^{(t-s) 
        (\mathcal{L}-\kappa)}\mathcal{N}[V(s,\mathbf x )]ds,\quad \mathbf x \in \Omega, t\in(0,t_1].
	\end{aligned}
\end{equation}
Taking  the supremum norm $\|\cdot\|_{\mathcal{X}}$ on both sides of  \eqref{eq:solution} and using Lemmas \ref{Nkappa} and \ref{lambda-Delta}, we obtain
\begin{equation}\label{ETDRK1}
      \begin{aligned}
            \|W(t,\mathbf x )\|_{\mathcal{X}}&\leq e^{-\kappa t}\|U_0(\mathbf x )\|_{\mathcal{X}}+\int_{0}^{t}e^{-\kappa(t-s)}\|\mathcal{N}[V(s,\mathbf x )]\|_{\mathcal{X}}ds\\
		&\leq e^{-\kappa t}\sqrt{m}+\int_{0}^{t}e^{-\kappa(t-s)}\kappa \sqrt{m}ds=\sqrt{m},\quad t\in(0,t_1].
       \end{aligned}
\end{equation}

Next, we define a mapping $\mathcal{A}: C([0,t_1];\mathcal{X}_m)\to C([0,t_1];\mathcal{X}_m)$, $V\mapsto W$ according to \eqref{eq:solution}. For any $V_1,V_2\in C([0,t_1];\mathcal{X}_m)$, we set $W_1=\mathcal{A}V_1$ and $W_2=\mathcal{A}V_2$. From \eqref{eq:solution} and Lemma \ref{NNkappa}, we have

\begin{equation}
      \begin{aligned}
	\|W_1(t,\mathbf x )-W_2(t,\mathbf x )\|_{\mathcal{X}}
 \leq&\int_{0}^{t}e^{-\kappa(t-s)}\|\mathcal{N}[V_1(s,\mathbf x )]-\mathcal{N}[V_2(s,\mathbf x )]\|_{\mathcal{X}}ds\\
	\leq& (\kappa+1+5m)\int_{0}^{t}e^{-\kappa(t-s)}\| V_1(s,\mathbf x )-V_2(s,\mathbf x)\|_{\mathcal{X}}ds\\
	\leq& (\kappa+1+5m)\left(\frac{1}{\kappa}-\frac{1}{\kappa}e^{-\kappa t_1}\right)\| V_1-V_2\|_{\mathcal{X}},\quad \forall t\in[0,t_1].
      \end{aligned}
\end{equation}
Thus, for $t_1<\kappa^{-1}\ln\left(1+\frac{\kappa}{1+5m}\right)$ such that $(\kappa+1+5m)\left(\frac{1}{\kappa}-\frac{1}{\kappa}e^{-\kappa t_1}\right)<1$, $\mathcal{A}$ is a contraction. Since $\mathcal{X}_m$ is closed in $\mathcal{X}$, we know that $C([0,t_1];\mathcal{X}_m)$ is complete with respect to the metric induced by the norm $\|\cdot\|_\mathcal{X}$. Then, by Banach's fixed point theorem we get that $\mathcal{A}$ has a unique fixed point in $C([0,t_1];\mathcal{X}_m)$, which is the unique solution to the matrix-valued Allen--Cahn equation (\ref{eq:mac_ka}) on $[0,t_1]$. We can extend this solution to the entire time domain $[0,T]$.
\end{proof}

\section{Exponential time differencing Runge--Kutta schemes}\label{section 3}
In this section, we use the ETD schemes to construct a temporal approximation of the matrix-valued Allen--Cahn equation \eqref{eq:mac}. We discuss the first- and second-order ETD schemes with the unconditional MBP-preserving property based on the equivalence equation \eqref{eq:mac_ka}. Then, we show that the ETD schemes obtained based on \eqref{eq:mac} has the unconditional energy dissipation property.

\subsection{ETD schemes}

Divide the total time by $\{t_n=n\tau\}_{n\geq 0}$ where $\tau>0$ is a time step size. To establish the ETD schemes for solving \eqref{eq:mac} on the time interval $[t_n,t_{n+1}]$, we let the exact solution $W(s,\mathbf x )=U(t_n+s,\mathbf x )$ satisfying the system
\begin{equation}\label{eq:NWn}
  \begin{aligned}
    \begin{cases}
    W_s+\kappa W=\mathcal{L}W+\mathcal{N}[W],\quad s\in(0,\tau],~\mathbf x \in\Omega,\\
    W(0,\mathbf x )=U(t_n,\mathbf x ).
    \end{cases}
  \end{aligned}
\end{equation}

The first-order ETD (ETD1) scheme is obtained by setting $\mathcal{N}[U(t_n+s)]\approx \mathcal{N}[U^n]$ in \eqref{eq:NWn} (here, it is assumed implicitly that the solution is sufficiently smooth in time). For $n\geq 0$ and given $U^n$, the ETD1 scheme reads
\begin{equation}\label{eq:ETDRK1}
   \begin{aligned}
     U^{n+1}=W^n(\tau)=e^{(\mathcal{L}-\kappa)\tau}U^n+\int_{0}^{\tau}e^{(\mathcal{L}-\kappa)(\tau-s)}\mathcal{N}[U^n]ds,
   \end{aligned}
\end{equation}
that is the solution of the following linear equation 
\begin{equation}\label{eq:3.2}
  \begin{aligned}
    \begin{cases}
    W_s^{n}+\kappa W^{n}=\mathcal{L}W^{n}+\mathcal{N}[U^n],\quad s\in(0,\tau],~\mathbf x \in \Omega,\\
    W^{n}(0,\mathbf x )=U^n(\mathbf x ),
    \end{cases}
  \end{aligned}
\end{equation}
with $U^0=U_0$. Furthermore, the equation \eqref{eq:ETDRK1} can be written as
\begin{equation}\label{eq:ETDRK1.2}
      \begin{aligned}
	U^{n+1}=e^{-\tau L}U^n+(\mathcal{I}-e^{-\tau L})L^{-1}\mathcal{N}[U^n],
      \end{aligned}
\end{equation}
where $L\coloneqq \kappa\mathcal{I}-\mathcal{L}=\kappa\mathcal{I}-\varepsilon^2\Delta$.

The second-order ETD Runge--Kutta (ETDRK2) scheme is obtained by setting $\mathcal{N}[U(t_n+s)]\approx \left(1-\frac{s}{\tau}\right)\mathcal{N}[U^n]+\frac{s}{\tau}\mathcal{N}[\widetilde U^{n+1}]$ in \eqref{eq:NWn}, where $\widetilde U^{n+1}$ is computed from the ETD1 scheme \eqref{eq:ETDRK1} (it is assumed implicitly that the solution is sufficiently smooth in time). For any $n\geq0$, the ETDRK2 scheme reads 

\begin{equation}\label{eq:ETDRK2}
  \begin{aligned}
  U^{n+1}=W^n(\tau)=e^{(\mathcal{L}-\kappa)\tau}U^n+\int_{0}^{\tau}e^{(\mathcal{L}-\kappa)(\tau-s)}\left(\left(1-\frac{s}{\tau}\right)\mathcal{N}[U^n]+\frac{s}{\tau}\mathcal{N}[\widetilde U^{n+1}]\right)ds,
  \end{aligned}
\end{equation}
that is the solution of the following linear equation 
\begin{equation}\label{eq:3.4}
  \begin{aligned}
    \begin{cases}
    W_s^{n}+\kappa W^{n}=\mathcal{L}W^{n}+\left(1-\frac{s}{\tau}\right)\mathcal{N}[U^n]+\frac{s}{\tau}\mathcal{N}[\widetilde U^{n+1}],\quad s\in(0,\tau],~\mathbf x \in \Omega,\\
    W^{n}(0,\mathbf x )=U^n(\mathbf x ),
    \end{cases}
  \end{aligned}
\end{equation}
with $U^0=U_0$. Then, the equation \eqref{eq:ETDRK2} can be written as
\begin{equation}\label{eq:ETDRK2.2}
  \begin{aligned}
  \widetilde{U}^{n+1}&=e^{-\tau L}U^n+(\mathcal{I}-e^{-\tau L})L^{-1}\mathcal{N}[U^n]\\
  U^{n+1}&=\widetilde U^{n+1}-\frac{1}{\tau}\left(e^{-\tau L}-\mathcal{I}+\tau L\right)L^{-2}\left(\mathcal{N}[U^n]-\mathcal{N}[\widetilde U^{n+1}]\right).
  \end{aligned}
\end{equation}

\subsection{Discrete maximum bound principles}
\begin{theorem}[MBP of ETD1 scheme]\label{MBP1} 
If $\|U^0\|_F\leq\sqrt{m}$ for any $\mathbf x\in \overline \Omega$ and $\kappa\geq \max\left\{\frac{3}{2}m-1,2\right\}$, the ETD1 scheme \eqref{eq:ETDRK1} preserves the discrete maximum bound principle unconditionally, i.e., for any time step size $\tau>0$ and $n\geq 1$, the ETD1 solution $U^n$ satisfies $\| U^n \|_F\leq \sqrt{m}$.
\end{theorem}
\begin{proof} Suppose that $\|U^n\|_F\leq\sqrt{m}$, we only need to prove $\| U^{n+1} \|_F\leq \sqrt{m}$. Taking the supremum norm $\|\cdot\|_{\mathcal{X}}$ on both sides of \eqref{eq:ETDRK1} and using Lemmas \ref{Nkappa} and \ref{lambda-Delta}, we obtain
\begin{equation}
  \begin{aligned}
  \| U^{n+1}\|_{\mathcal{X}}=&\| e^{(\mathcal{L}-\kappa)\tau}U^n+\int_{0}^{\tau}e^{(\mathcal{L}-\kappa)(\tau-s)}\mathcal{N}[U^n]ds\|_{\mathcal{X}}\\
  \leq&e^{-\kappa\tau}\sqrt{m}+\left(\frac{1}{\kappa}-\frac{1}{\kappa}e^{-\kappa\tau}\right)\kappa\sqrt{m}=\sqrt{m},
  \end{aligned}
\end{equation}
which verifies the maximum bound principle of ETD1 scheme \eqref{eq:ETDRK1}.
\end{proof}
	
\begin{theorem} [MBP of ETDRK2 scheme]\label{MBP2}
 If $\|U^0\|_F\leq\sqrt{m}$  for any $\mathbf x\in \overline \Omega$ and $\kappa\geq \max\left\{\frac{3}{2}m-1,2\right\}$, the ETDRK2 scheme \eqref{eq:ETDRK2} preserves the discrete maximum bound principle unconditionally, i.e., for any time step size $\tau>0$ and $n\geq 1$, the ETDRK2 solution $U^n$ satisfies $\| U^n \|_F\leq \sqrt{m}$.
\end{theorem}
\begin{proof} Suppose that $\|U^n\|_F\leq\sqrt{m}$. From Theorem \ref{MBP1}, we have $\|\widetilde{U}^{n+1}\|_F\leq \sqrt{m}$. Similar to the proof of Theorem \ref{MBP1}, we take the supremum norm $\|\cdot\|_{\mathcal{X}}$ on both sides of \eqref{eq:ETDRK2} and use Lemmas \ref{Nkappa} and \ref{lambda-Delta} to obtain

\begin{equation}
	\begin{aligned}
	    \| U^{n+1}\|_{\mathcal{X}}&= \Big\| e^{(\mathcal{L}-\kappa)\tau}U^n+\int_{0}^{\tau}e^{(\mathcal{L}-\kappa)(\tau-s)}\left((1-\frac{s}{\tau})\mathcal{N}[U^n]+\frac{s}{\tau}\mathcal{N}[\tilde{U}^{n+1}]\right)ds\Big\|_{\mathcal{X}}\\
		&\leq e^{-\kappa\tau}\sqrt{m}+\int_{0}^{\tau}e^{-\kappa(\tau-s)}\left(\left(1-\frac{s}{\tau}\right)\kappa \sqrt{m}+\frac{s}{\tau}\kappa \sqrt{m}\right) ds\\
		&\leq e^{-\kappa\tau}\sqrt{m}+\left(\frac{1}{\kappa}-\frac{1}{\kappa}e^{-\kappa\tau}\right)\kappa\sqrt{m}=\sqrt{m},
	\end{aligned}
\end{equation}
which completes the proof.
\end{proof}

\subsection{Unconditional energy dissipation}

In the following content, we will prove the unconditional energy dissipation of ETD1 \eqref{eq:ETDRK1.2} and ETDRK2 \eqref{eq:ETDRK2.2}. 
Our proof is inspired by the proof for scalar case (i.e., $m=1$) in \cite{fu2022energy}. The key difference is that we shall prove an inequality (that is trivial in the scalar case) as stated in Lemma \ref{Fkappa}. 
\begin{lemma}\label{Fkappa}
For any $U_1,U_2\in \mathbb{R}^{m\times m}$ with $\| U_1\|_F,~\| U_2\|_F\leq \sqrt m$, if $\kappa\geq 3m-1$, there holds
\begin{align*}
\langle F(U_1)-F(U_2), I \rangle_F+\left\langle U_1-U_2,f(U_2)\right \rangle_F\leq\frac{\kappa}{2}\| U_1-U_2\|_F^2,
\end{align*}
where $F(U)=\frac{1}{4}\left(U^\text{T}U-I\right)^2$, $f(U)=U-UU^\text{T}U$, and $\langle\cdot,\cdot\rangle_F$ represents the Frobenius inner product.
\end{lemma}
\begin{proof} For any $U_1,U_2\in \mathbb{R}^{m\times m}$, we have
\begin{equation}
   \begin{aligned}\label{eq:3.3.1}
   \langle F(U_1)-F(U_2), I \rangle_F
   =&\frac{1}{4}\left\langle (U_1U_1^{\text{T}}-I)^2-(U_2U_2^{\text{T}}-I)^2,I\right\rangle_{F}\\
   =&\frac{1}{4}\left\langle U_1U_1^{\text{T}}U_1U_1^{\text{T}}-2U_1U_1^{\text{T}}-U_2U_2^{\text{T}}U_2U_2^{\text{T}}+2U_2U_2^{\text{T}},I\right\rangle_{F},
   \end{aligned}
\end{equation}
and 
\begin{equation}
   \begin{aligned}\label{eq:3.3.2}
   \left\langle U_1-U_2,f(U_2)\right\rangle_F=&\left\langle U_1-U_2,U_2-U_2U_2^{\text{T}}U_2\right\rangle_F\\
   =&\left\langle U_1U_2^{\text{T}}-U_2U_2^{\text{T}}-U_1U_2^{\text{T}}U_2U_2^{\text{T}}+U_2U_2^{\text{T}}U_2U_2^{\text{T}},I\right\rangle_{F}.
   \end{aligned}
\end{equation}
	
Since $\| U_1\|_F,~\| U_2\|_F\leq \sqrt m$ and $\kappa\geq 3m-1$, combing \eqref{eq:3.3.1}-\eqref{eq:3.3.2}, we get
  \begin{equation}
    \begin{aligned}\label{eq:3.3.3}
	&\langle F(U_1)-F(U_2), I \rangle_F+\left\langle U_1-U_2,f(U_2)\right\rangle_F\\
	=&\frac{1}{2}\left\langle -U_1U_1^{\text{T}}+2U_1U_2^{\text{T}}-U_2U_2^{\text{T}},I\right\rangle_{F}\\
&+\frac{1}{4}\left\langle U_1U_1^{\text{T}}U_1U_1^{\text{T}}-4U_1U_2^{\text{T}}U_2U_2^{\text{T}}+3U_2U_2^{\text{T}}U_2U_2^{\text{T}},I\right\rangle_{F}\\
 \leq&\frac{1}{2}(3m-1)\| U_1-U_2\|_F^2\leq \frac{\kappa}{2}\| U_1-U_2\|_F^2,
    \end{aligned}
\end{equation}
where we use
\begin{equation}
  \begin{aligned}\label{eq:3.3.4}
    \left\langle -U_1U_1^{\text{T}}+2U_1U_2^{\text{T}}-U_2U_2^{\text{T}},I\right\rangle_{F}=&\left\langle -U_1U_1^{\text{T}}+U_1U_2^{\text{T}}+U_2U_1^{\text{T}}-U_2U_2^{\text{T}},I\right\rangle_{F}\\
    =&-\| U_1-U_2\|_F^2
  \end{aligned}
\end{equation}
and
\begin{equation}
     \begin{aligned}\label{eq:3.3.5}	
         &\left\langle U_1U_1^{\text{T}}U_1U_1^{\text{T}}-4U_1U_2^{\text{T}}U_2U_2^{\text{T}}+3U_2U_2^{\text{T}}U_2U_2^{\text{T}},I\right\rangle_{F}\\
	=&\left\langle (U_1U_1^{\text{T}}U_1-U_2U_2^{\text{T}}U_2)U_1^{\text{T}}-3U_2U_2^{\text{T}}U_2(U_1^{\text{T}}-U_2^{\text{T}}),I\right\rangle_{F}\\
	=&\left\langle (U_1-U_2)(U_1^\text{T}U_1+U_2^\text{T}U_2- 
         U_2^\text{T}U_1)U_1^{\text{T}}+U_1U_2^\text{T}(U_1-U_2)U_1^{\text{T}}\right.\\
         &\left.+U_2(U_1^\text{T}-          U_2^\text{T})U_1U_1^{\text{T}}-3U_2U_2^{\text{T}}U_2(U_1^{\text{T}}-U_2^{\text{T}}),I\right\rangle_{F}\\
	=&\left\langle (U_1-U_2)        (U_1^{\text{T}}U_1U_1^{\text{T}}+U_2^{\text{T}}U_2U_1^{\text{T}}-U_2^{\text{T}}U_1U_1^{\text{T}}+U_1^{\text{T}}U_1U_2^{\text{T}}\right.\\
&\left.+U_2^{\text{T}}U_1U_1^{\text{T}}-3U_2^{\text{T}}U_2U_2^{\text{T}}),I\right\rangle_F\\
	=&\left\langle (U_1-U_2)((U_1^{\text{T}}-U_2^{\text{T}})U_1U_1^{\text{T}}+U_2^{\text{T}}U_2(U_1^{\text{T}}-U_2^{\text{T}})+(U_1^{\text{T}}-U_2^{\text{T}})U_1U_2^{\text{T}}\right.\\
	&\left.+U_2^{\text{T}}(U_1-U_2)U_2^{\text{T}}+U_2^{\text{T}}(U_1-U_2)U_1^{\text{T}}+U_2^{\text{T}}U_2(U_1^{\text{T}}- 
         U_2^{\text{T}})),I\right\rangle_F\\
         \leq&\| U_1-U_2\|_F^2\left(\|U_1U_1^{\text{T}}\|_F+\|U_2^{\text{T}}U_2\|_F+\|U_1U_2^{\text{T}}\|_F+\|U_2U_2\|_F\right.\\	&\left.+\|U_2U_1\|_F+\|U_2^{\text{T}}U_2\|_F\right)\\
	\leq&6m\| U_1-U_2\|_F^2.\\
     \end{aligned}
\end{equation}\end{proof}

Based on Lemma \ref{Fkappa}, we are now ready to prove the original energy dissipation law for both ETD1 and ETDRK2 schemes that is inspired by the proof in \cite{fu2022energy}.
	
\begin{theorem}[Original energy dissipation law]\label{Energy}
    For the matrix-valued Allen--Cahn equation \eqref{eq:mac_ka}, if $\kappa\geq 3m-1$, both the ETD1 \eqref{eq:ETDRK1.2} and ETDRK2 \eqref{eq:ETDRK2.2} schemes  satisfy the energy dissipation law, i.e.,
   \begin{align}
    E(U^{n+1})\leq E(U^n),~~\forall n\geq0,
   \end{align}
where $E(U)$ is the original energy defined by \eqref{energy}.
\end{theorem}
\begin{proof} This proof is similar to the scalar case given in \cite{fu2022energy}. For the sake of completeness, we still show the details here.  

Let $V= U^{n+1}$ where $U^{n+1}$ is obtained from the ETD1 scheme \eqref{eq:ETDRK1.2}. 
For the first part of the energy, we have
\begin{equation}
	\begin{aligned}\label{eq:4.1.2}
	&\frac{\varepsilon^2}{2}\int_{\Omega}\left(\| \nabla V\|_F^2-\| 
        \nabla U^n\|_F^2\right)d\mathbf x 
	=\frac{\varepsilon^2}{2}\int_{\Omega}\left(-\left\langle 
        V,\Delta V\right\rangle_{F}+\left\langle U^n,\Delta U^n\right\rangle_F\right)d\mathbf x \\
	=&\int_{\Omega}\left(-\varepsilon^2\left\langle V-U^n,\Delta 
        V\right\rangle_{F}-\frac{\varepsilon^2}{2}\| \nabla(V-U^{n})\|_{F}^2\right)d\mathbf x \\
	=&\int_{\Omega}\left\langle V- 
        U^n,LV\right\rangle_{F}d\mathbf x -\int_{\Omega}\left\langle V-U^n,\kappa V\right\rangle_{F}d\mathbf x -\frac{\varepsilon^2}{2}\int_{\Omega}\| \nabla(V-U^{n})\|_{F}^2d\mathbf x ,
	\end{aligned}
 \end{equation}
 where $L=\kappa\mathcal{I}-\varepsilon^2\Delta$. 
 For the second part of the energy, according to Lemma \ref{Fkappa}, we have
 
\begin{equation}
     \begin{aligned}\label{eq:4.1.1}
	&\int_{\Omega}\langle F(V)-F(U^n), I \rangle_Fd\mathbf x 
	\leq\int_{\Omega}\left\langle V-U^{n},- 
            f(U^n)\right\rangle_Fd\mathbf x +\frac{\kappa}{2}\int_{\Omega}\| V-U^{n}\|_{F}^2d\mathbf x \\
        =&-\int_{\Omega}\left\langle V-U^{n},\mathcal{N}[U^n]\right\rangle_Fd\mathbf x +\int_{\Omega}\left\langle V-U^{n},\kappa U^n\right\rangle_Fd\mathbf x +\frac{\kappa}{2}\int_{\Omega}\| V-U^{n}\|_{F}^2d\mathbf x ,
    \end{aligned}
\end{equation}
where $\mathcal{N}=\kappa\mathcal{I}+f$. Combining \eqref{eq:4.1.2}-\eqref{eq:4.1.1} and according to \eqref{eq:ETDRK1.2}, we obtain
 
\begin{equation}
	\begin{aligned}\label{eq:4.1.3}
	&E(V)-E(U^n)\\
	\leq&\int_{\Omega}\left\langle V-U^{n},LV-\mathcal{N} 
             [U^n]\right\rangle_Fd\mathbf x -\frac{\kappa}{2}\int_{\Omega}\| V-U^{n}\|_{F}^2d\mathbf x -\frac{\varepsilon^2}{2}\int_{\Omega}\| \nabla(V-U^{n})\|_{F}^2d\mathbf x \\
             \leq&\int_{\Omega}\left\langle V-U^{n},LV-\mathcal{N}[U^n]\right\rangle_Fd\mathbf x \\
        =&\int_{\Omega}\left\langle V-U^{n},LV-L(\mathcal{I}-e^{-\tau L})^{-1}(V-e^{-\tau 
           L}U^n)\right\rangle_Fd\mathbf x \\
	=&\int_{\Omega}\left\langle V-U^{n},\left[L-L(\mathcal{I}-e^{-\tau L})^{-1}\right] 
          (V-U^n)\right\rangle_Fd\mathbf x \\
	=&\int_{\Omega}\left\langle V-U^{n},\Delta_1(V-U^n)\right\rangle_Fd\mathbf x ,
	\end{aligned}
\end{equation}
 where $\mathcal{N}[U^n]=L(\mathcal{I}-e^{-\tau L})^{-1}(V-e^{-\tau L}U^n)$ and $\Delta_1=L-L(\mathcal{I}-e^{-\tau L})^{-1}$. Since the function $y_1(x)=x-\frac{x}{1-e^x}\geq 0$  for all $x\in\mathbb R$ (here we make a convention that $y_1(0)=1$) and $\Delta_1=-\frac{1}{\tau} y_1(-\tau L)$, we obtain that the operator $\Delta_1$ is non-positive definite which indicates that the energy is decreasing from $U^n$ to $V$. Thus, we have proved that the ETD1 \eqref{eq:ETDRK1.2} scheme satisfies the energy dissipation law.

 Consider the ETDRK2 scheme \eqref{eq:ETDRK2.2}. 
 Let $V = \widetilde U^{n+1}$ that is obtained from the ETD1 scheme. 
 The same as \eqref{eq:4.1.3}, we can obtain 
 \begin{equation}
	\begin{aligned}\label{eq:3.19}
	E(V)-E(U^n)
	\leq \int_{\Omega}\left\langle V-U^{n},\Delta_1(V-U^n)\right\rangle_Fd\mathbf x .
	\end{aligned}
\end{equation}
Now $U^{n+1}$ is obtained from the ETDRK2 scheme \eqref{eq:ETDRK2.2}. 
Similar to \eqref{eq:4.1.3}, we can compute the energy difference between $U^{n+1}$ and $V$:

\begin{equation}
\begin{aligned}\label{eq:4.1.5}
&E(U^{n+1})-E(V)\leq\int_{\Omega}\left\langle U^{n+1}-V,LU^{n+1}-\mathcal{N}[V]\right\rangle_Fd\mathbf x \\
=&\int_{\Omega}\left\langle U^{n+1}-V,LU^{n+1}-\mathcal{N}[U^n]-\tau(e^{-\tau L}-\mathcal{I}+\tau L)^{-1}L^2(U^{n+1}-V)\right\rangle_Fd\mathbf x \\
=&\int_{\Omega}\left\langle U^{n+1}-V,LV-\mathcal{N}[U^n]+LU^{n+1}-LV\right.\\
&\left.~~~~~-\tau(e^{-\tau L}-\mathcal{I}+\tau L)^{-1}L^2(U^{n+1}-V)\right\rangle_Fd\mathbf x \\
=&\int_{\Omega}\left\langle U^{n+1}-V,\Delta_1(V-U^n)+\left(L-\tau(e^{-\tau L}-\mathcal{I}+\tau L)^{-1}L^2\right)(U^{n+1}-V)\right\rangle_Fd\mathbf x \\
=&\int_{\Omega}\left\langle U^{n+1}-V,\Delta_1(V-U^n)\right\rangle_Fd\mathbf x +\int_{\Omega}\left\langle U^{n+1}-V,\Delta_2(U^{n+1}-V)\right\rangle_Fd\mathbf x ,
\end{aligned}
\end{equation}
where $\Delta_2=L-\tau\left(e^{-\tau L}-\mathcal{I}+\tau L\right)^{-1}L^2$. Combing \eqref{eq:3.19} with \eqref{eq:4.1.5}, we can obtain that

\begin{equation}
    \begin{aligned}\label{eq:4.1.6}
     & E(U^{n+1})-E(U^n)\\
      \leq&\int_{\Omega}\left\langle V-U^{n},\Delta_1(V-U^n)\right\rangle_Fd\mathbf x +\int_{\Omega}\left\langle U^{n+1}-V,\Delta_1(V-U^n)\right\rangle_Fd\mathbf x \\
      &+\int_{\Omega}\left\langle U^{n+1}-V,\Delta_2(U^{n+1}-V)\right\rangle_Fd\mathbf x \\
      =&\frac{1}{2}\int_{\Omega}\left\langle V-U^{n},\Delta_1(V-U^n)\right\rangle_Fd\mathbf x +\int_{\Omega}\left\langle U^{n+1}-V,(\Delta_2-\frac{1}{2}\Delta_1)(U^{n+1}-V)\right\rangle_Fd\mathbf x \\
      &+\frac{1}{2}\int_{\Omega}\left[\left\langle V-U^{n},\Delta_1(V-U^n)\right\rangle_F+2\left\langle U^{n+1}-V,\Delta_1(V-U^n)\right\rangle_F\right.\\
      &\left.+\left\langle U^{n+1}-V,\Delta_1(U^{n+1}-V)\right\rangle_F\right]d\mathbf x \\
      =&\frac{1}{2}\int_{\Omega}\left\langle V-U^{n},\Delta_1(V-U^n)\right\rangle_Fd\mathbf x +\int_{\Omega}\left\langle U^{n+1}-V,(\Delta_2-\frac{1}{2}\Delta_1)(U^{n+1}-V)\right\rangle_Fd\mathbf x \\
      &+\frac{1}{2}\int_{\Omega}\left\langle U^{n+1}-U^{n},\Delta_1(U^{n+1}-U^n)\right\rangle_Fd\mathbf x .
    \end{aligned}
\end{equation}
 Since $\Delta_1$ is non-positive definite, both the first and third terms of \eqref{eq:4.1.6} are non-positive. For the second term of \eqref{eq:4.1.6}, it is already proved in \cite{fu2022energy}  that 
 \begin{equation}
     y_2(x)=x+\frac{x^2}{e^x-1-x}-\frac{1}{2}y_1(x)=\frac{x(e^x(e^x+x-3)+2)}{2(e^x-1)(e^x-1-x)}\geq 0,\quad \forall x\in \mathbb R,
 \end{equation} 
 where we make a convention that $y_2(0)=\frac 32$. Thus, the operator $\Delta_2-\frac{1}{2}\Delta_1=-\frac{1}{\tau}y_2(-\tau L)$ is non-positive definite. The second term of \eqref{eq:4.1.6} is also non-positive. Therefore, we have $E(U^{n+1})-E(U^n)\leq 0$.
\end{proof} 

\subsection{Convergence analysis}
In this section, we give the convergence results for the ETD1 scheme \eqref{eq:ETDRK1} and ETDRK2 scheme \eqref{eq:ETDRK2}. Compared to the convergence analysis for the scalar case (see \cite[Theorem 3.4, 3.5]{du2021maximum}), the norm $\|\cdot\|=\sup _{\mathbf x  \in \hat{\Omega}}|\cdot|$ in \cite{du2021maximum} shall be replaced by $\|\cdot\|_{\mathcal X}=\max _{\mathbf x  \in \overline{\Omega}}\|\cdot\|_F$ for the matrix-valued case.
For the sake of completeness, we show the detailed proofs.

\begin{theorem}[Convergence of ETD1]\label{Conv1}
    Suppose that $\kappa\geq \max\{\frac 32 m-1,2\}$ and $\|U^0\|_{\mathcal{X}}\leq \sqrt{m}$. For the fixed terminal time $T>0$, assume that the exact solution $U$ to \eqref{eq:mac_ka} belongs to $C^1([0,T];\mathcal{X})$ and $\{U^n\}_{n\geq0}$ is generated by the ETD1 scheme \eqref{eq:ETDRK1}. Then, we have
    \begin{equation}\label{eq:conv1}
        \begin{aligned}
            \|U^n-U(t_n)\|_{\mathcal{X}}\leq Cm\left(\frac{\kappa}{1+5m}+1\right)e^{(1+5m)t_n}\tau,
        \end{aligned}
    \end{equation}
for any $\tau>0$ and $t_n\leq T$, where the constant $C>0$ is the $C^1([0,T];\mathcal{X})$-norm of $U$ and is independent of $\tau$.
\end{theorem}
\begin{proof}
 
  We know that the ETD1 solution is given by $U^{n+1}=W^n(\tau)$ with the function $W^n:[0,\tau]\rightarrow\mathcal{X}$ solving \eqref{eq:3.2}. Setting $e_1^n=U^n-U(t_n)$ and subtracting \eqref{eq:mac_ka} from \eqref{eq:3.2} yields
    \begin{equation}\label{eq:e1}
        \begin{aligned}
            e_1^{n+1}=e^{(\mathcal{L}-\kappa)\tau}e_1^n+\int_0^{\tau}e^{(\mathcal{L}-\kappa)(\tau-s)}\left(\mathcal{N}[U^n]-\mathcal{N}[U(t_n)]+R_1(s)\right)ds,
        \end{aligned}
    \end{equation}
    where $R_1(s)=\mathcal{N}[U(t_n)]-\mathcal{N}[U(t_n+s)],~s\in [0,\tau]$.
   
    According to Theorem \ref{unique u}, we have $\|U(t_n)\|_F\leq \sqrt{m}$ and $\|U(t_n+s)\|_F\leq \sqrt{m}$. Then, taking the supremum norm $\|\cdot\|_{\mathcal{X}}$ to $R_1(s)$ and using Lemma \ref{NNkappa}, we derive
    \begin{equation}\label{eq:R1}
        \begin{aligned}
            \|R_1(s)\|_{\mathcal{X}}&=\max _{\mathbf x  \in \overline{\Omega}}\|\mathcal{N}[U(t_n)]-\mathcal{N}[U(t_n+s)]\|_F\\
            &\leq(\kappa+1+5m)\max _{\mathbf x  \in \overline{\Omega}}\|U(t_n)-U(t_n+s)\|_F\\
            &\leq (\kappa+1+5m)\left(\sum_{i,j=1}^{m} \left(\max_{\mathbf x  \in \overline{\Omega}}\left|U_{ij}^{'}(t_n+\xi_{ij})\right|\right)^2\right)^{\frac 12}s\\
            &\leq(\kappa+1+5m)\left(m^2C^2\right)^{\frac 12}s\\
           &\leq Cm(\kappa+1+5m)\tau,
        \end{aligned}
    \end{equation}
where $U=(U_{ij})_{m\times m}$, $\xi_{ij}\in[0,\tau]$, and the constant $C$ is the $C^1([0,T];\mathcal{X})$-norm of $U$. Similarly, according to Theorem \ref{MBP1} and Lemma \ref{NNkappa}, we have $\|U^n\|_F\leq\sqrt{m}$ and 
     \begin{equation}\label{eq:3.26}
        \begin{aligned}
        \|\mathcal{N}[U^n]-\mathcal{N}[U(t_n)]\|_{\mathcal{X}}\leq(\kappa+1+5m)\|U^n-U(t_n)\|_{\mathcal{X}}=(\kappa+1+5m)\|e_1^n\|_{\mathcal{X}}.
        \end{aligned}
    \end{equation}
   Taking the supremum norm $\|\cdot\|_{\mathcal{X}}$ to \eqref{eq:e1} and using \eqref{eq:R1}--\eqref{eq:3.26} and Lemma \ref{lambda-Delta}, we can derive
    \begin{equation}\label{eq:e1_n+1}
        \begin{aligned}
            \|e_1^{n+1}\|_{\mathcal{X}}\leq& e^{-\kappa\tau}\|e_1^n\|_{\mathcal{X}}+\int_0^{\tau}e^{-\kappa(\tau-s)}\|\mathcal{N}[U^n]-\mathcal{N}[U(t_n)]\|_{\mathcal{X}}+\|R_1(s)\|_{\mathcal{X}}ds\\
            \leq&e^{-\kappa\tau}\|e_1^n\|_{\mathcal{X}}+\left((\kappa+1+5m)\|e_1^n\|_{\mathcal{X}}+Cm(\kappa+1+5m)\tau\right)\int_0^{\tau}e^{-\kappa(\tau-s)}ds\\
            =&e^{-\kappa\tau}\|e_1^n\|_{\mathcal{X}}+\left((\kappa+1+5m)\|e_1^n\|_{\mathcal{X}}+Cm(\kappa+1+5m)\tau\right)\frac{1-e^{-\kappa\tau}}{\kappa}\\
            \leq&(1+\tau+5m\tau)\|e_1^n\|_{\mathcal{X}}+Cm(\kappa+1+5m)\tau^2,
        \end{aligned}
    \end{equation}
    where we use $1-e^{-a}\leq a$ for any $a>0$.
    By induction, we get
    \begin{equation}
        \begin{aligned}
            \|e_1^n\|_{\mathcal{X}}&\leq (1+\tau+5m\tau)^n\|e_1^0\|_{\mathcal{X}}+Cm(\kappa+1+5m)\tau^2\sum_{k=0}^{n-1}(1+\tau+5m\tau)^k\\
&\leq(1+\tau+5m\tau)^n\|e_1^0\|_{\mathcal{X}}+Cm(\kappa+1+5m)\tau\frac{(1+\tau+5m\tau)^n-1}{1+5m}\\
            &\leq e^{(1+5m)n\tau}\|e_1^0\|_{\mathcal{X}}+Cm\left(\frac{\kappa}{1+5m}+1\right)e^{(1+5m)n\tau}\tau.
        \end{aligned}
    \end{equation}
    where we use $1+a\leq e^a$ for $a\in \mathbb{R}$.
    Since $e_1^0=0$ and $n\tau=t_n$, we can obtain \eqref{eq:conv1}.
\end{proof}


\begin{theorem}[Convergence of ETDRK2]\label{Conv2}
Suppose that $\kappa\geq \max\{\frac 32 m-1,2\}$ and $\|U^0\|_{\mathcal{X}}\leq \sqrt{m}$. For the fixed terminal time $T>0$, assume that the exact solution $U$ to \eqref{eq:mac_ka} belongs to $C^2([0,T];\mathcal{X})$ and $\{U^n\}_{n\geq0}$ is generated by the ETDRK2 scheme \eqref{eq:ETDRK2}. Then, we have
    \begin{equation}\label{eq:conv2}
        \begin{aligned}
            \|U^n-U(t_n)\|_{\mathcal{X}}\leq \tilde C e^{(1+5m)t_n}\tau^2,
        \end{aligned}
    \end{equation}
for any $\tau>0$ and $t_n\leq T$, where the constant $\tilde C$ depends on $\kappa$, $m$, $t_n$, and the $C^2([0,T];\mathcal{X})$-norm of $U$, but is independent of $\tau$.
\end{theorem}
\begin{proof} 
Setting $e_2^n=U^n-U(t_n)$ and subtracting \eqref{eq:3.4} from \eqref{eq:mac_ka} yields
\begin{equation}\label{eq:e2}
    \begin{aligned}
          e_2^{n+1}=e^{(\mathcal{L}-\kappa)\tau}e_2^n+\int_0^{\tau}&e^{(\mathcal{L}-\kappa)(\tau-s)}\left(\left(1-\frac{s}{\tau}\right)(\mathcal{N}[U^n]-\mathcal{N}[U(t_n)])\right.\\
          &\left.+\frac{s}{\tau}(\mathcal{N}[\widetilde{U}^{n+1}]-\mathcal{N}[U(t_{n+1})])+R_2(s)\right)ds,
    \end{aligned}
\end{equation}
where $R_2(s)=\left(1-\frac{s}{\tau}\right)\mathcal{N}[U(t_n)]+\frac{s}{\tau}\mathcal{N}[U(t_{n+1})]-\mathcal{N}[U(t_n+s)]$ is the truncation error.

According to Theorem \ref{unique u}, we have $\|U(t_n)\|_F\leq \sqrt{m}$, $\|U(t_{n+1})\|_F\leq \sqrt{m}$, and $\|U(t_n+s)\|_F\leq \sqrt{m}$. Then, taking the supremum norm $\|\cdot\|_{\mathcal{X}}$ to $R_2(s)$ and using the Taylor's formula, we derive
    \begin{equation}\label{eq:R2}
        \begin{aligned}
            \|R_2(s)\|_{\mathcal{X}}=&\max _{\mathbf x  \in \overline{\Omega}}\left\|\left(1-\frac{s}{\tau}\right)\mathcal{N}[U(t_n)]+\frac{s}{\tau}\mathcal{N}[U(t_{n+1})]-\mathcal{N}[U(t_n+s)]\right\|_F\\
            \leq&\max _{\mathbf x  \in \overline{\Omega}}\left\|\left(1-\frac{s}{\tau}\right)\mathcal{N}[U(t_n)]+\frac{s}{\tau}\mathcal{N}\left[U(t_{n})+U^{'}(t_n)\tau+\frac{1}{2}\left(U_{ij}^{''}(t_n+\xi_{ij})\right)_{m\times m}\tau^2\right]\right.\\
            &\left.~~~~~~~-\mathcal{N}\left[U(t_n)+U^{'}(t_n)s+\frac{1}{2}\left(U_{ij}^{''}(t_n+\widetilde{\xi}_{ij})\right)_{m\times m}s^2\right]\right\|_F\\
            \leq& C_0\left[\left(\sum_{i,j=1}^{m} \left(\max_{\mathbf x  \in \overline{\Omega}}\left|U_{ij}^{''}(t_n+\xi_{ij})\right|\right)^2\right)^{\frac 12}+\left(\sum_{i,j=1}^{m} \left(\max_{\mathbf x  \in \overline{\Omega}}\left|U_{ij}^{''}(t_n+\widetilde{\xi}_{ij})\right|\right)^2\right)^{\frac 12}\right.\\
            &+\left.\left(\sum_{i,j=1}^{m}\sum_{k,l=1}^{m}\left(\max_{\mathbf x  \in \overline{\Omega}}\left|U_{ij}^{'}(t_n)U_{kl}^{'}(t_n)\right|\right)^2\right)^{\frac 12}\right]\tau^2+C_1(\tau^3+\tau^4+\tau^5+\tau^6)\\
            \leq& C_2\tau^2,
        \end{aligned}
    \end{equation}
    where $U=(U_{ij})_{m\times m}$, $\xi_{ij}, \widetilde{\xi}_{ij}\in[0,\tau]$, the constant $C_0$ depends on $\kappa$ and $m$, $C_1$ depends on $m$ and $C^2([0,T];\mathcal{X})$-norm of $U$, and $C_2$ depends on $\kappa$, $m$, $t_n$, and the $C^2([0,T];\mathcal{X})$-norm of $U$. Here we use the fact that $\tau\leq t_n$.
Using \eqref{eq:e1_n+1} and Lemma \ref{NNkappa}, we can obtain
\begin{equation}\label{eq:03.32}
        \|\widetilde{U}^{n+1}-U(t_{n+1})\|_{\mathcal{X}}\leq
        (1+\tau+5m\tau)\|U^n-U(t_n)\|_{\mathcal{X}}+Cm(\kappa+1+5m)\tau^2,
\end{equation}
and
\begin{equation}\label{eq:3.32}
    \begin{aligned}
        &\|\mathcal{N}[\widetilde{U}^{n+1}]-\mathcal{N}[U(t_{n+1})]\|_{\mathcal{X}}\\
        \leq&(\kappa+1+5m)\|\widetilde{U}^{n+1}-U(t_{n+1})\|_{\mathcal{X}}\\
        \leq&(\kappa+1+5m)(1+\tau+5m\tau)\|e_2^n\|_{\mathcal{X}}+Cm(\kappa+1+5m)^2\tau^2,
    \end{aligned}
\end{equation}
where the constant $C$ is the $C^1([0,T];\mathcal{X})$-norm of $U$ defined in Theorem \ref{Conv1}.

Combining \eqref{eq:3.26} and \eqref{eq:3.32}, we can derive
\begin{equation}\label{eq:3.34}
    \begin{aligned}
    &\left\|\left(1-\frac{s}{\tau}\right)(\mathcal{N}[U^n]-\mathcal{N}[U(t_n)])+\frac{s}{\tau}(\mathcal{N}[\widetilde{U}^{n+1}]-\mathcal{N}[U(t_{n+1})])\right\|_{\mathcal{X}}\\
    \leq&(\kappa+1+5m)\left(1-\frac{s}{\tau}\right)\|e_2^n\|_{\mathcal{X}}+\frac{s}{\tau}\left((\kappa+1+5m)(1+\tau+5m\tau)\|e_2^n\|_{\mathcal{X}}+Cm(\kappa+1+5m)^2\tau^2\right)\\
    =&(\kappa+1+5m)\|e_2^n\|_{\mathcal{X}}+(\kappa+1+5m)(1+5m)s\|e_2^n\|_{\mathcal{X}}+Cm(\kappa+1+5m)^2\tau s.
    \end{aligned}
\end{equation}
Then, taking the supremum norm $\|\cdot\|_{\mathcal{X}}$ to \eqref{eq:e2} and using \eqref{eq:R2}--\eqref{eq:3.32} and Lemma \ref{lambda-Delta}, we obtain
\begin{equation}
    \begin{aligned}
        \|e_2^{n+1}\|_{\mathcal{X}}\leq&e^{-\kappa \tau}\|e_2^n\|_{\mathcal{X}}+\int_0^{\tau}e^{-\kappa(\tau-s)}[(\kappa+1+5m)\|e_2^n\|_{\mathcal{X}}+(\kappa+1+5m)(1+5m)s\|e_2^n\|_{\mathcal{X}}\\
        &+Cm(\kappa+1+5m)^2\tau s+{C_2\tau^2}]ds\\
        =&e^{-\kappa \tau}\|e_2^n\|_{\mathcal{X}}+[(\kappa+1+5m)\|e_2^n\|_{\mathcal{X}}+{C_2\tau^2}]\int_0^{\tau}e^{-\kappa(\tau-s)}ds\\
        &+((\kappa+1+5m)(1+5m)\|e_2^n\|_{\mathcal{X}}+Cm(\kappa+1+5m)^2\tau)\int_0^{\tau}se^{-\kappa(\tau-s)}ds\\
        =&e^{-\kappa\tau}\|e_2^n\|_{\mathcal{X}}+\frac{1-e^{-\kappa\tau}}{\kappa}[(\kappa+1+5m)\|e_2^n\|_{\mathcal{X}}+{C_2\tau^2}]\\
        &+\frac{e^{-\kappa\tau}-1+\kappa\tau}{\kappa^2}((\kappa+1+5m)(1+5m)\|e_2^n\|_{\mathcal{X}}+Cm(\kappa+1+5m)^2\tau)\\
        \leq&\left[e^{-\kappa\tau}+(1-e^{-\kappa\tau})+\frac{1-e^{-\kappa\tau}}{\kappa}(1+5m)+\frac{e^{-\kappa\tau}-1-\kappa\tau}{\kappa}(1+5m)\right.\\
        &\left.+\frac{e^{-\kappa\tau}-1-\kappa\tau}{\kappa}(1+5m)^2\right]\|e_2^n\|_{\mathcal{X}}+\left({C_2}+\frac{C}{2}m(\kappa+1+5m)^2\right)\tau^3\\
        \leq&\left(1+(1+5m)\tau+\frac 12(1+5m)^2\tau^2\right)\|e_2^n\|_{\mathcal{X}}+\left({C_2}+\frac{C}{2}m(\kappa+1+5m)^2\right)\tau^3,
    \end{aligned}
\end{equation}
where we use $1-e^{-a}\leq a$ and $e^{-a}- 1+a\leq\frac{a^2}{2}$ for any $a>0$.

By induction, we get
\begin{equation}
    \begin{aligned}
        \|e_2^{n}\|_{\mathcal{X}}\leq&\left(1+(1+5m)\tau+\frac 12(1+5m)^2\tau^2\right)^n\|e_2^0\|_{\mathcal{X}}\\
        &+\left({C_2}+\frac{C}{2}m(\kappa+1+5m)^2\right)\tau^3\sum_{k=0}^{n-1}\left(1+(1+5m)\tau+\frac 12(1+5m)^2\tau^2\right)^k\\
        \leq&\left(1+(1+5m)\tau+\frac 12(1+5m)^2\tau^2\right)^n\|e_2^0\|_{\mathcal{X}}\\
        &+\left({C_2}+\frac{C}{2}m(\kappa+1+5m)^2\right)\tau^2\frac{\left(1+(1+5m)\tau+\frac 12(1+5m)^2\tau^2\right)^n-1}{1+5m+\frac 12(1+5m)^2\tau}\\
        \leq&e^{(1+5m)n\tau}\|e_2^0\|_{\mathcal{X}}+\frac{{C_2}+\frac{C}{2}m(\kappa+1+5m)^2}{1+5m}e^{(1+5m)n\tau}\tau^2.
    \end{aligned}
\end{equation}
Since $e_2^0=0$ and $n\tau=t_n$, we can obtain \eqref{eq:conv2}.\end{proof}


\section{Numerical Experiments}\label{section 4}

In this section, we present a series of diverse numerical examples in two and three dimensions to validate the convergence, MBP preservation, and energy dissipation properties of time-discrete ETD schemes applied to the matrix-valued Allen--Cahn equation \eqref{eq:mac} with the periodic boundary condition under the Frobenius norm. We simply use the central finite difference method for the discretization of spatial Laplacian. The products of matrix exponents with vectors are implemented by the fast Fourier transform (FFT). Additionally, we simulate the corresponding matrix field and interface evolution processes. The interface here is defined as the zero level set of the corresponding determinate field of the matrix field. In all examples, the computational domain is defined as $\Omega=[-\frac 12, \frac 12]^m$, where $m=2,3$.

\subsection{Convergence, MBP, and energy dissipation tests}
This section tests the convergence order, discrete maximum bound principles, and unconditional energy dissipation of the ETD schemes. For the test of the convergence order in time, we use $1024\times 1024$ grid points to discretize the computational domain $\Omega=[-\frac 12,\frac 12]^2$ to reduce the impact of spatial discretization errors.

{\bf Example 1.} We consider the matrix-valued Allen--Cahn equation \eqref{eq:mac} in $\Omega$ with $\varepsilon=0.01$ and set the initial condition to be
\begin{equation}\label{eq:4.1}
    \begin{aligned}
	U^0(x,y)=
	\begin{bmatrix}
	\cos\alpha&-\sin\alpha\\
	\sin\alpha&\cos\alpha
	\end{bmatrix}
    \end{aligned}
\end{equation}
where $\alpha=\alpha(x,y)=1+\frac{\pi}{2}\sin(2\pi(x+y))$. The stabilizing parameter is chosen as $\kappa=5$.

We take several different time step sizes $\tau=0.1\times 2^{-k}$ with $k=0,1,\ldots,7$ to obtain the corresponding numerical solution at $t=1$. Furthermore, since the exact solution of \eqref{eq:mac} is unavailable, we take a small time size $\tau=0.1\times 10^{-10}$ to obtain an reference solution at $t=1$. The $L^2$ and $L^{\infty}$ errors between these numerical solutions and the reference solution are shown in Tables \ref{tab:1} and \ref{tab:2}. It can be observed that the convergence rates of ETD1 and ETDRK2 schemes are about $1$ and $2$ (corresponds to Theorems \ref{Conv1}, \ref{Conv2}), respectively. Consistently, the values of the errors in the ETDRK2 scheme is more accurate than those of the ETD1 scheme.
\begin{table}[!htpb]
    \centering
	\normalsize
	\caption{Errors and convergence rates at time $t=1$ computed by the ETD1 scheme in Example 1.}
	\label{tab:1}
	\vspace{10pt}
	\begin{tabular}{ccccc}
	\hline
	$\tau$ & $L^2$ error & Rate&$L^{\infty}$ error & Rate \\
 \hline
	$0.1\times 2^{0}$&$8.399\times 10^{-2}$&-&$3.706\times 10^{-3}$&-\\
	$0.1\times 2^{-1}$&$4.524\times 10^{-2}$&0.8926&$2.001\times 10^{-3}$&0.8892\\
	$0.1\times 2^{-2}$&$2.347\times 10^{-2}$&0.9465&$1.039\times 10^{-3}$&0.9450\\
	$0.1\times 2^{-3}$&$1.193\times 10^{-2}$&0.9770&$5.283\times 10^{-4}$&0.9763\\
	$0.1\times 2^{-4}$&$5.976\times 10^{-3}$&0.9669&$2.648\times 10^{-4}$&0.9966\\
	$0.1\times 2^{-5}$&$2.955\times 10^{-3}$&1.0158&$1.309\times 10^{-4}$&1.0157\\
  	$0.1\times 2^{-6}$&$1.434\times 10^{-3}$&1.0437&$6.353\times 10^{-5}$&1.0436\\
  	$0.1\times 2^{-7}$&$6.699\times 10^{-4}$&1.0977&$2.969\times 10^{-5}$&1.0977\\
	\hline
	\end{tabular}
\end{table}
\begin{table}
    \centering
	\normalsize
	\caption{Errors and convergence rates at time $t=1$ computed by the ETDRK2 scheme in Example 1.}
	\label{tab:2}
	\vspace{10pt}
	\begin{tabular}{ccccc}
 \hline
	$\tau$ &$L^2$ error & Rate&$L^{\infty}$ error & Rate \\
 \hline
	$0.1\times 2^{0}$&$1.783\times 10^{-2}$&-&$7.904\times 10^{-4}$&-\\
	$0.1\times 2^{-1}$&$5.210\times 10^{-3}$&1.7753&$2.310\times 10^{-4}$&1.7744\\
	$0.1\times 2^{-2}$&$1.412\times 10^{-3}$&1.8839&$6.261\times 10^{-5}$&1.8836\\
	$0.1\times 2^{-3}$&$3.676\times 10^{-4}$&1.9411&$1.631\times 10^{-5}$&1.9409\\
	$0.1\times 2^{-4}$&$9.381\times 10^{-5}$&1.9705&$4.161\times 10^{-6}$&1.9704\\
	$0.1\times 2^{-5}$&$2.368\times 10^{-5}$&1.9861&$1.050\times 10^{-6}$&1.9861\\
  	$0.1\times 2^{-6}$&$5.933\times 10^{-6}$&1.9968&$2.632\times 10^{-7}$&1.9966\\
  	$0.1\times 2^{-7}$&$1.469\times 10^{-6}$&2.0133&$6.523\times 10^{-8}$&2.0127\\
	\hline
	\end{tabular}
\end{table}

In the follows, we will test the discrete MBP and unconditional energy dissipation of the ETD schemes. We use $256\times 256$ mesh points to discretize the computational domain $\Omega=[-\frac 12,\frac 12]^2$ and set the time step size $\tau=0.01$ and the terminal time $T=50$.

Figures \ref{fig:4.1} and \ref{fig:4.2} plot the evolution of the supremum norm and the energy of the ETD1 and ETDRK2 schemes, respectively. As shown in Theorems \ref{MBP1} and \ref{MBP2}, the supremum norm $\|\cdot\|_{\mathcal{X}}$ of the numerical solutions at any time does not exceed $\sqrt{2}$. Moreover, we see that the energy function decreases monotonically with time in both ETD1 and ETDRK2 schemes, which verifies the results we proved in Theorem \ref{Energy}.

\begin{figure}[ht!]
    \centering
		\subfigure{\includegraphics[width=0.42\textwidth,
		height=45mm]{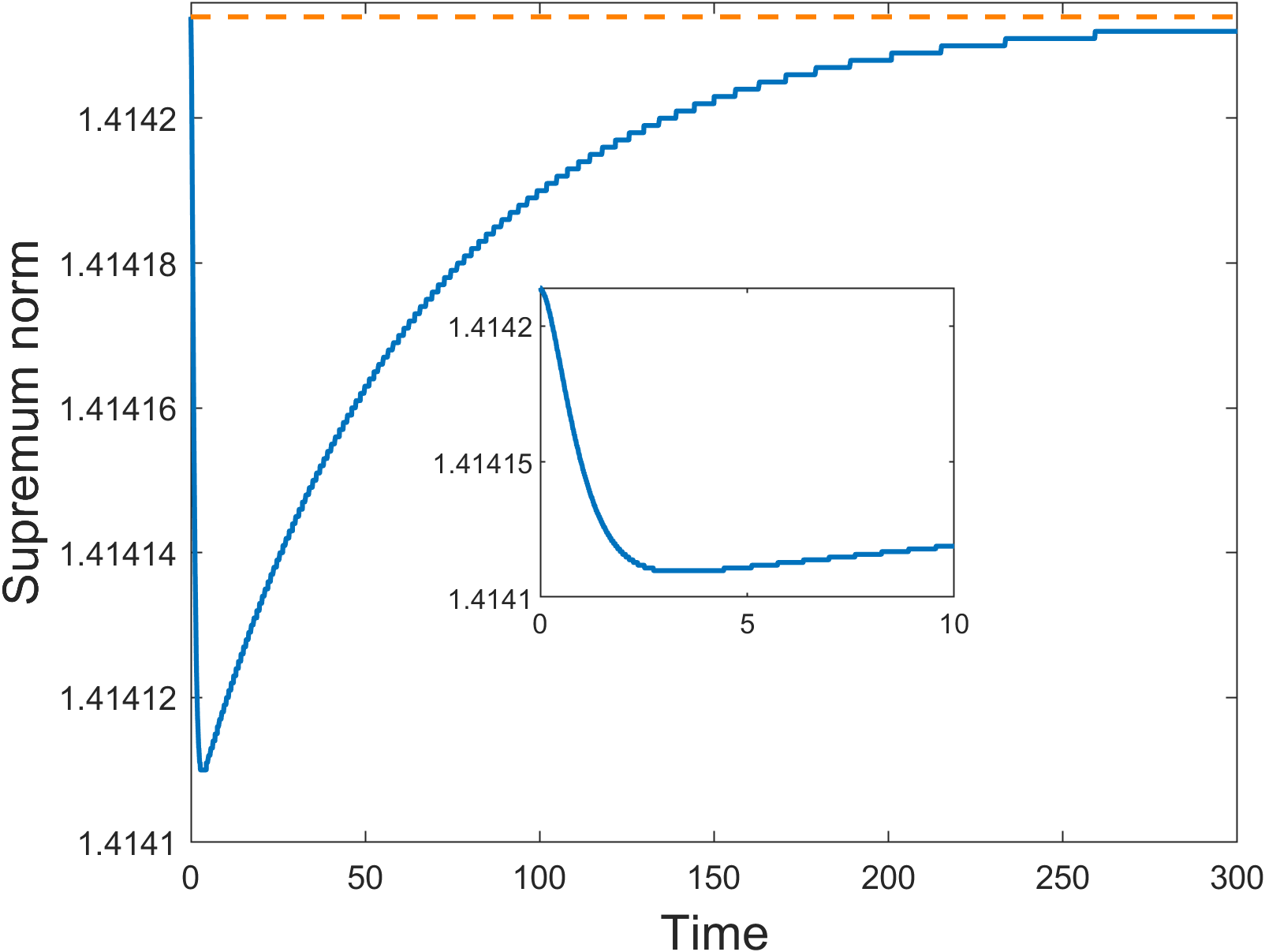}}\quad
		\subfigure{\includegraphics[width=0.40\textwidth,
		height=47mm]{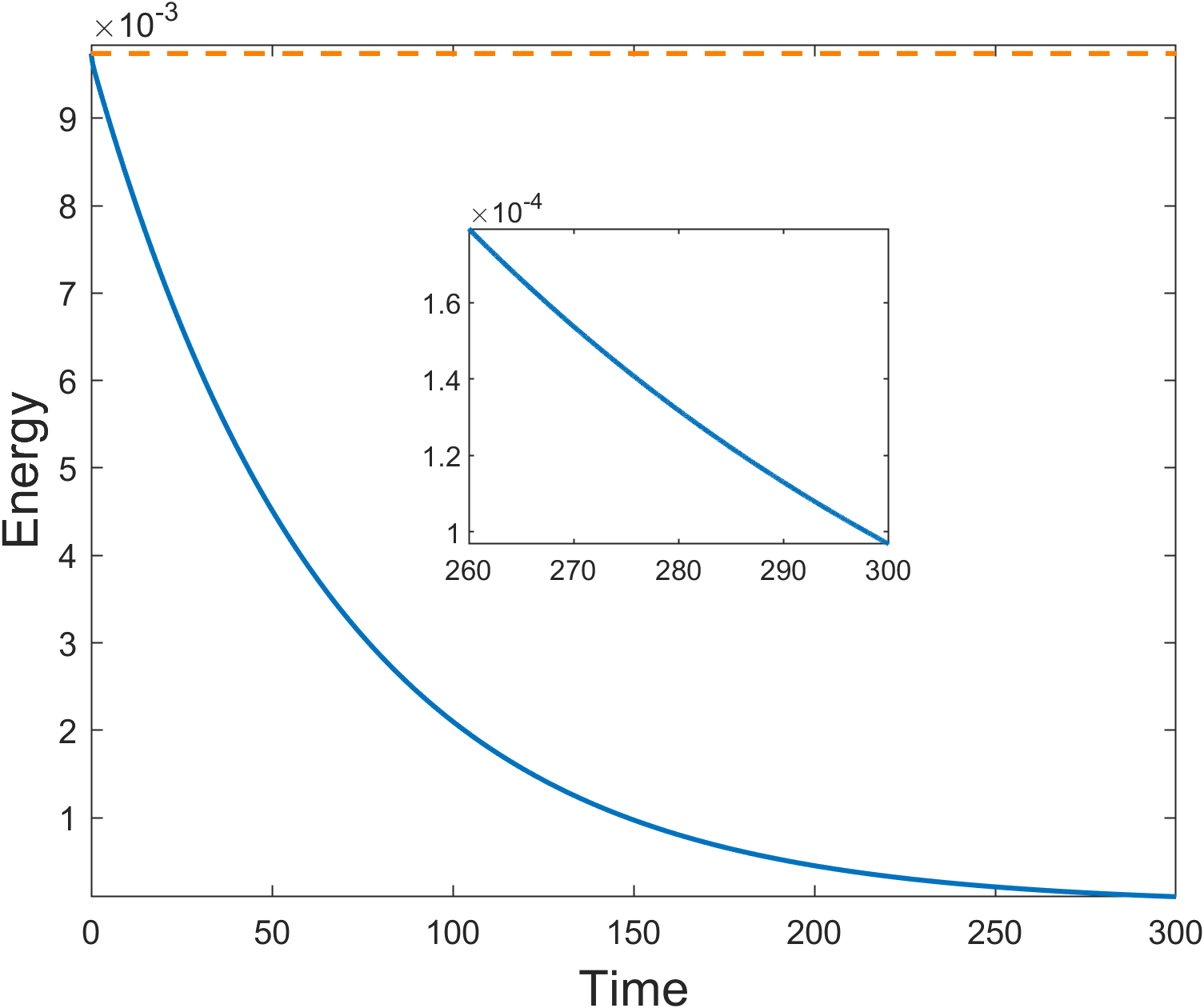}}
		\caption{Evolution of the supremum norm $\|\cdot\|_{\mathcal{X}}$ and energy computed by the ETD1 scheme in Example 1. The dashed line in the left figure is the maximum bound $\sqrt m$ while the dashed line in the right figure is the initial energy.  }
  \label{fig:4.1}
\end{figure}
\begin{figure}[ht!]
    \centering
		\subfigure{\includegraphics[width=0.42\textwidth,
		height=45mm]{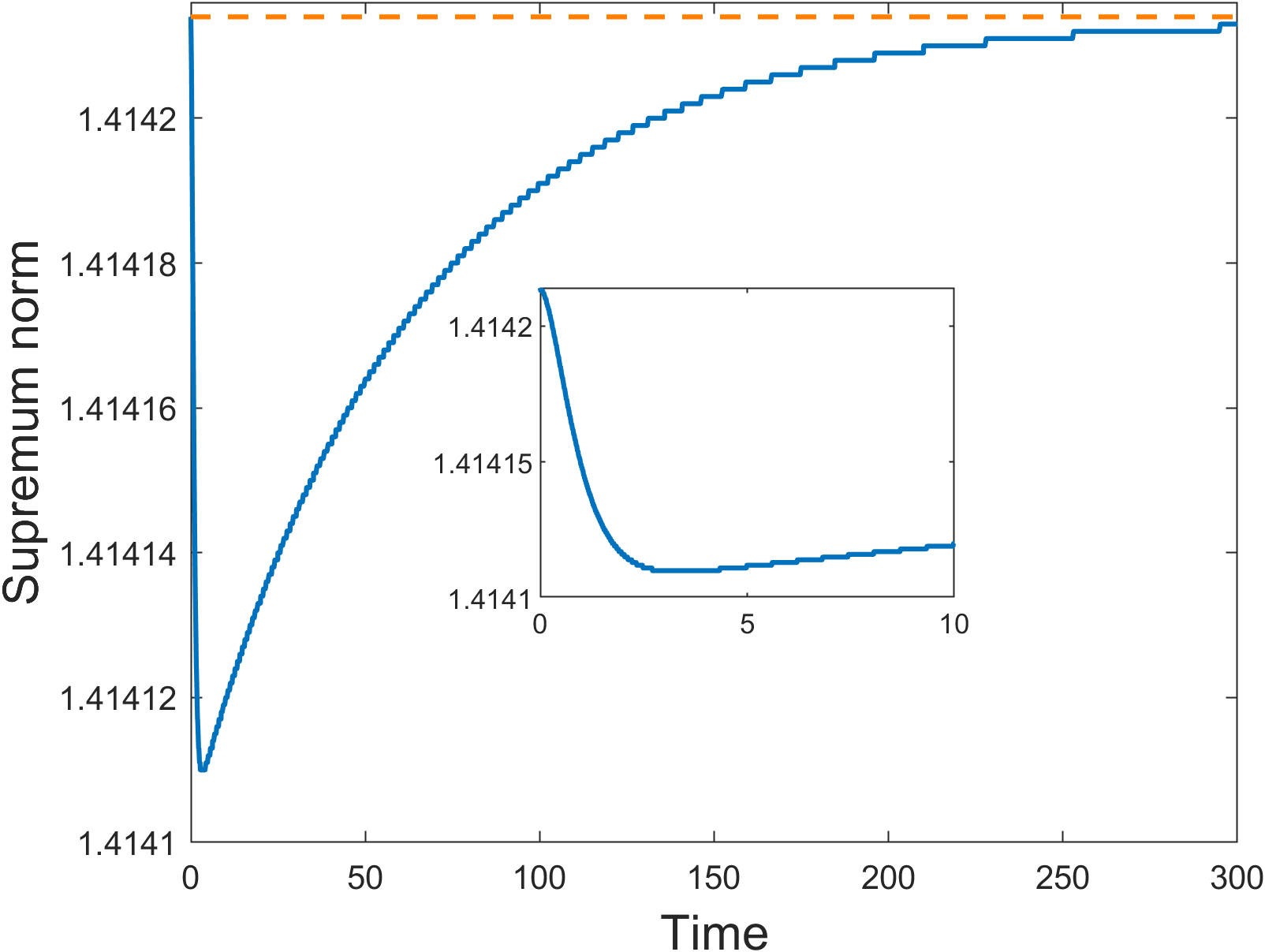}}\quad
		\subfigure{\includegraphics[width=0.40\textwidth,
		height=47mm]{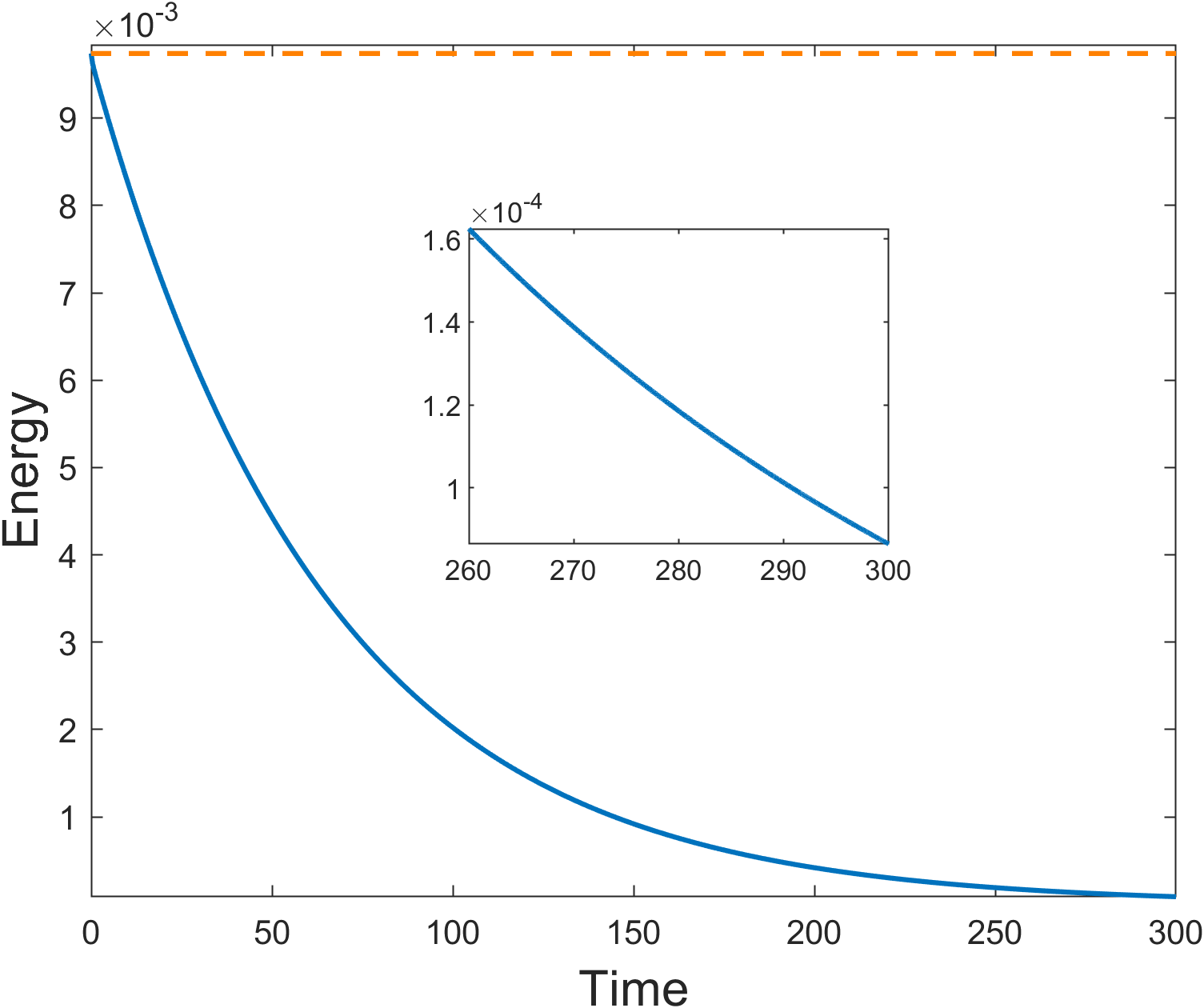}}
		\caption{Evolution of the supremum norm $\|\cdot\|_{\mathcal{X}}$ and energy computed by the ETDRK2 scheme in Example 1. The dashed line in the left figure is the maximum bound $\sqrt m$ while the dashed line in the right figure is the initial energy.}
  \label{fig:4.2}
\end{figure}

\subsection{Simulation of matrix-valued field and interface evolution}

In this section, we present numerous numerical examples in two and three dimensions to simulate the evolution of matrix-valued fields and interfaces for the matrix-valued Allen--Cahn equation with initial values of orthogonal matrix fields. 
The article \cite{wang2019interface} investigates the initial value problem for the matrix-valued Allen--Cahn equation using an asymptotic approach. If the initial conditions have regions with positive and negative determinants, an interface is formed. Away from the interface, the matrix-valued field in each region behaves as a single-sign determinant case, reducing to the complex Ginzburg-Landau case \cite{ma2023energy}. The interface is then driven by curvature and jumps of the squared tangential derivative of the phase spanning the interface along the interface. In the following, we verify that the results of the discretization of the matrix-valued Allen--Cahn equation using the ETDRK2 scheme satisfy the laws of motion of this interface as presented in \cite{wang2019interface}.

\subsubsection{Two-dimensional matrix-valued Allen--Cahn equation}

In the following two-dimensional examples, we consider the spatial domain $\Omega=[-\frac 12,\frac 12]^2$ to be the flat torus in dimension two and use $256\times 256$ grid points to discretize $\Omega$ with various initial conditions. In the subsequent figures of the matrix-valued field and interface evolution for the matrix-valued Allen--Cahn equation, the domain is colored by the determinant of the matrix $U(t,x,y)$ and the vector field is generated by the first column vector of the matrix $U(t,x,y)$. We set the stabilizing parameter $\kappa=5$.

{\bf Example 2.} We take $\varepsilon=0.01$ and set the initial condition to be
\begin{equation}\label{eq:4.2}
   \begin{aligned}
	U^0(x,y)=\begin{cases}
	\begin{bmatrix}
	\cos\alpha&-\sin\alpha\\
	\sin\alpha&\cos\alpha
	\end{bmatrix}\quad |x-y|< 0.5,\\
	\\
	\begin{bmatrix}
	\cos\alpha&\sin\alpha\\
	\sin\alpha&-\cos\alpha
	\end{bmatrix}\quad\text{otherwise},\\
	\end{cases}
    \end{aligned}
\end{equation}
where $\alpha=\alpha(x,y)=0$ or $\frac{\pi}{2}\sin(2\pi(x+y))$. The time step is fixed as $\tau=0.01$. In order to observe the effect of the matrix-valued field on the motion law of the interface, we perform two experiments where the initial condition has the same interface but different initial matrix-valued fields. 

Figures \ref{fig:4.3} and \ref{fig:4.5} give the interface simulation for two different initial matrix fields, respectively. We can observe that the matrix field influences the interface motion.

Figures \ref{fig:4.4} and \ref{fig:4.6} present the evolution of the supremum norm and energy for two different initial value matrix fields, respectively. It can be seen that the discrete maximum bound principle (Theorem \ref{MBP2}) and the energy dissipation law (Theorem \ref{Energy}) are satisfied for both examples with different initial matrix fields.

\begin{figure}[ht!]
    \centering
	\subfigure{\includegraphics[width=0.30\textwidth,
		height=39mm]{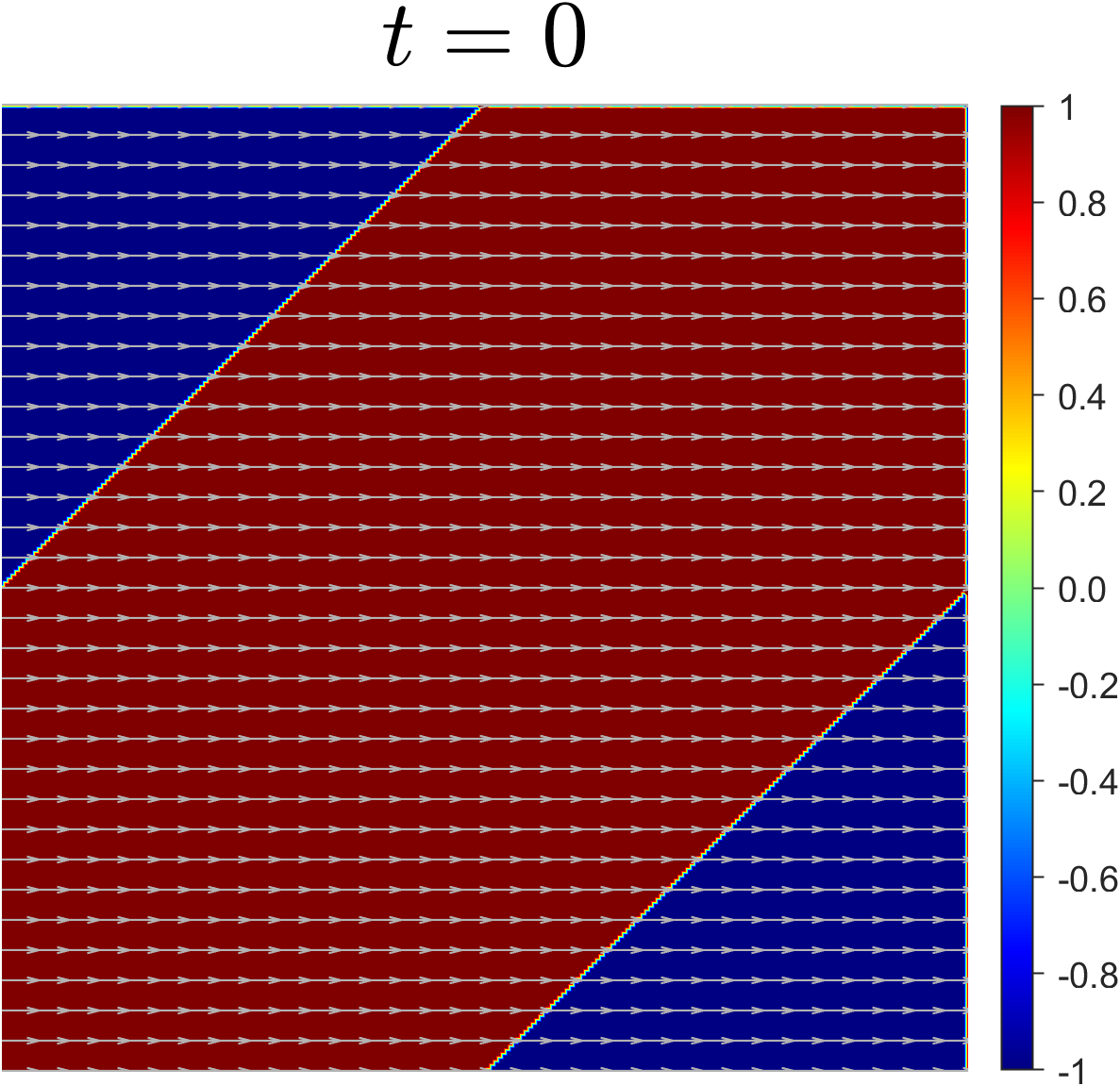}}\quad
	\subfigure{\includegraphics[width=0.30\textwidth,
		height=39mm]{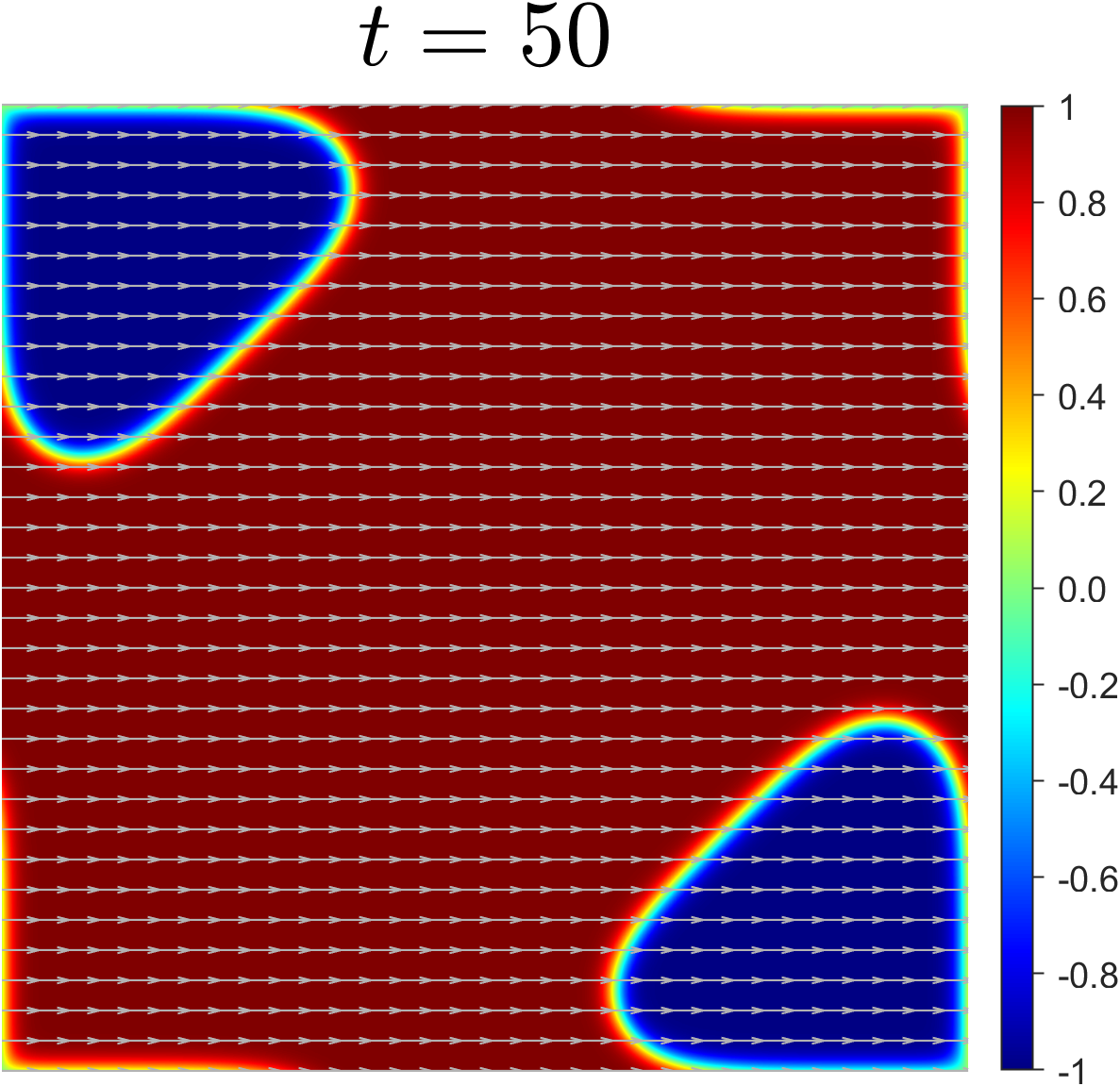}}\quad
	\subfigure{\includegraphics[width=0.30\textwidth,
		height=39mm]{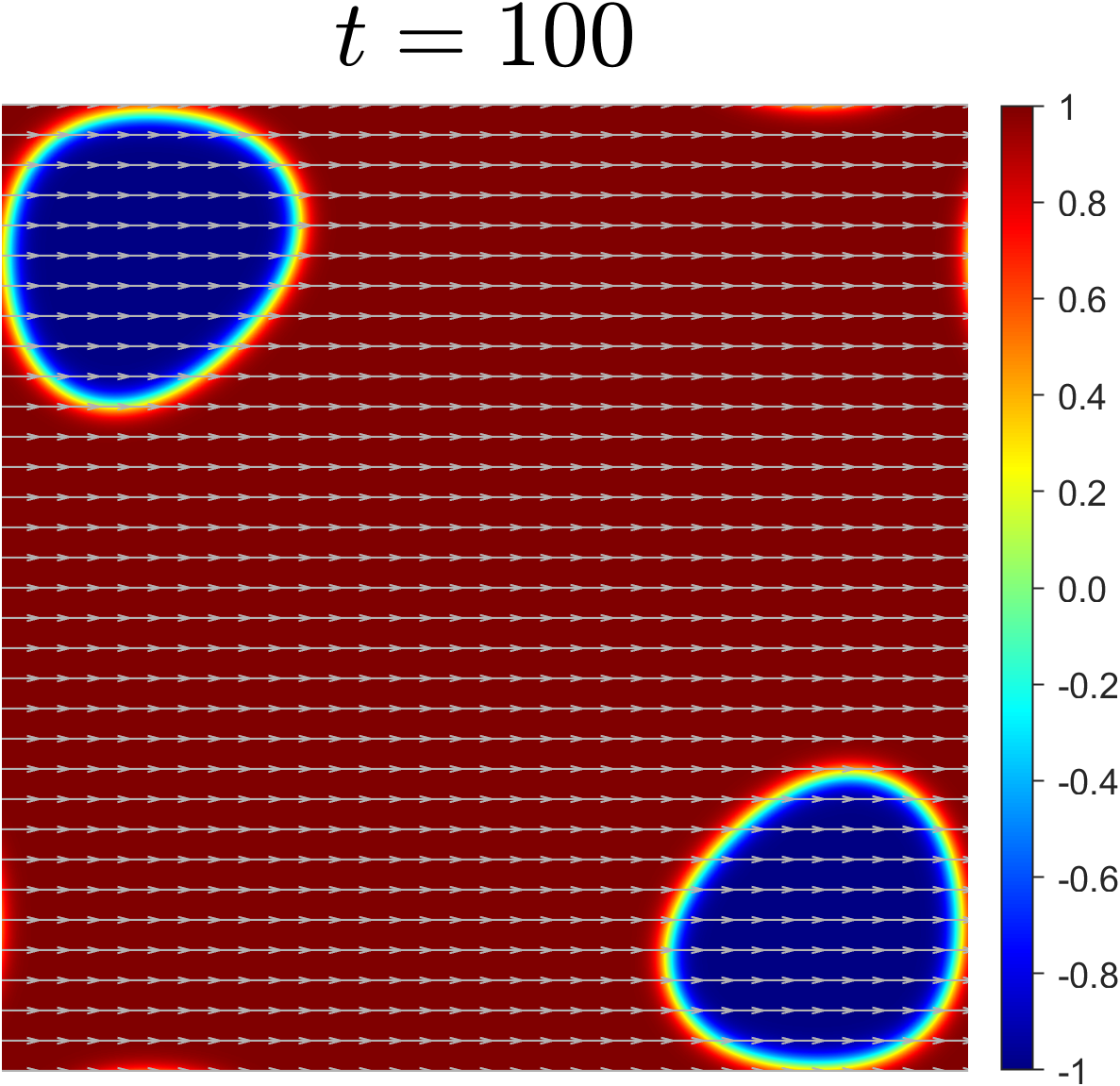}}\quad
	\subfigure{\includegraphics[width=0.30\textwidth,
		height=39mm]{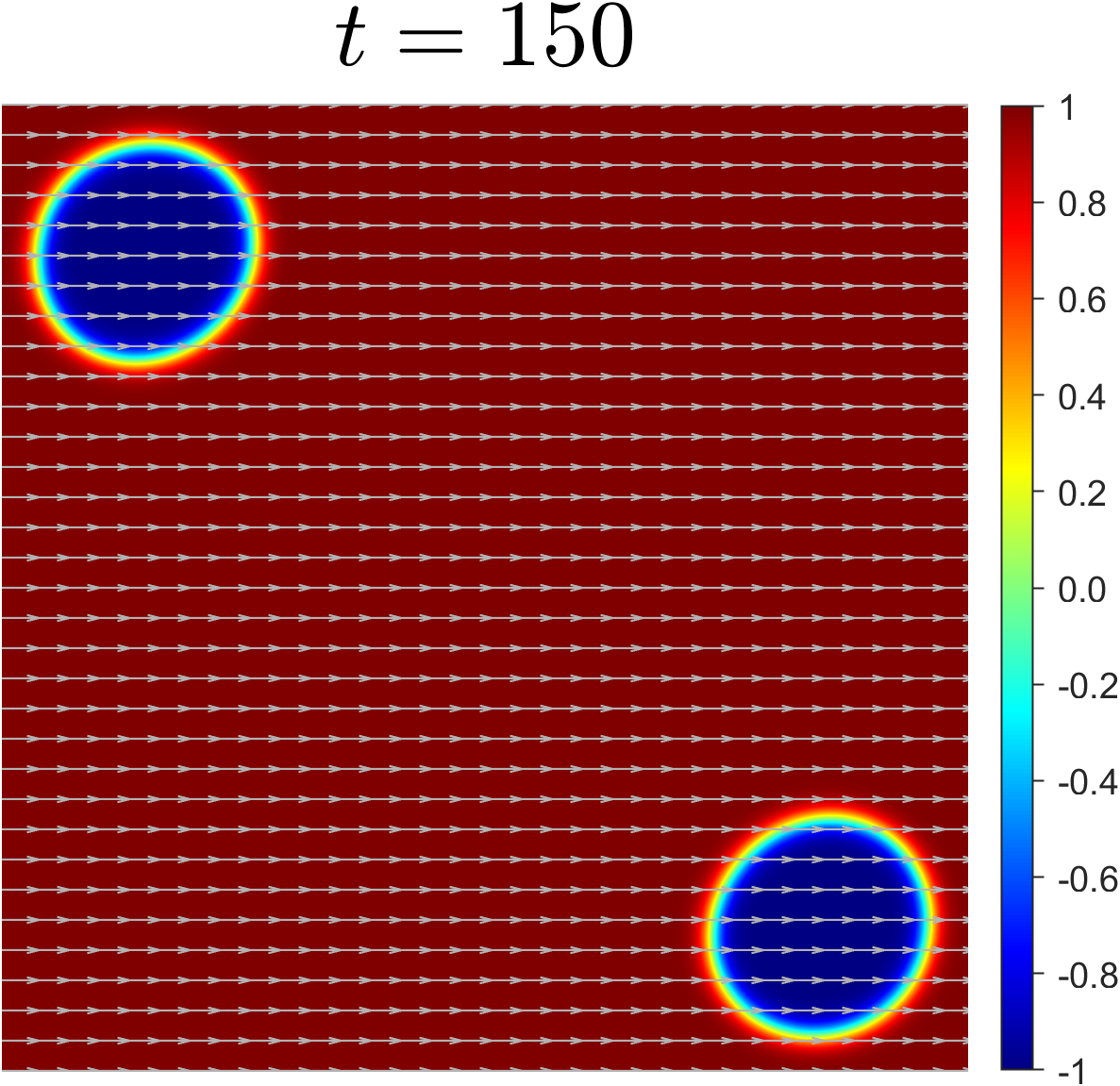}}\quad
	\subfigure{\includegraphics[width=0.30\textwidth,
		height=39mm]{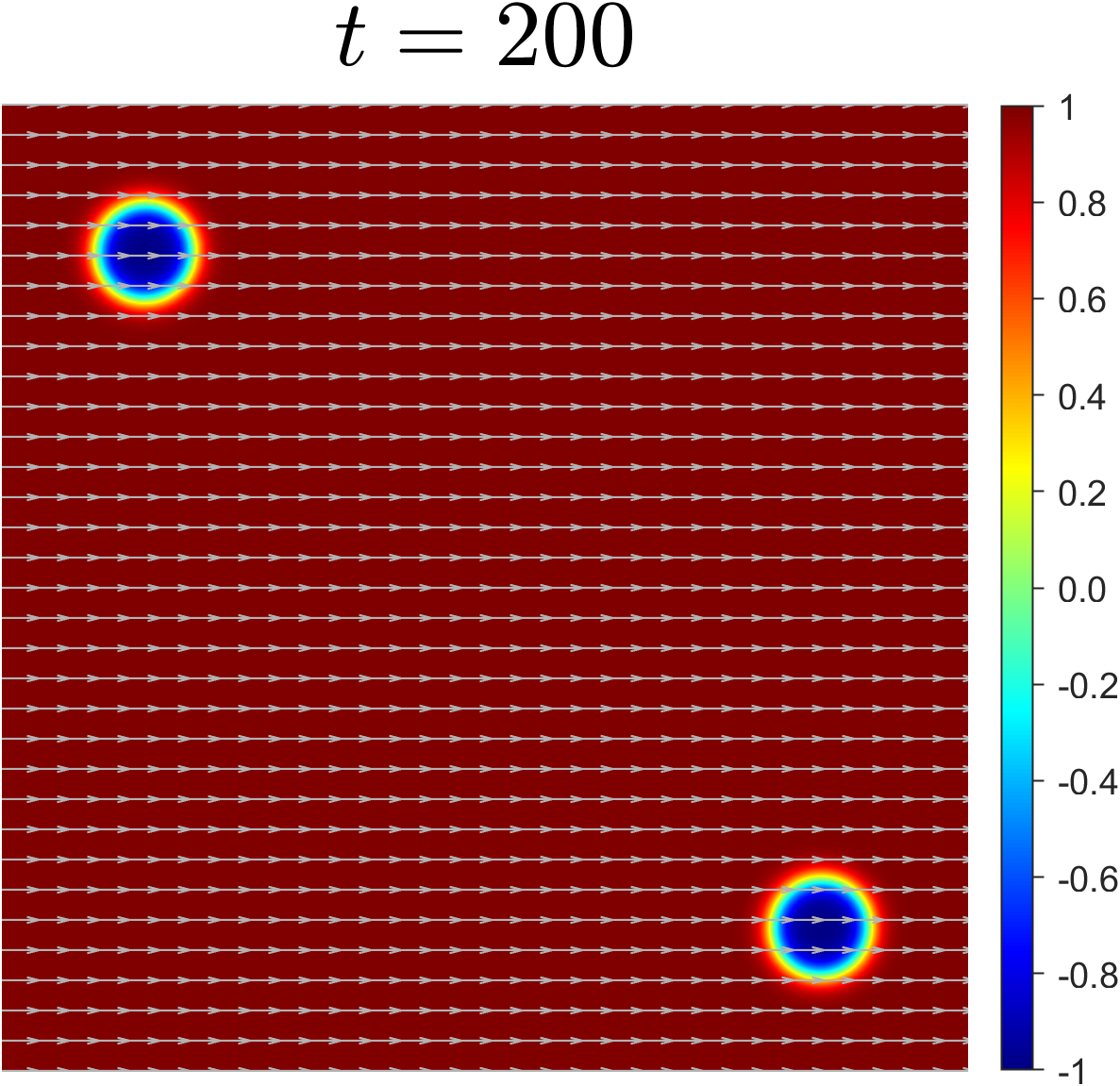}}\quad
	\subfigure{\includegraphics[width=0.30\textwidth,
		height=39mm]{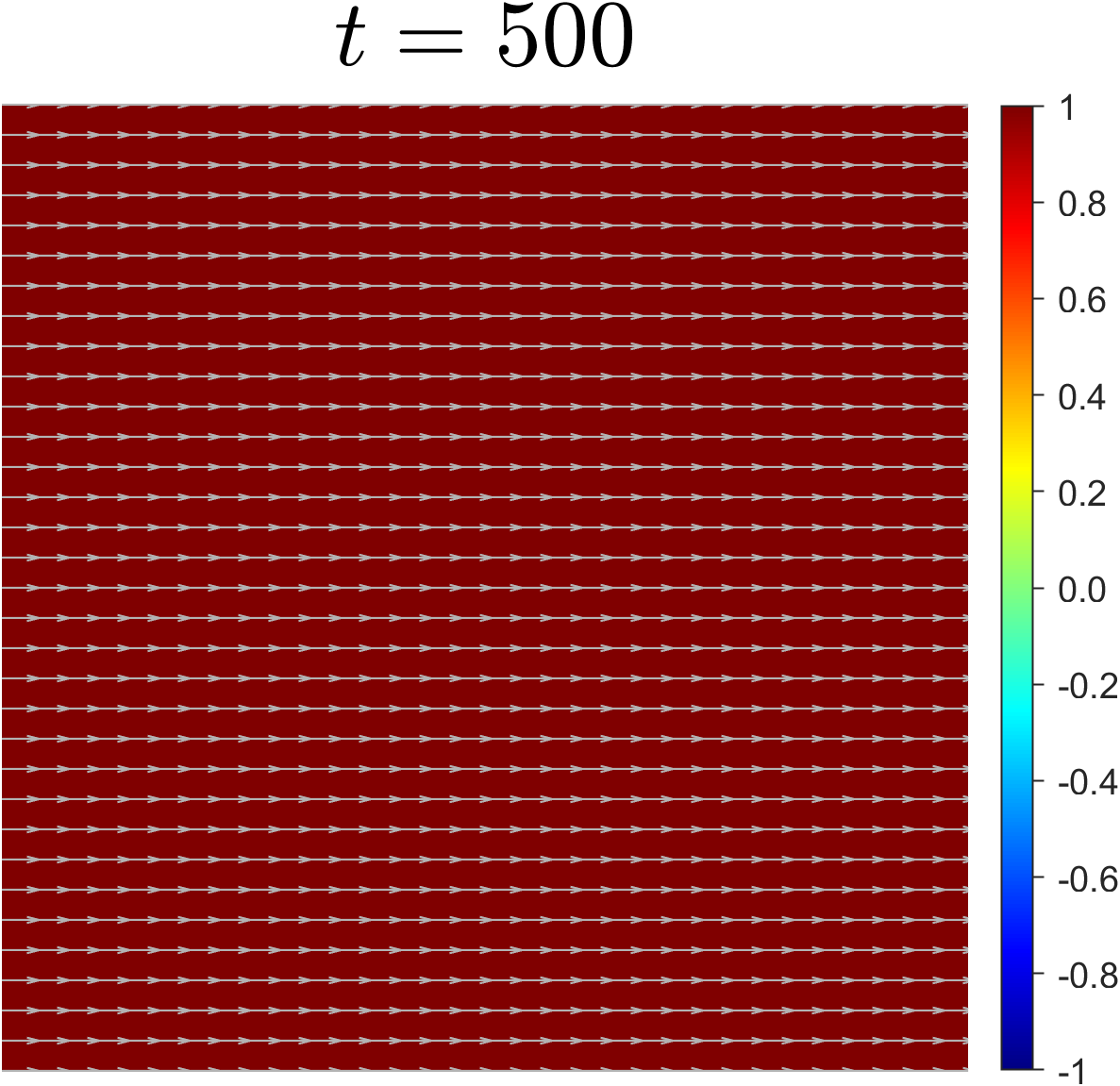}}
	\caption{ Evolution of the matrix-valued field and interface at $t=0,50,100,150,200,500$. The initial field is given in \eqref{eq:4.2} with $\alpha(x,y)=0$.}
        \label{fig:4.3}
\end{figure}

\begin{figure}[ht!]
    \begin{center}
		\subfigure{\includegraphics[width=0.40\textwidth,
		height=45mm]{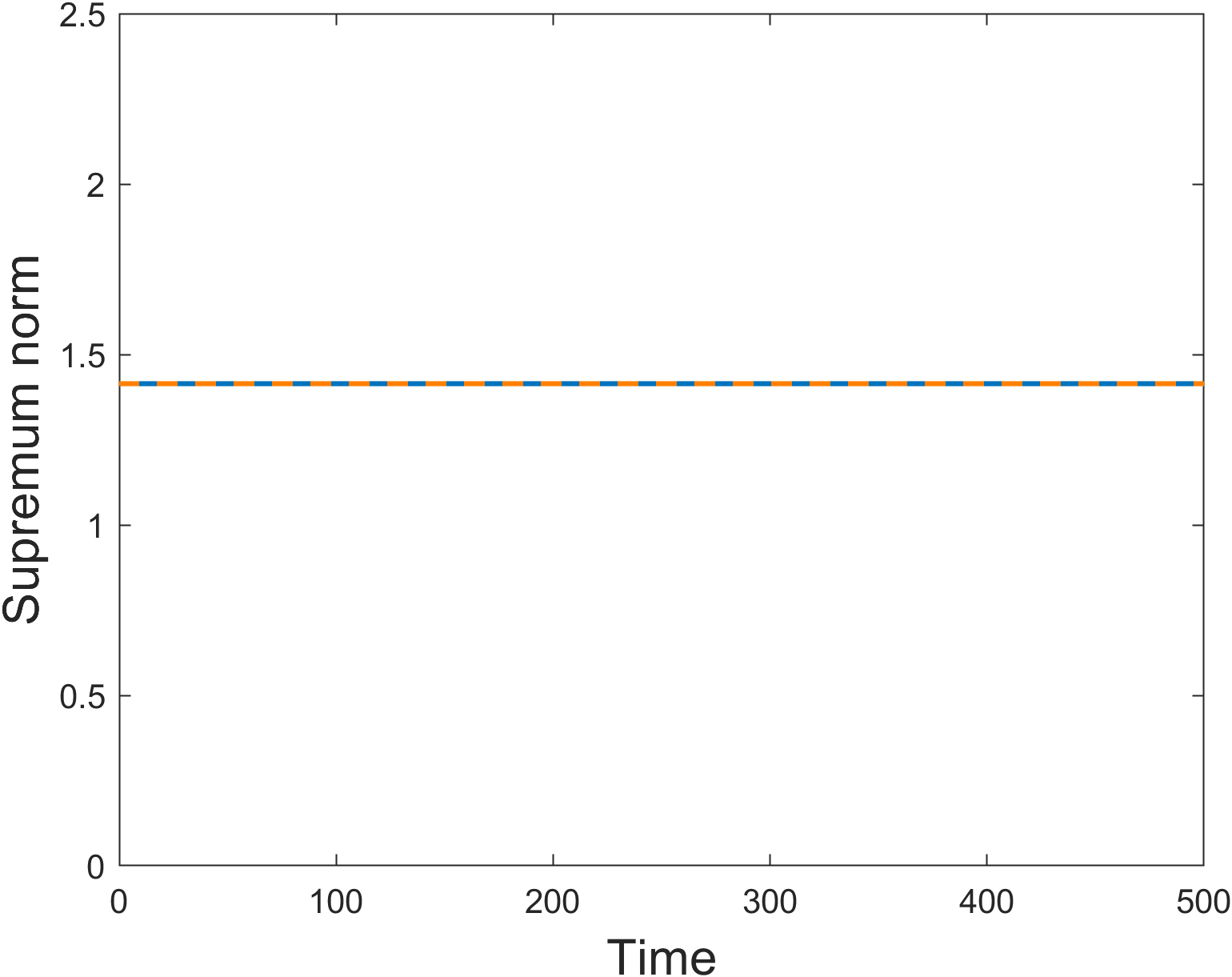}}\quad
		\subfigure{\includegraphics[width=0.40\textwidth,
		height=45mm]{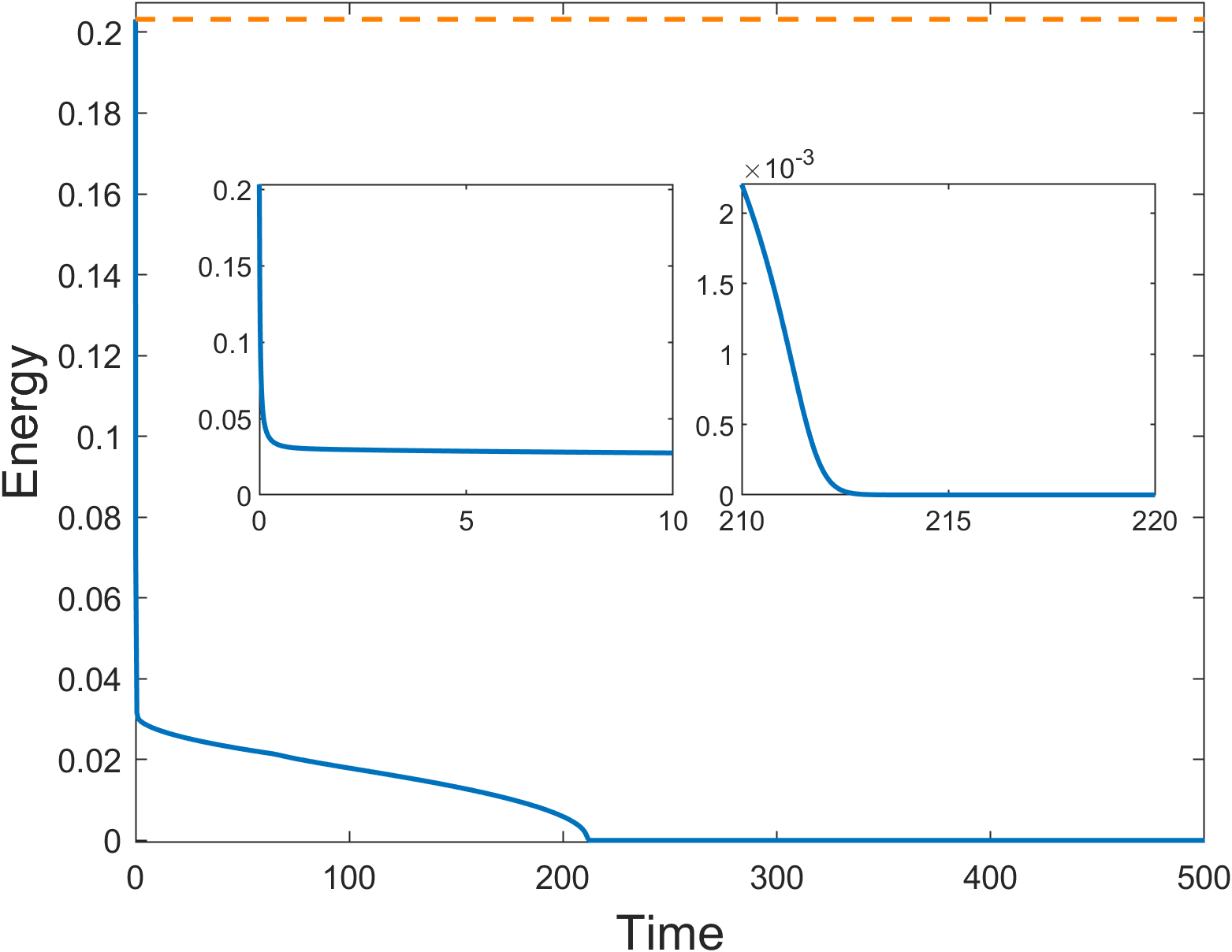}}
		\caption{Evolution of the supremum norm $\|\cdot\|_{\mathcal{X}}$ and energy with initial condition \eqref{eq:4.2} and $\alpha(x,y)=0$. The dashed line in the left figure is the maximum bound $\sqrt m$ while the dashed line in the right figure is the initial energy.}
  \label{fig:4.4}
   \end{center}
\end{figure}

 \begin{figure}[ht!]
    \centering
	\subfigure{\includegraphics[width=0.30\textwidth,
		height=39mm]{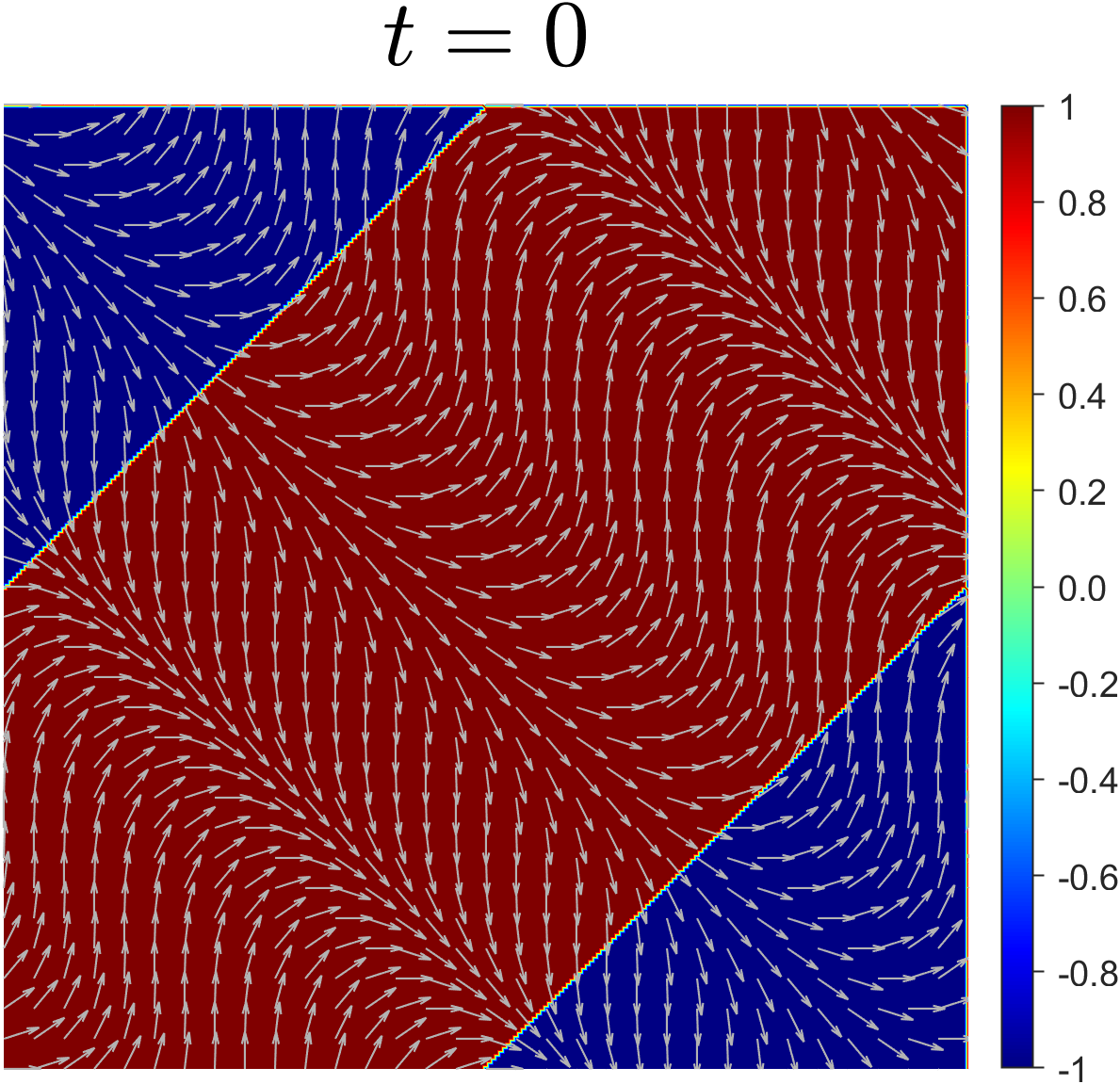}}\quad
	\subfigure{\includegraphics[width=0.30\textwidth,
		height=39mm]{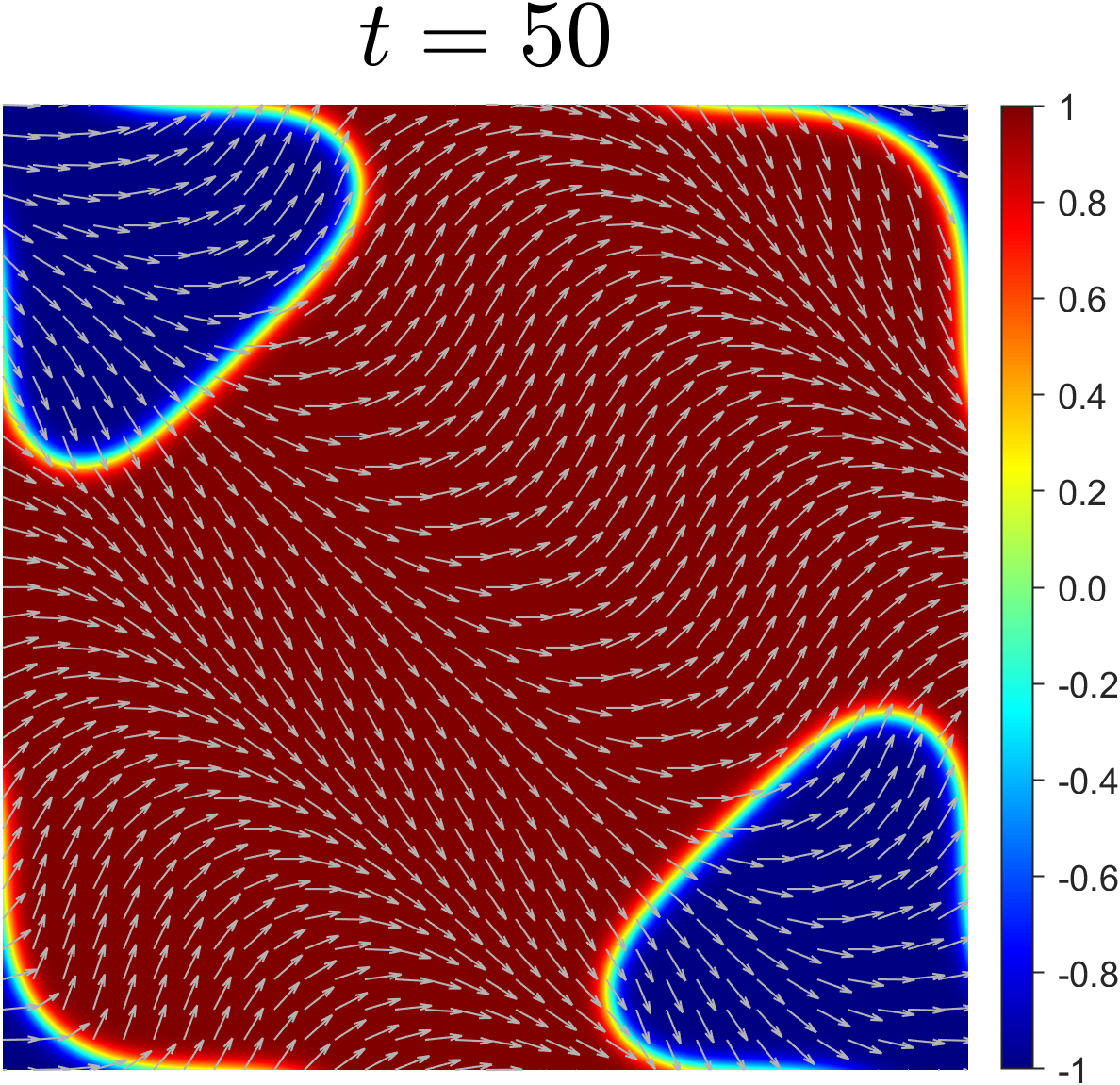}}\quad
	\subfigure{\includegraphics[width=0.30\textwidth,
		height=39mm]{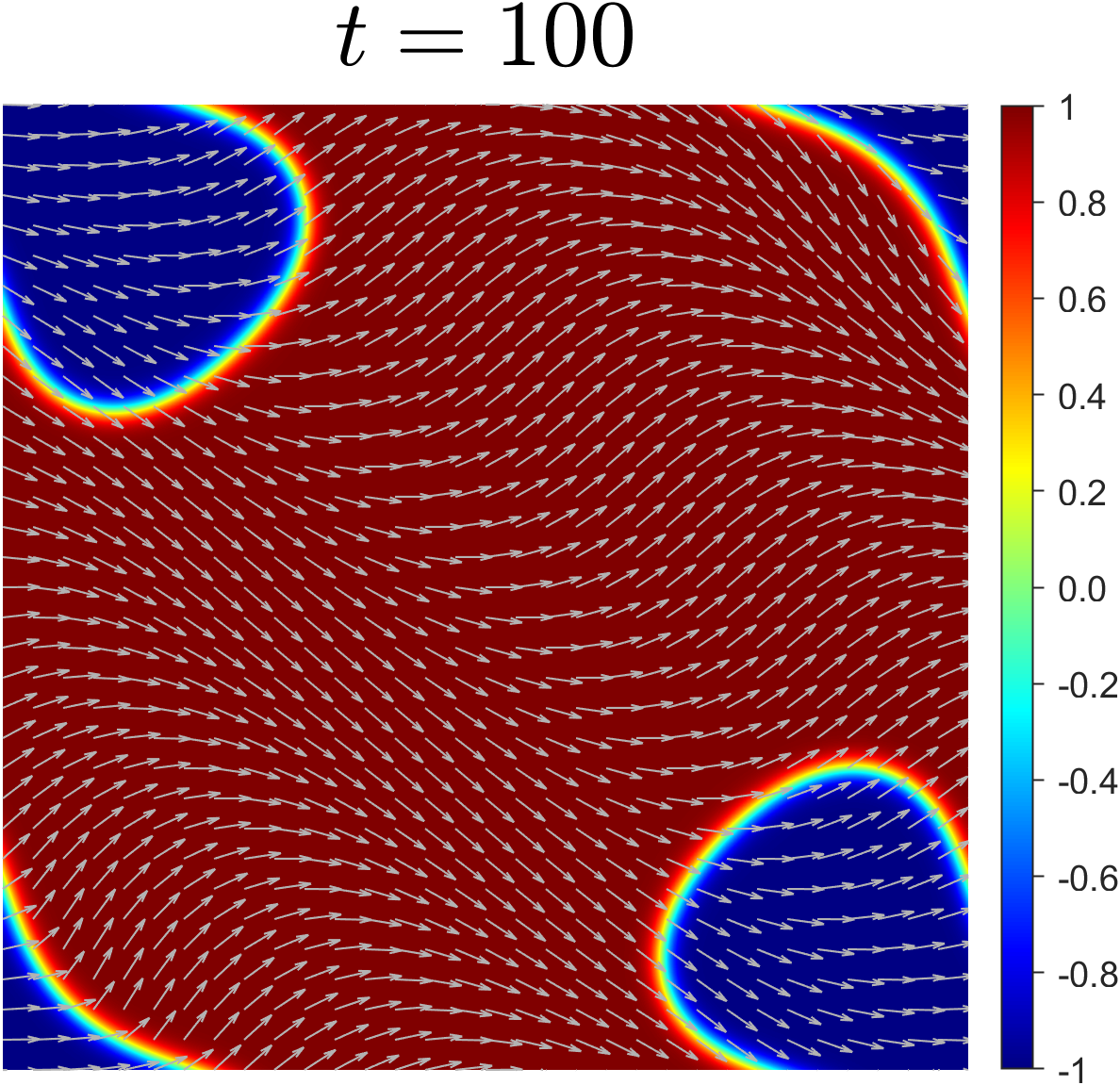}}\quad
	\subfigure{\includegraphics[width=0.30\textwidth,
		height=39mm]{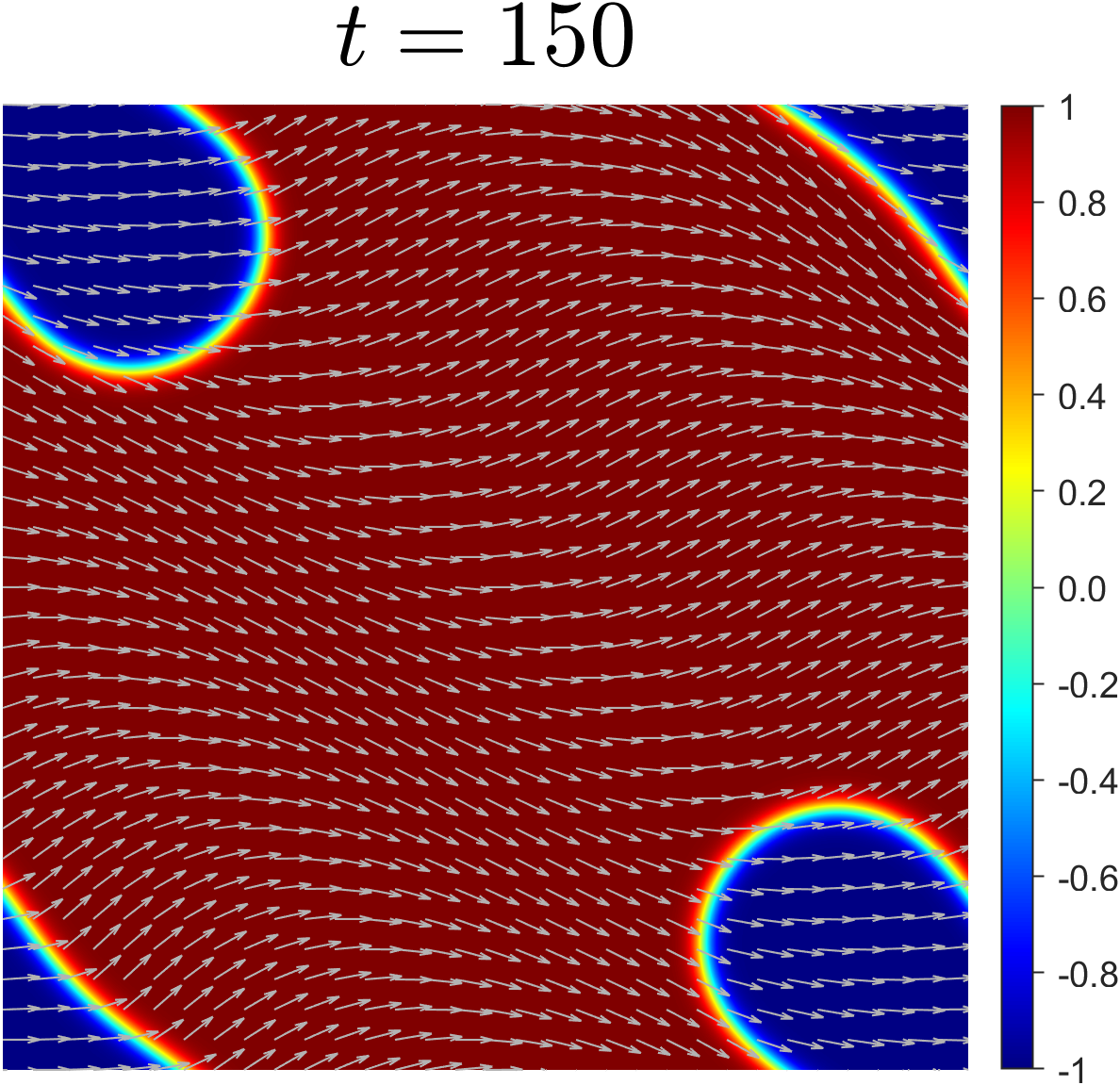}}\quad
	\subfigure{\includegraphics[width=0.30\textwidth,
		height=39mm]{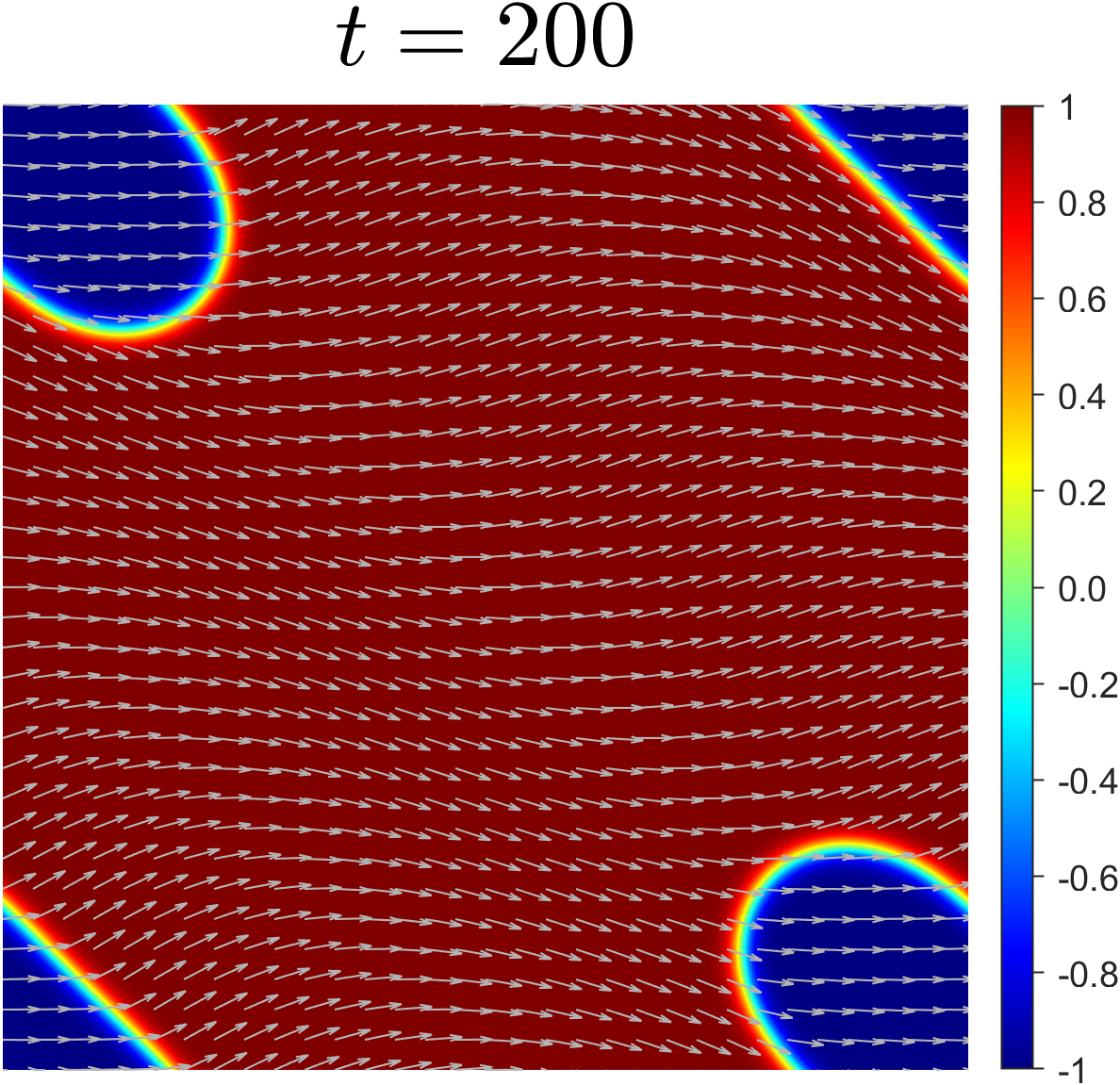}}\quad
	\subfigure{\includegraphics[width=0.30\textwidth,
		height=39mm]{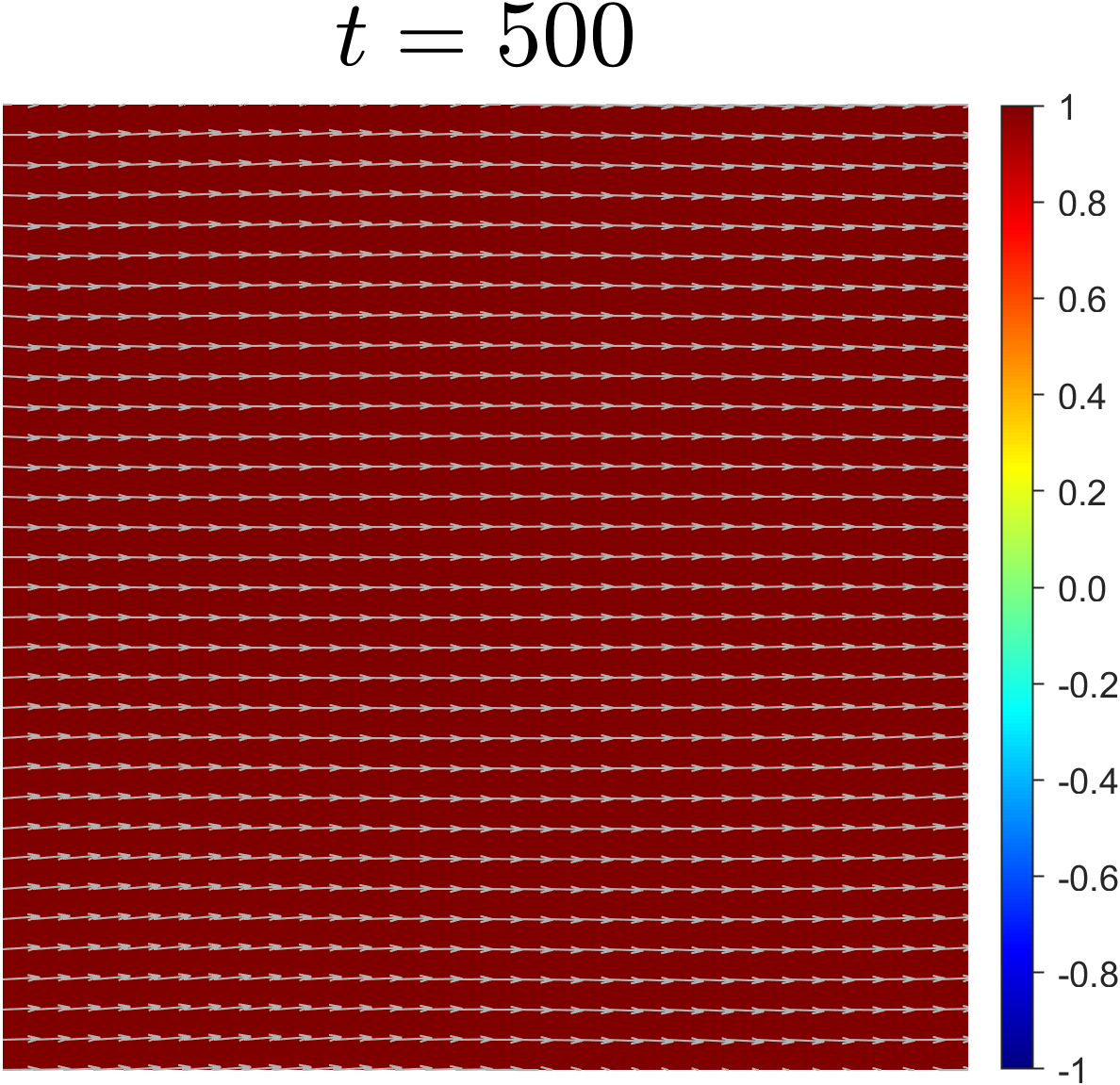}}
	\caption{ Evolution of the matrix-valued field and interface at $t=0,50,100,150,200,500$. The initial field is given in \eqref{eq:4.2} with $\alpha(x,y)=\frac{\pi}{2}\sin(2\pi(x+y))$.}
        \label{fig:4.5}
\end{figure}

\begin{figure}
   \centering
		\subfigure{\includegraphics[width=0.42\textwidth,
		height=45mm]{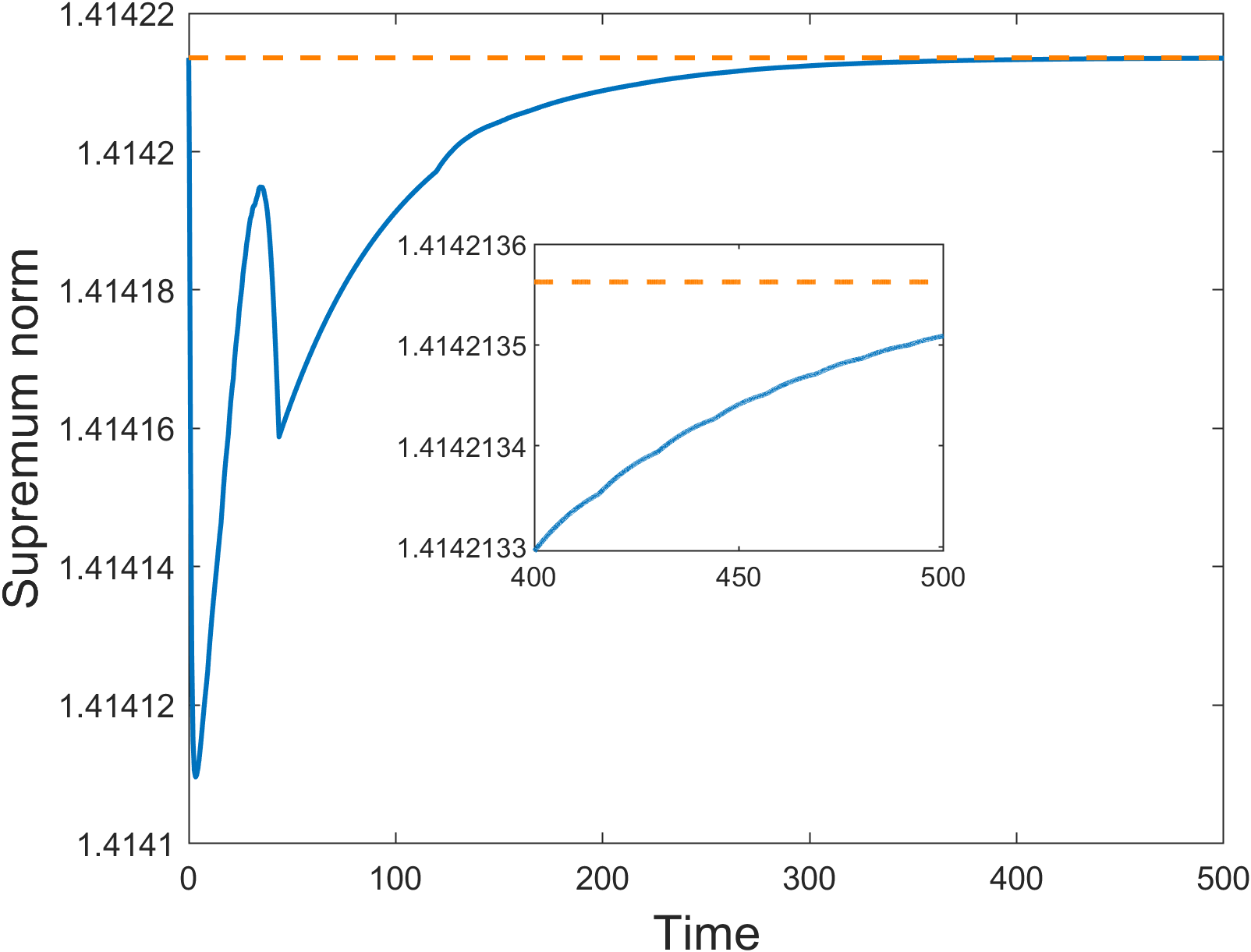}}\quad
		\subfigure{\includegraphics[width=0.42\textwidth,
		height=45mm]{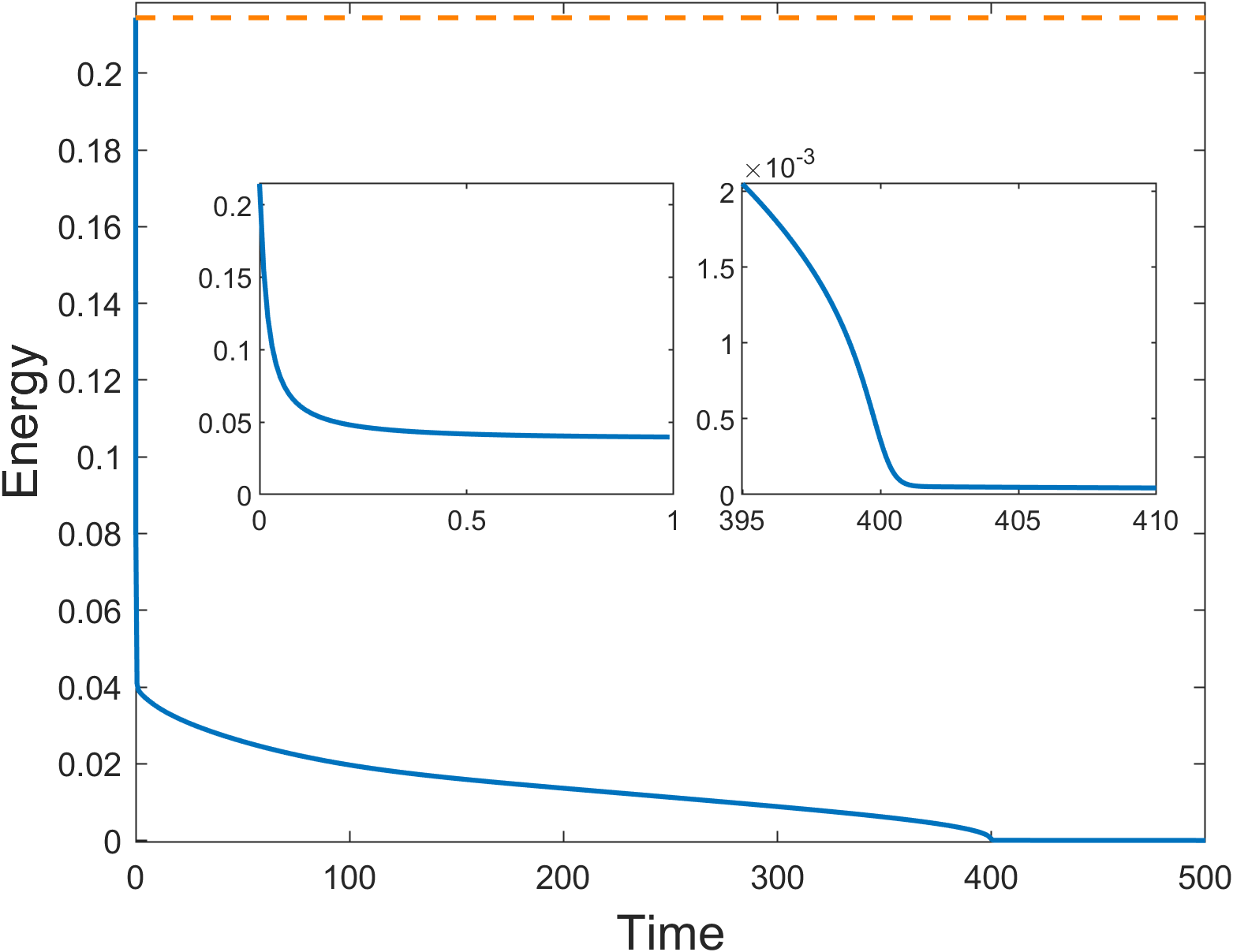}}
		\caption{Evolution of the supremum norm $\|\cdot\|_{\mathcal{X}}$ and energy with initial condition \eqref{eq:4.2} and $\alpha(x,y)=\frac{\pi}{2}\sin(2\pi(x+y))$. The dashed line in the left figure is the maximum bound $\sqrt m$ while the dashed line in the right figure is the initial energy.}
  \label{fig:4.6}
\end{figure}

{\bf Example 3.}  We consider the matrix-valued Allen--Cahn equation \eqref{eq:mac} within the domain $\Omega$, where $\varepsilon$ is set to 0.01, and the initial condition is specified as
\begin{equation}\label{eq:4.3}
     \begin{aligned}
	U^0(x,y)=\begin{cases}
	\begin{bmatrix}
		\cos\alpha_1&-\sin\alpha_1\\
		\sin\alpha_1&\cos\alpha_1
	\end{bmatrix}\quad \mbox{if}~ |x|>0.25,\\
	\\
	\begin{bmatrix}
		\cos\alpha_2&\sin\alpha_2\\
		\sin\alpha_2&-\cos\alpha_2
	\end{bmatrix}\quad\text{otherwise},\\
	\end{cases}
     \end{aligned}
\end{equation}
for different choices of $\alpha_1(x,y)$ and $\alpha_2(x,y)$ satisfying $\Delta \alpha_1=0$ and $\Delta\alpha_2=0$. The time step is fixed as $\tau=0.01$. Here, we perform several experiments where the initial condition has two straight parallel interfaces and consider the following three cases:
\begin{align*}
(\mbox{\romannumeral1})\quad\alpha_1(x,y)=2\pi y, \alpha_2(x,y)=4\pi y;\\
(\mbox{\romannumeral2})\quad\alpha_1(x,y)=2\pi y, \alpha_2(x,y)=8\pi y;\\
(\mbox{\romannumeral3})\quad\alpha_1(x,y)=8\pi y, \alpha_2(x,y)=2\pi y.
\end{align*}

Figures \ref{fig:4.7}, \ref{fig:4.9}, and \ref{fig:4.11} show the evolution of the matrix-valued field and interface at different times for the three cases, respectively. We can observe that the straight interfaces have nonzero speed along their normal directions and the speed of the interface in case (\romannumeral2) is about five times higher than that in case (\romannumeral1). In addition, the dynamics in case (\romannumeral3) has an opposite direction to that in case (\romannumeral2). In all three cases, the matrix-valued field evolves towards a uniform matrix-valued field. On one hand, this phenomenon is different from that of the classical scalar Allen--Cahn. In the case of scalar Allen--Cahn, such initial solution can be a stationary solution. On the other hand, one can also interpret the scalar Allen--Cahn system as a matrix-valued Allen--Cahn with constant matrix-valued fields. All above observations are consistent with the analysis of the laws of motion in \cite{wang2019interface}. 

Figures \ref{fig:4.8}, \ref{fig:4.10}, and \ref{fig:4.12} present the evolution of the supremum norm and energy for three cases, respectively. Notably, all three cases satisfy the discrete maximum bound principle (Theorem \ref{MBP2}) and the energy dissipation law (Theorem \ref{Energy}).

\begin{figure}
	\centering
		\subfigure{\includegraphics[width=0.30\textwidth,
		height=39mm]{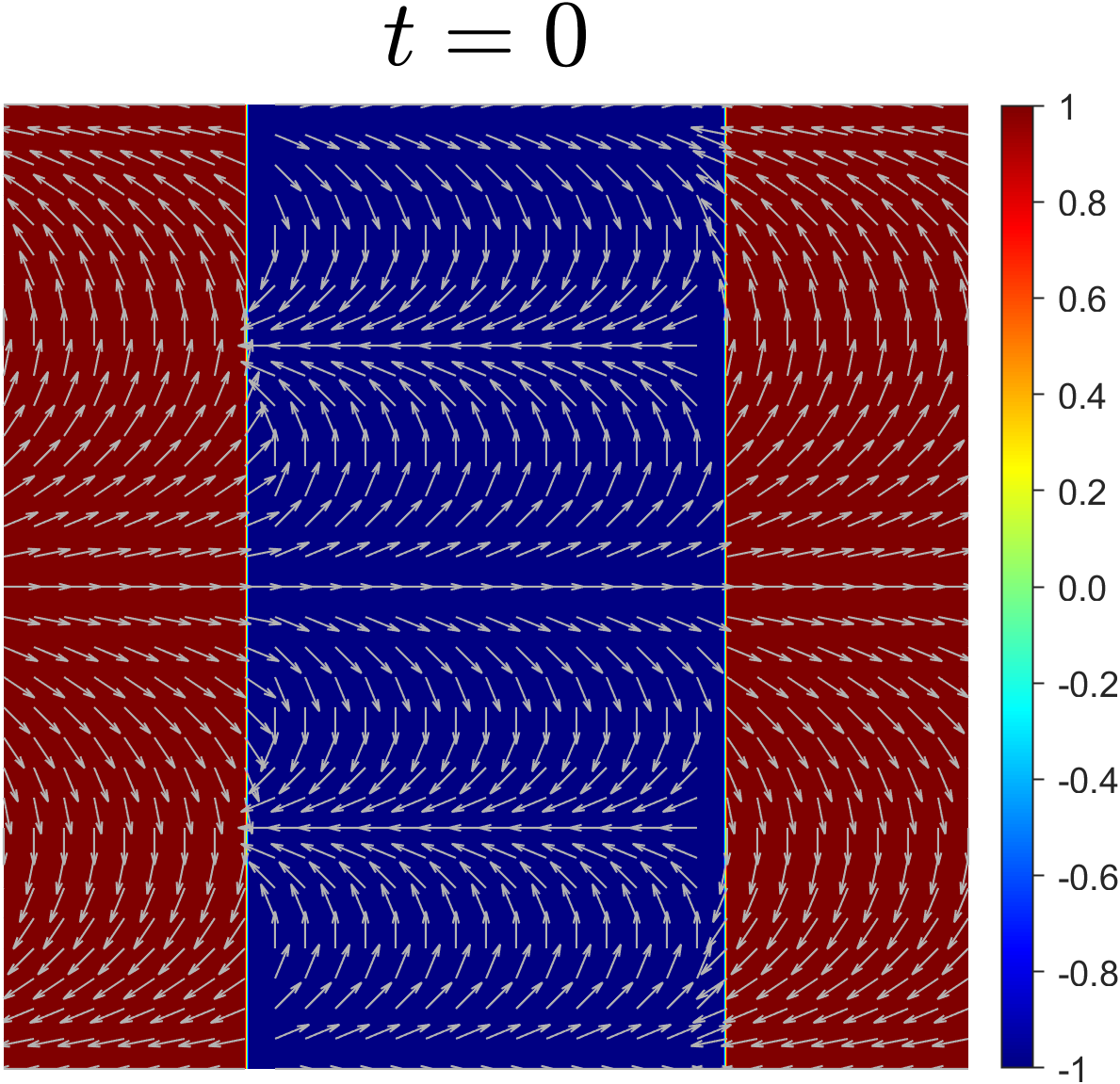}}\quad
		\subfigure{\includegraphics[width=0.30\textwidth,
		height=39mm]{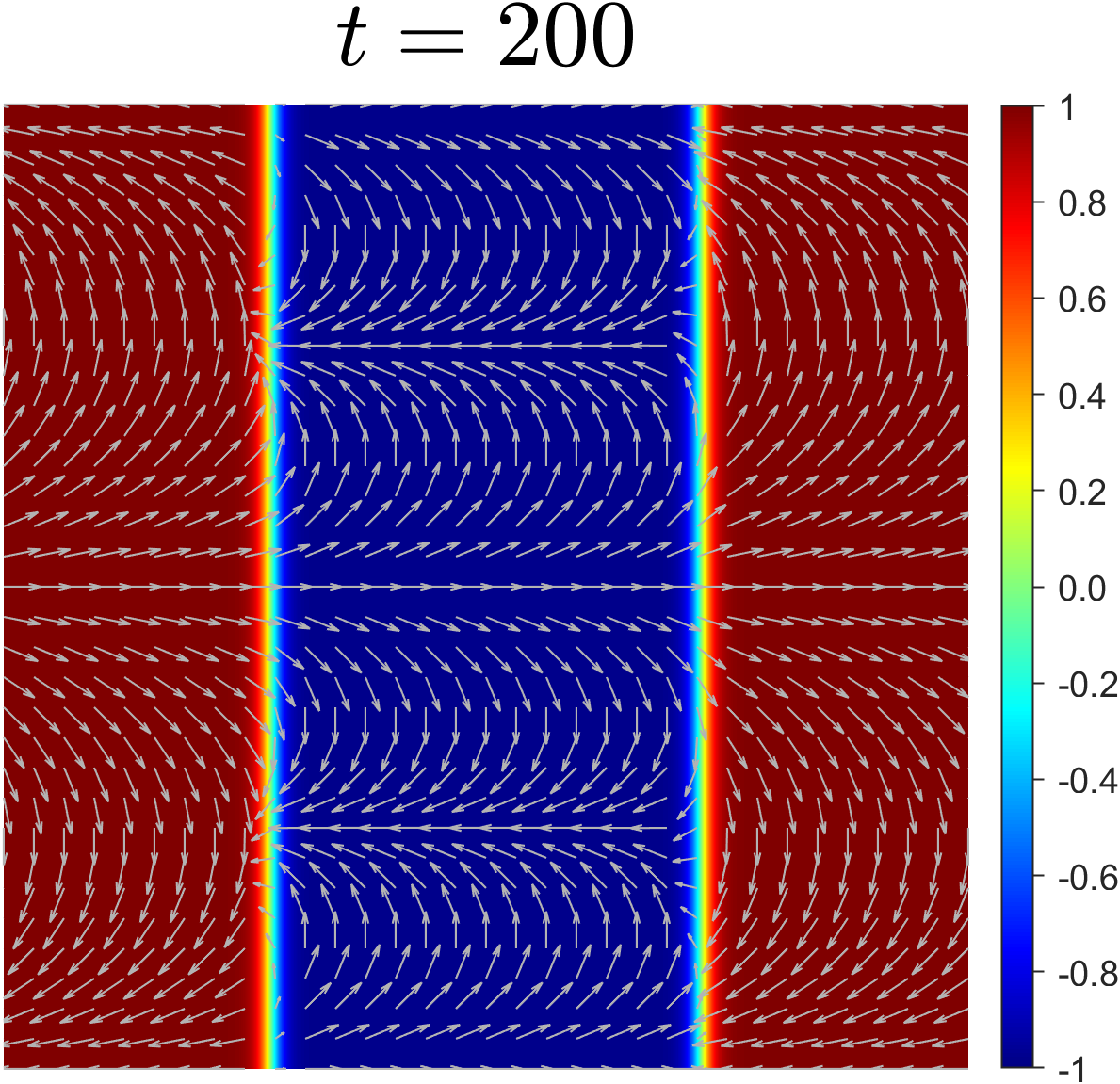}}\quad
		\subfigure{\includegraphics[width=0.30\textwidth,
		height=39mm]{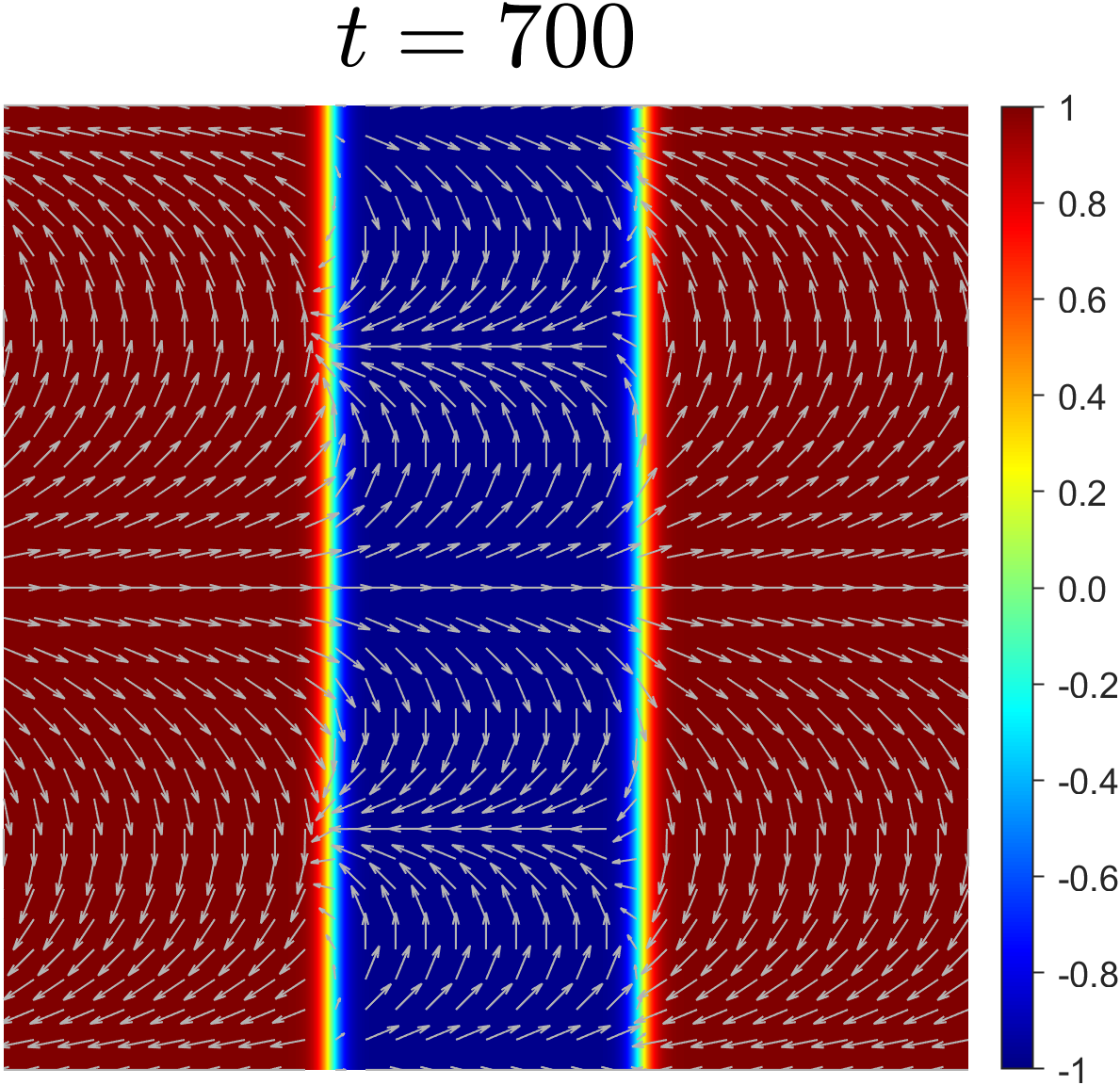}}\quad
		\subfigure{\includegraphics[width=0.30\textwidth,
		height=39mm]{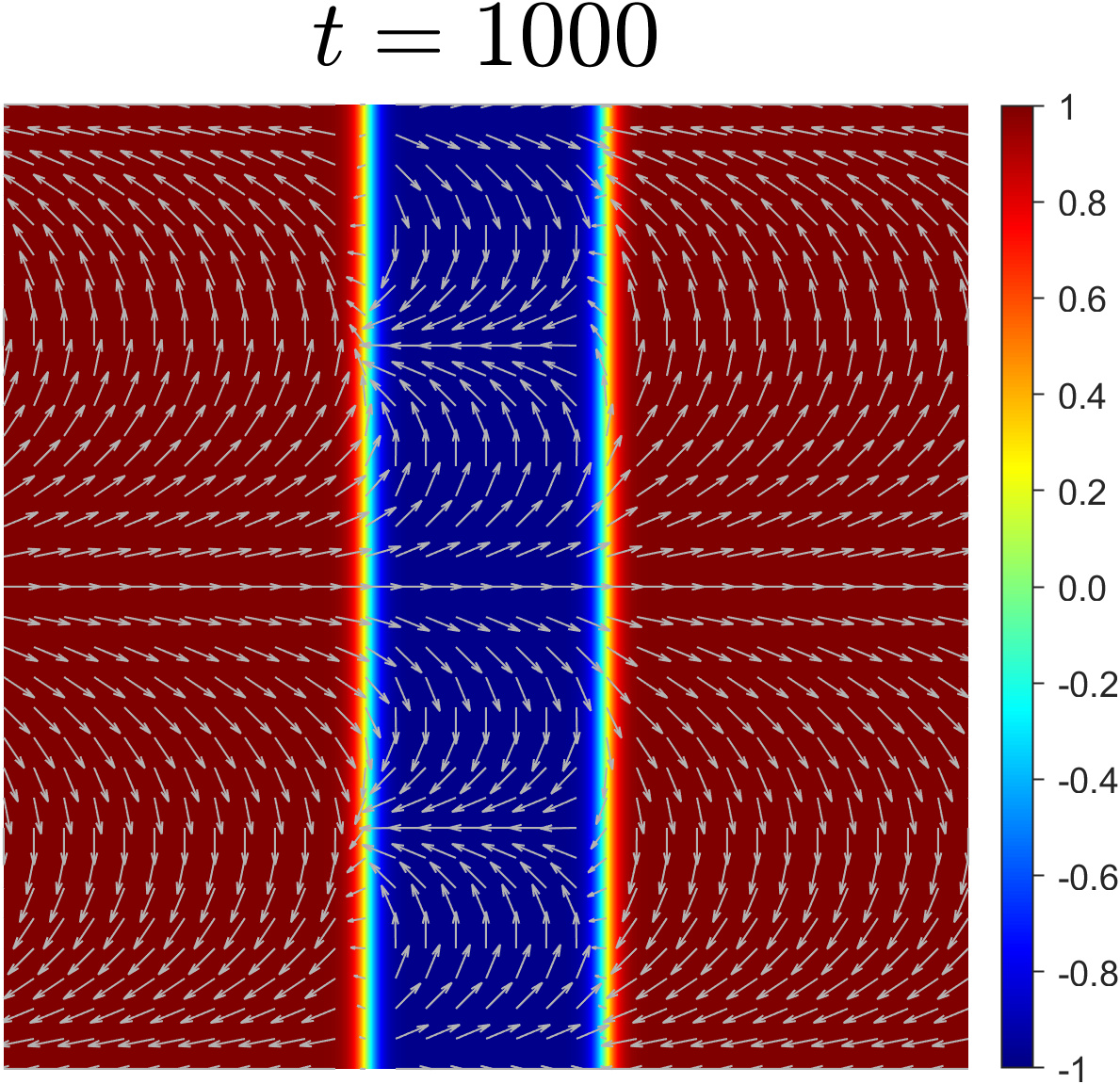}}\quad
		\subfigure{\includegraphics[width=0.30\textwidth,
		height=39mm]{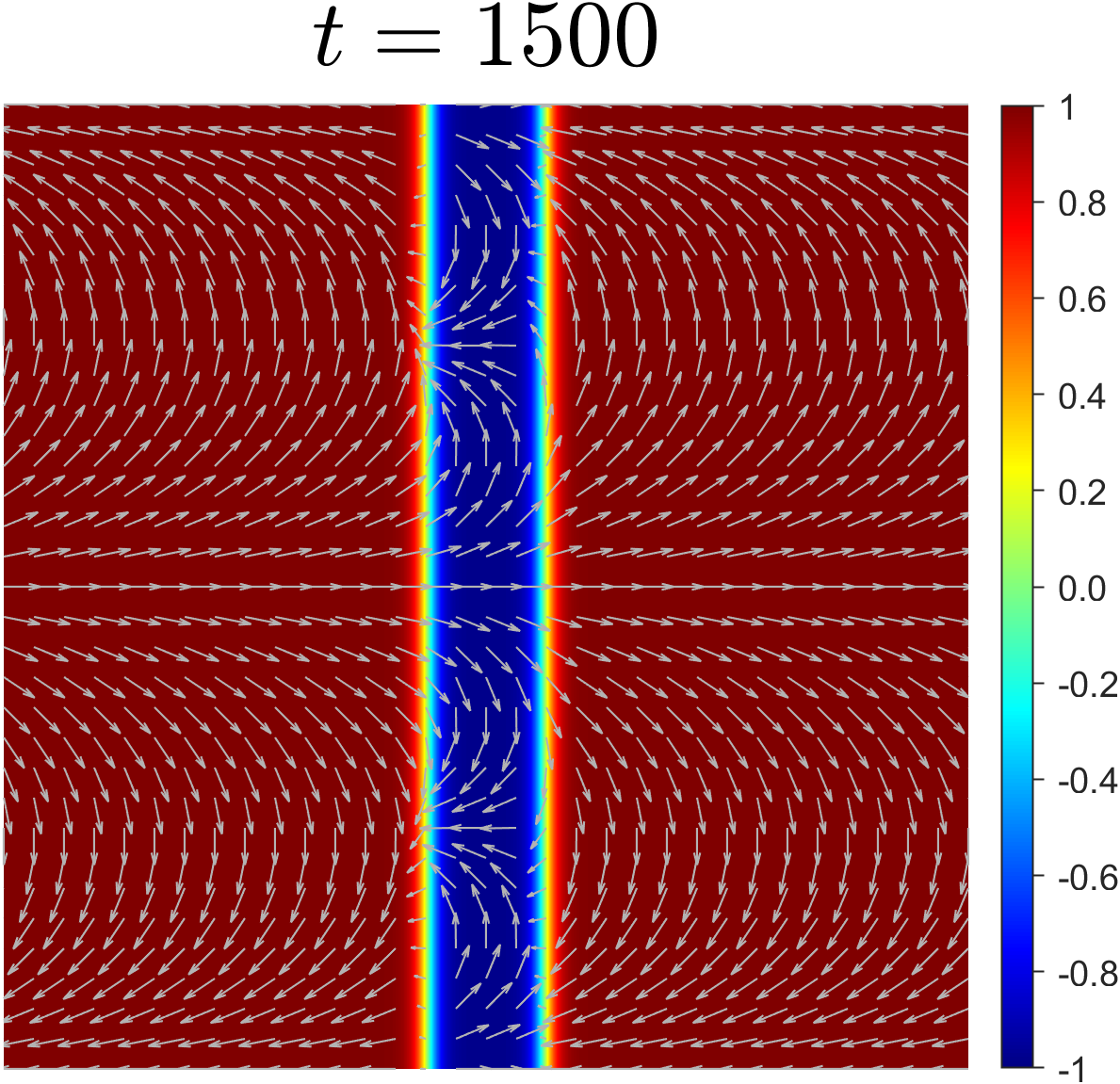}}\quad
		\subfigure{\includegraphics[width=0.30\textwidth,
		height=39mm]{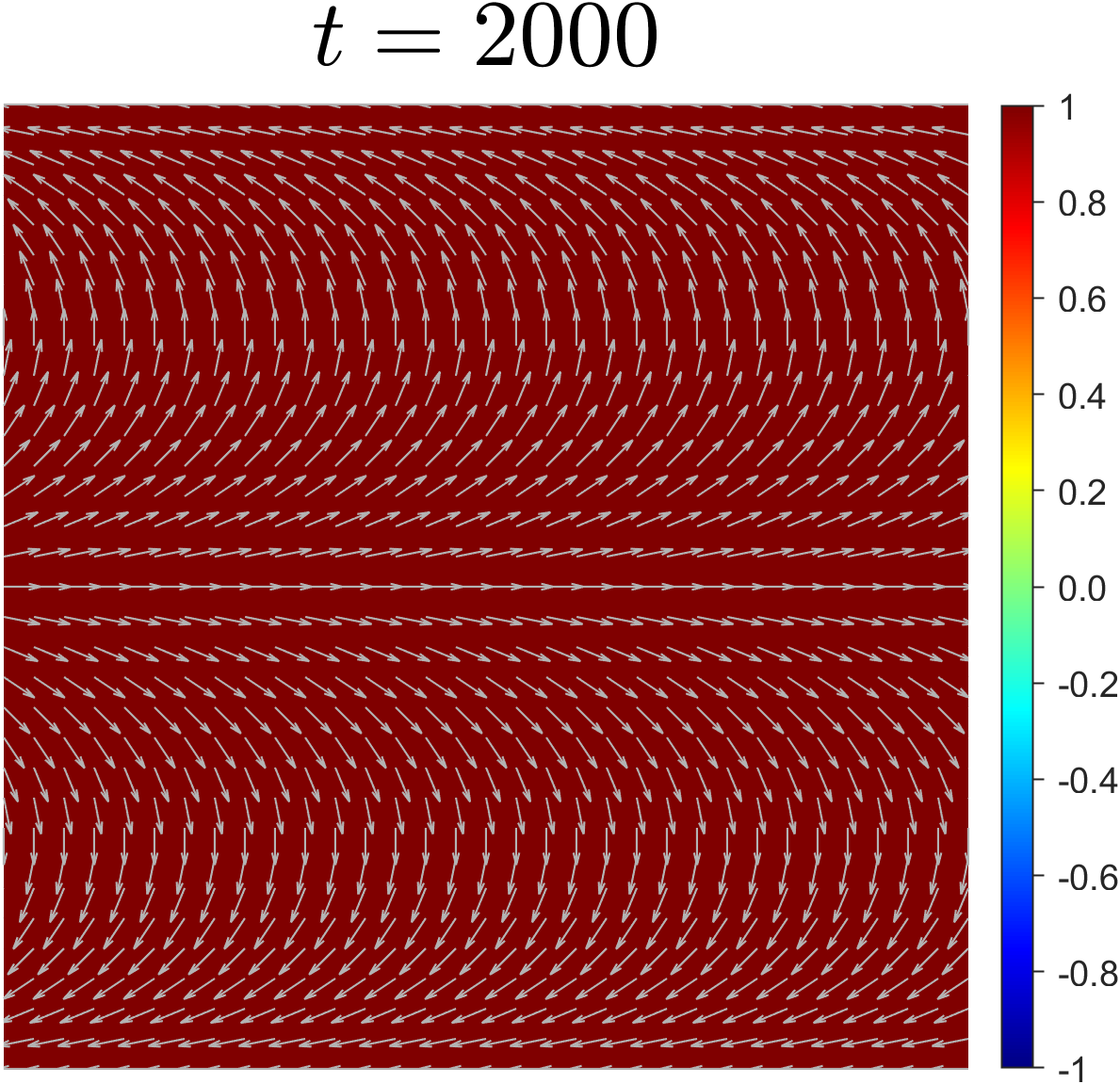}}
		\caption{Evolution of the matrix-valued field and interface at $t=0,200,700,1000,1500,2000$. The initial field is given in \eqref{eq:4.3} with $(\alpha_1,\alpha_2)=(2\pi y,4\pi y)$.}
  \label{fig:4.7}
\end{figure}

\begin{figure}
	\centering
		\subfigure{\includegraphics[width=0.42\textwidth,
		height=45mm]{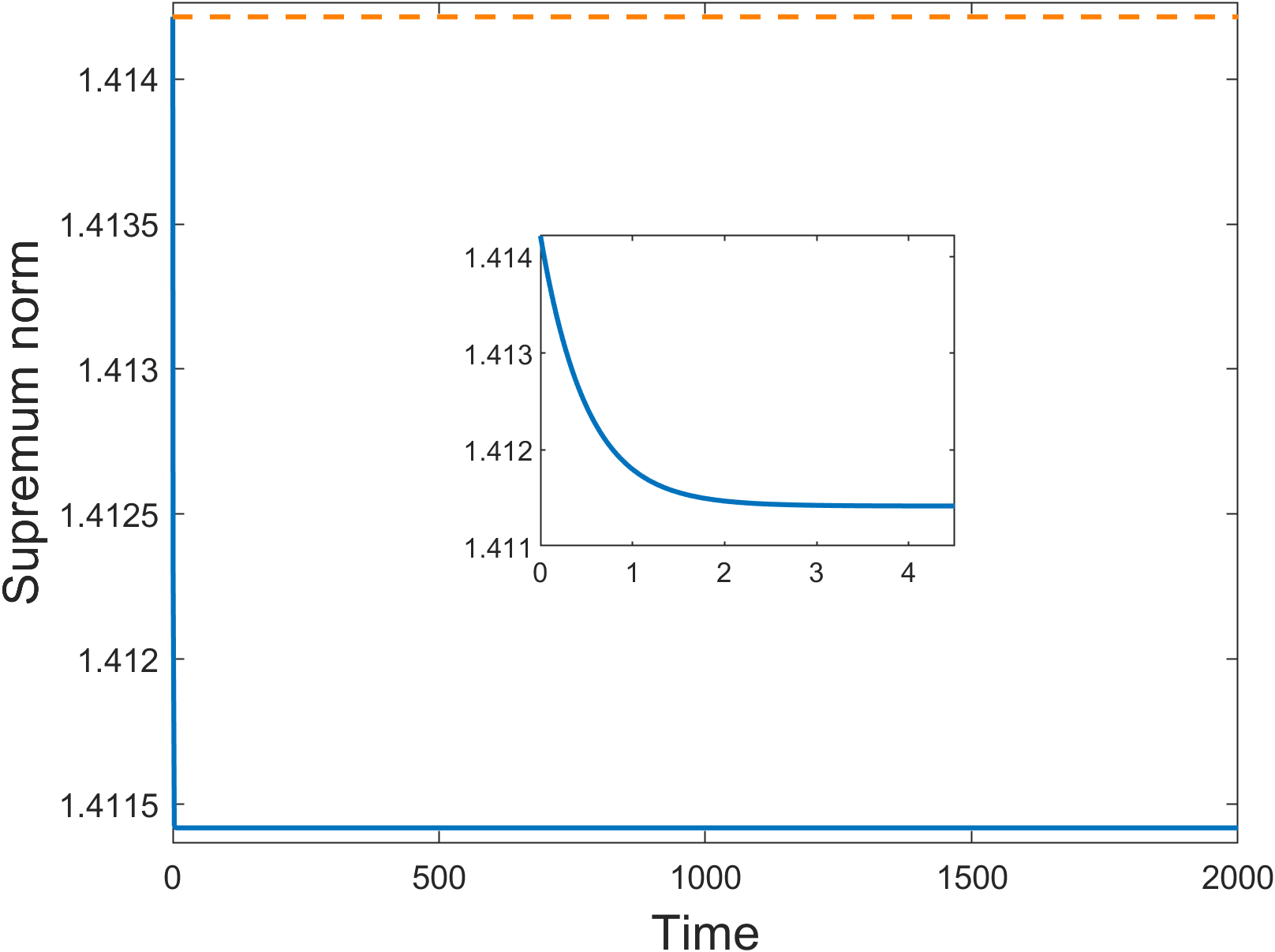}}\quad
		\subfigure{\includegraphics[width=0.42\textwidth,
		height=45mm]{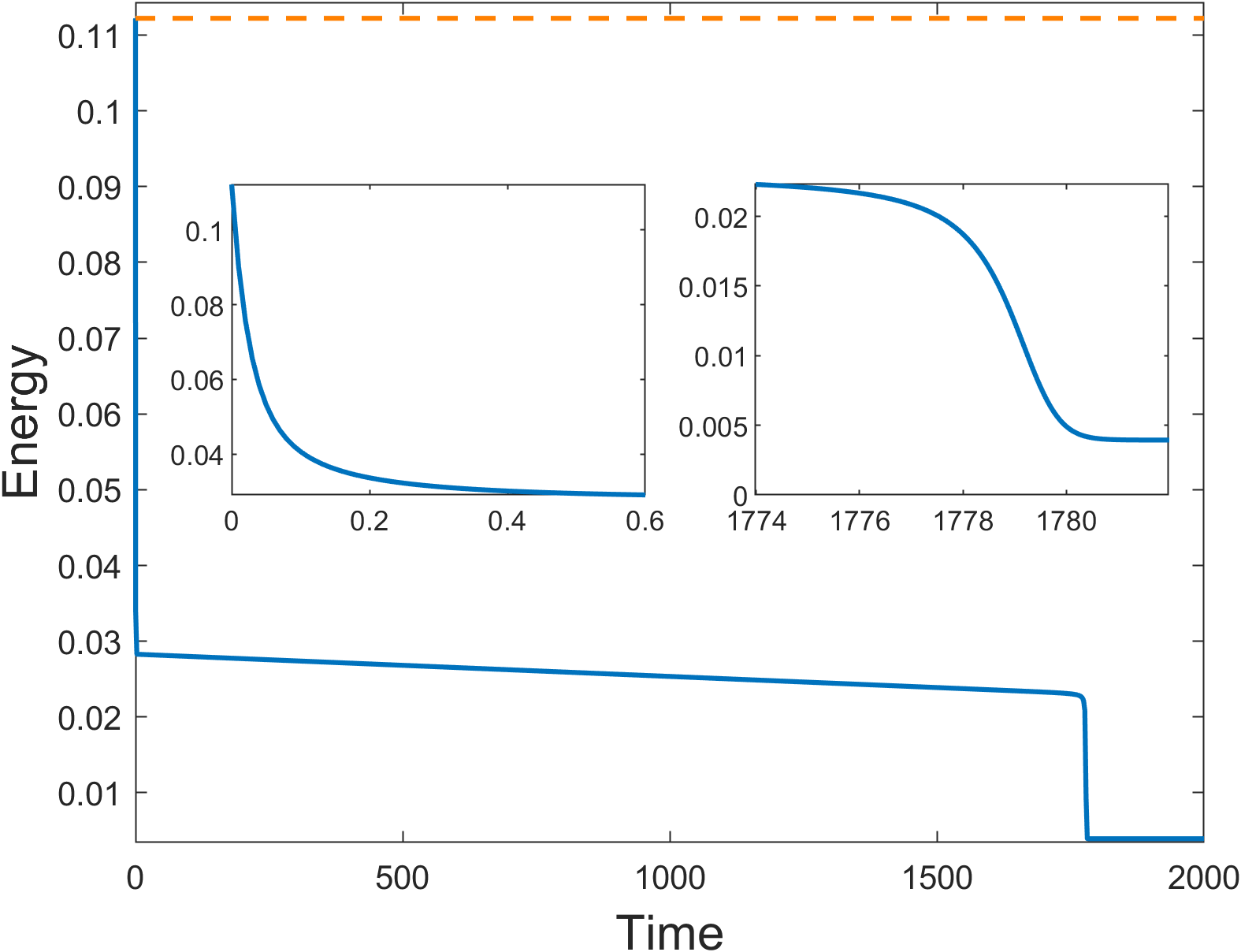}}
		\caption{Evolution of the supremum norm $\|\cdot\|_{\mathcal{X}}$ and energy with initial condition \eqref{eq:4.3} and $(\alpha_1,\alpha_2)=(2\pi y,4\pi y)$. The dashed line in the left figure is the maximum bound $\sqrt m$ while the dashed line in the right figure is the initial energy.}
  \label{fig:4.8}
\end{figure}
 
 \begin{figure}
	\centering
		\subfigure{\includegraphics[width=0.30\textwidth,
		height=39mm]{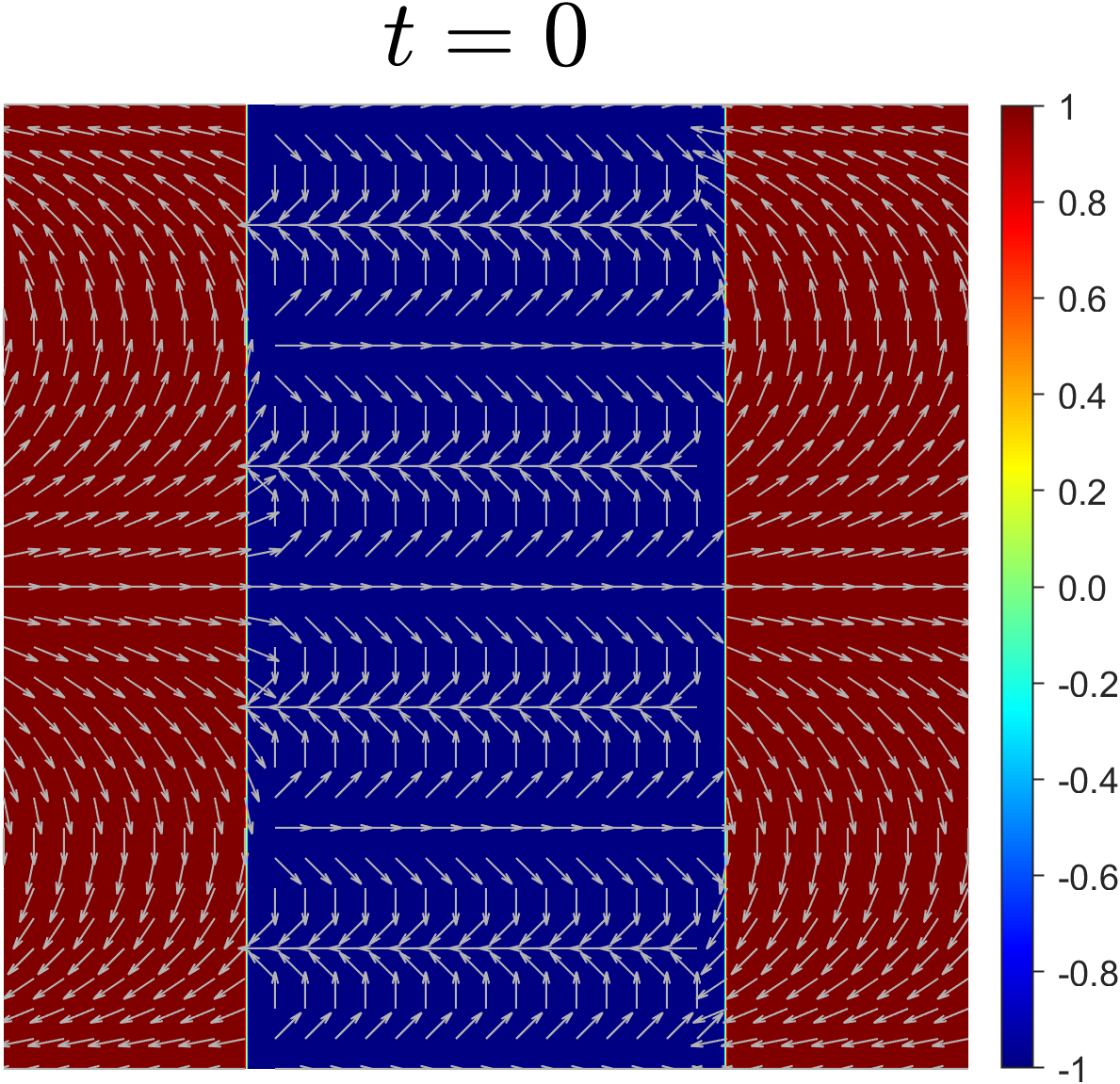}}\quad
		\subfigure{\includegraphics[width=0.30\textwidth,
		height=39mm]{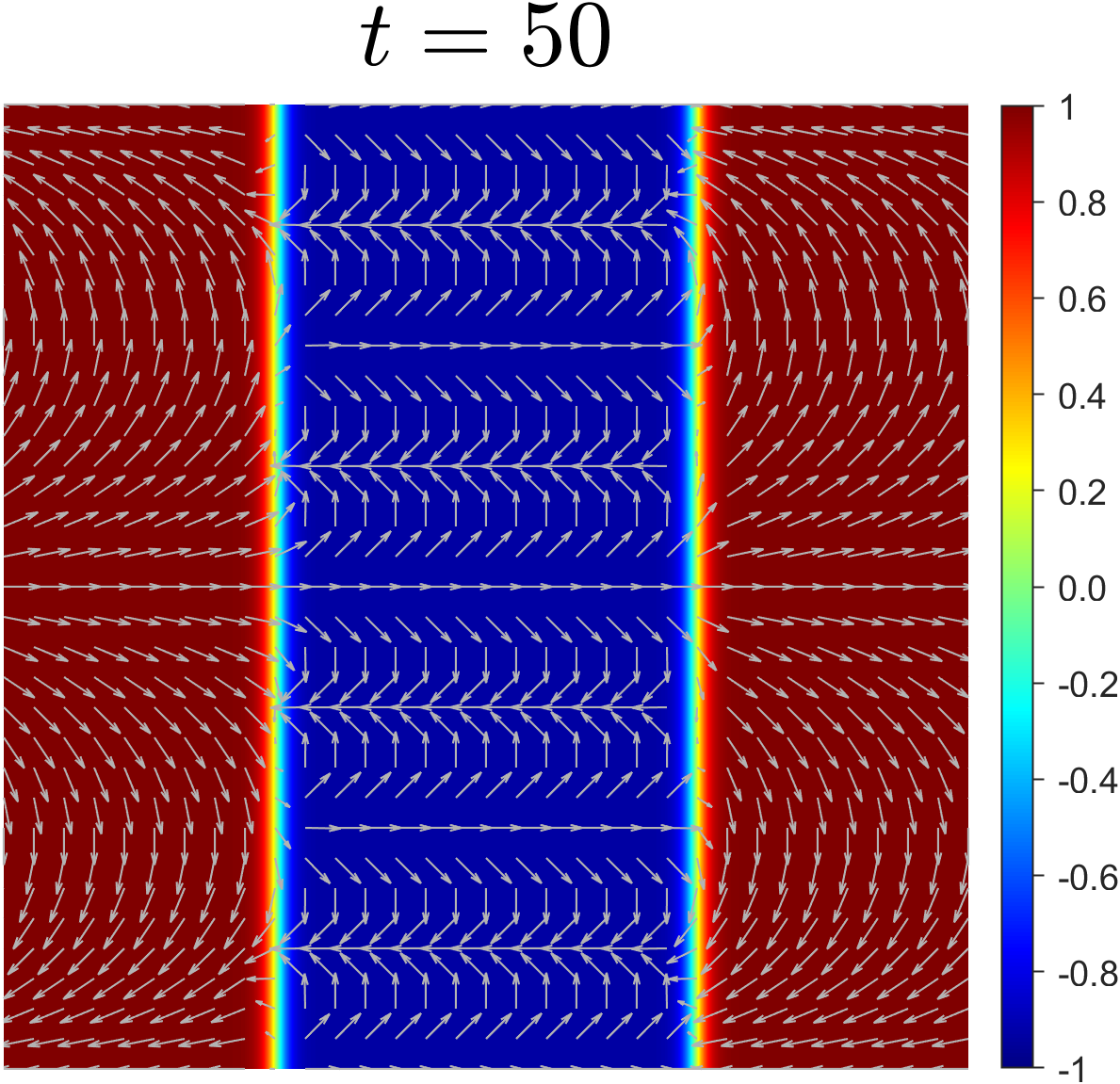}}\quad
		\subfigure{\includegraphics[width=0.30\textwidth,
		height=39mm]{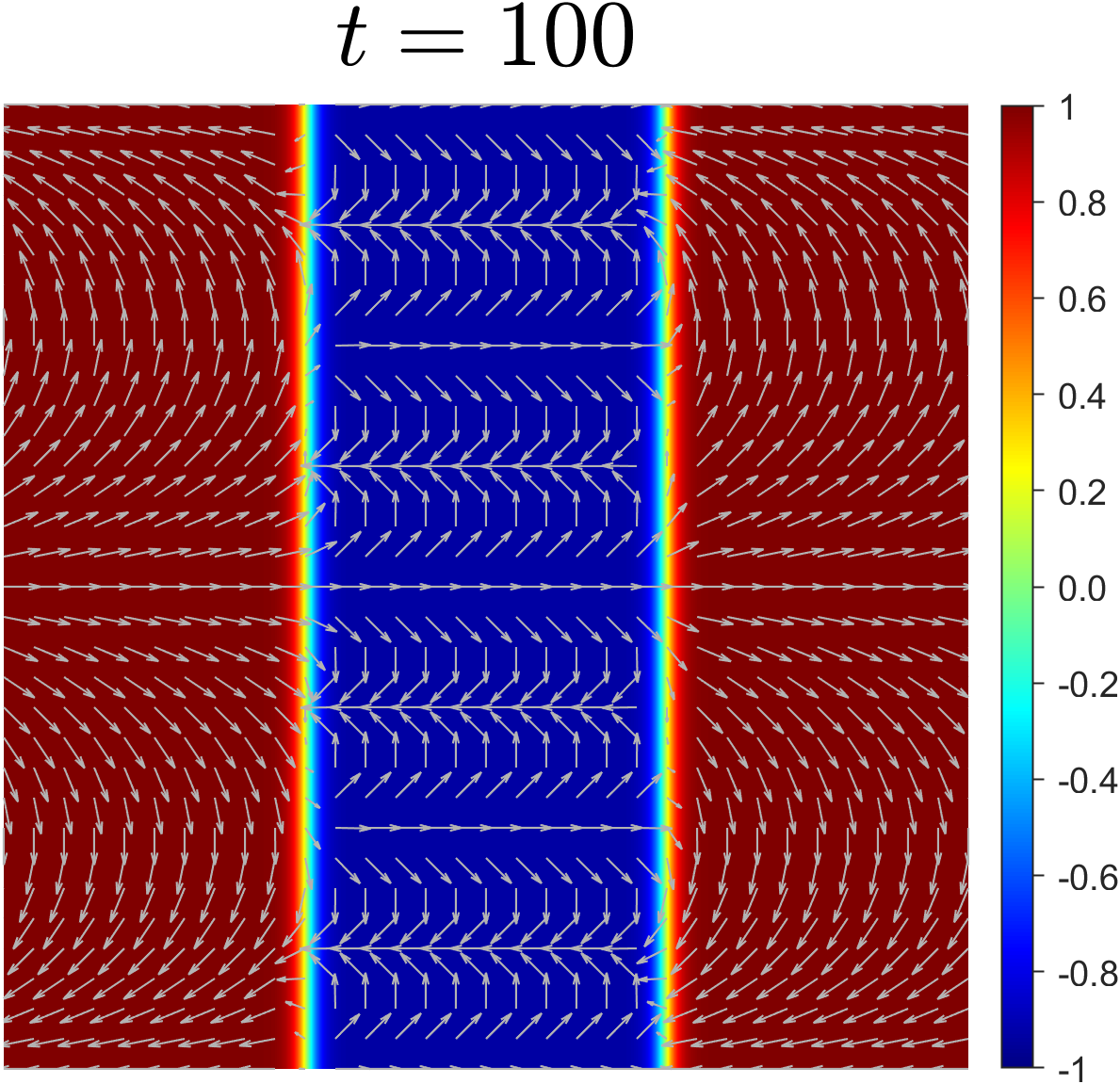}}\quad
		\subfigure{\includegraphics[width=0.30\textwidth,
		height=39mm]{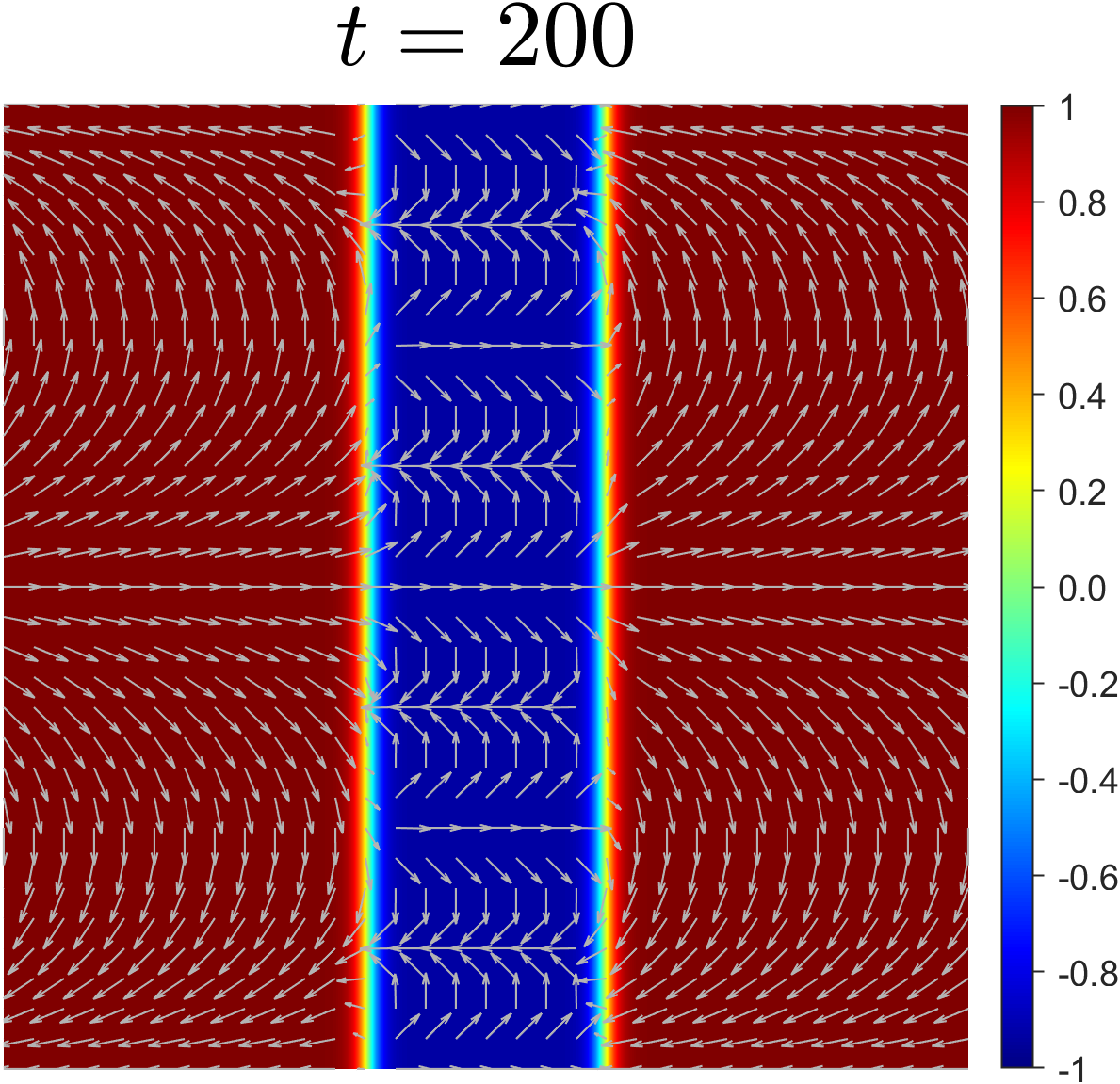}}\quad
		\subfigure{\includegraphics[width=0.30\textwidth,
		height=39mm]{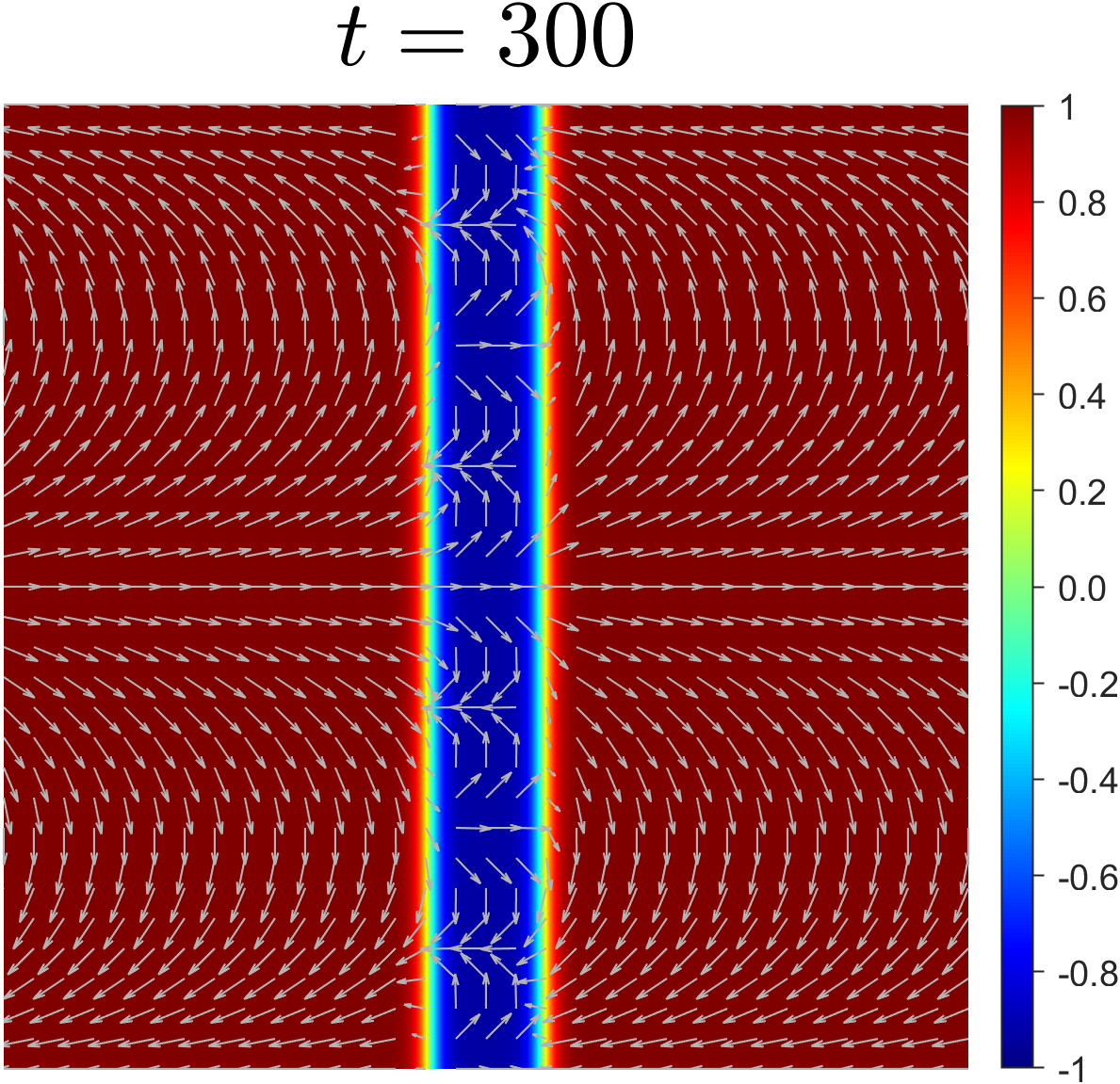}}\quad
		\subfigure{\includegraphics[width=0.30\textwidth,
		height=39mm]{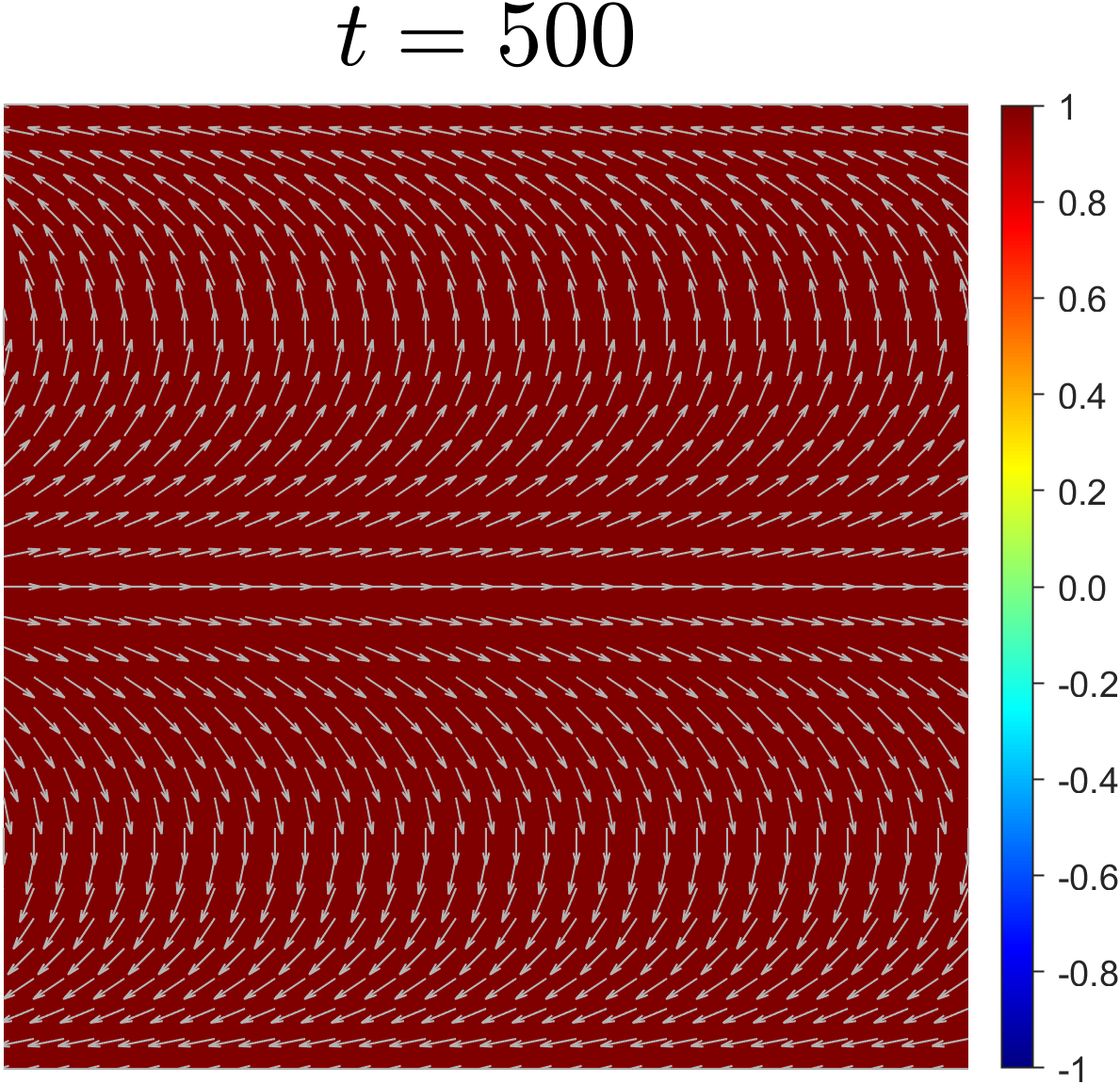}}
		\caption{Evolution of the matrix-valued field and interface at $t=0,50,100,200,300,500$. The initial field is given in \eqref{eq:4.3} with $(\alpha_1,\alpha_2)=(2\pi y,8\pi y)$.}
   \label{fig:4.9}
\end{figure}
\begin{figure}
	\centering
		\subfigure{\includegraphics[width=0.42\textwidth,
		height=45mm]{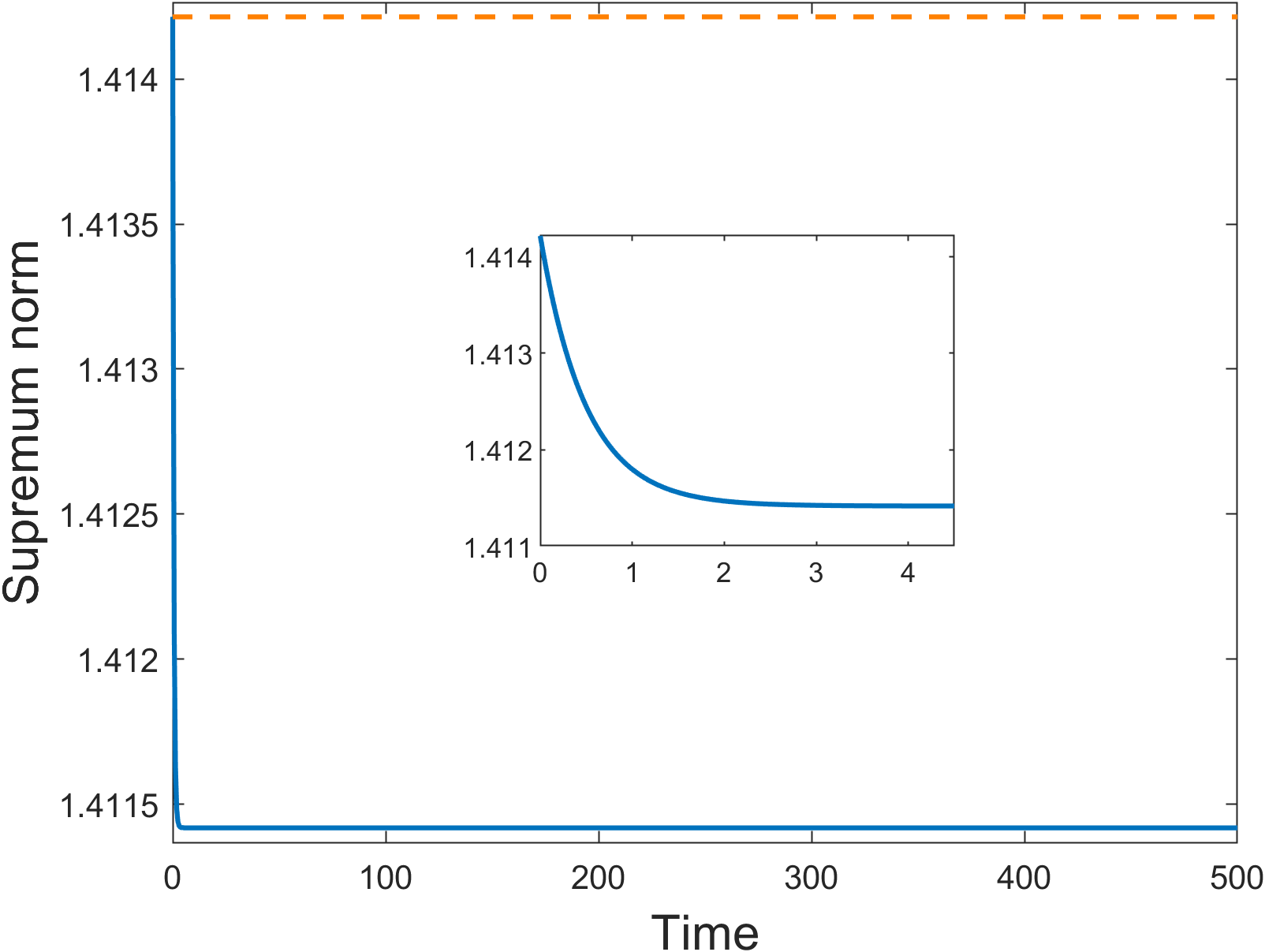}}\quad
		\subfigure{\includegraphics[width=0.42\textwidth,
		height=45mm]{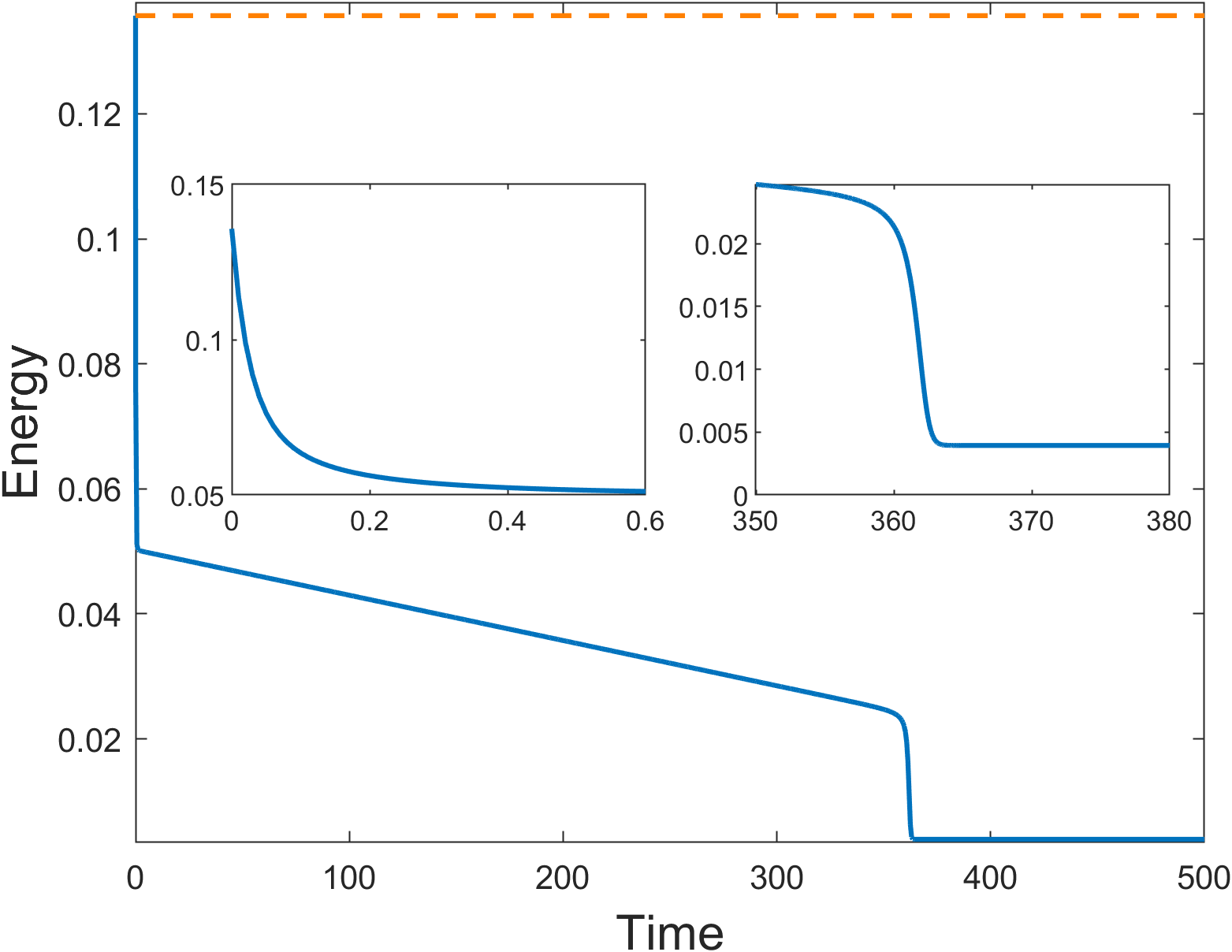}}
		\caption{Evolution of the supremum norm $\|\cdot\|_{\mathcal{X}}$ and energy with initial condition \eqref{eq:4.3} and $(\alpha_1,\alpha_2)=(2\pi y,8\pi y)$. The dashed line in the left figure is the maximum bound $\sqrt m$ while the dashed line in the right figure is the initial energy.}
    \label{fig:4.10}
\end{figure}

 \begin{figure}
	\centering
		\subfigure{\includegraphics[width=0.30\textwidth,
		height=39mm]{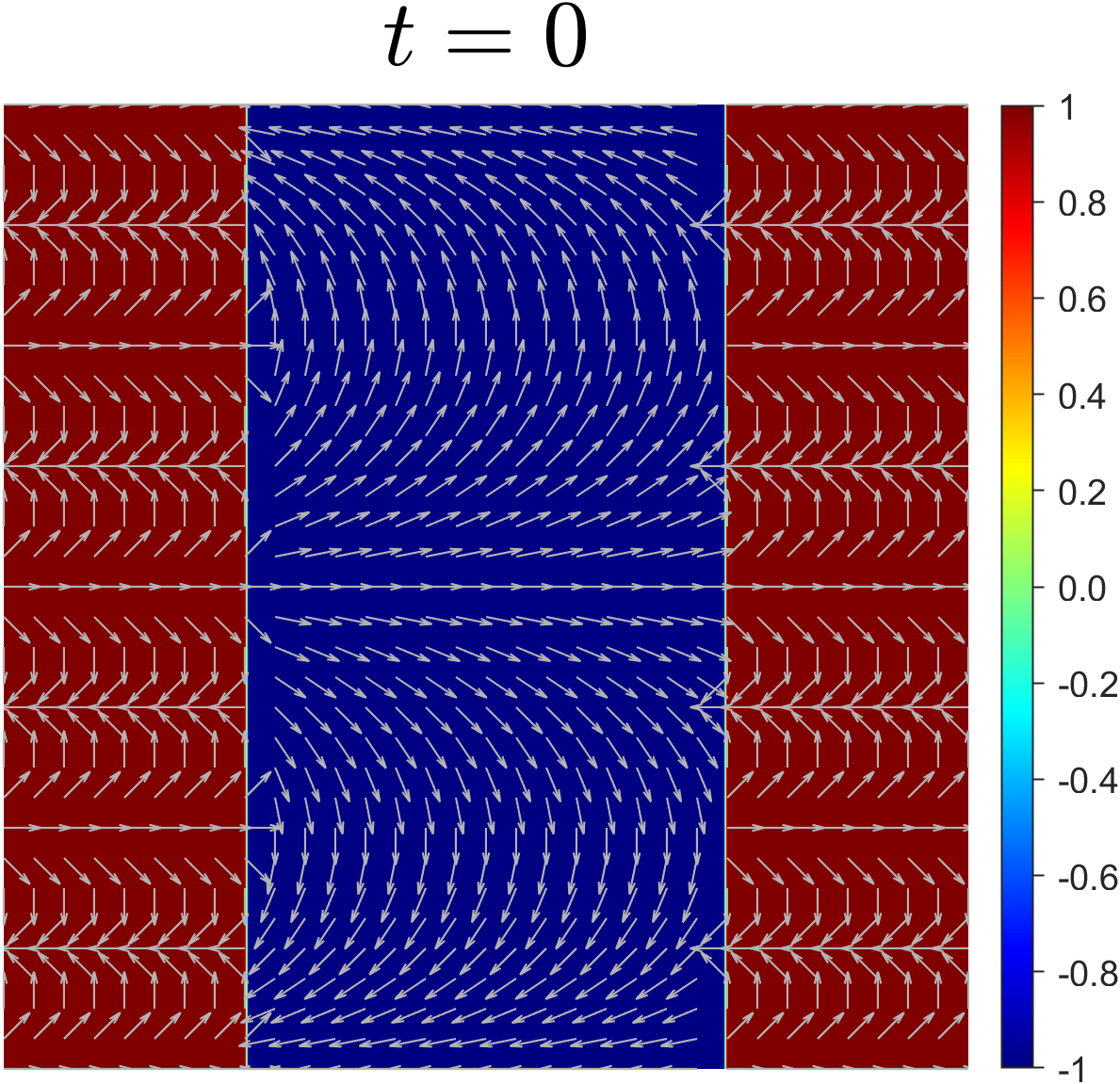}}\quad
		\subfigure{\includegraphics[width=0.30\textwidth,
		height=39mm]{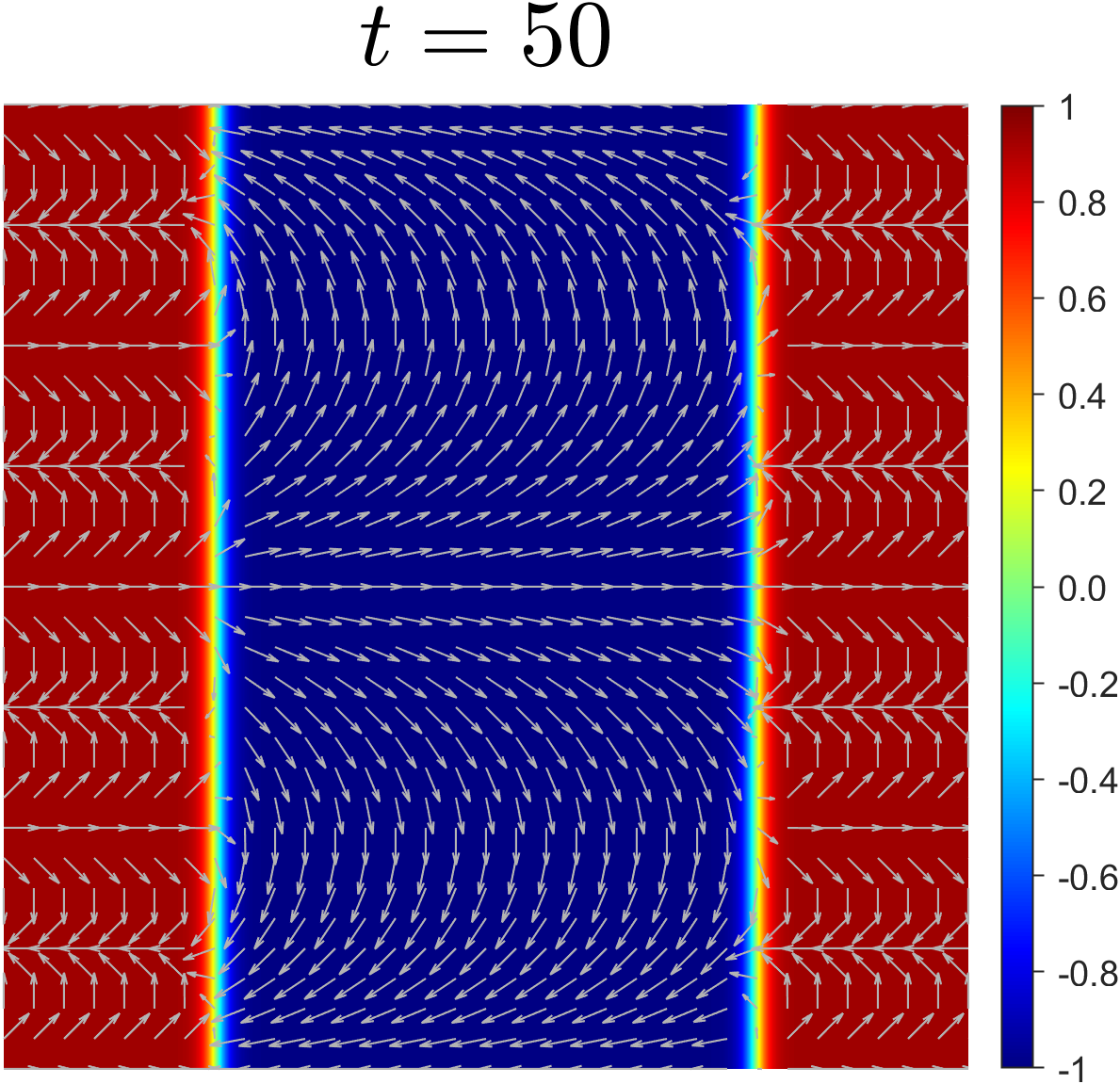}}\quad
		\subfigure{\includegraphics[width=0.30\textwidth,
		height=39mm]{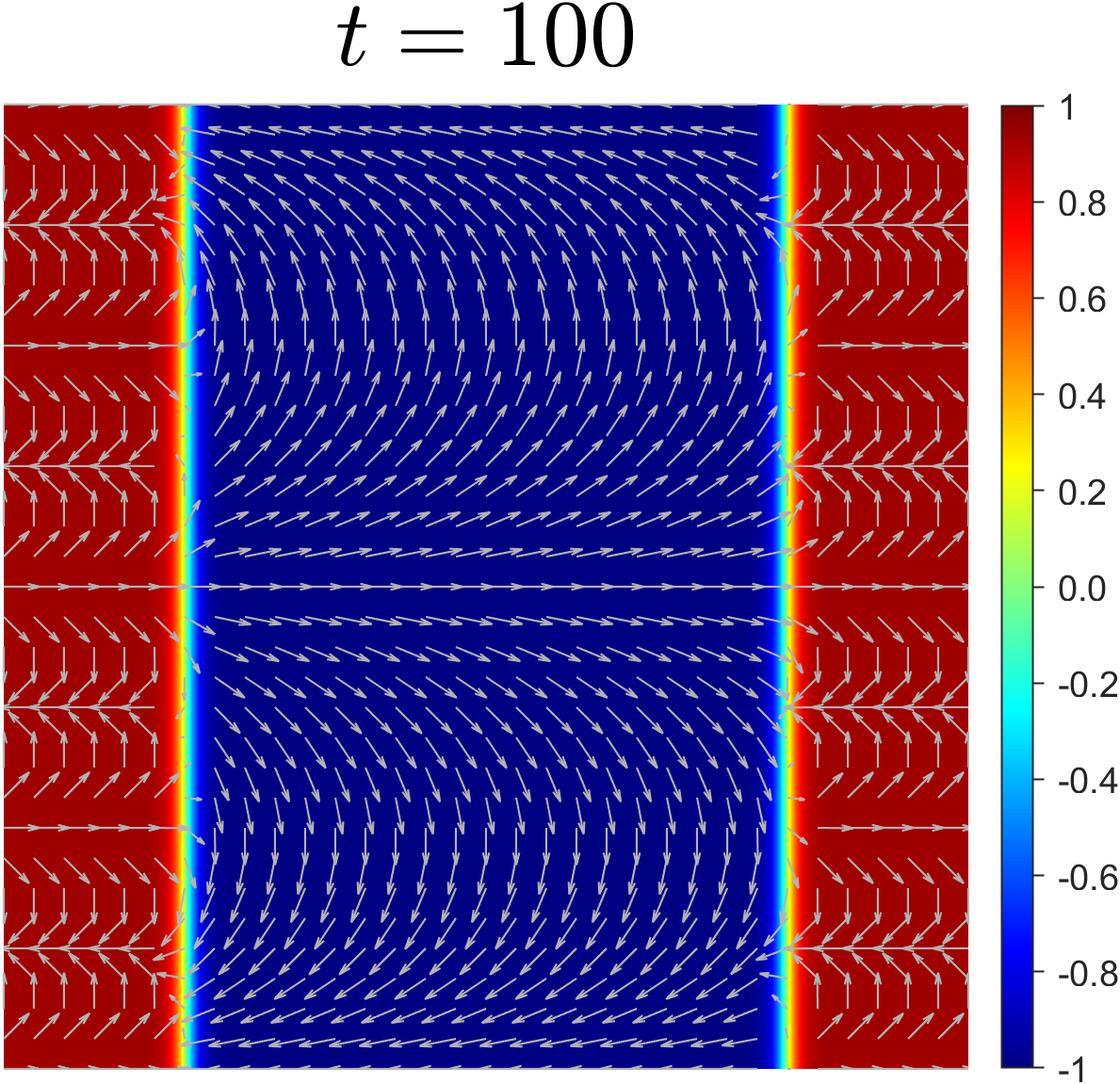}}\quad
		\subfigure{\includegraphics[width=0.30\textwidth,
		height=39mm]{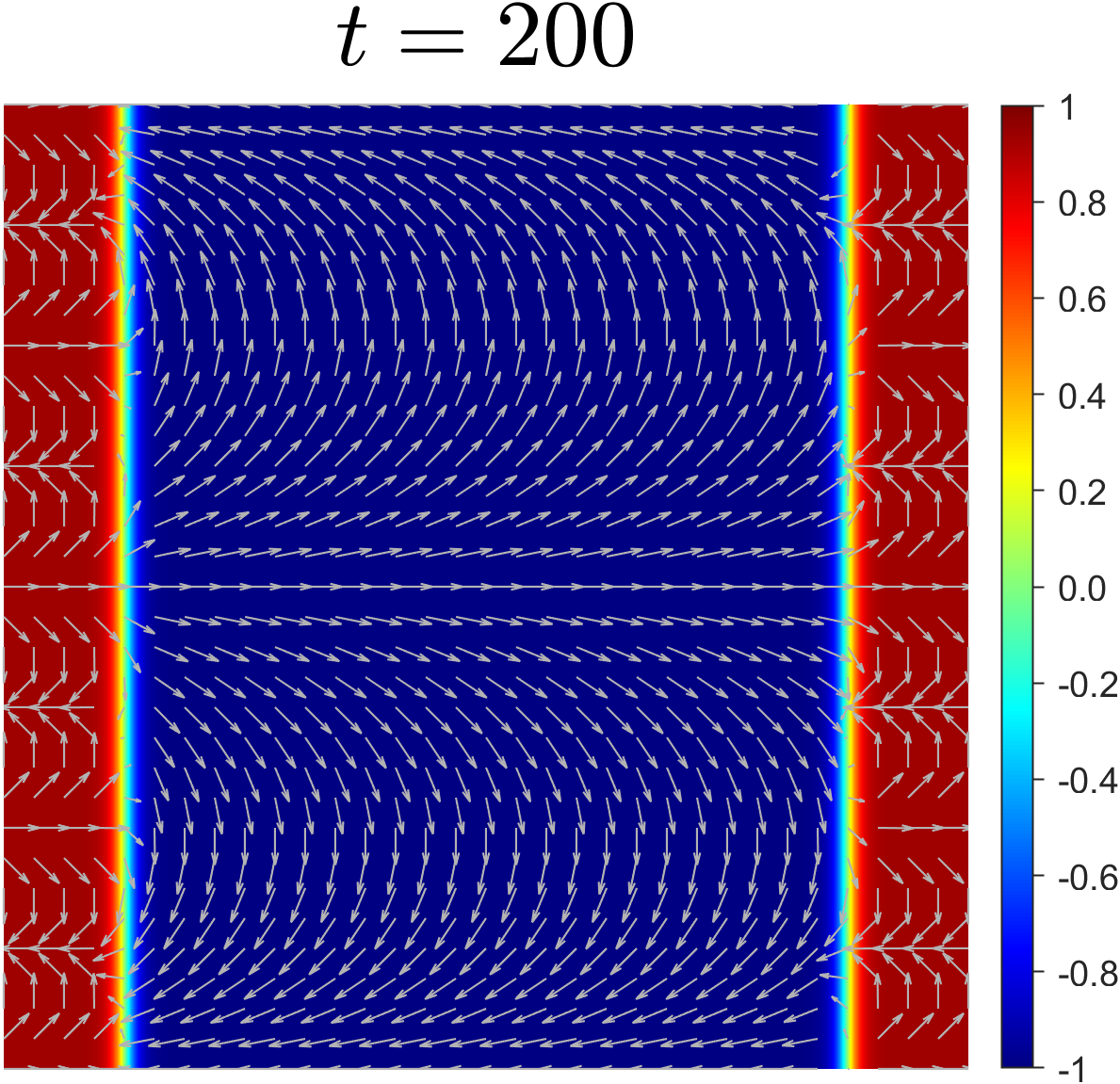}}\quad
		\subfigure{\includegraphics[width=0.30\textwidth,
		height=39mm]{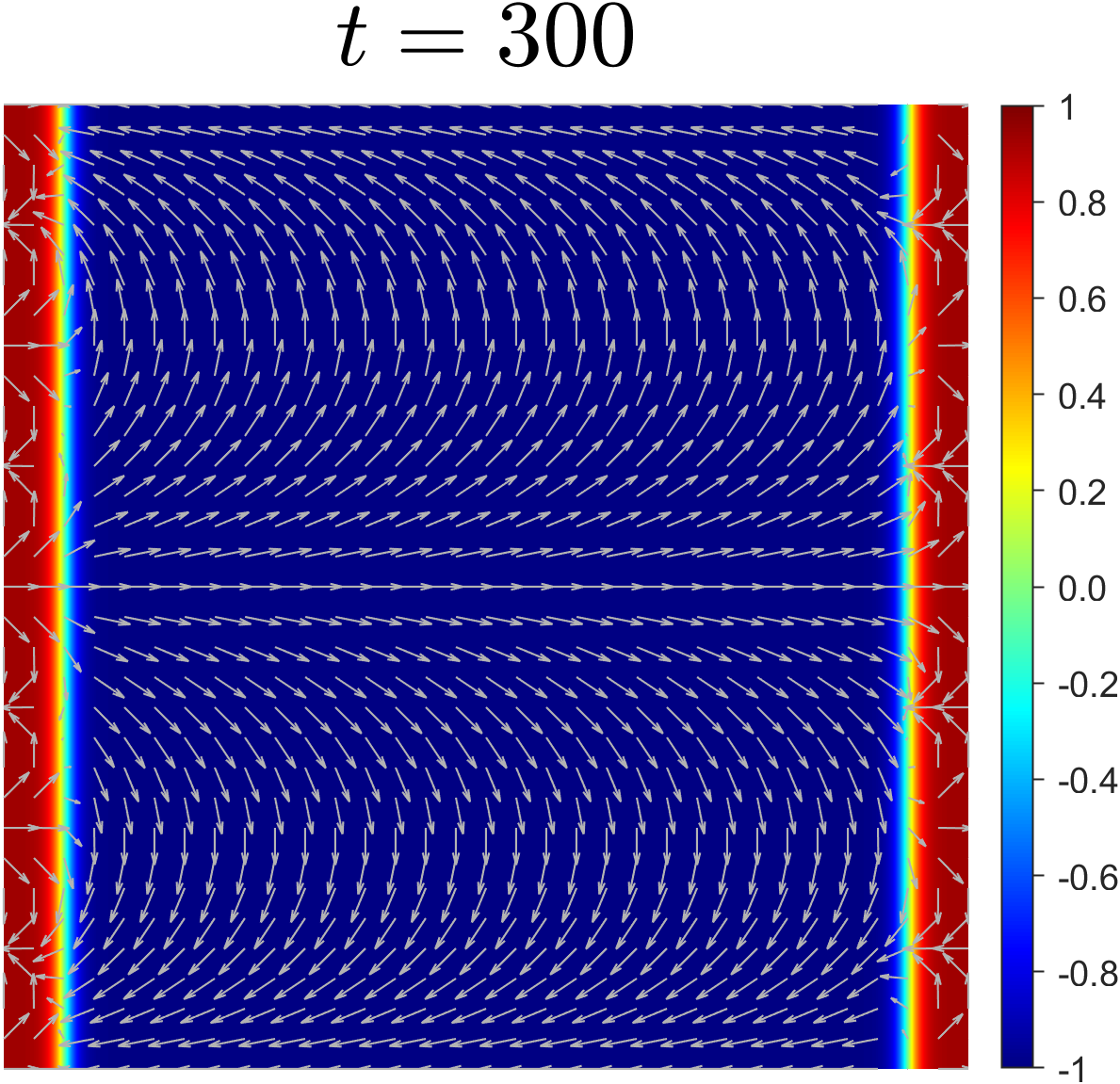}}\quad
		\subfigure{\includegraphics[width=0.30\textwidth,
		height=39mm]{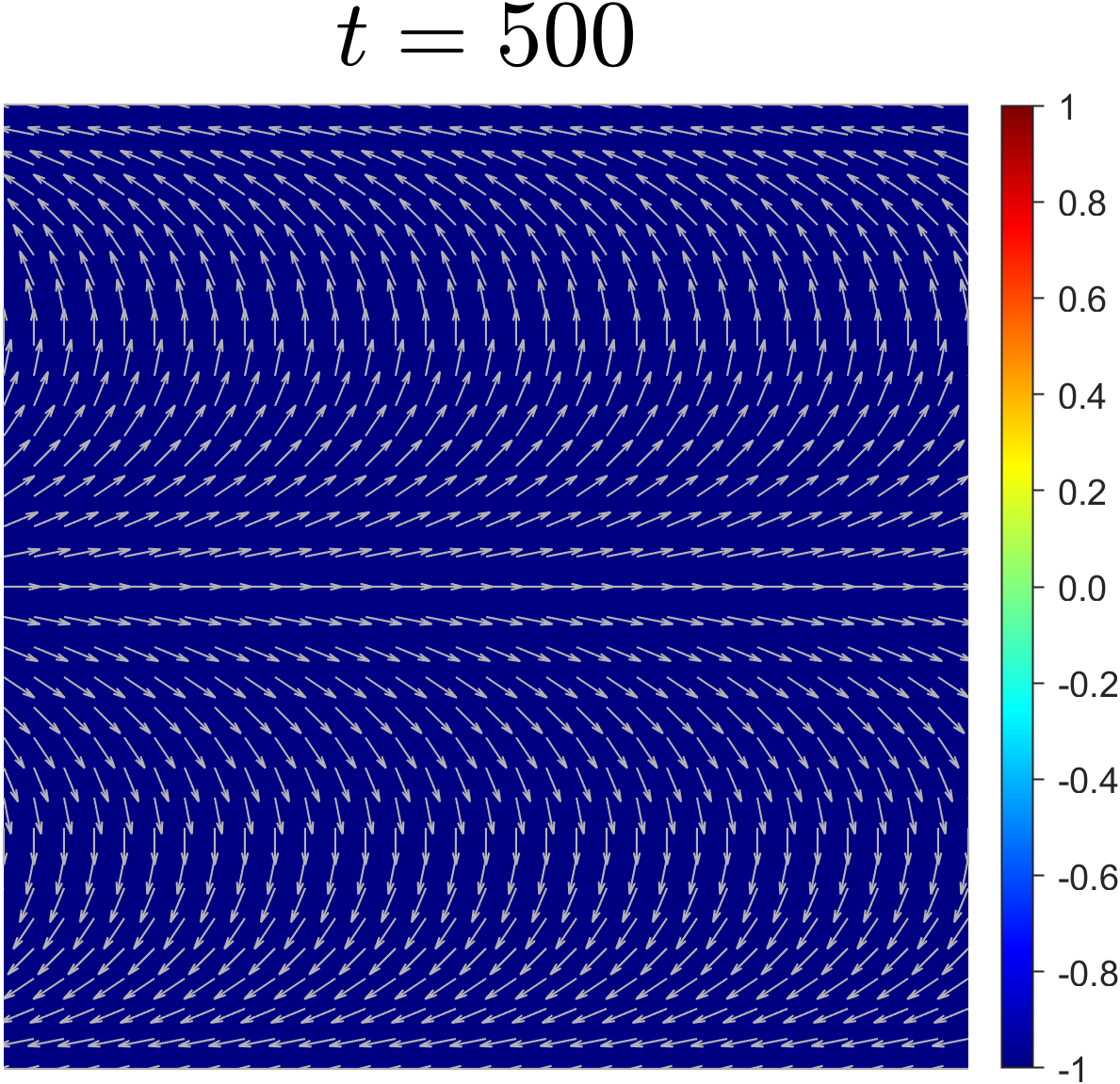}}
		\caption{Evolution of the matrix-valued field and interface at $t=0,50,100,200,300,500$. The initial field is given in \eqref{eq:4.3} with $(\alpha_1,\alpha_2)=(8\pi y,2\pi y)$.}
	\label{fig:4.11}
\end{figure}

\begin{figure}
    \centering
       \subfigure{\includegraphics[width=0.42\textwidth,
		height=45mm]{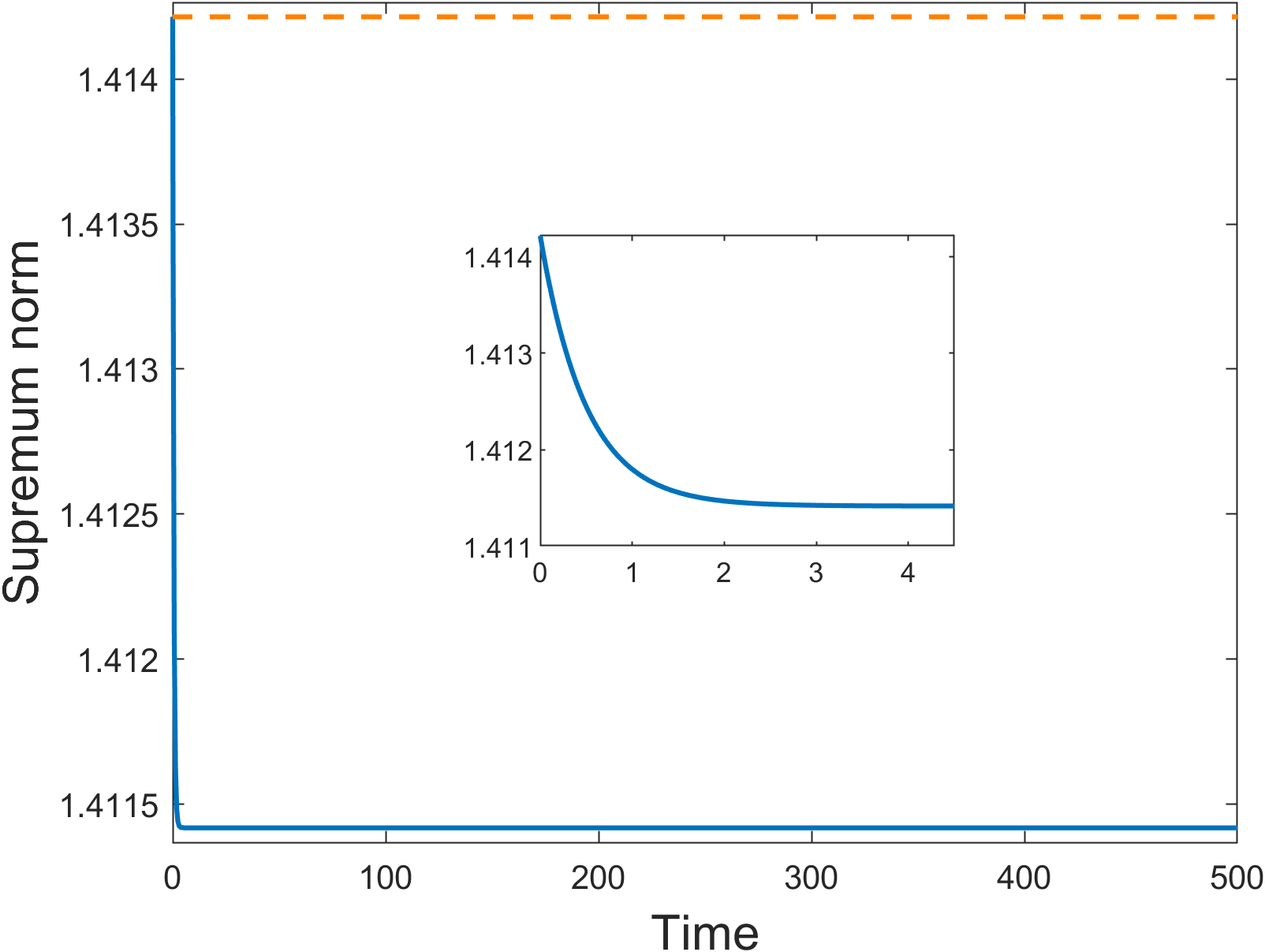}}\quad
	\subfigure{\includegraphics[width=0.42\textwidth,
		height=45mm]{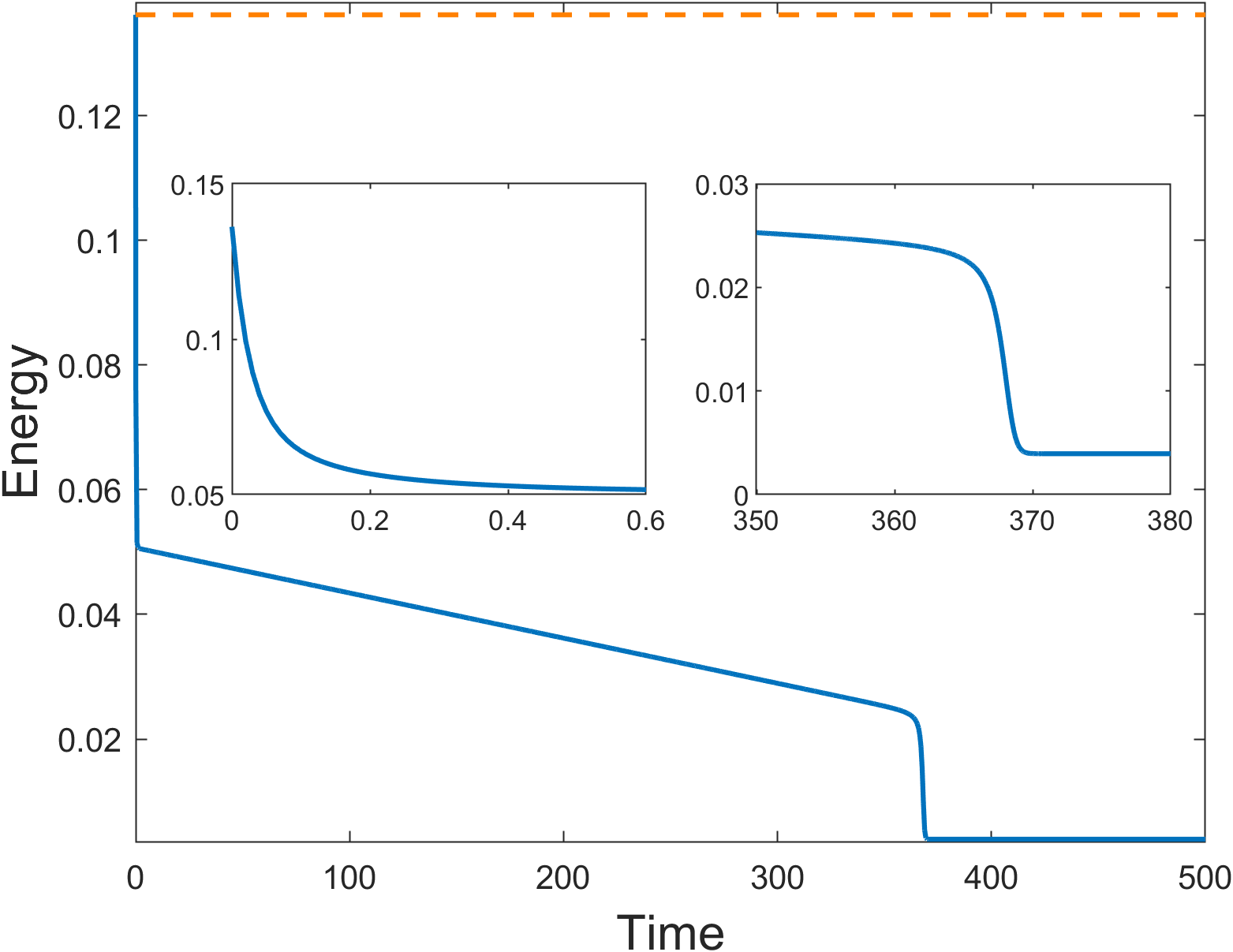}}
		\caption{Evolution of the supremum norm $\|\cdot\|_{\mathcal{X}}$ and energy with initial condition \eqref{eq:4.3} and $(\alpha_1,\alpha_2)=(8\pi y,2\pi y)$. The dashed line in the left figure is the maximum bound $\sqrt m$ while the dashed line in the right figure is the initial energy.}
   \label{fig:4.12}
\end{figure}

{\bf Example 4.} We explore the matrix-valued Allen--Cahn equation \eqref{eq:mac} across the domain $\Omega$, with $\varepsilon$ fixed at 0.01, and initializing with condition
	\begin{equation}\label{eq:4.4}
		\begin{aligned}
			U^0(x,y)=\begin{cases}
				\begin{bmatrix}
					\cos\alpha&-\sin\alpha\\
					\sin\alpha&\cos\alpha
				\end{bmatrix}\quad \mbox{if}~ xy\leq 0,\\
				\\
				\begin{bmatrix}
					\cos\alpha&\sin\alpha\\
					\sin\alpha&-\cos\alpha
				\end{bmatrix}\quad\text{otherwise}.\\
			\end{cases}
		\end{aligned}
	\end{equation}
where $\alpha=\alpha(x,y)=\frac{\pi}{2}\sin(2\pi x)\sin(2\pi y)$. 

Figure \ref{fig:4.13} gives the evolution of the matrix-valued field and the interface. We can observe that the matrix-valued field evolves towards a uniform matrix-valued field, while the region with determinant $-1$ gradually shrinks to a circle as the interface evolves.

Figure \ref{fig:4.14} presents the evolution of the supremum norm and energy. It can be observed that the discrete maximum bound principle (Theorem \ref{MBP2}) and the energy dissipation law (Theorem \ref{Energy}) also hold under the initial condition \eqref{eq:4.4}.

\begin{figure}
    \centering
	\subfigure{\includegraphics[width=0.30\textwidth,
		height=39mm]{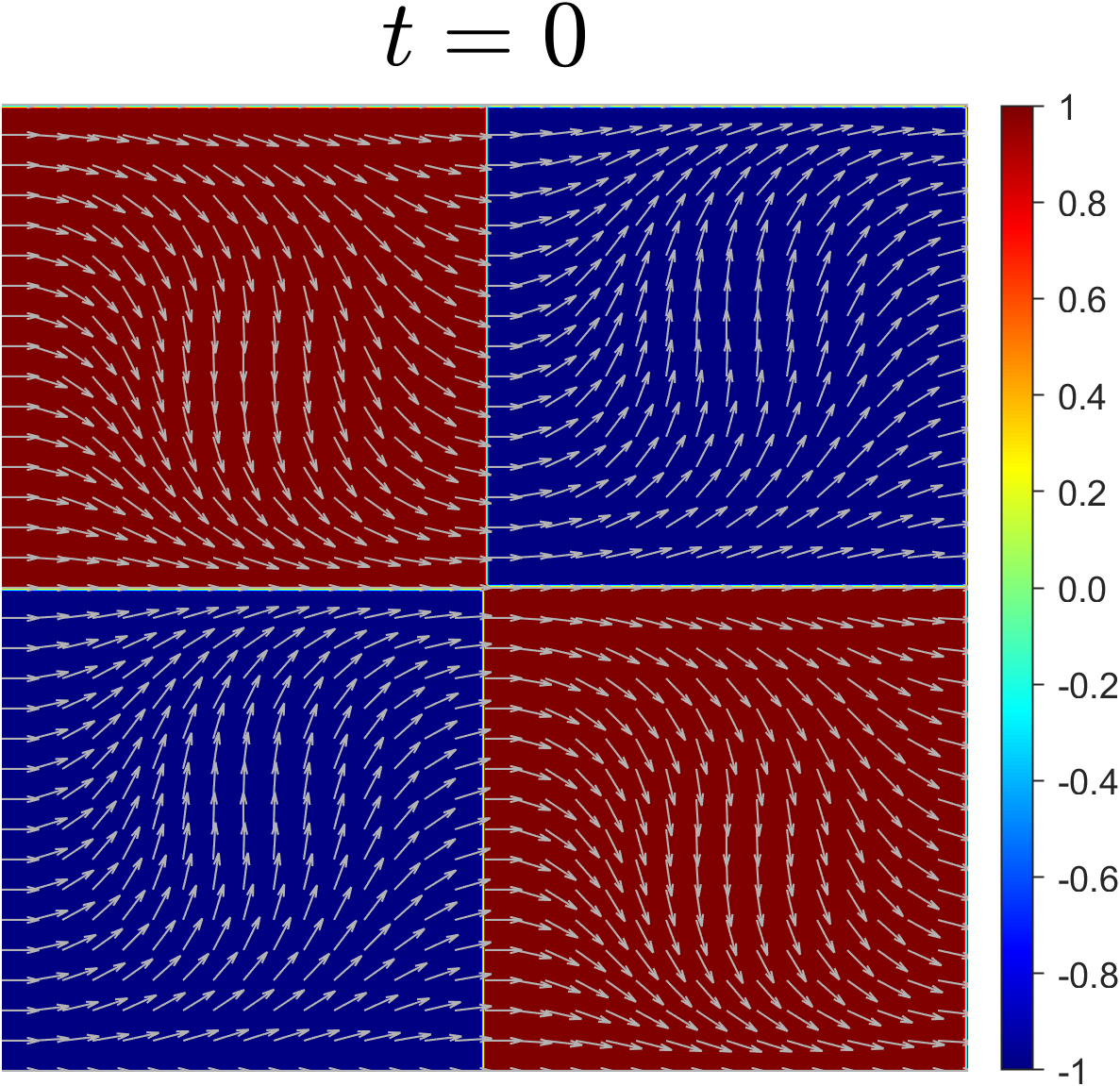}}\quad
	\subfigure{\includegraphics[width=0.30\textwidth,
		height=39mm]{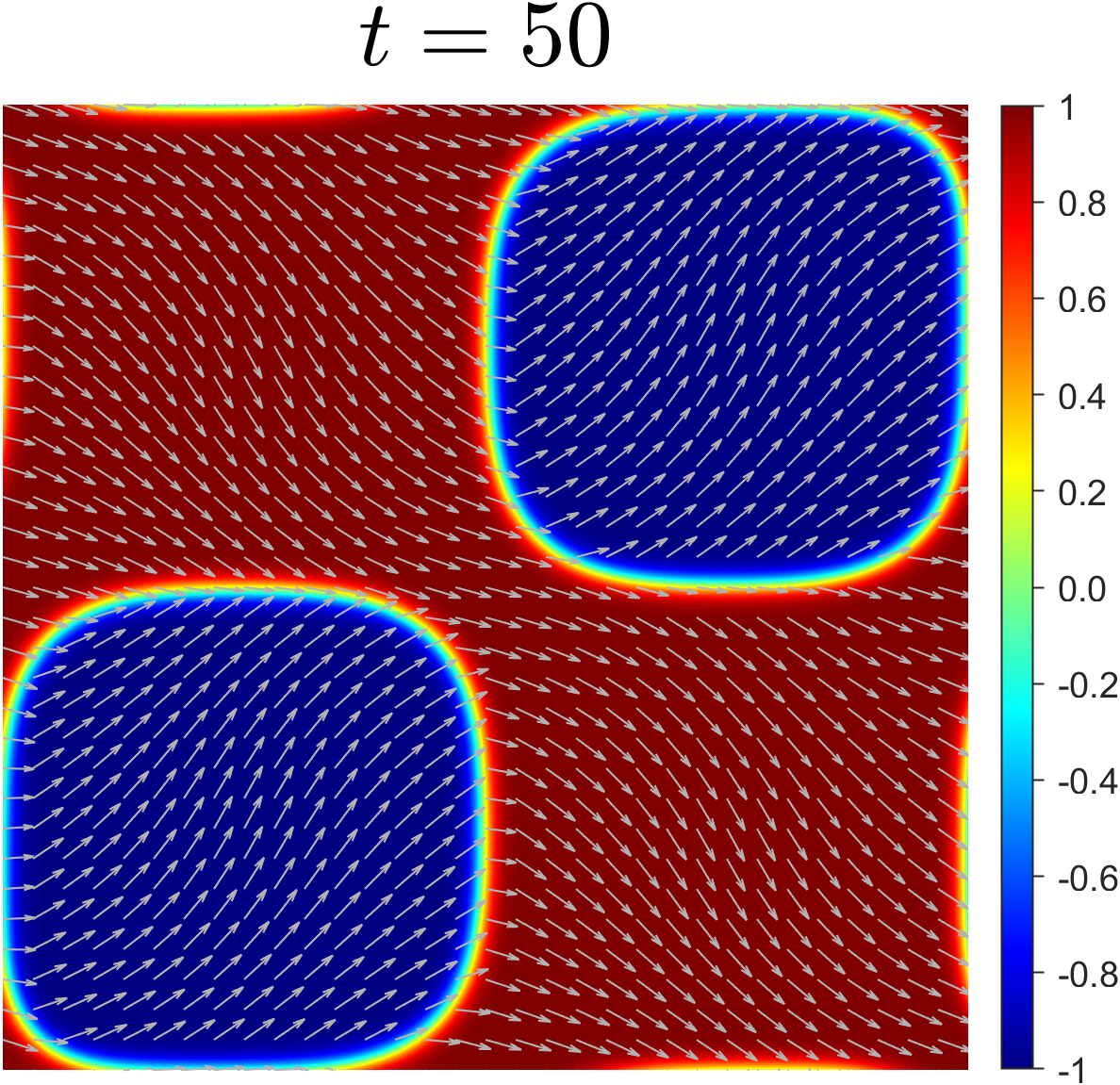}}\quad
	\subfigure{\includegraphics[width=0.30\textwidth,
		height=39mm]{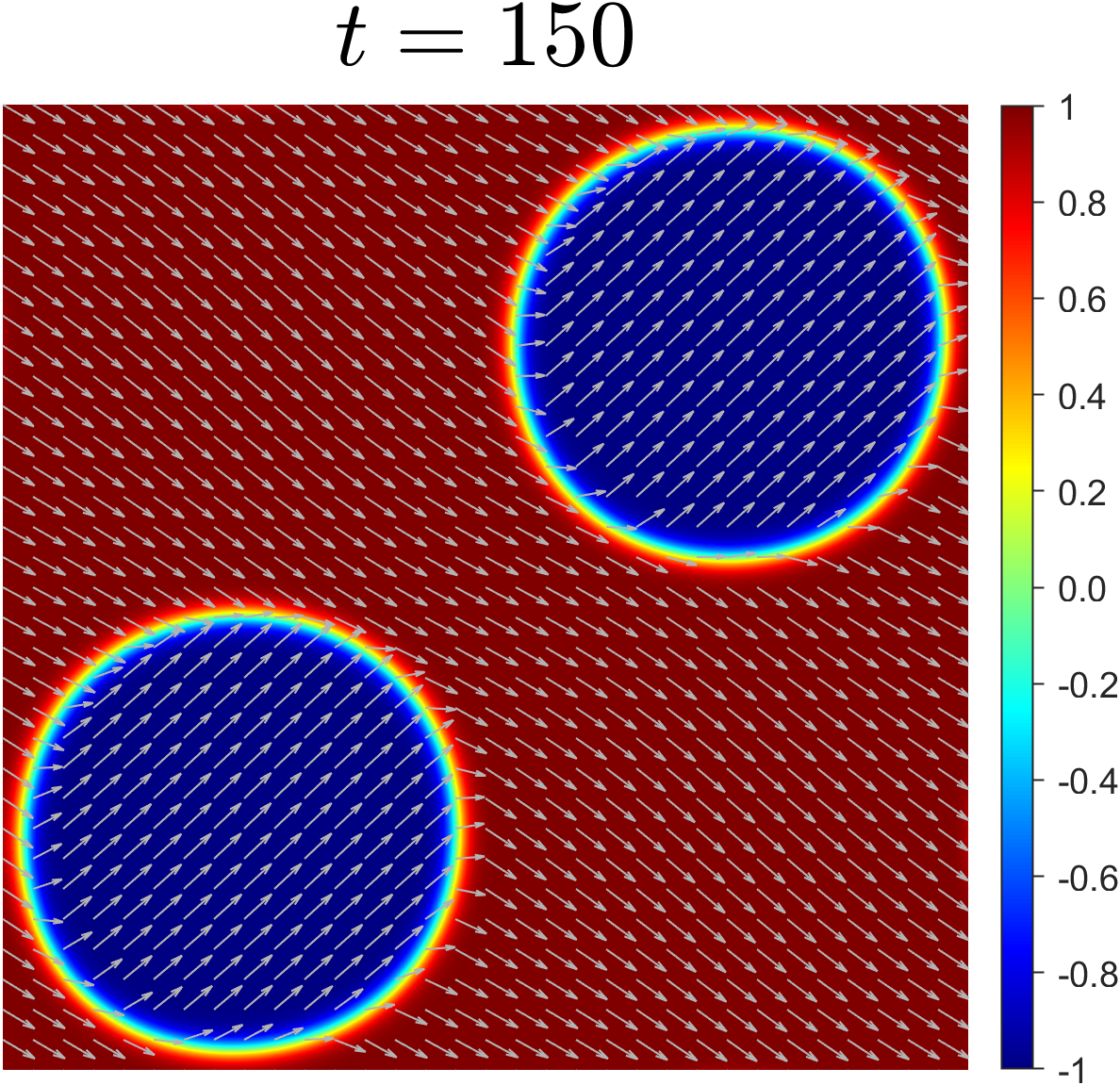}}\quad
	\subfigure{\includegraphics[width=0.30\textwidth,
		height=39mm]{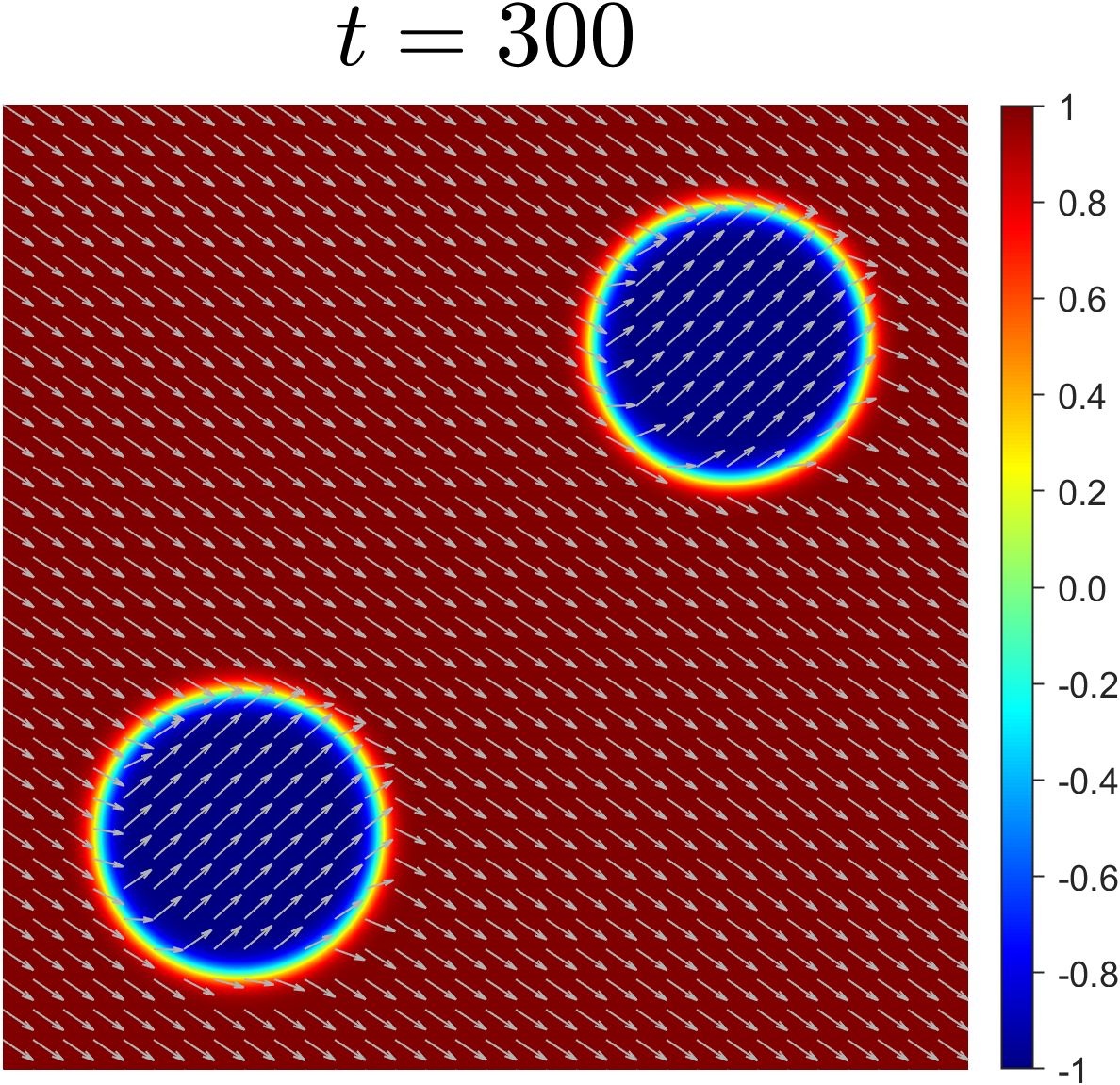}}\quad
	\subfigure{\includegraphics[width=0.30\textwidth,
		height=39mm]{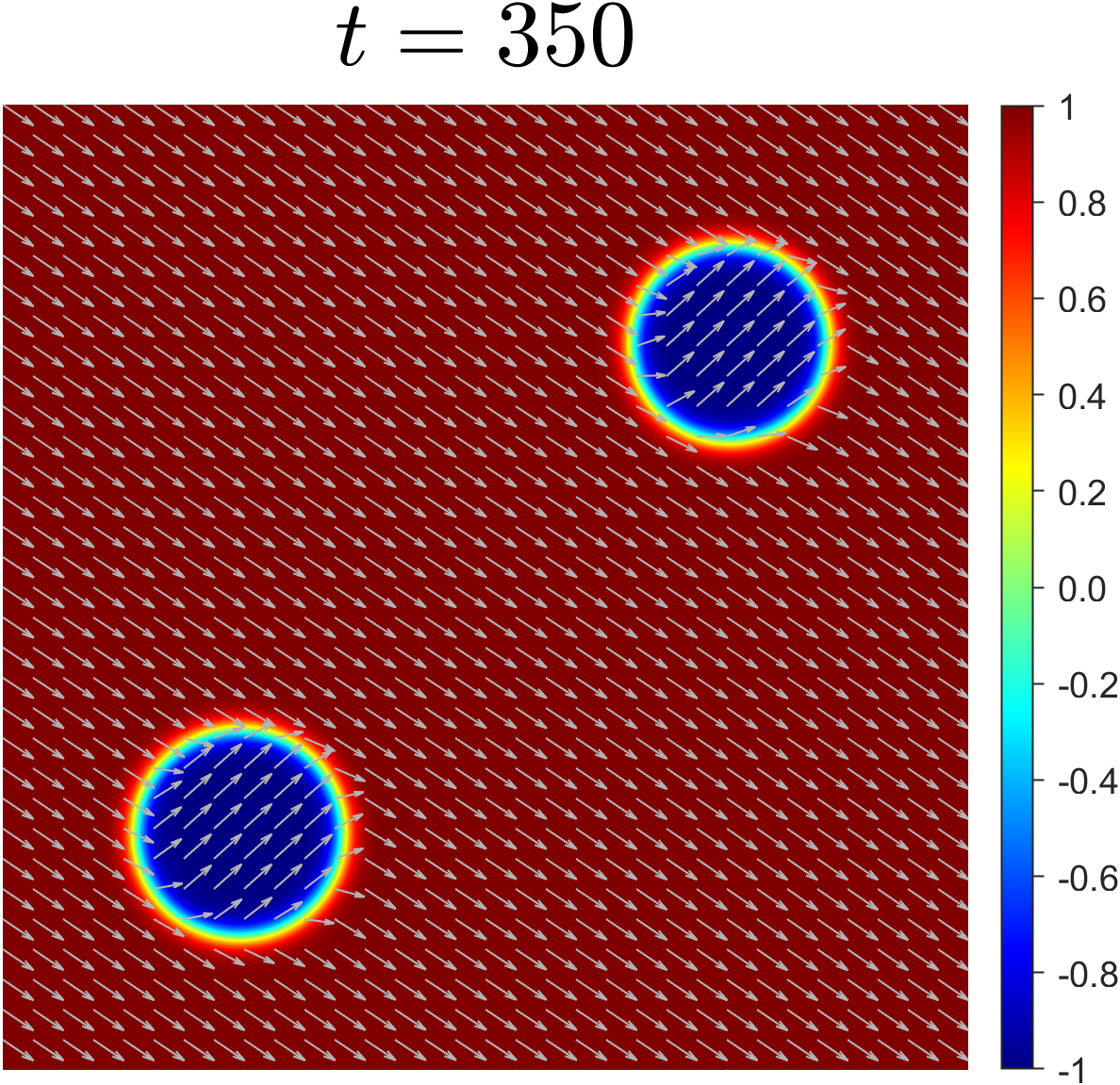}}\quad
	\subfigure{\includegraphics[width=0.30\textwidth,
		height=39mm]{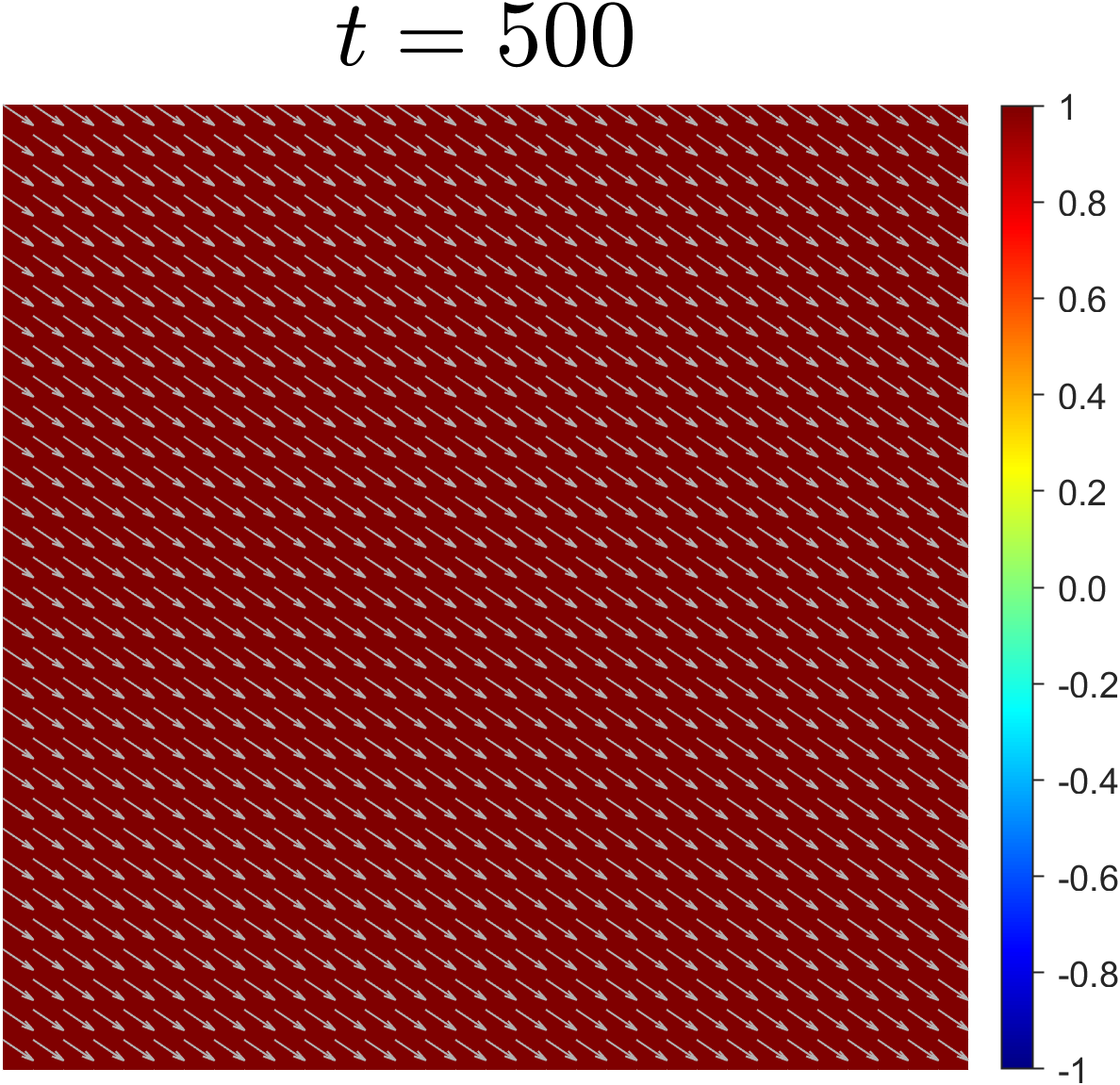}}
	\caption{Evolution of the matrix-valued field and interface at $t=0,50,150,300,350,500$. The initial field is given in \eqref{eq:4.4} with $\alpha(x,y)=\frac{\pi}{2}\sin(2\pi x)\sin(2\pi y)$.}
   \label{fig:4.13}
\end{figure}
 \begin{figure}
      \centering
	\subfigure{\includegraphics[width=0.42\textwidth,
	height=45mm]{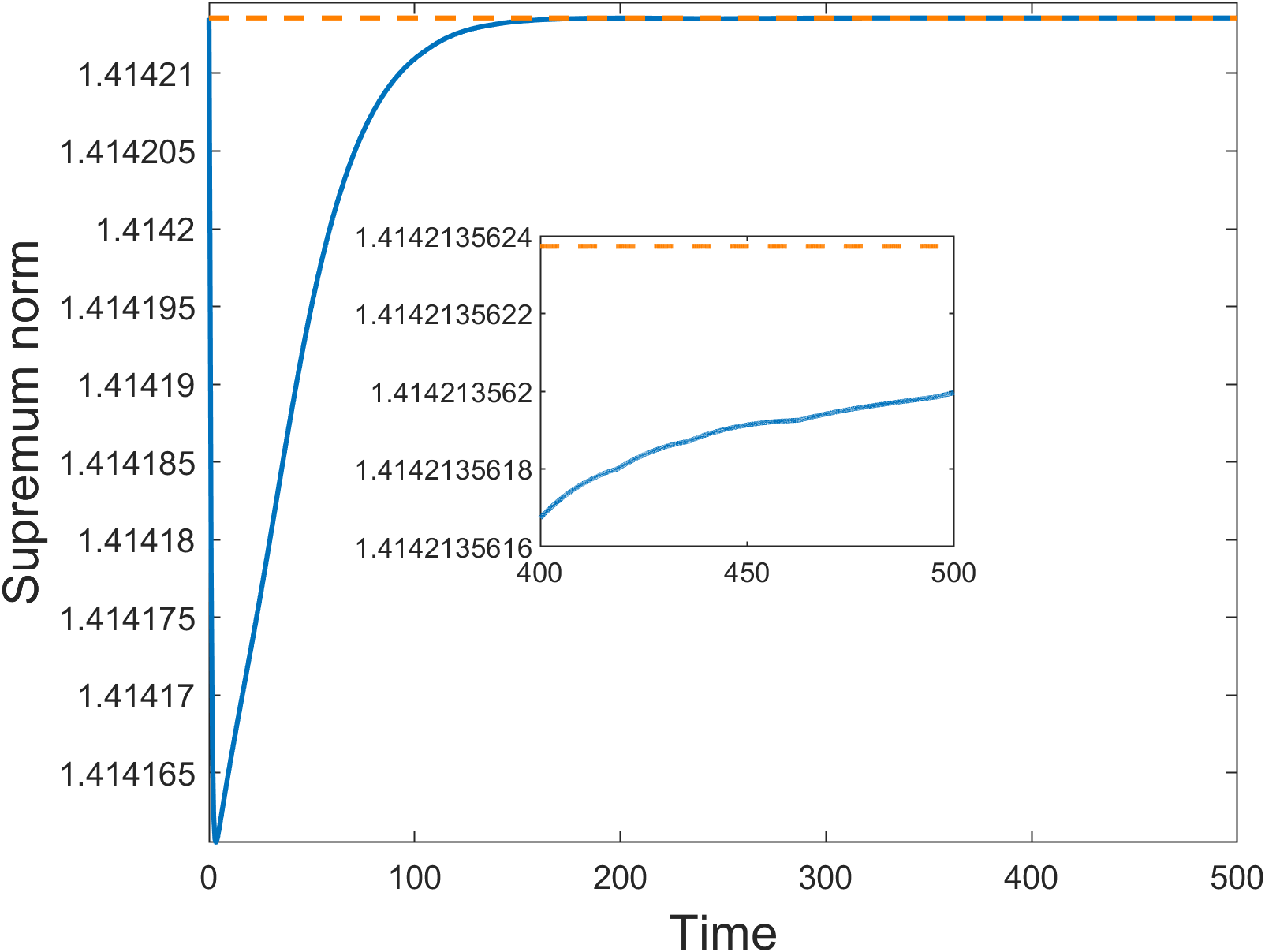}}\quad
	\subfigure{\includegraphics[width=0.42\textwidth,
	height=45mm]{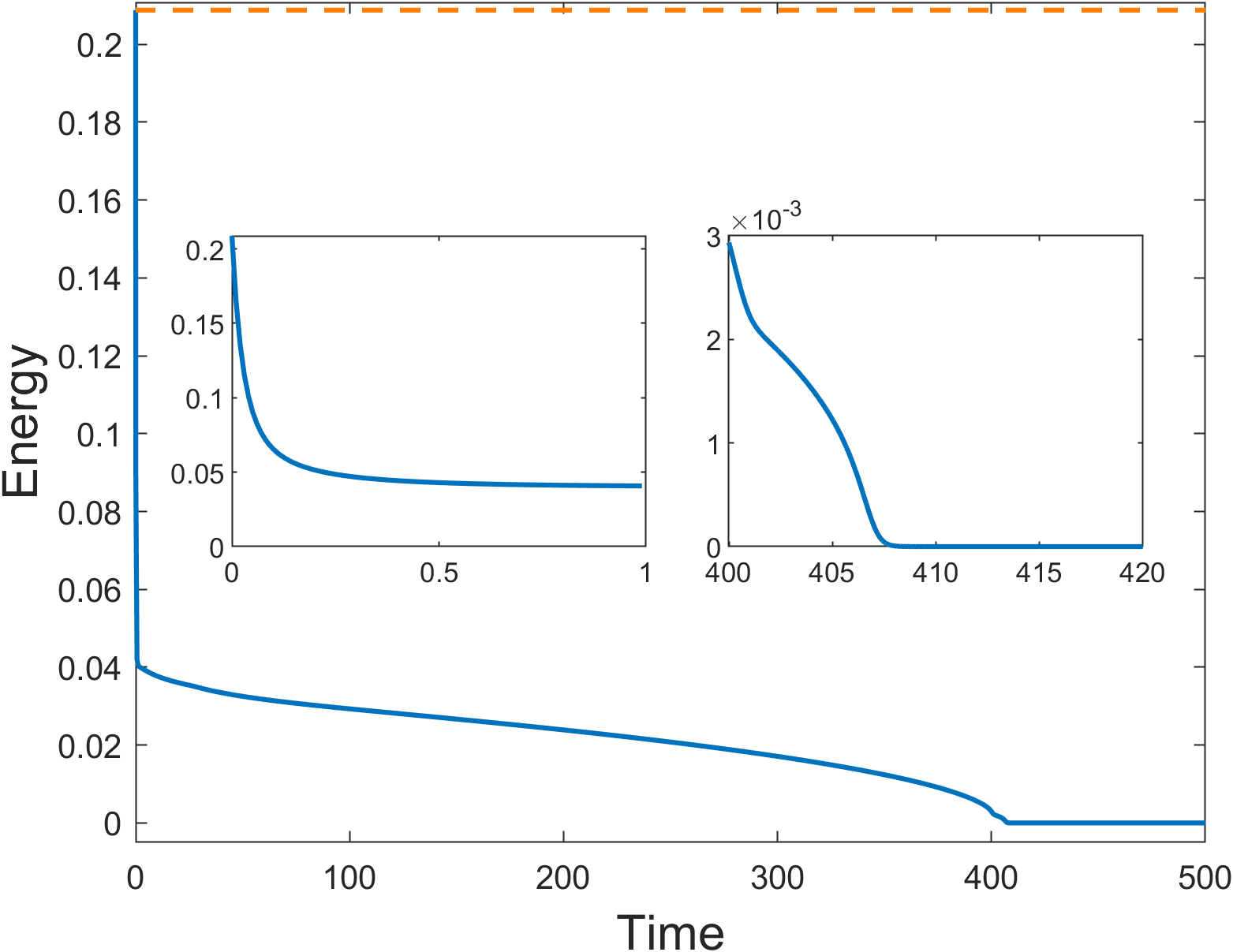}}
    \caption{Evolution of the supremum norm $\|\cdot\|_{\mathcal{X}}$ and energy with initial condition \eqref{eq:4.4} and $\alpha(x,y)=\frac{\pi}{2}\sin(2\pi x)\sin(2\pi y)$. The dashed line in the left figure is the maximum bound $\sqrt m$ while the dashed line in the right figure is the initial energy.}
    \label{fig:4.14}
\end{figure}

\subsubsection{Three-dimensional matrix-valued Allen--Cahn equation}

In the following three-dimensional examples, we consider the spatial domain $\Omega=\left[-\frac 12,\frac 12\right]^3$ and discretize $\Omega$ using $80\times 80\times 80$ mesh points for various initial conditions. Due to the specified initial conditions, the determinant of the matrix exhibits both positive and negative values. We use the zero level set to display its interface and observe the evolution. The vector field is generated by the first column vector of the matrix $U(t,x,y,z)$. We set the stabilization parameter $\kappa=8$ and perform a pointwise projection of the discretized initial matrix during the numerical computation, as \cite[Lemma 1.1]{osting2020diffusion}  shows.

{\bf Example 5.} In this example, we numerically simulate the evolution of a three-dimensional ring, governed by the matrix-valued Allen--Cahn equation \eqref{eq:mac} with $\varepsilon=0.01$. The initial condition is given by
\begin{equation}\label{eq:4.5}
     \begin{aligned}
	U^0(x,y,z)=
       \begin{cases}
		\mathcal{P}\begin{bmatrix}
		\frac{1}{2}\cos\alpha&\frac{\sqrt{6}}{2}\cos\alpha&-\frac{\sqrt{2}}{2}\sin\alpha\\
		\frac{1}{2}\sin\alpha&\frac{\sqrt{6}}{2}\sin\alpha&\frac{\sqrt{2}}{2}\cos\alpha\\
            \frac{\sqrt{3}}{2}&-\frac{\sqrt{2}}{2}&0
		\end{bmatrix}\quad \mbox{if}~$(x,y,z)$~\mbox{satisfy condition A},\\
  \\
		\mathcal{P}\begin{bmatrix}
		-\frac{1}{2}\cos\alpha&\frac{\sqrt{6}}{2}\cos\alpha&\frac{\sqrt{2}}{2}\sin\alpha\\
		-\frac{1}{2}\sin\alpha&\frac{\sqrt{6}}{2}\sin\alpha&-\frac{\sqrt{2}}{2}\cos\alpha\\
            \frac{\sqrt{3}}{2}&\frac{\sqrt{2}}{2}&0
		\end{bmatrix}\quad\text{otherwise},\\
	\end{cases}
     \end{aligned}
\end{equation}
where $\alpha=\alpha(x,y,z)=2\pi x(y+z)$, condition A is $\left(0.2-\sqrt{x^2+y^2}\right)^2+z^2< 0.15^2$, and $\mathcal{P}$ is the pointwise projection.

We use the ETDRK 2 scheme with time step $\tau=0.1$ and plot the results with the initial condition \eqref{eq:4.5} in Figures \ref{fig:4.15} and \ref{fig:4.16}. We observe from Figure \ref{fig:4.15} that the initial circular region with a value of $1$ for the determinant eventually evolves into an ellipsoid eventually disappearing, and the initial matrix field evolves from a disordered state to a uniform matrix field.

In Figure \ref{fig:4.17}, we plot the evolution of the original energy and the supremum norm with initial matrix \eqref{eq:4.5}. We find that the original energy decreases and the supremum norm does not exceed $\sqrt{3}$.

\begin{figure}
	\begin{center}
		\subfigure{\includegraphics[width=0.31\textwidth,
			height=48mm]{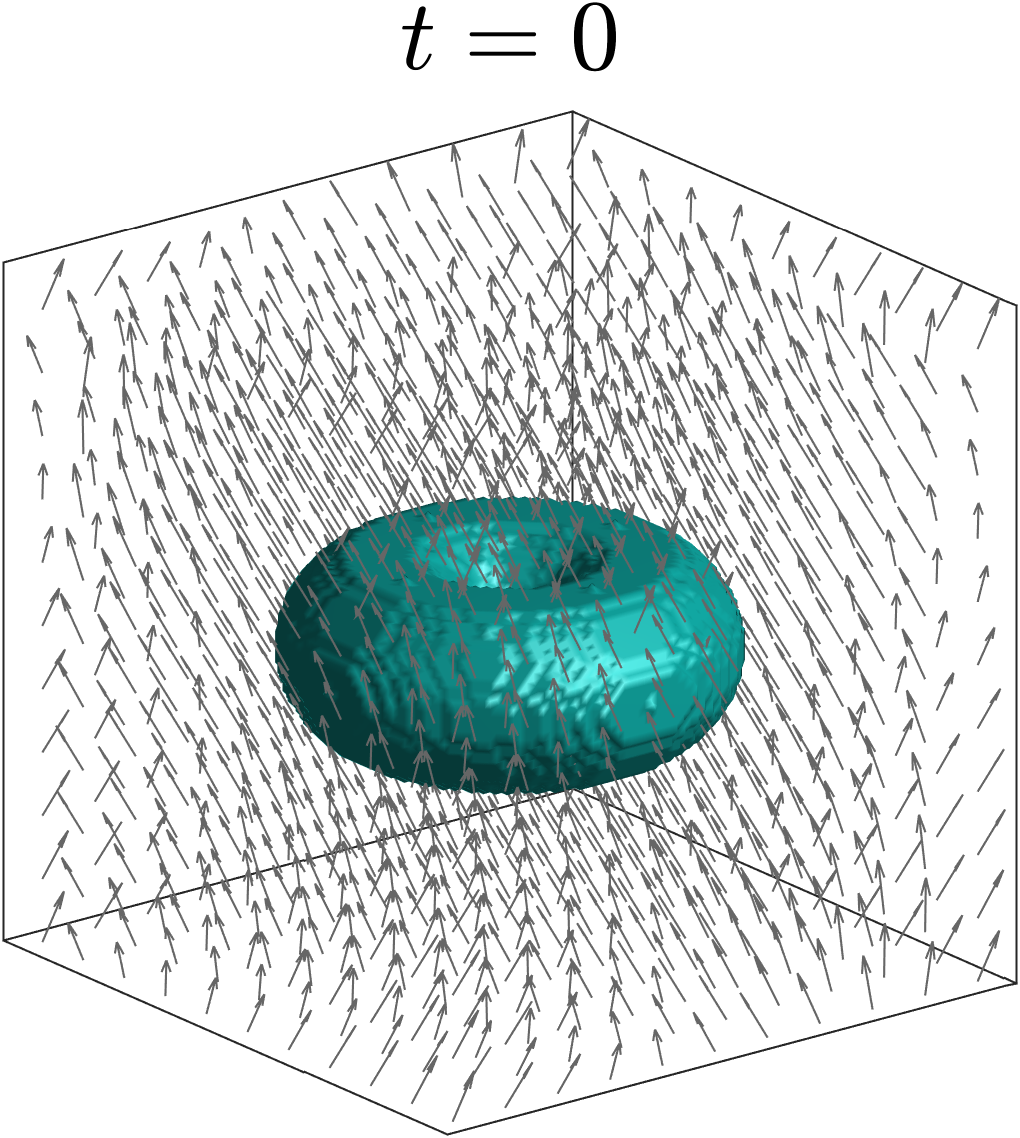}}\quad
		\subfigure{\includegraphics[width=0.31\textwidth,
			height=48mm]{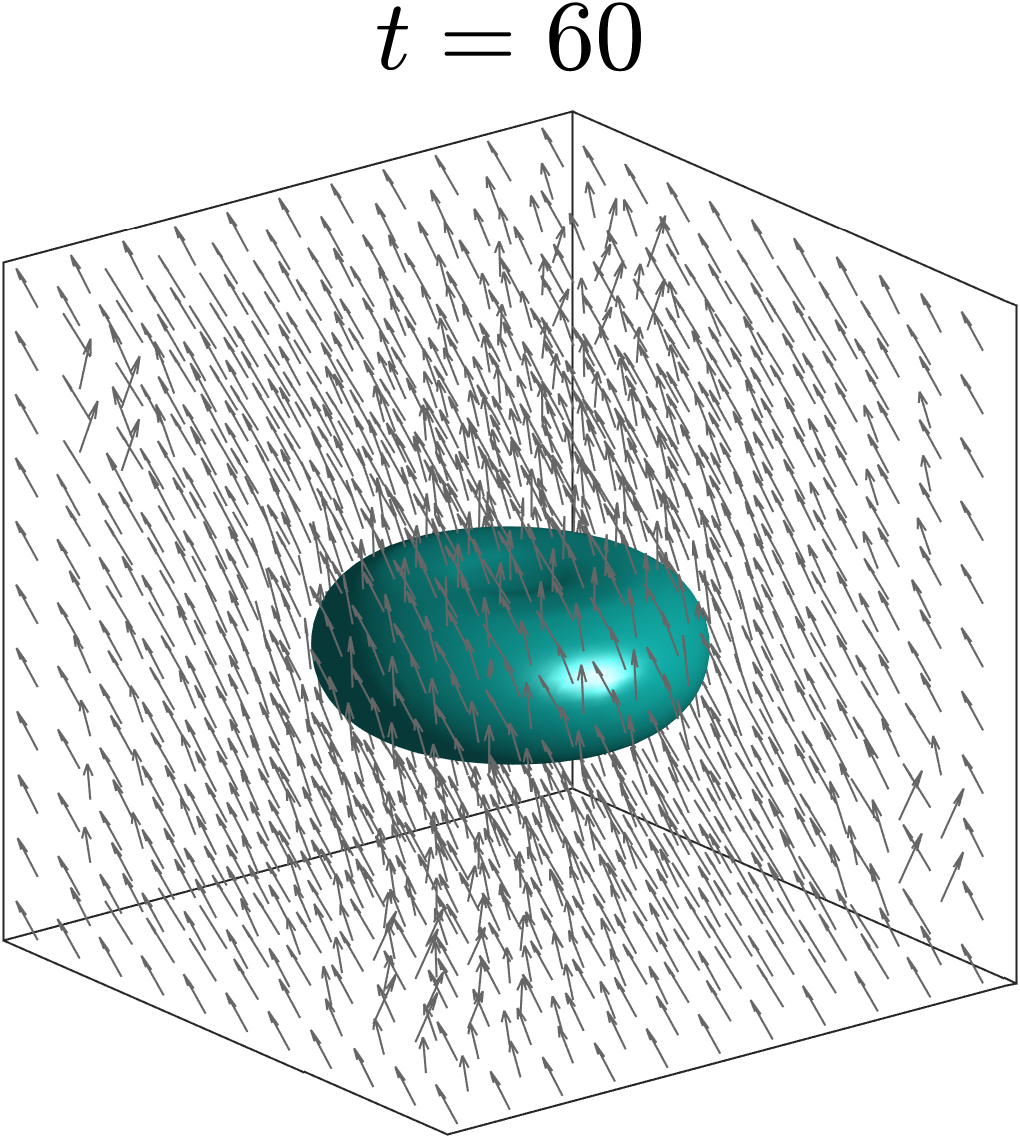}}\quad
		\subfigure{\includegraphics[width=0.31\textwidth,
			height=48mm]{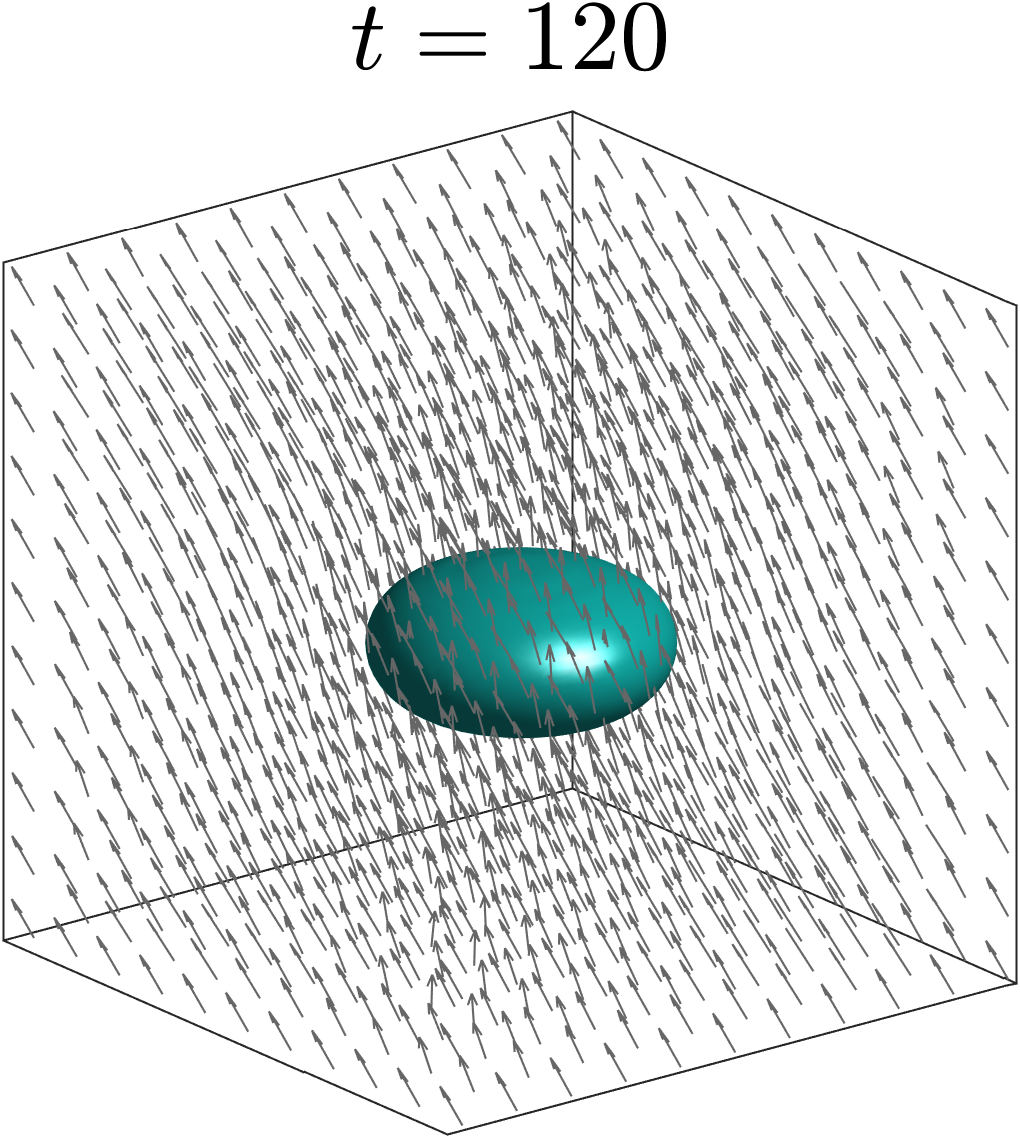}}\quad
		\subfigure{\includegraphics[width=0.31\textwidth,
			height=48mm]{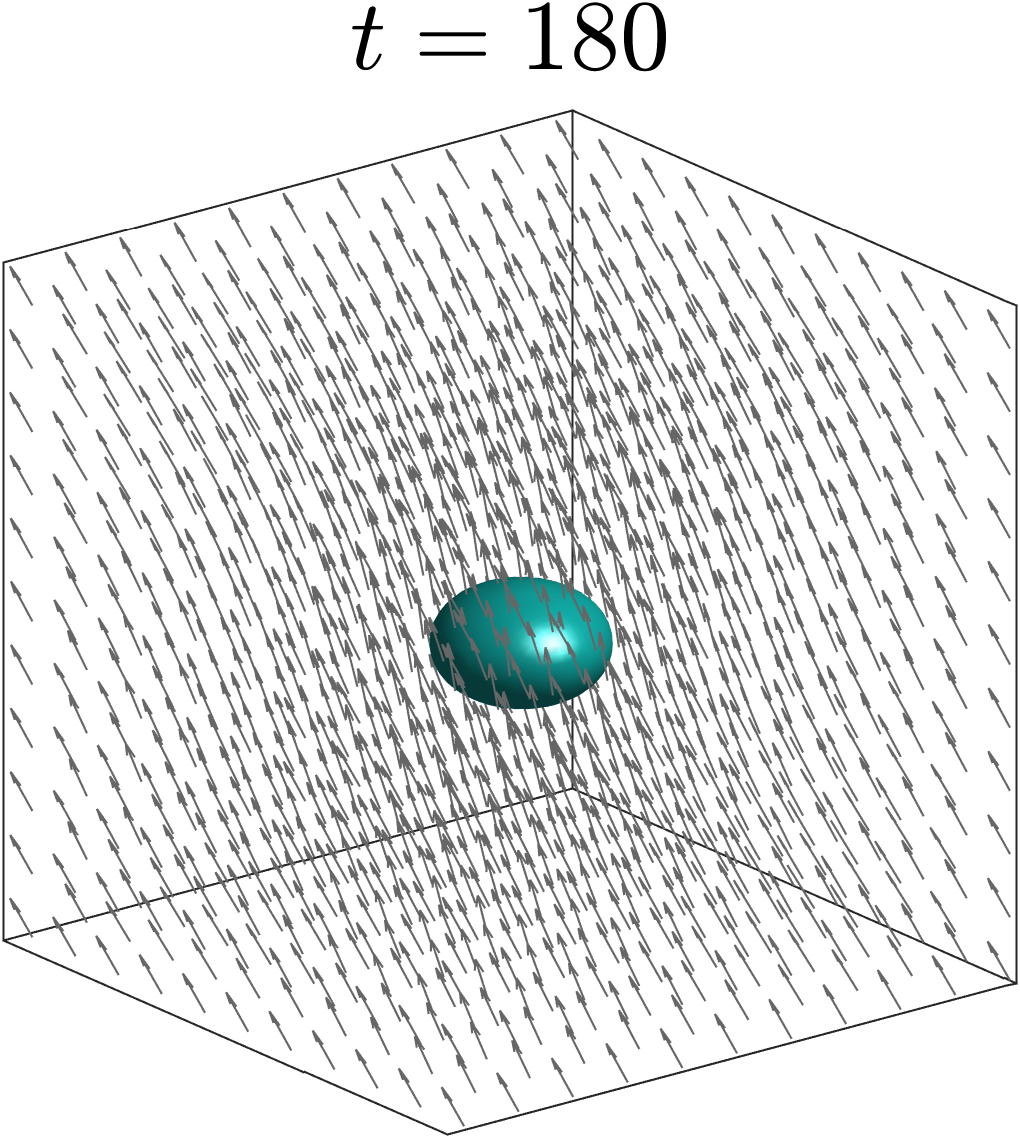}}\quad
		\subfigure{\includegraphics[width=0.31\textwidth,
			height=48mm]{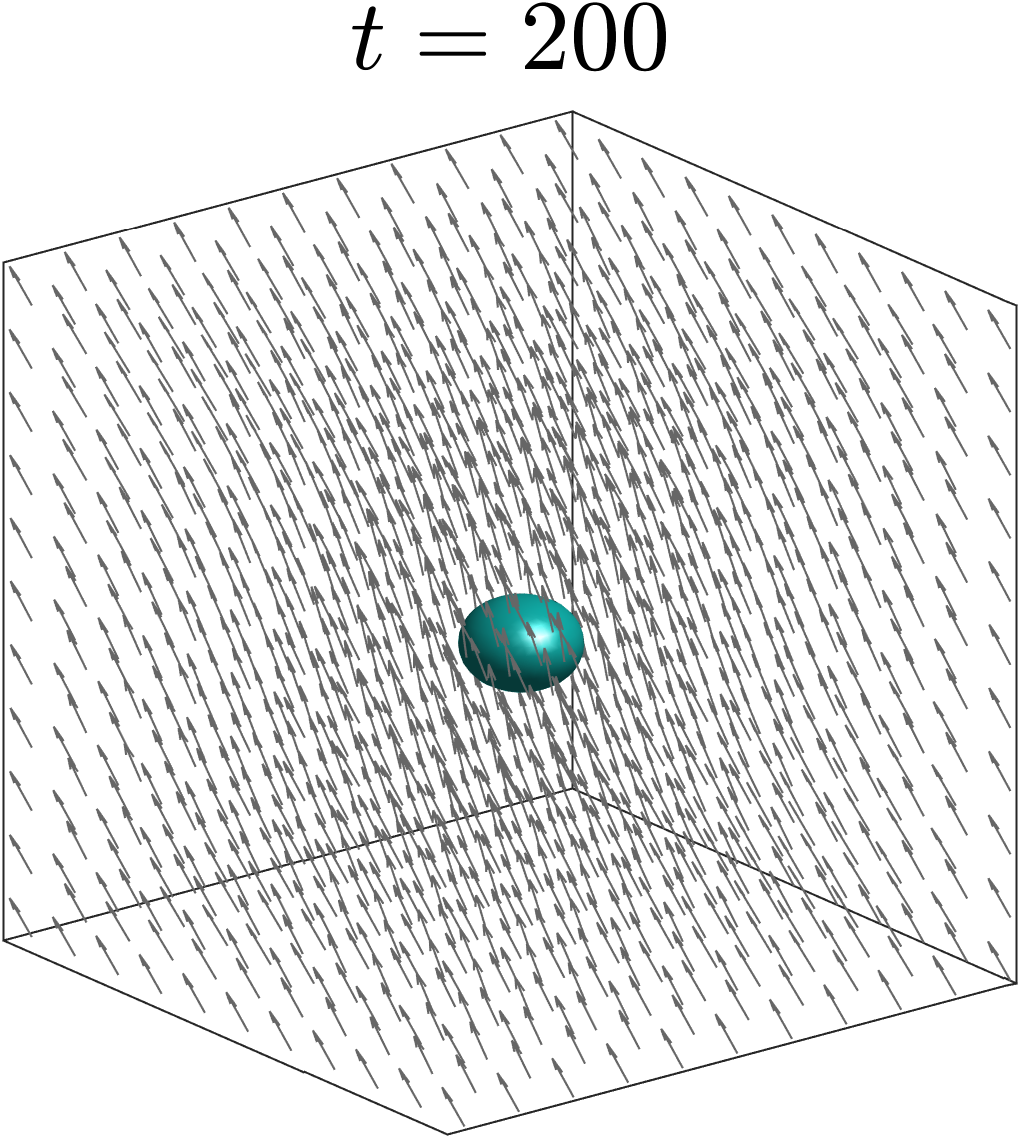}}\quad
		\subfigure{\includegraphics[width=0.31\textwidth,
			height=48mm]{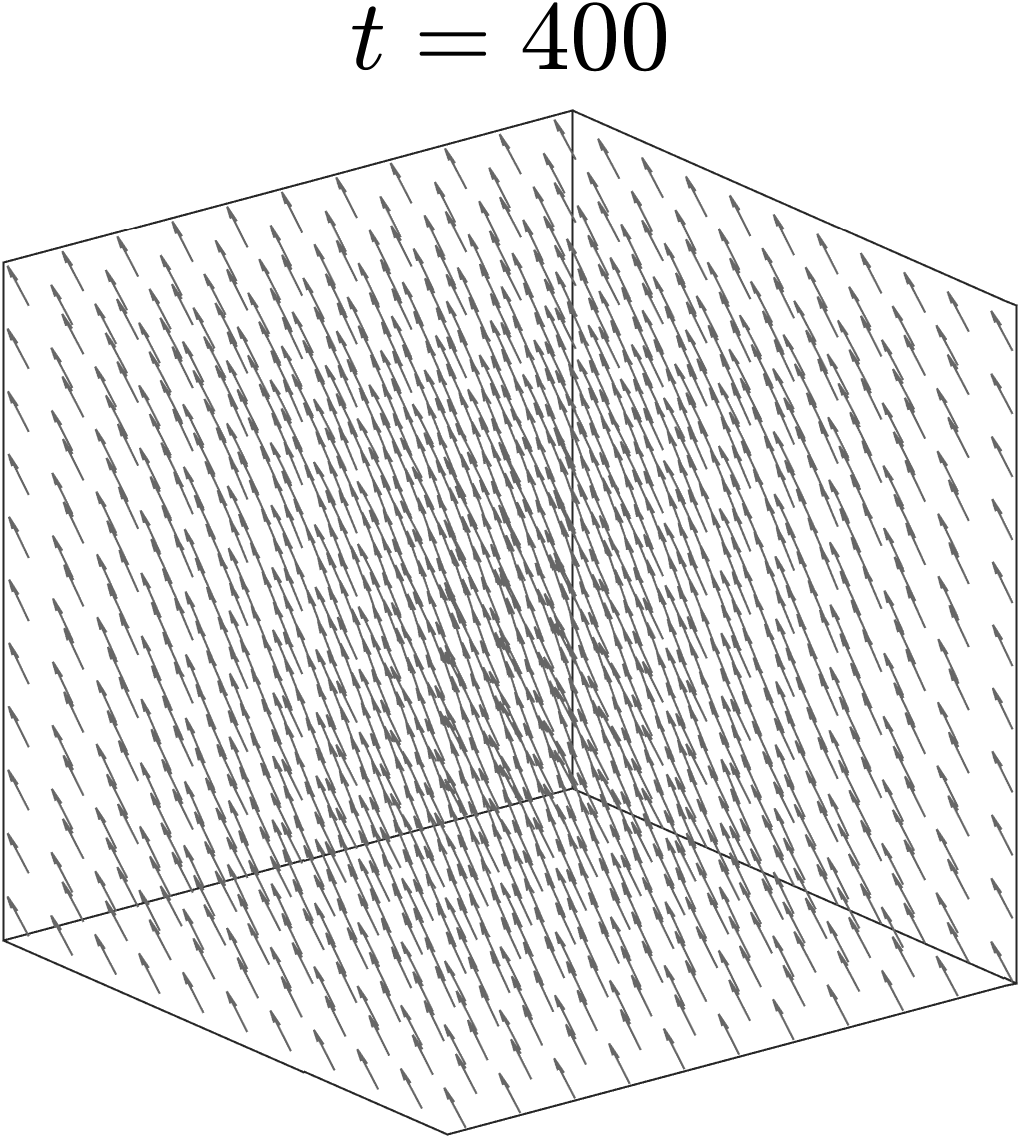}}
	\caption{Evolution of the matrix-valued field and interface at $t=0,60,120,180,200,400$. The initial field is given in \eqref{eq:4.5} with $\alpha(x,y,z)=2\pi x(y+z)$.}
 \label{fig:4.15}
	\end{center}
\end{figure}

\begin{figure}
	\begin{center}
		\subfigure{\includegraphics[width=0.28\textwidth,
			height=40mm]{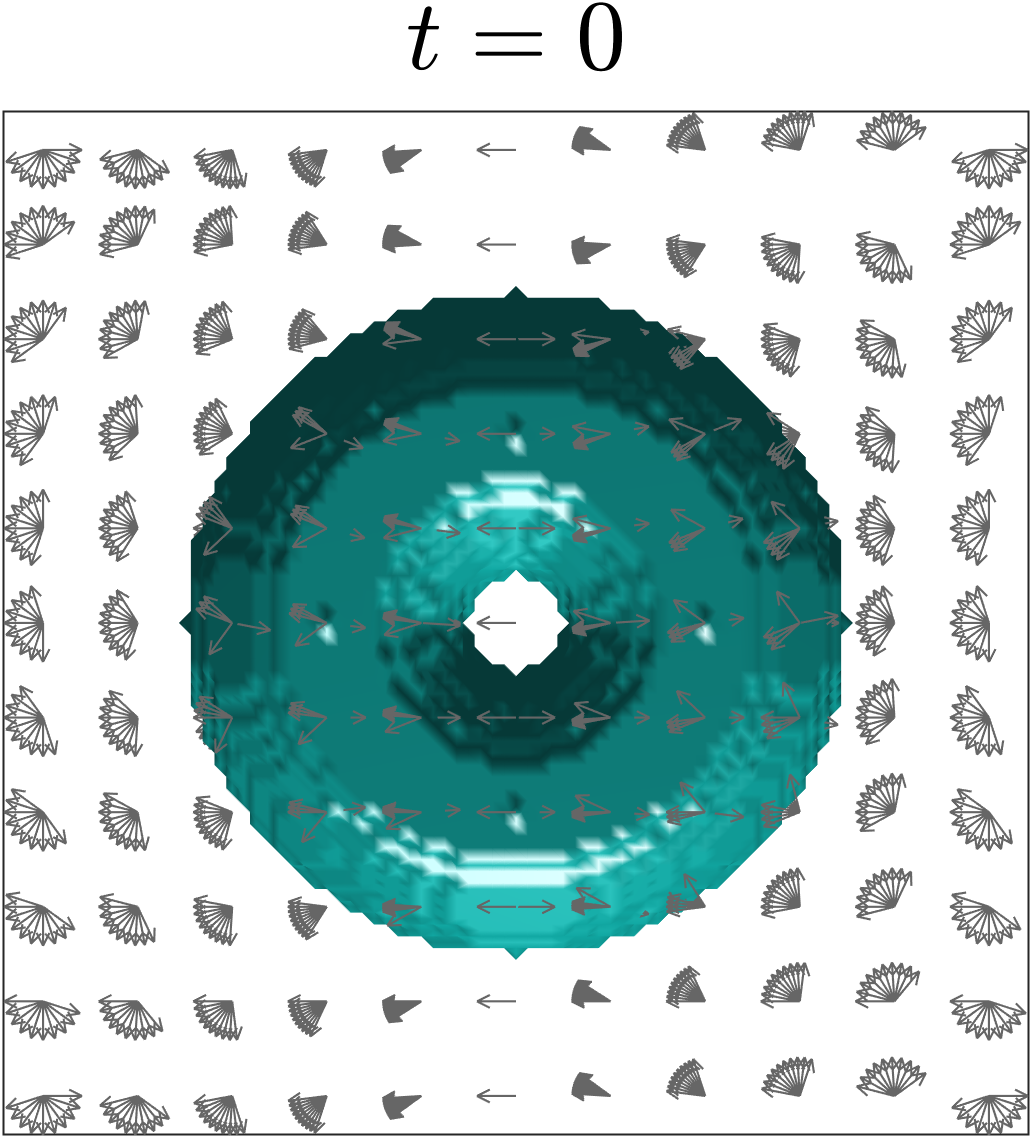}}\quad
		\subfigure{\includegraphics[width=0.28\textwidth,
			height=40mm]{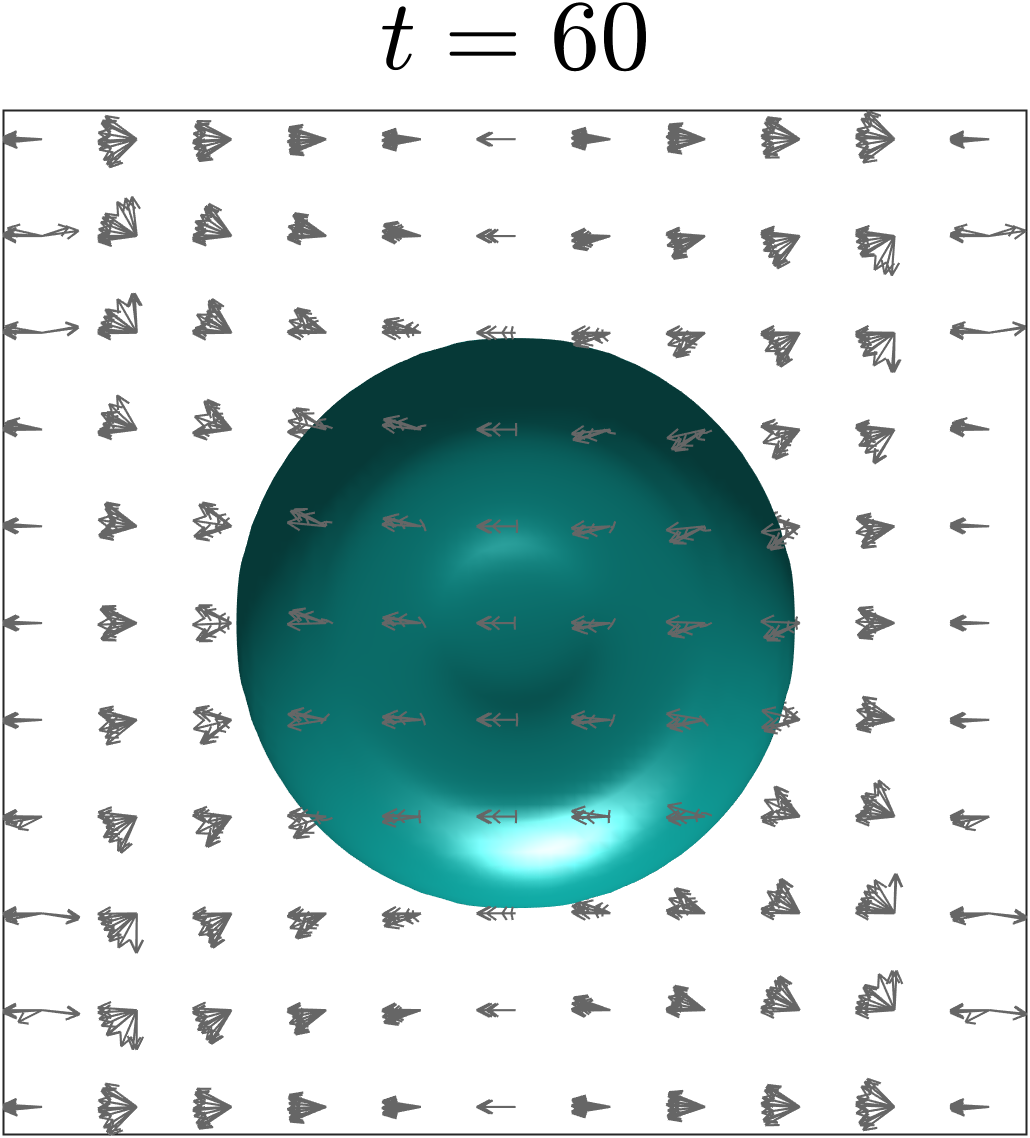}}\quad
    	\subfigure{\includegraphics[width=0.28\textwidth,
			height=40mm]{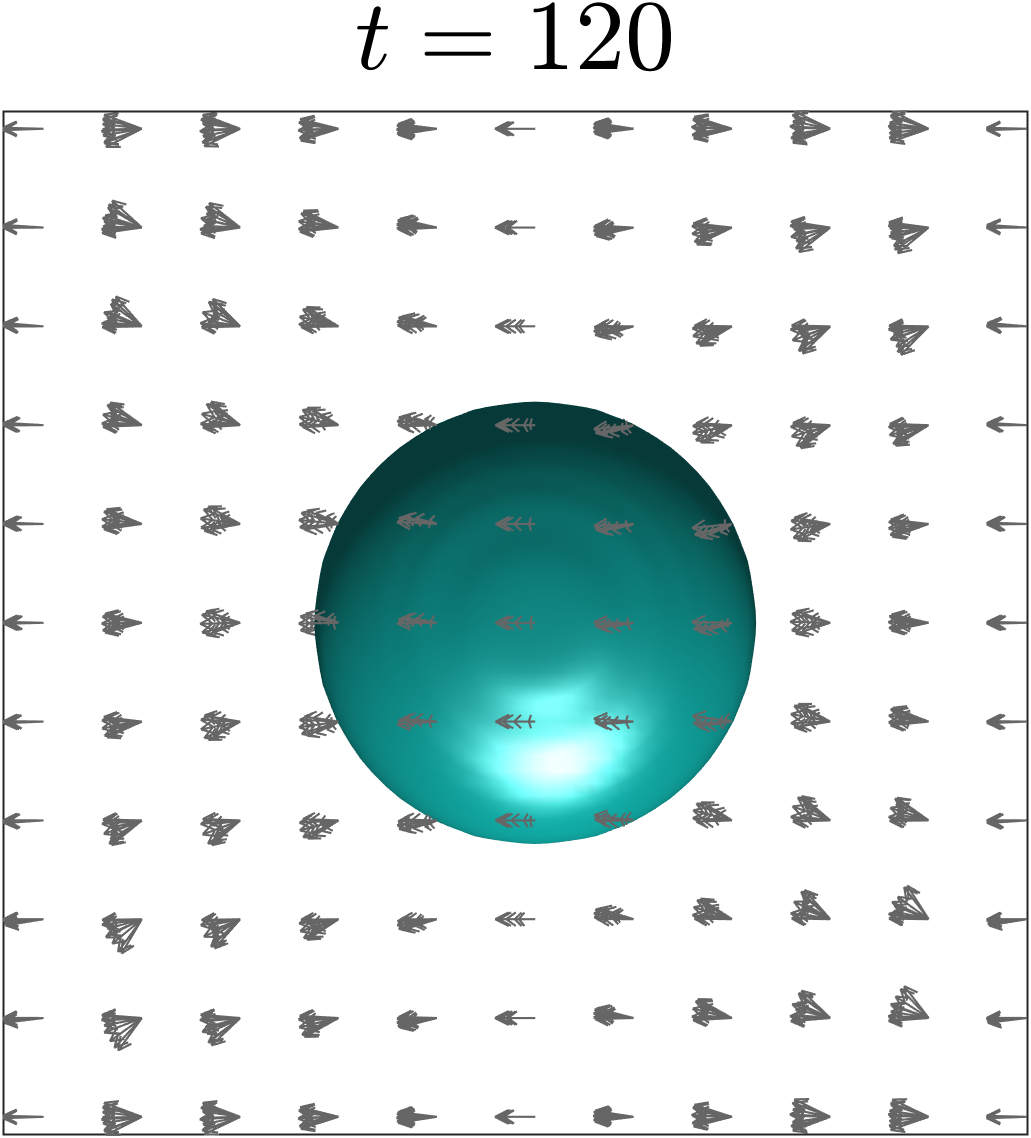}}\quad
		\subfigure{\includegraphics[width=0.28\textwidth,
			height=40mm]{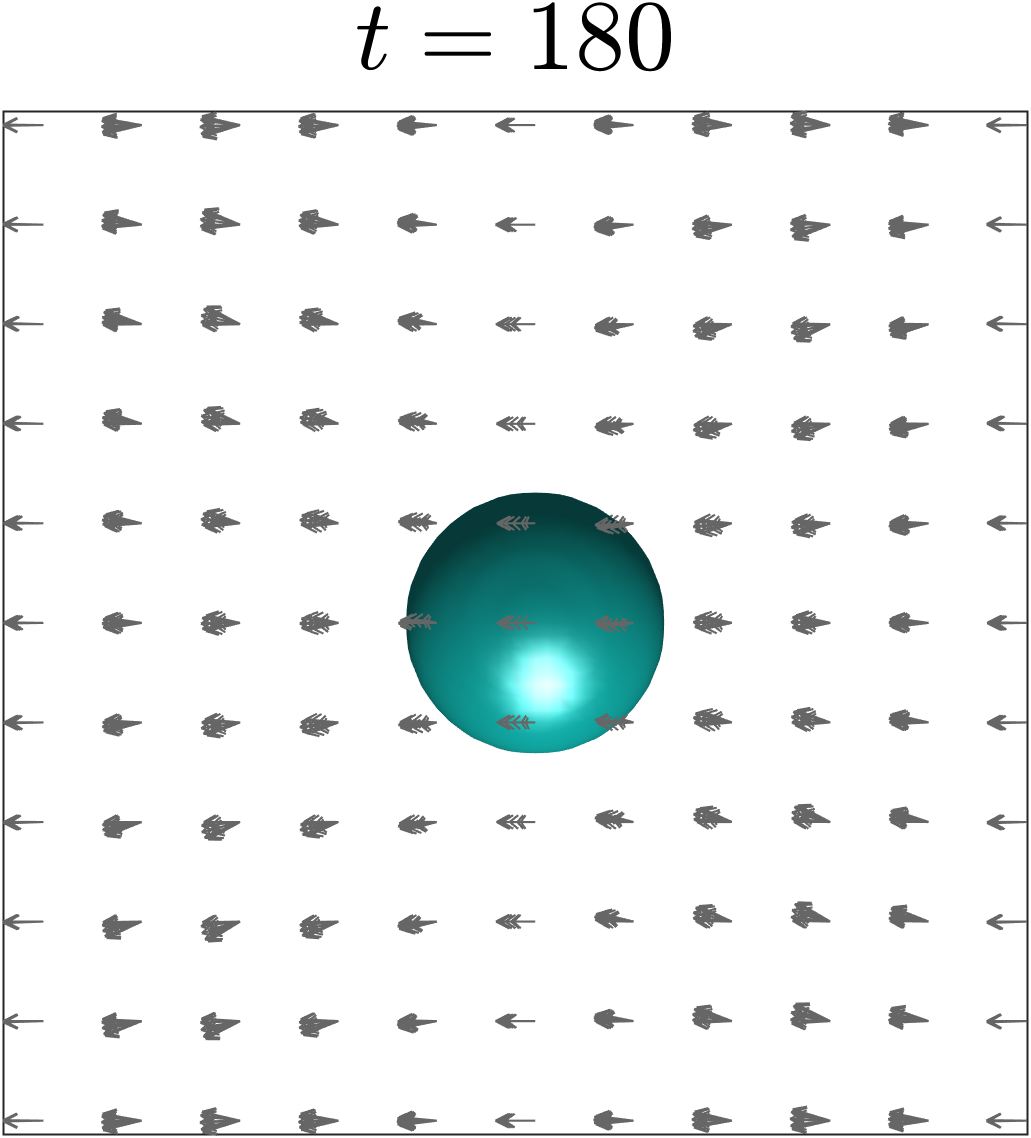}}\quad
    	\subfigure{\includegraphics[width=0.28\textwidth,
			height=40mm]{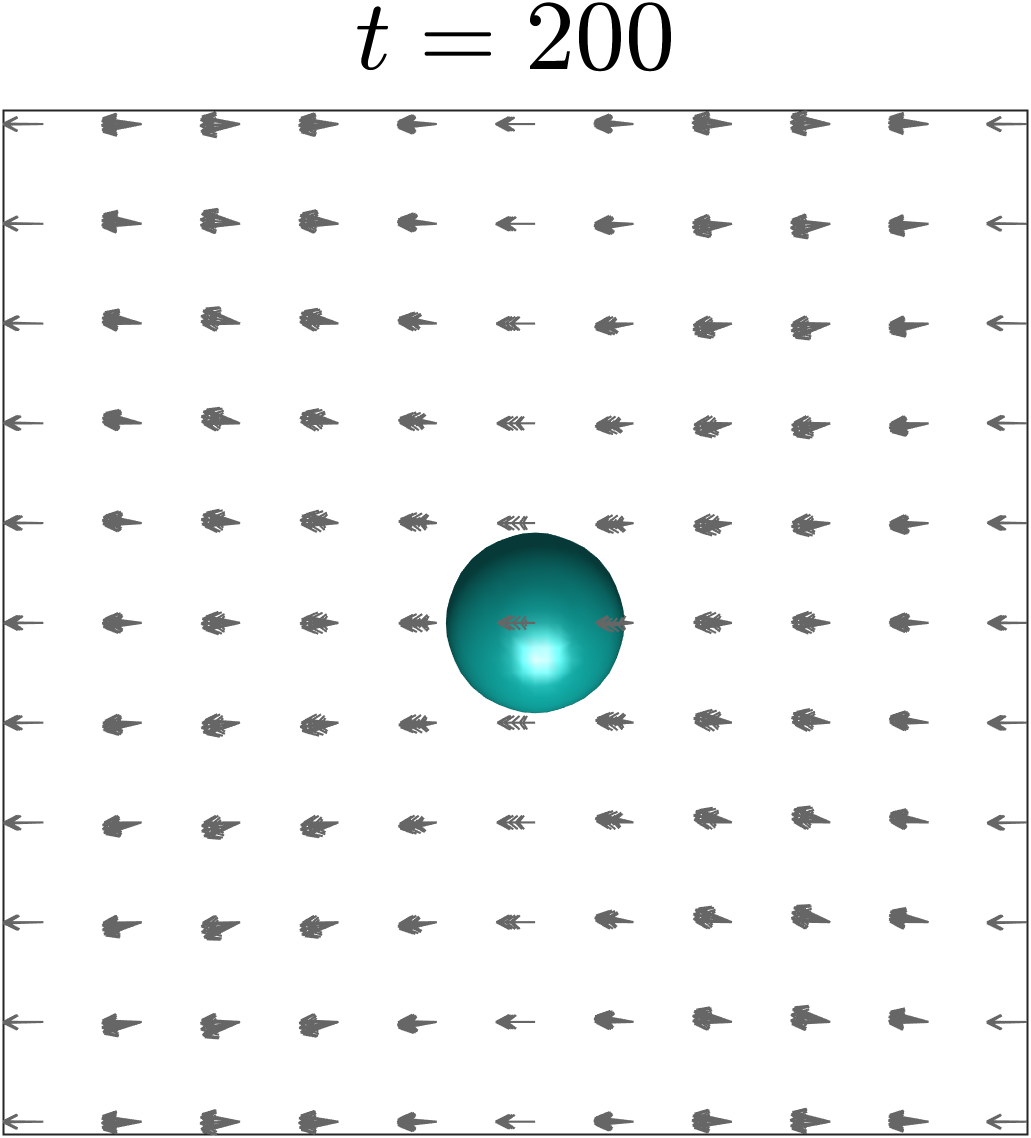}}\quad
		\subfigure{\includegraphics[width=0.28\textwidth,
			height=40mm]{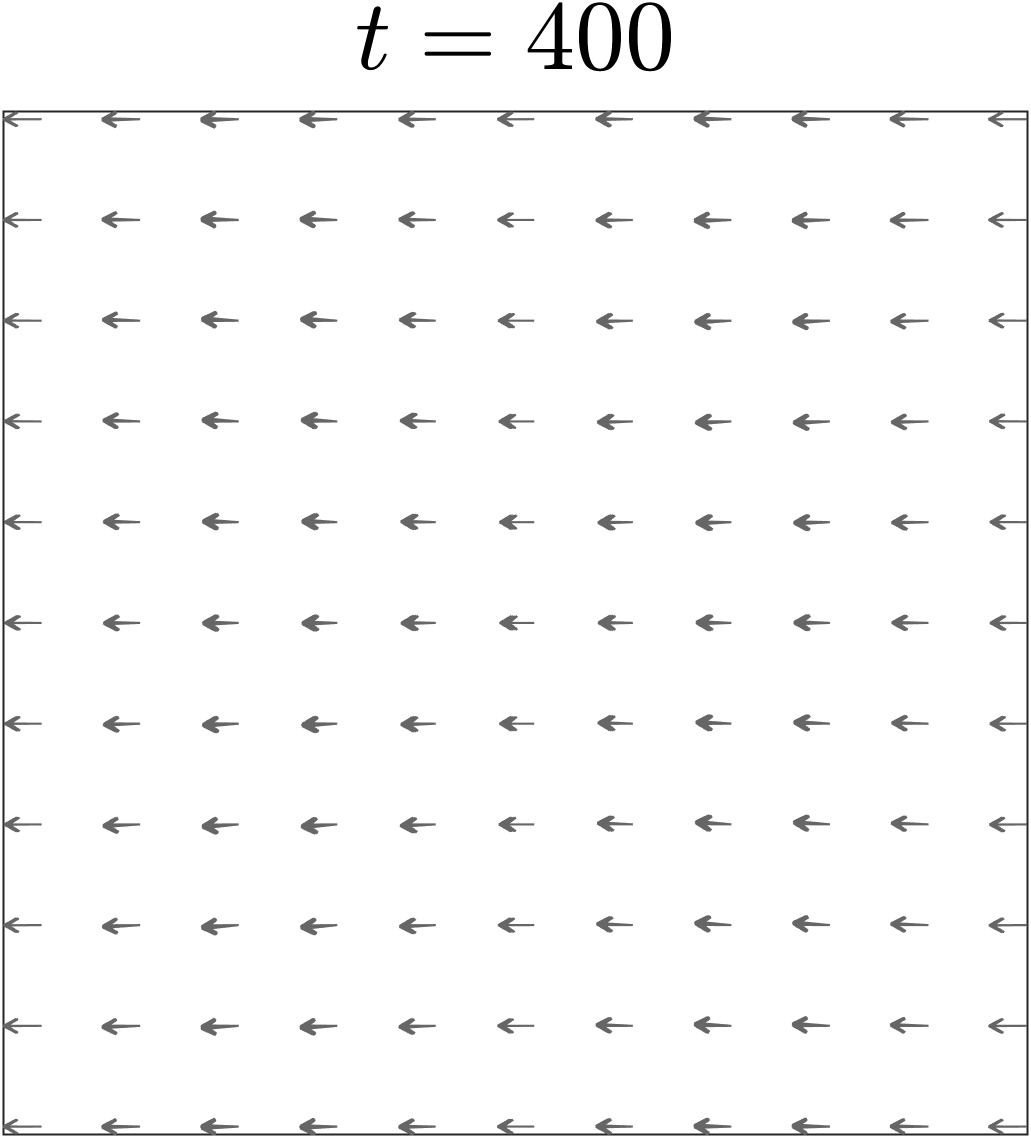}}
	\caption{A vertical view of the evolution of matrix-valued field and interface at $t=0,60,120,180,200,400$. The initial field is given in \eqref{eq:4.5} with $\alpha(x,y,z)=2\pi x(y+z)$.}
 \label{fig:4.16}
	\end{center}
\end{figure}

 \begin{figure}
	\begin{center}
		\subfigure{\includegraphics[width=0.42\textwidth,
			height=45mm]{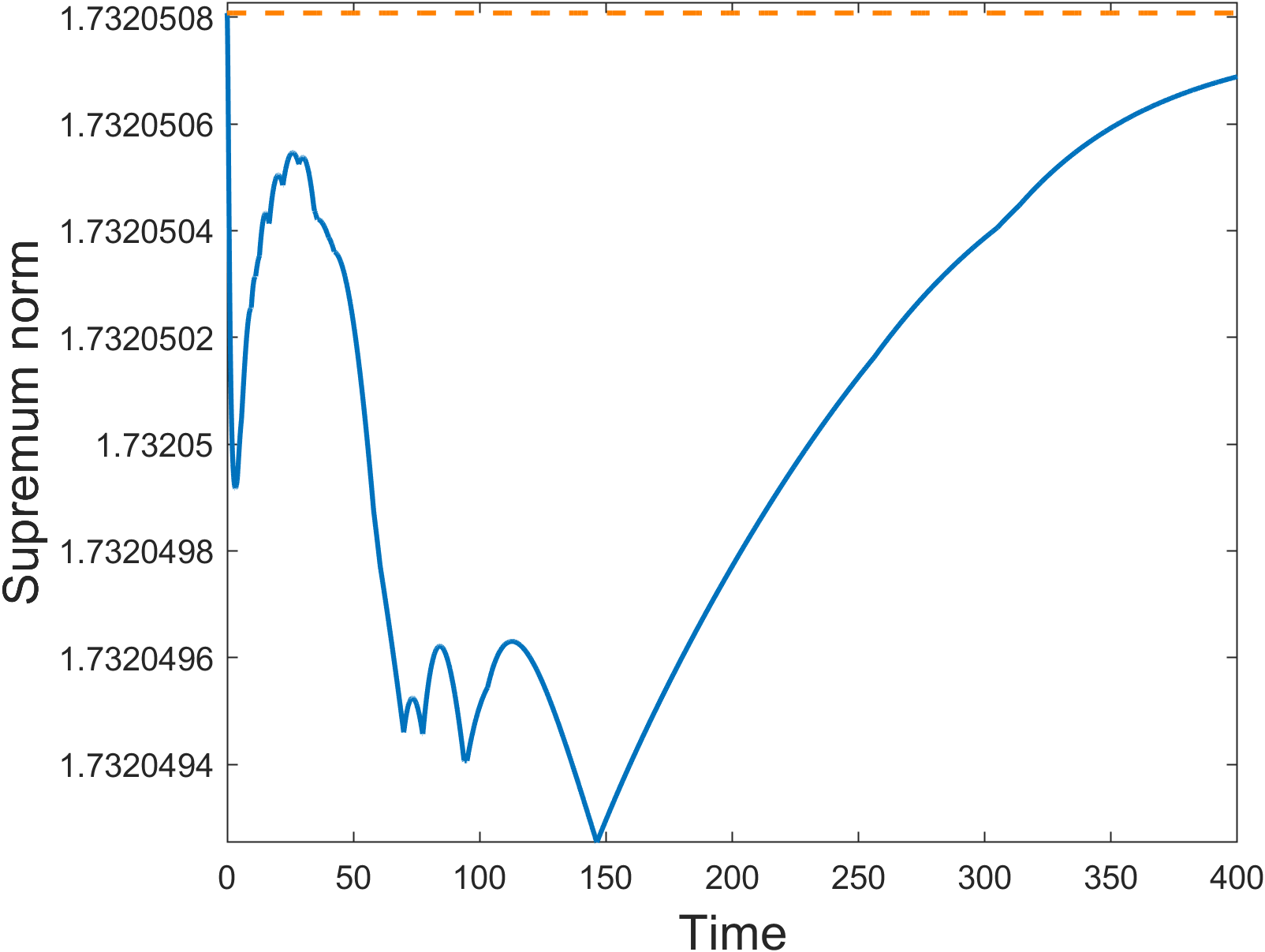}}\quad
		\subfigure{\includegraphics[width=0.42\textwidth,
			height=45mm]{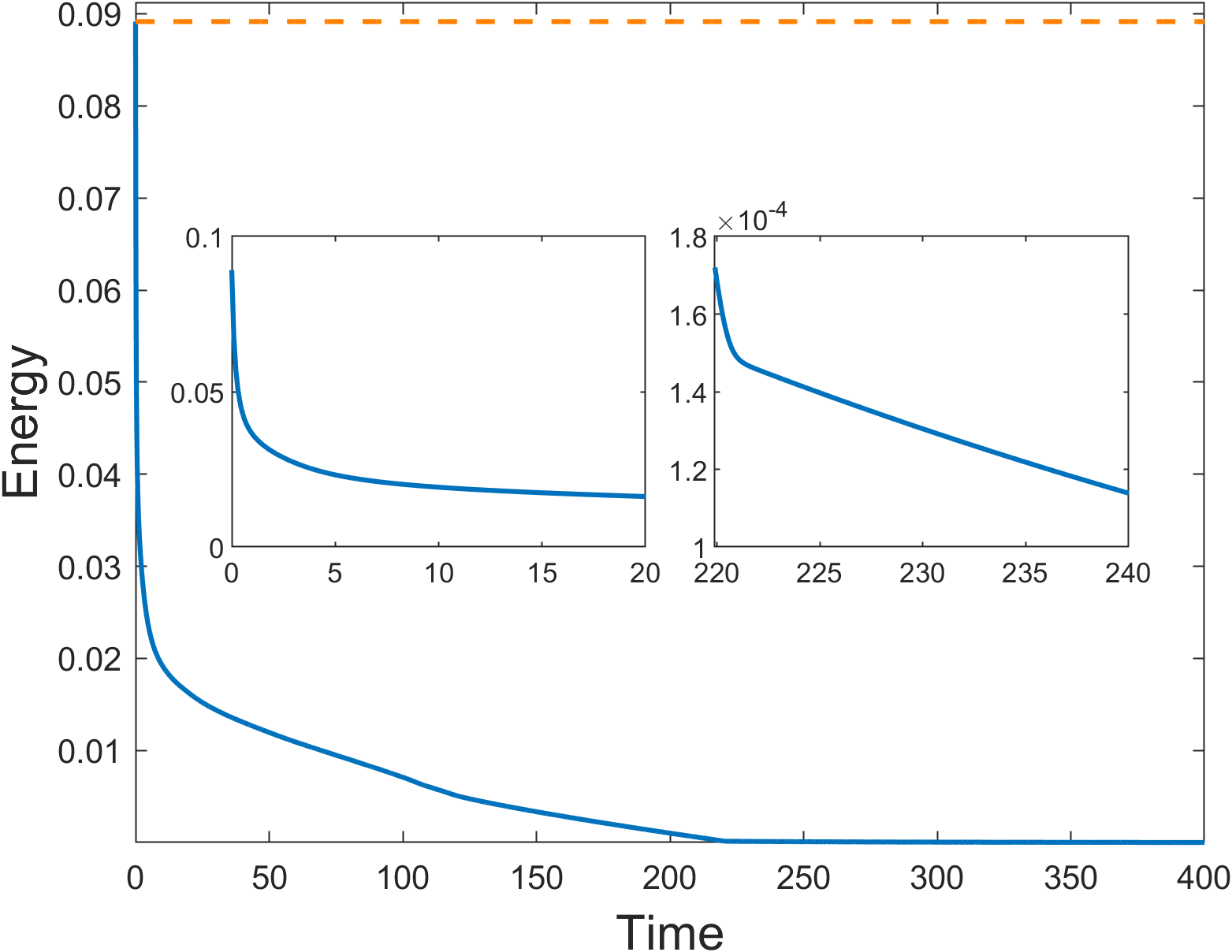}}
		\caption{Evolution of the supremum norm $\|\cdot\|_{\mathcal{X}}$ and energy with initial condition \eqref{eq:4.5} and $\alpha(x,y,z)=2\pi x(y+z)$. The dashed line in the left figure is the maximum bound $\sqrt m$ while the dashed line in the right figure is the initial energy.}
  \label{fig:4.17}
	\end{center}
\end{figure}

{\bf Example 6.} In this example, we study the evolution of the three interlocking rings, noted as condition B: $\left(0.15-\left(x^2+(y+0.21)^2\right)^{\frac{1}{2}}\right)^2+z^2<r^2$ or $\left(0.2-\left(y^2+z^2\right)^{\frac{1}{2}}\right)^2+x^2<r^2$ or $\left(0.15-\left(x^2+(y-0.21)^2\right)^{\frac{1}{2}}\right)^2+z^2<r^2$. We observe the effect of the value of the radius of the rings on the evolutionary process of the matrix-valued field and interface. Here, we consider two cases with $r=0.04$ and $0.06$. The initial matrix domain is denoted as
\begin{equation}\label{eq:4.6}
     \begin{aligned}
	U^0(x,y,z)=
       \begin{cases}
		\mathcal{P}\begin{bmatrix}
		\frac{1}{2}\cos\alpha&\frac{\sqrt{6}}{2}\cos\alpha&-\frac{\sqrt{2}}{2}\sin\alpha\\
		\frac{1}{2}\sin\alpha&\frac{\sqrt{6}}{2}\sin\alpha&\frac{\sqrt{2}}{2}\cos\alpha\\
            \frac{\sqrt{3}}{2}&-\frac{\sqrt{2}}{2}&0
		\end{bmatrix}\quad \mbox{if}~$(x,y,z)$~\mbox{satisfy condition B},\\
  \\
		\mathcal{P}\begin{bmatrix}
		-\frac{1}{2}\cos\alpha&\frac{\sqrt{6}}{2}\cos\alpha&\frac{\sqrt{2}}{2}\sin\alpha\\
		-\frac{1}{2}\sin\alpha&\frac{\sqrt{6}}{2}\sin\alpha&-\frac{\sqrt{2}}{2}\cos\alpha\\
            \frac{\sqrt{3}}{2}&\frac{\sqrt{2}}{2}&0
		\end{bmatrix}\quad\text{otherwise}.\\
	\end{cases}
     \end{aligned}
\end{equation}
where $\alpha=\alpha(x,y,z)=4\pi xyz$ and $\mathcal{P}$ is the pointwise projection. We take $\varepsilon=0.01$ and use the ETDRK2 scheme with time step $\tau = 0.1$ to study the evolution of three interlocking rings at two different radius.

In Figures \ref{fig:4.18} and \ref{fig:4.19}, we describe the evolution of three interlocking rings with $r=0.06$. We observe that the three rings gradually merge and finally split into two teardrop-shaped spheres, which eventually disappear. The vector field evolves from a disordered state to a uniform state. Then, we simulated the case with $r=0.04$ and plotted the results in Figures \ref{fig:4.21} and \ref{fig:4.22}, respectively. We observe that the three rings are not compatible, but contract independently until they disappear. The vector field evolves from a disordered state to a uniform state in both cases.

In Figures \ref{fig:4.20} and \ref{fig:4.23}, we plot the evolution of the original energy and the supremum norm for each of the two cases. We observe that in both cases the original energy decays with time development and that the supremum norm does not exceed $\sqrt{3}$.
\begin{figure}
	\begin{center}
		\subfigure{\includegraphics[width=0.31\textwidth,
			height=48mm]{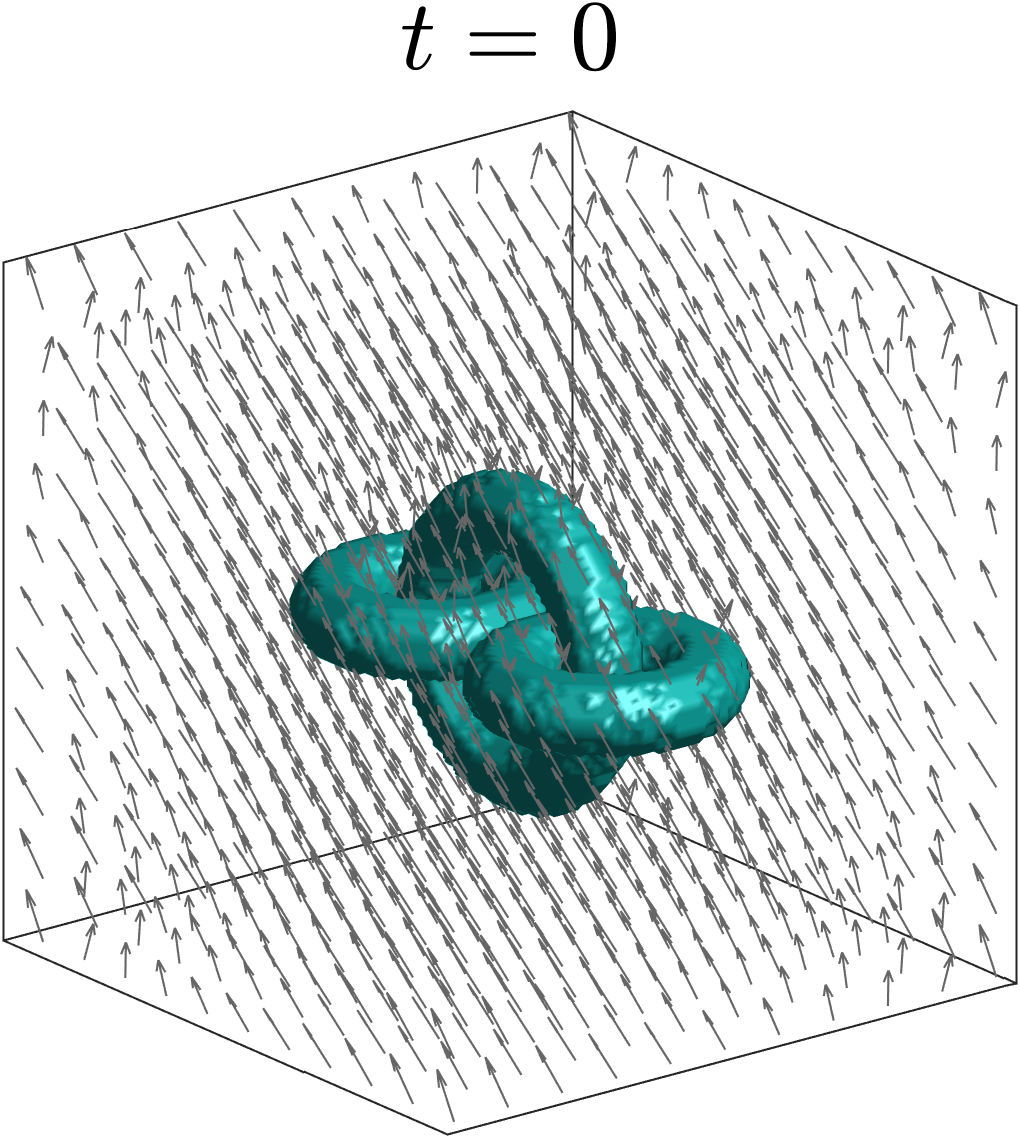}}\quad
		\subfigure{\includegraphics[width=0.31\textwidth,
			height=48mm]{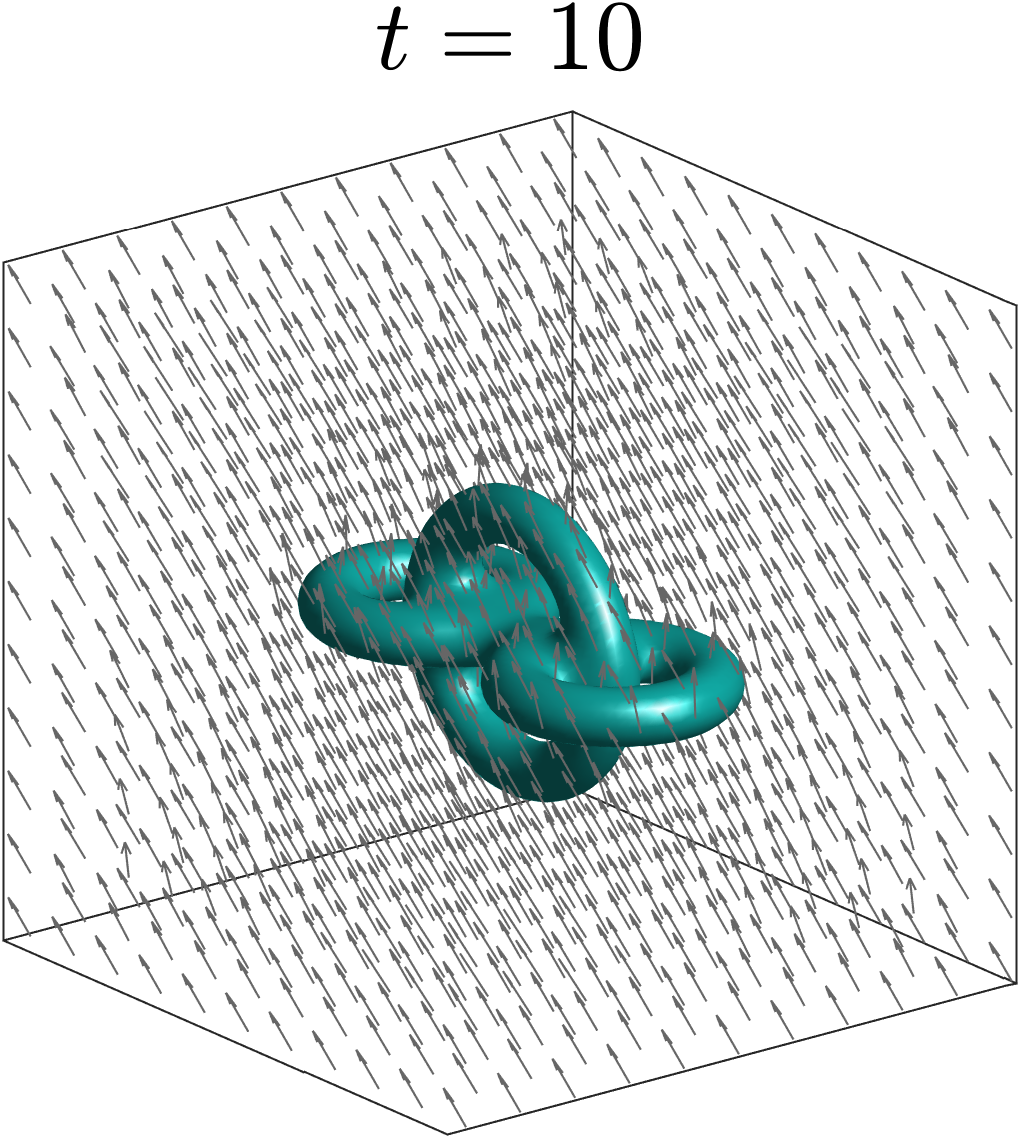}}\quad
		\subfigure{\includegraphics[width=0.31\textwidth,
			height=48mm]{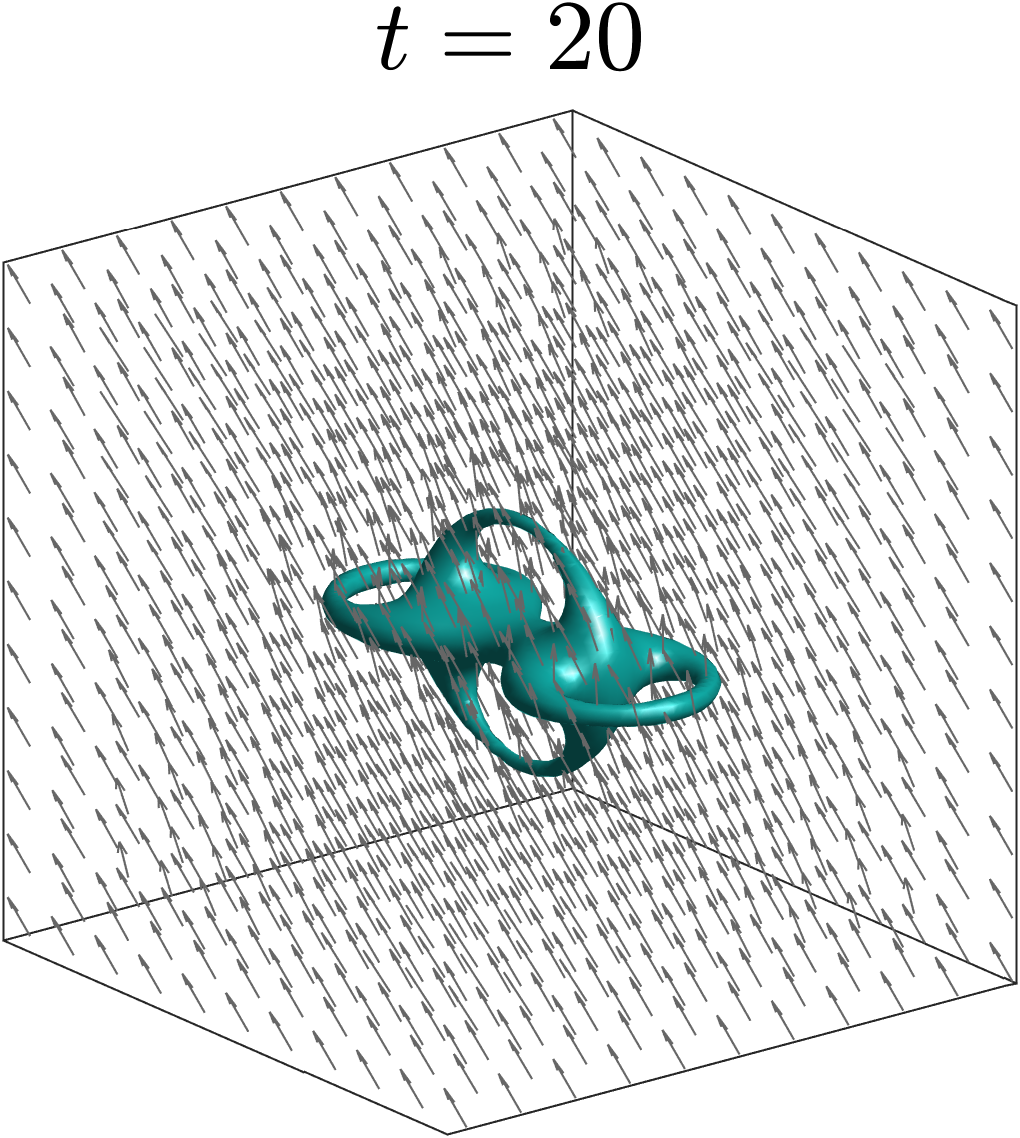}}\quad
		\subfigure{\includegraphics[width=0.31\textwidth,
			height=48mm]{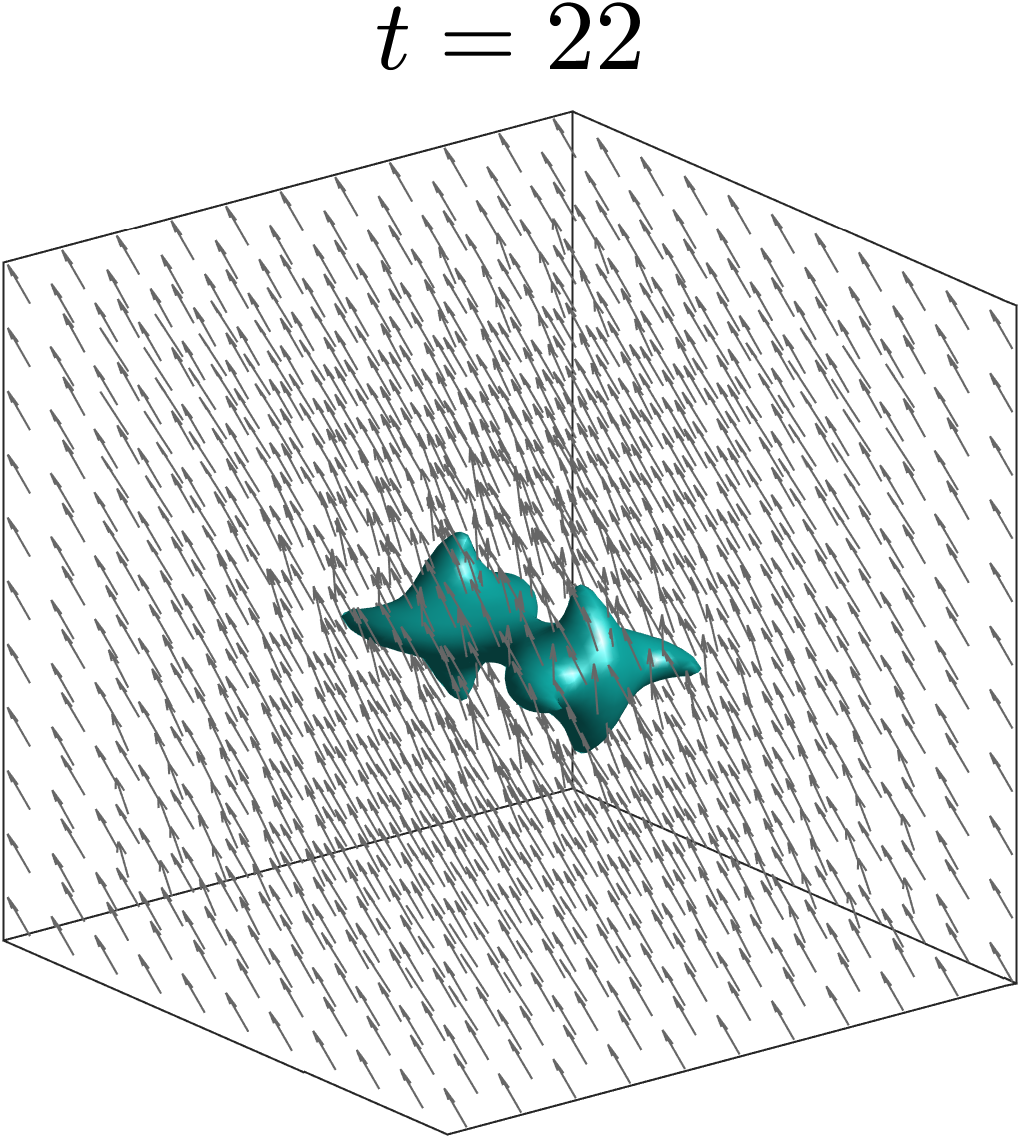}}\quad
		\subfigure{\includegraphics[width=0.31\textwidth,
			height=48mm]{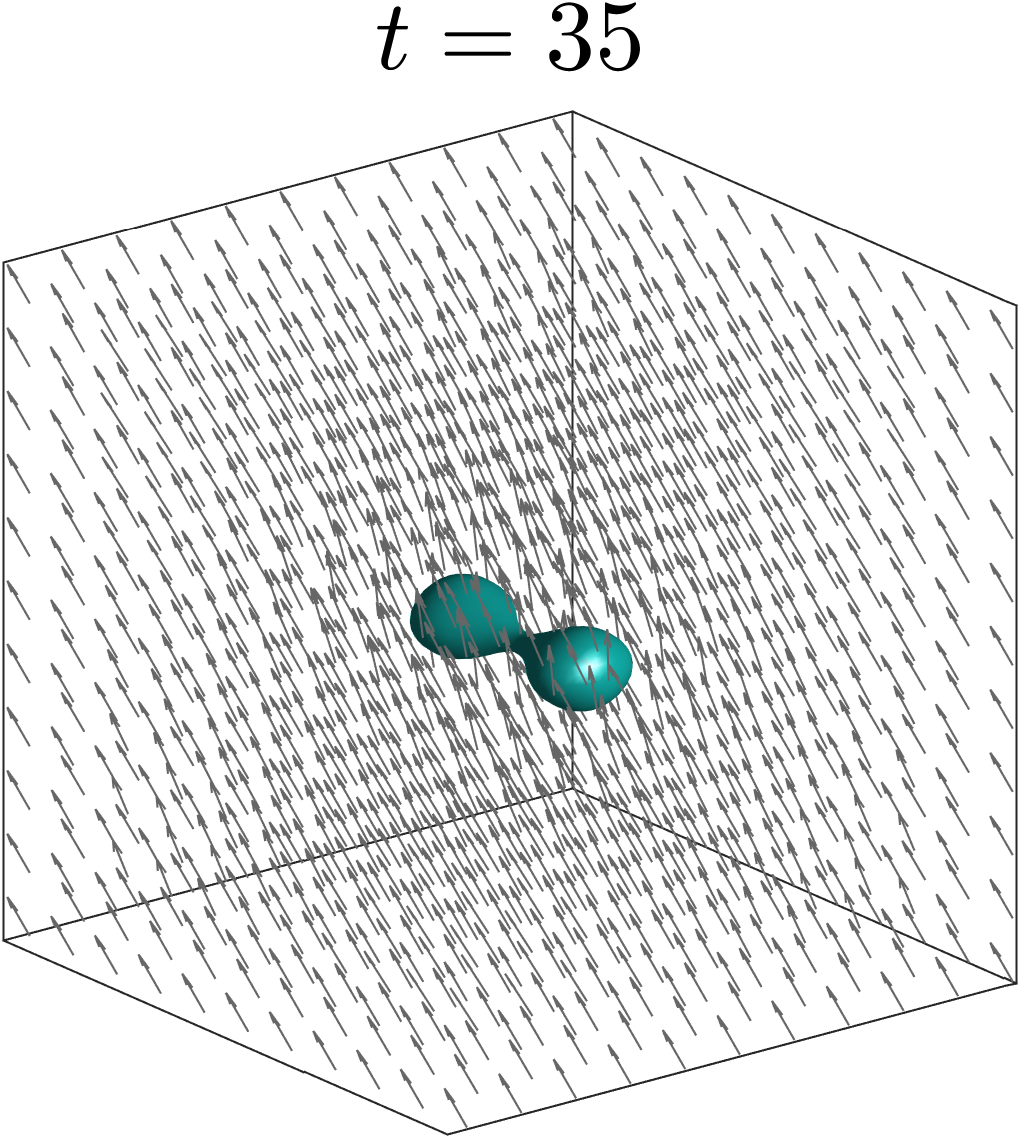}}\quad
		\subfigure{\includegraphics[width=0.31\textwidth,
			height=48mm]{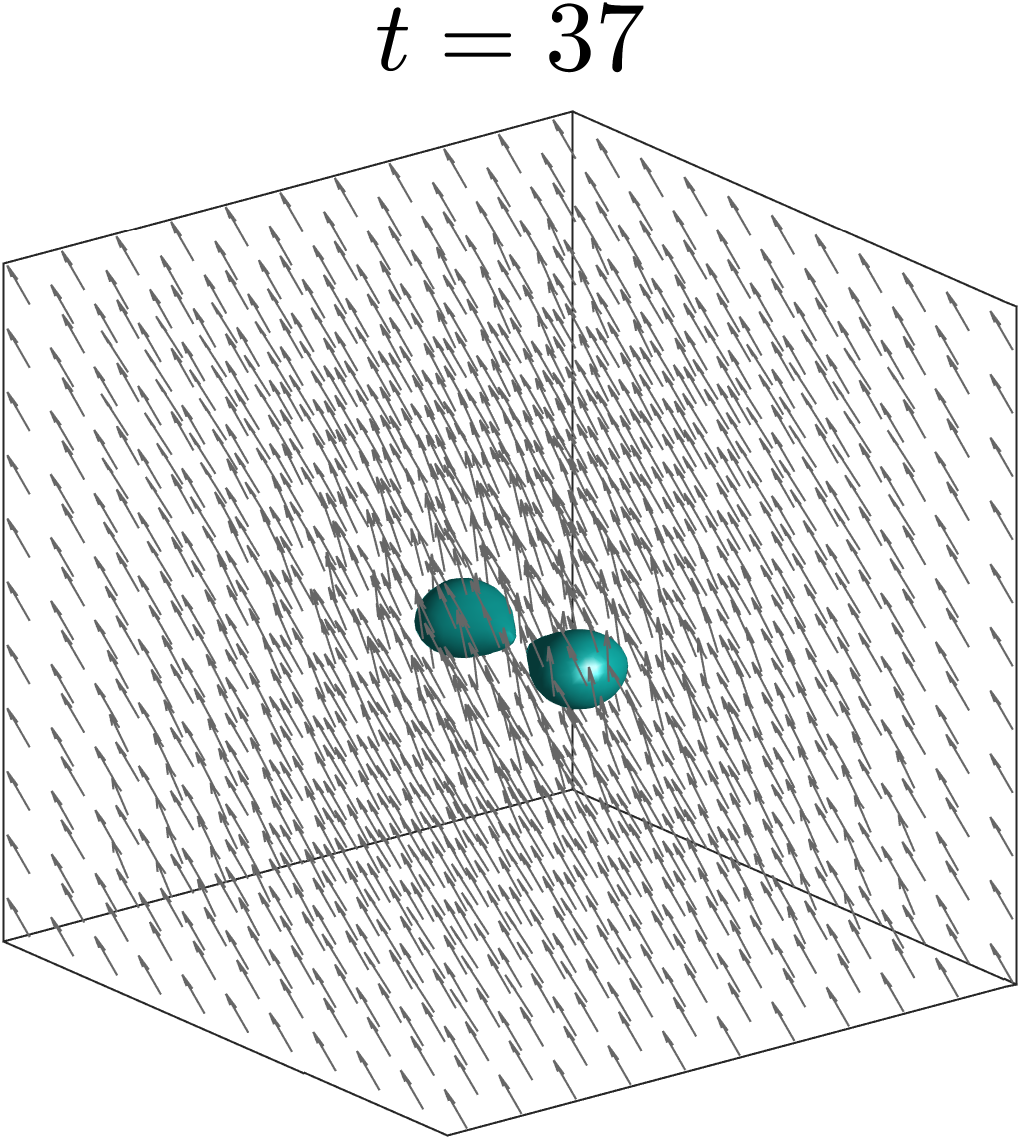}}\quad
    	\subfigure{\includegraphics[width=0.31\textwidth,
			height=48mm]{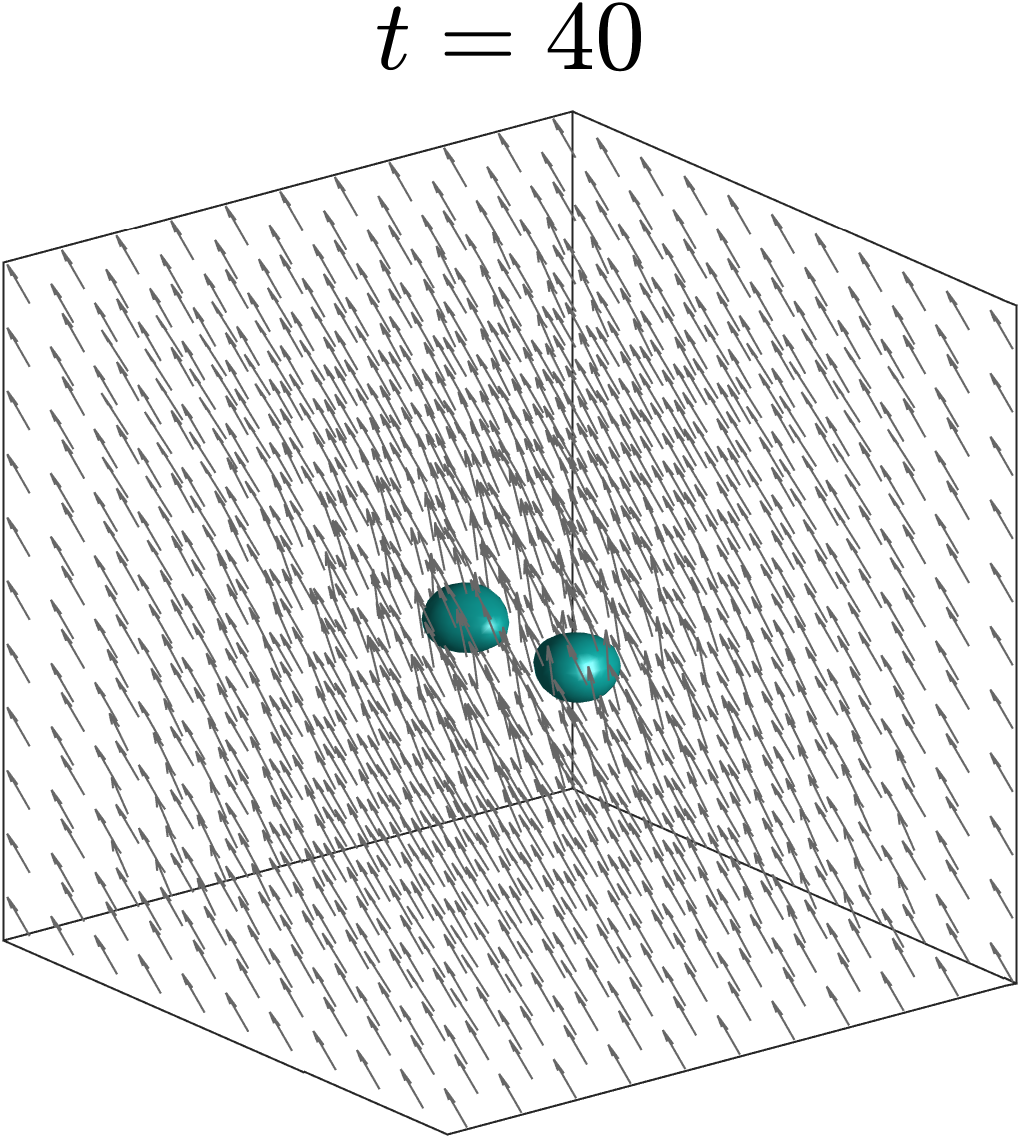}}\quad
		\subfigure{\includegraphics[width=0.31\textwidth,
			height=48mm]{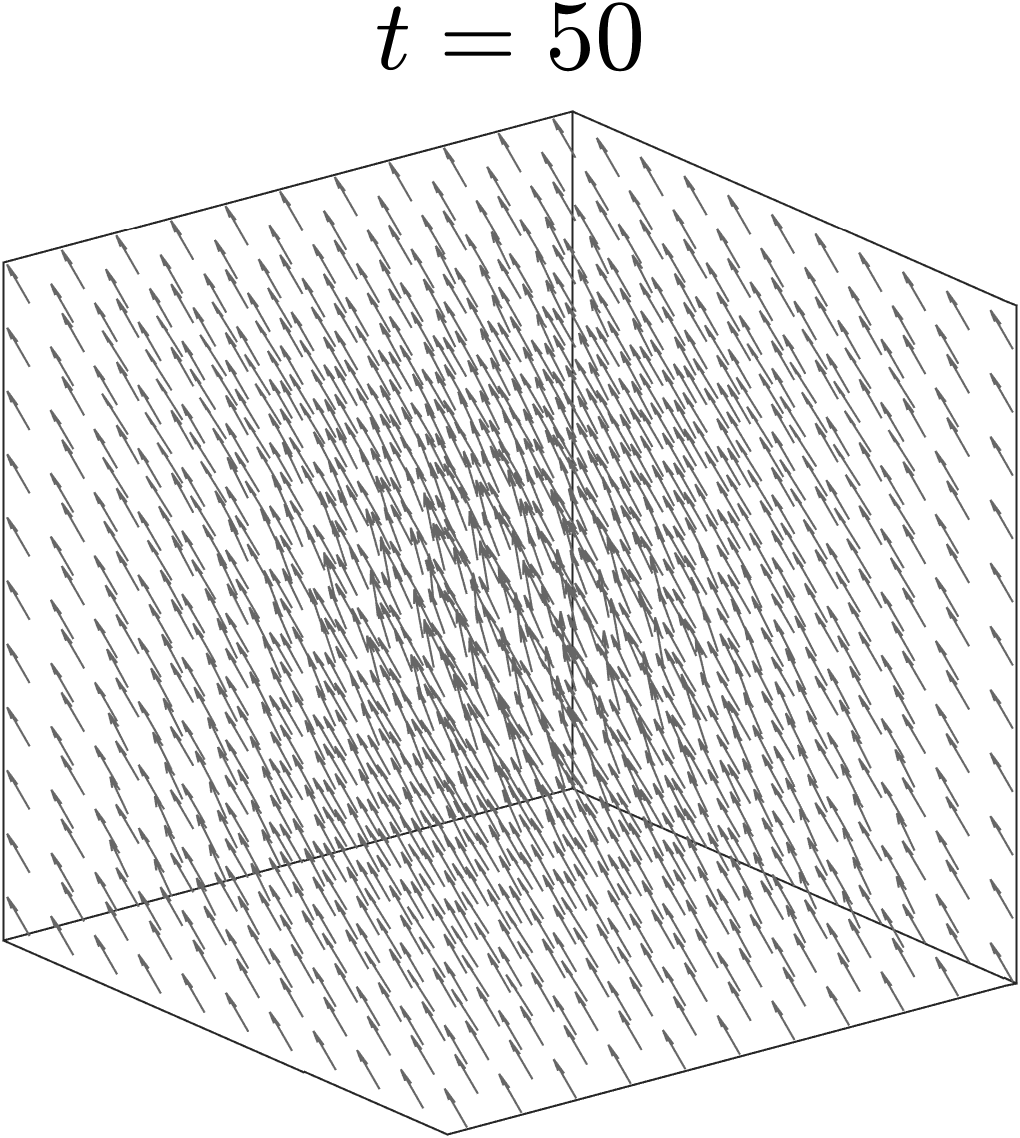}}\quad
		\subfigure{\includegraphics[width=0.31\textwidth,
			height=48mm]{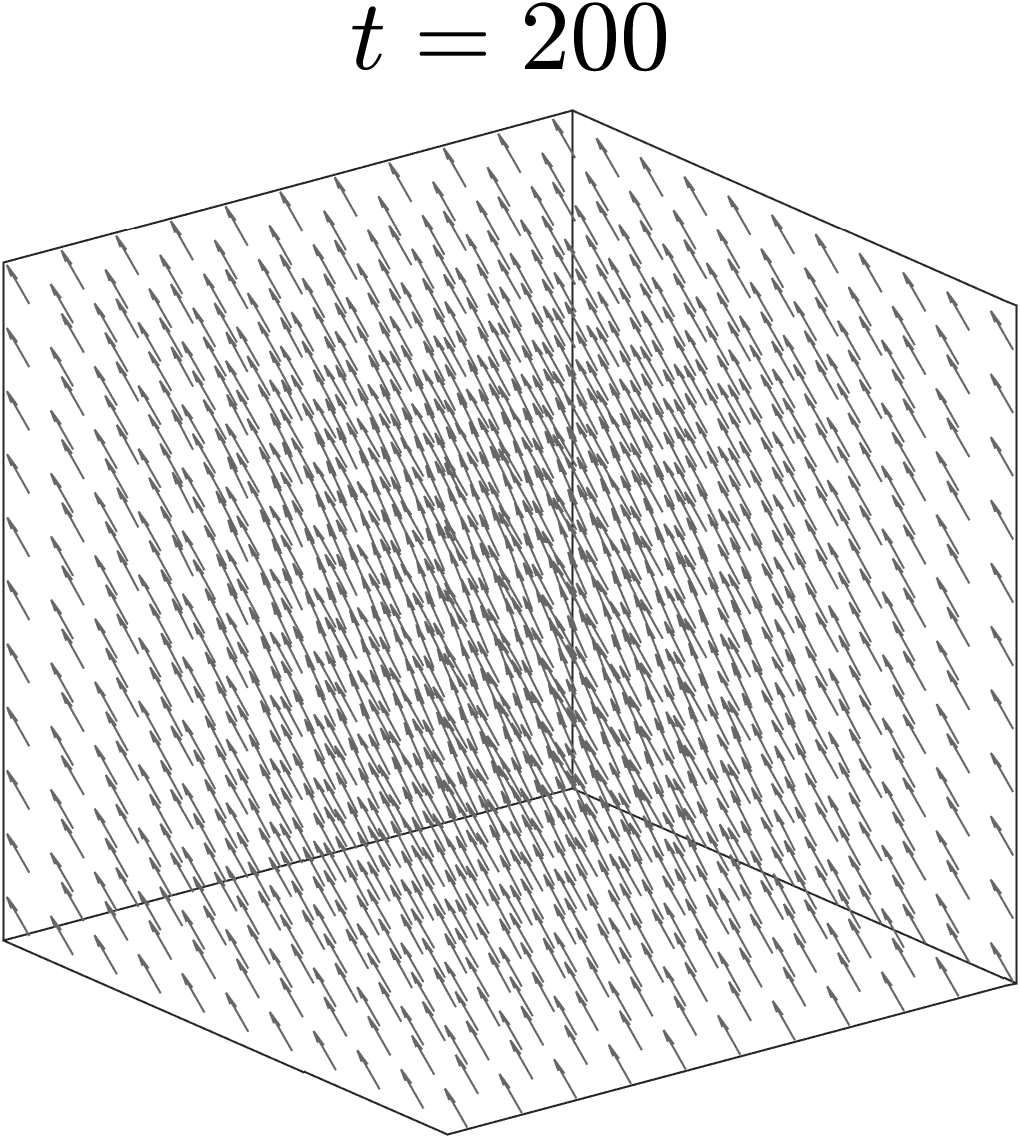}}
\caption{Evolution of the matrix-valued field and interface at $t=0,10,20,22,35,37,40,$ $50,200$. The initial field is given in \eqref{eq:4.6} with $\alpha(x,y,z)=4\pi xyz$ and $r=0.06$.}
  \label{fig:4.18}
	\end{center}
\end{figure}

\begin{figure}
    \begin{center}
		\subfigure{\includegraphics[width=0.28\textwidth,
			height=40mm]{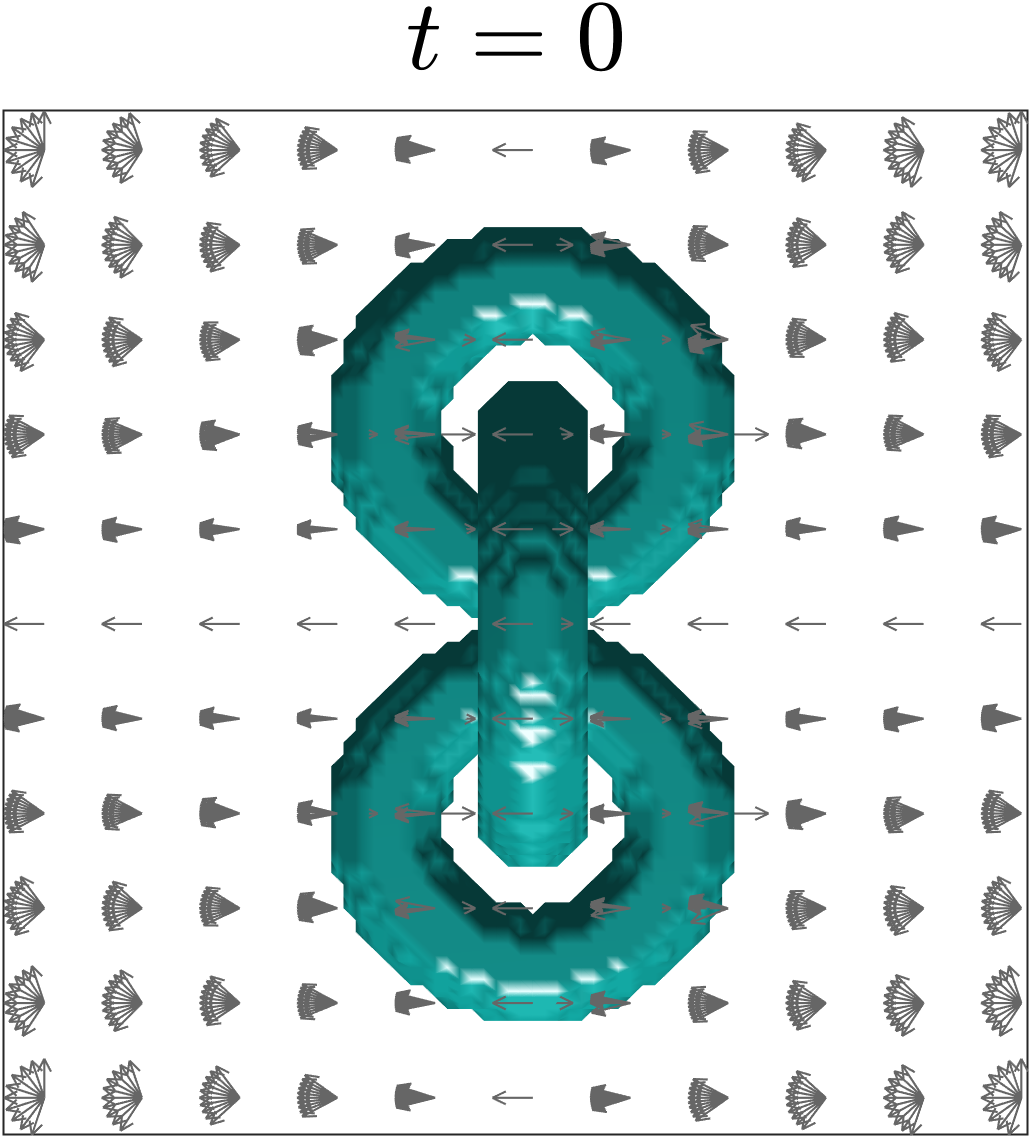}}\quad
		\subfigure{\includegraphics[width=0.28\textwidth,
			height=40mm]{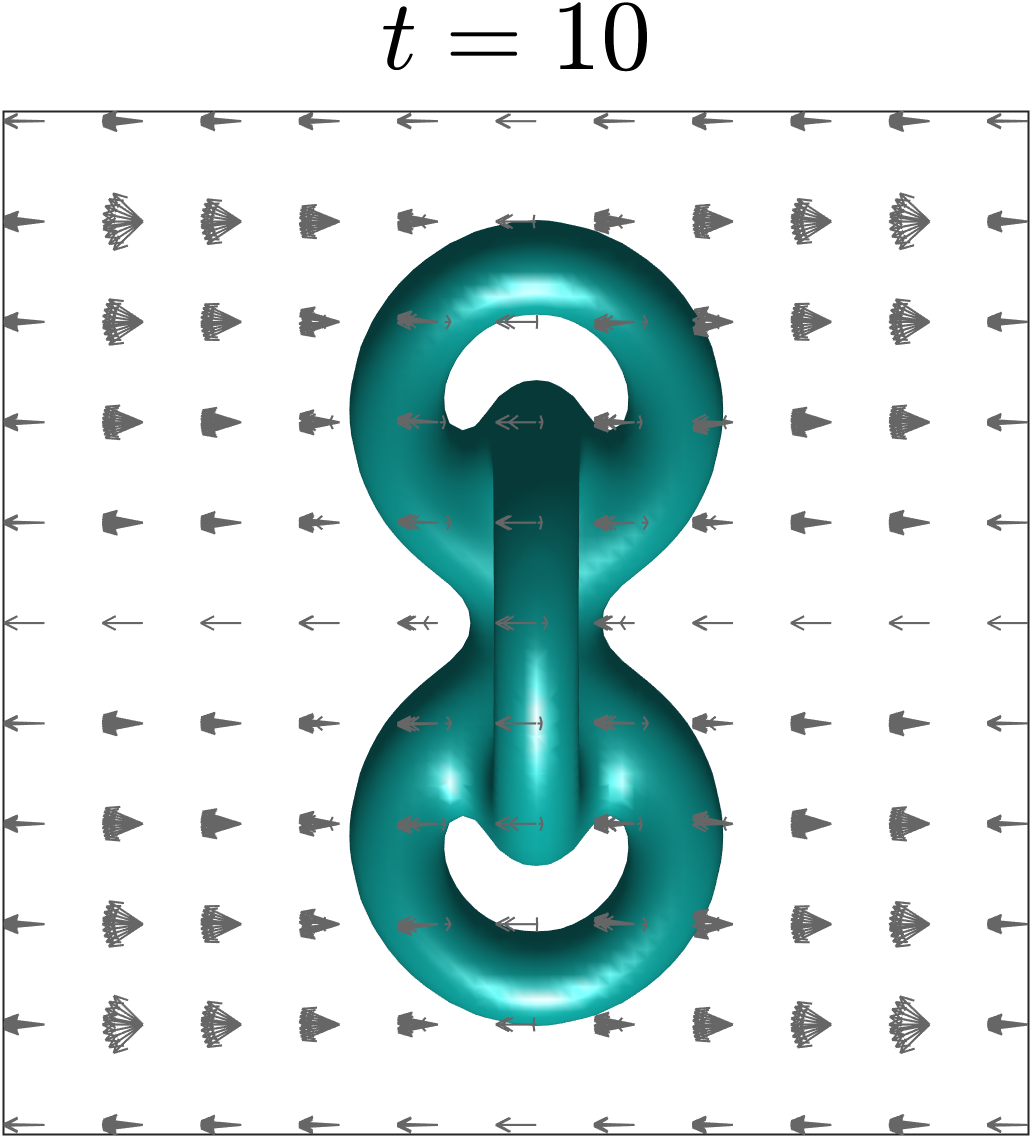}}\quad
    	\subfigure{\includegraphics[width=0.28\textwidth,
			height=40mm]{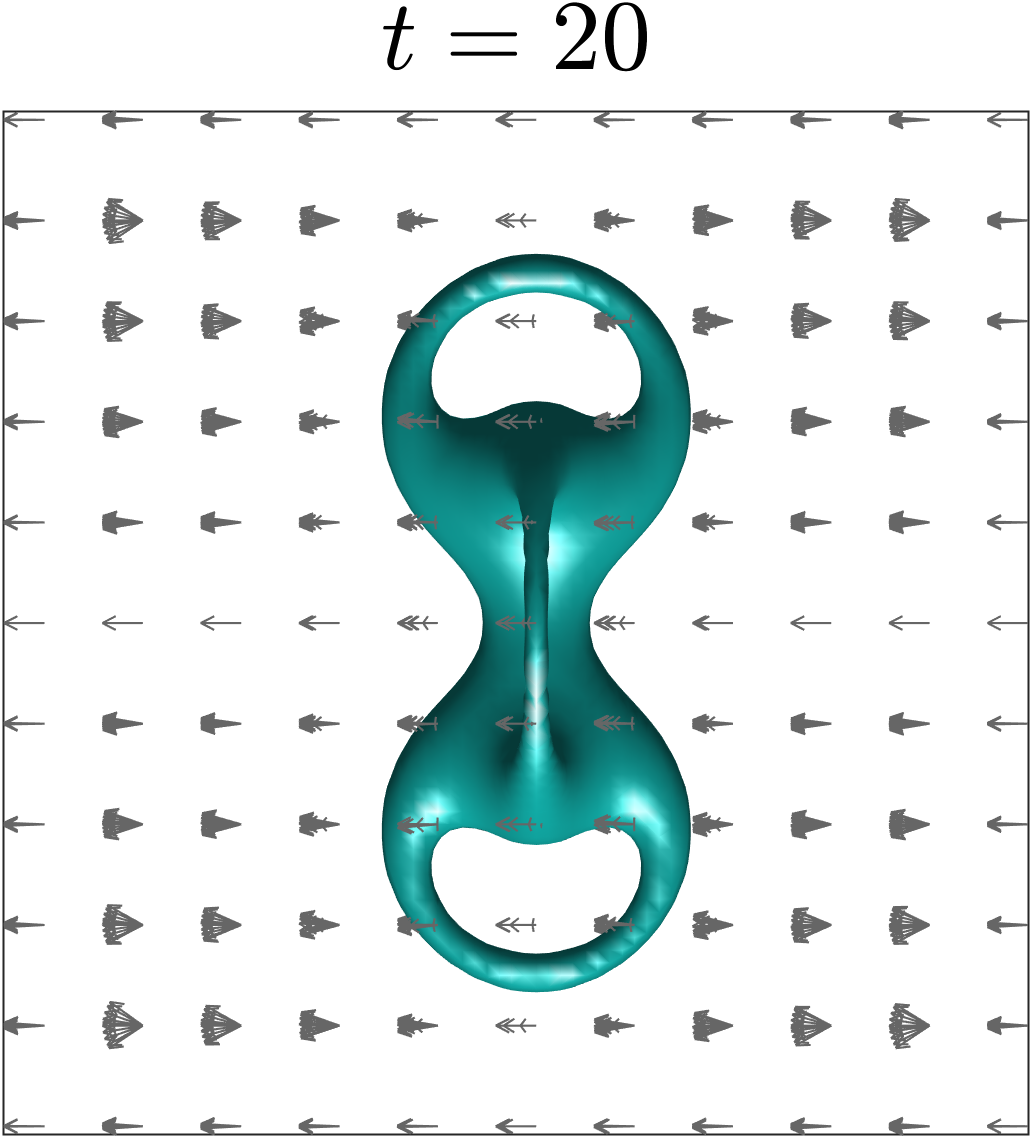}}\quad
		\subfigure{\includegraphics[width=0.28\textwidth,
			height=40mm]{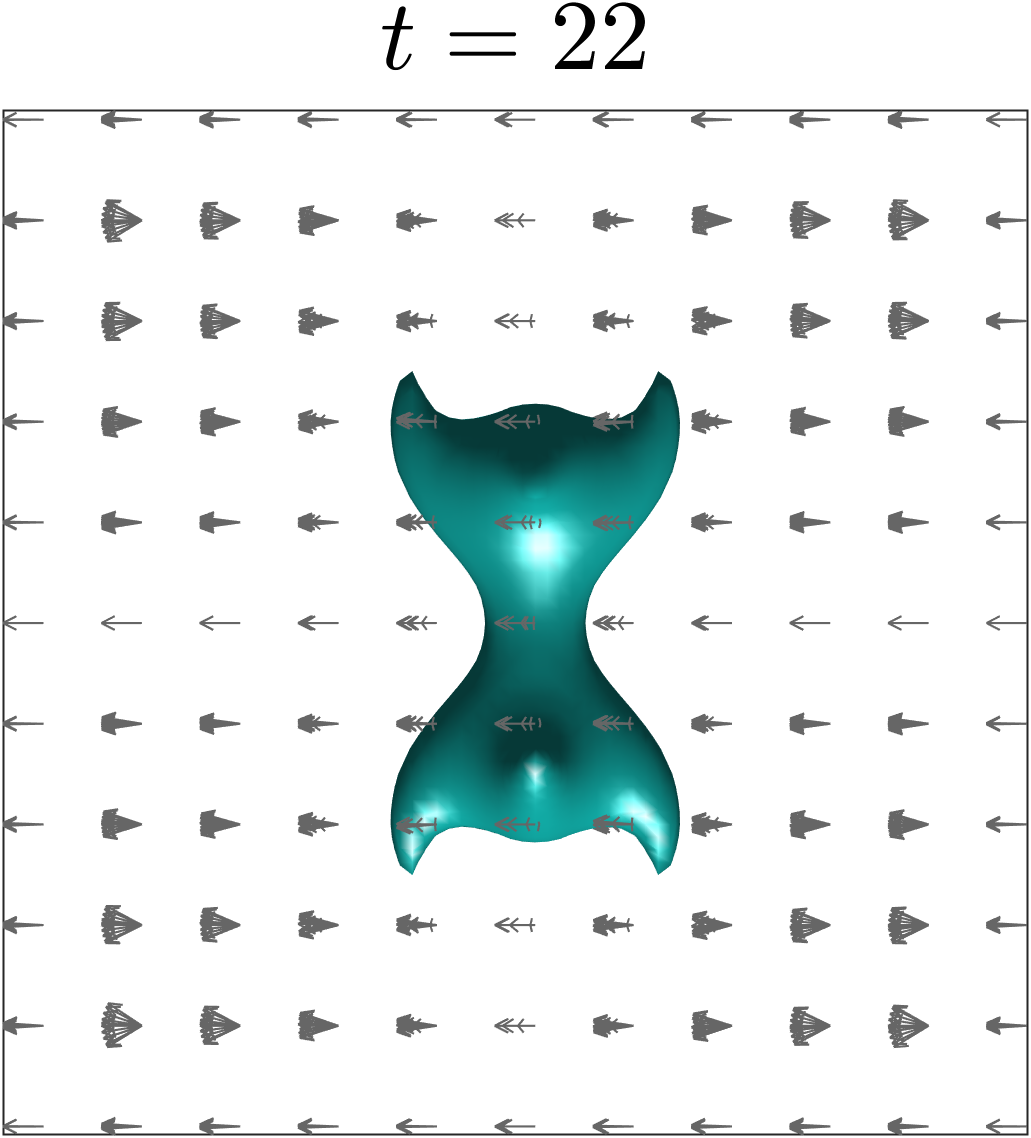}}\quad
    	\subfigure{\includegraphics[width=0.28\textwidth,
			height=40mm]{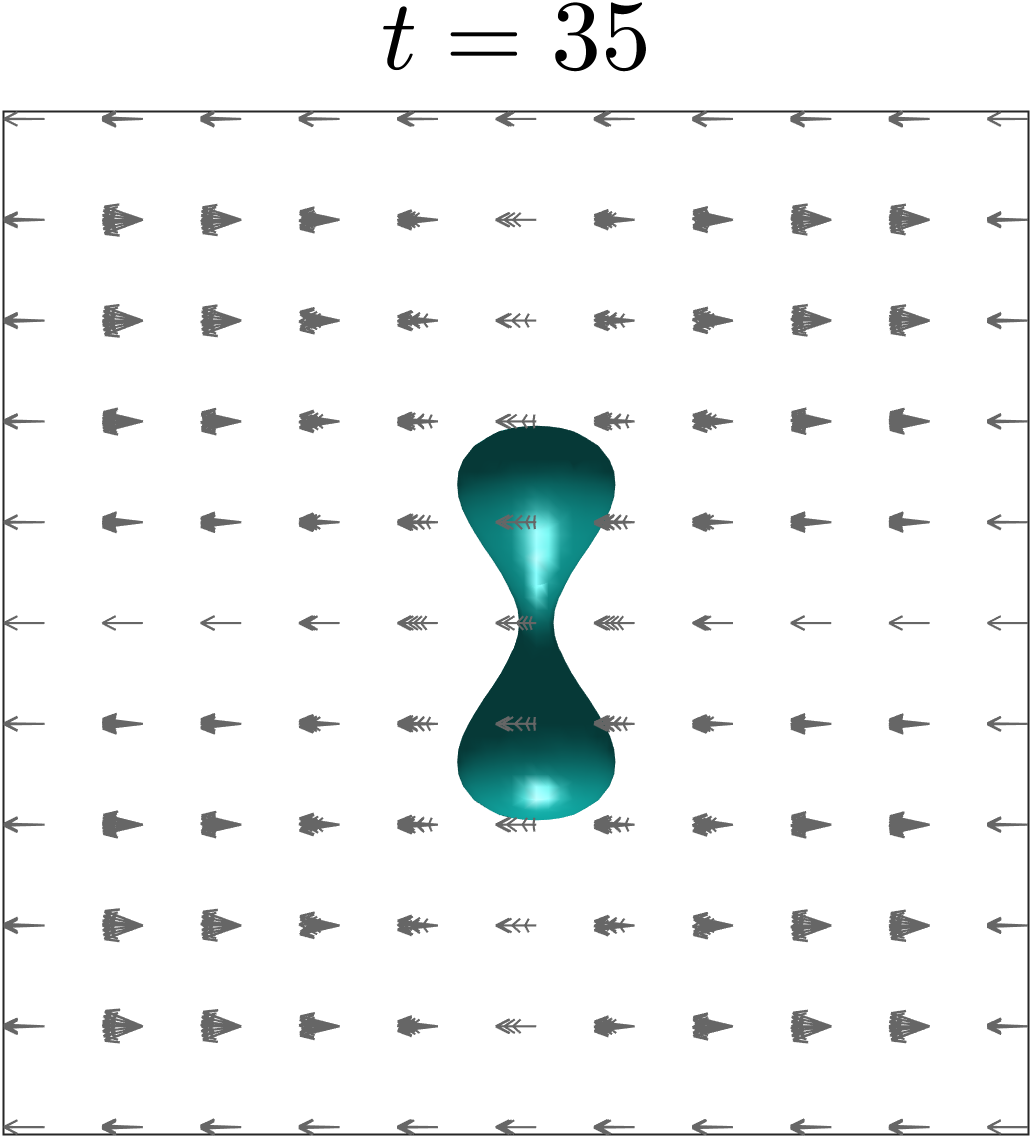}}\quad
		\subfigure{\includegraphics[width=0.28\textwidth,
			height=40mm]{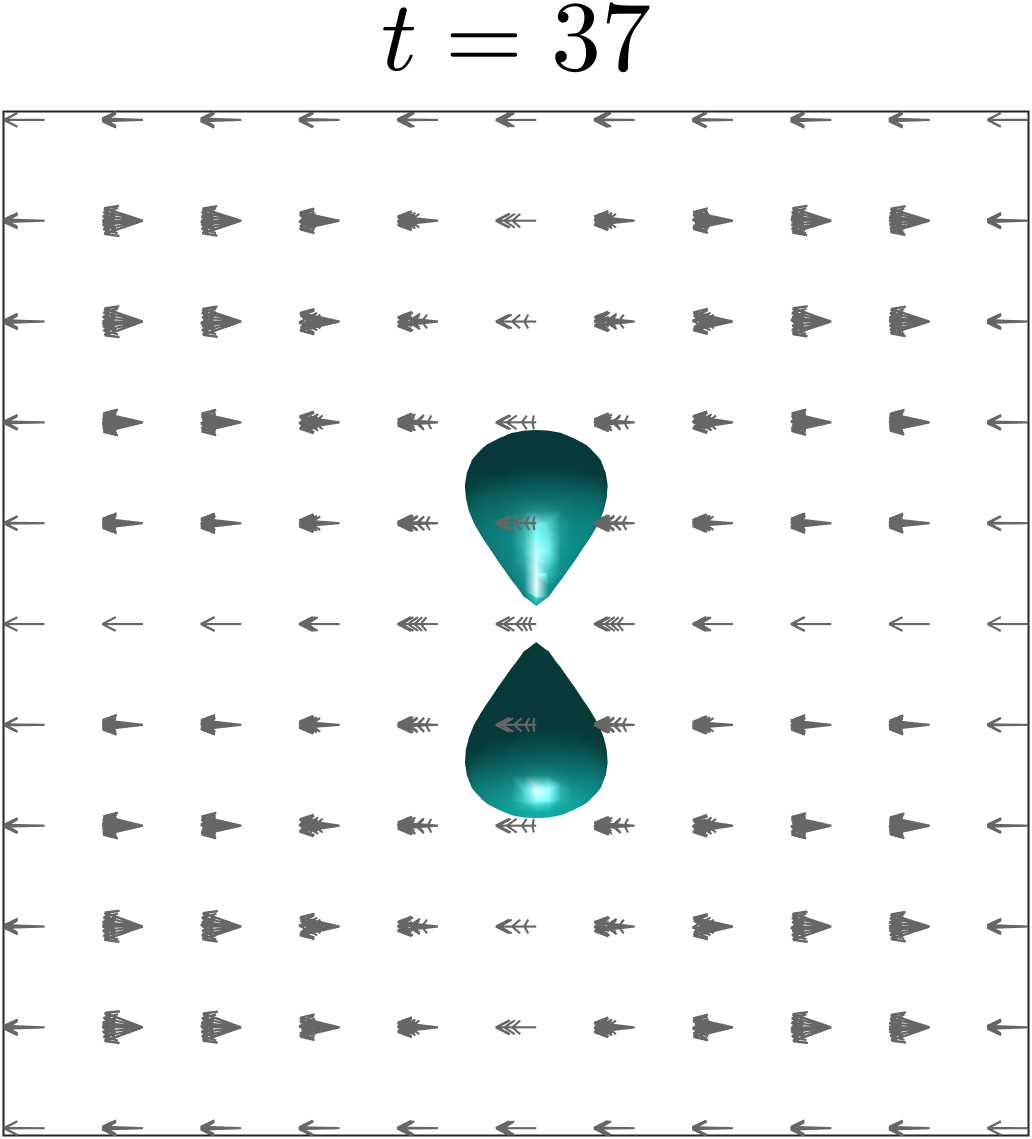}}\quad
    	\subfigure{\includegraphics[width=0.28\textwidth,
			height=40mm]{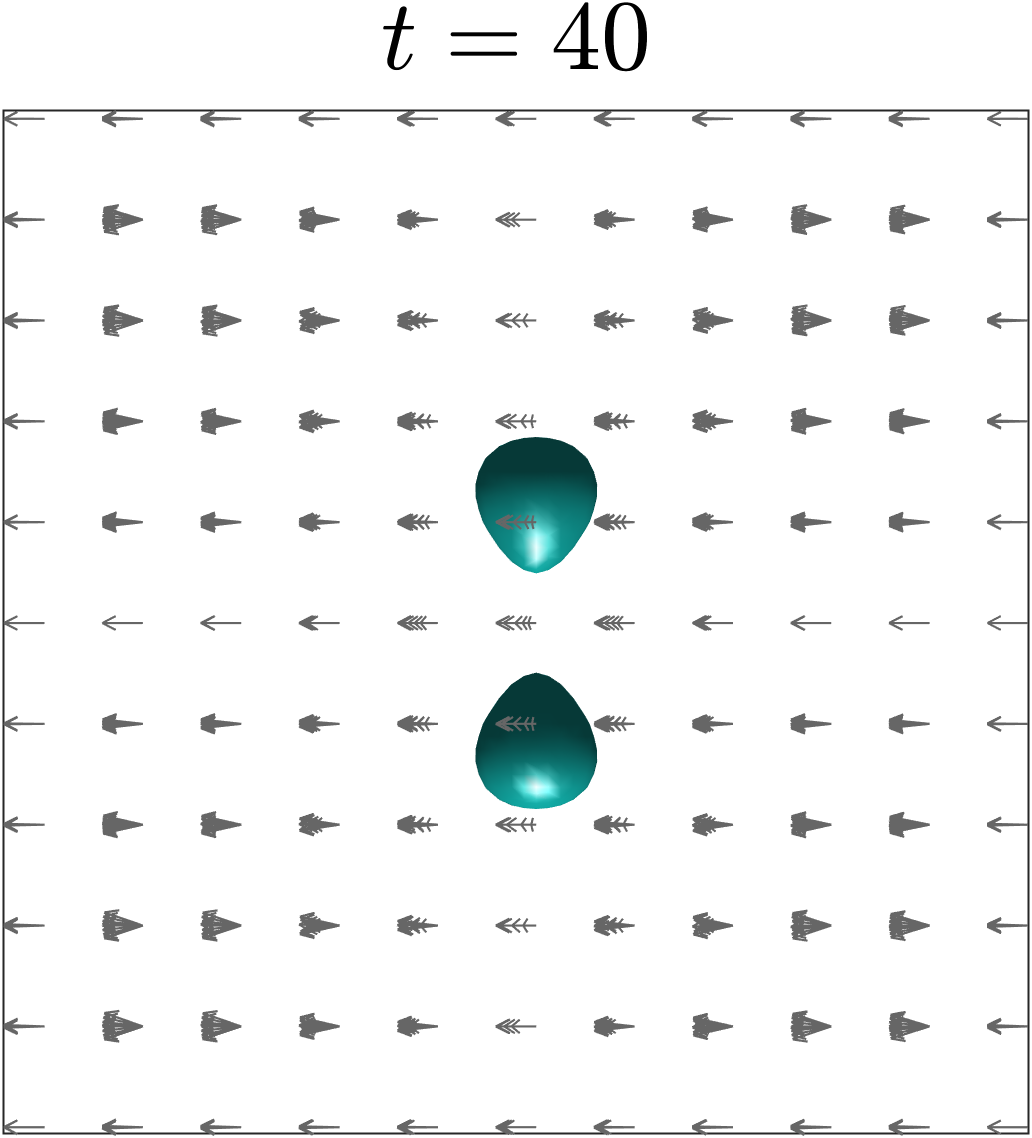}}\quad
		\subfigure{\includegraphics[width=0.28\textwidth,
			height=40mm]{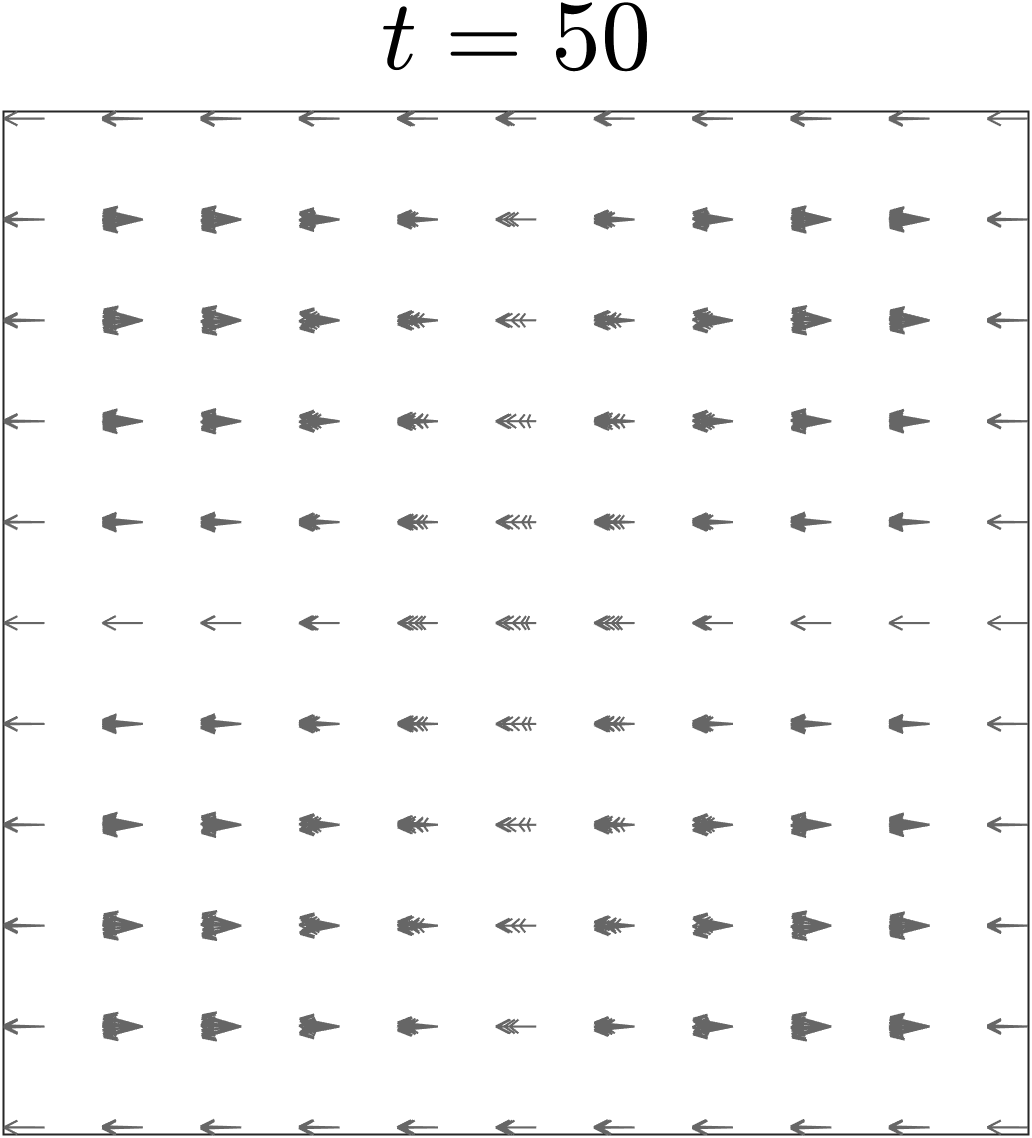}}\quad
		\subfigure{\includegraphics[width=0.28\textwidth,
			height=40mm]{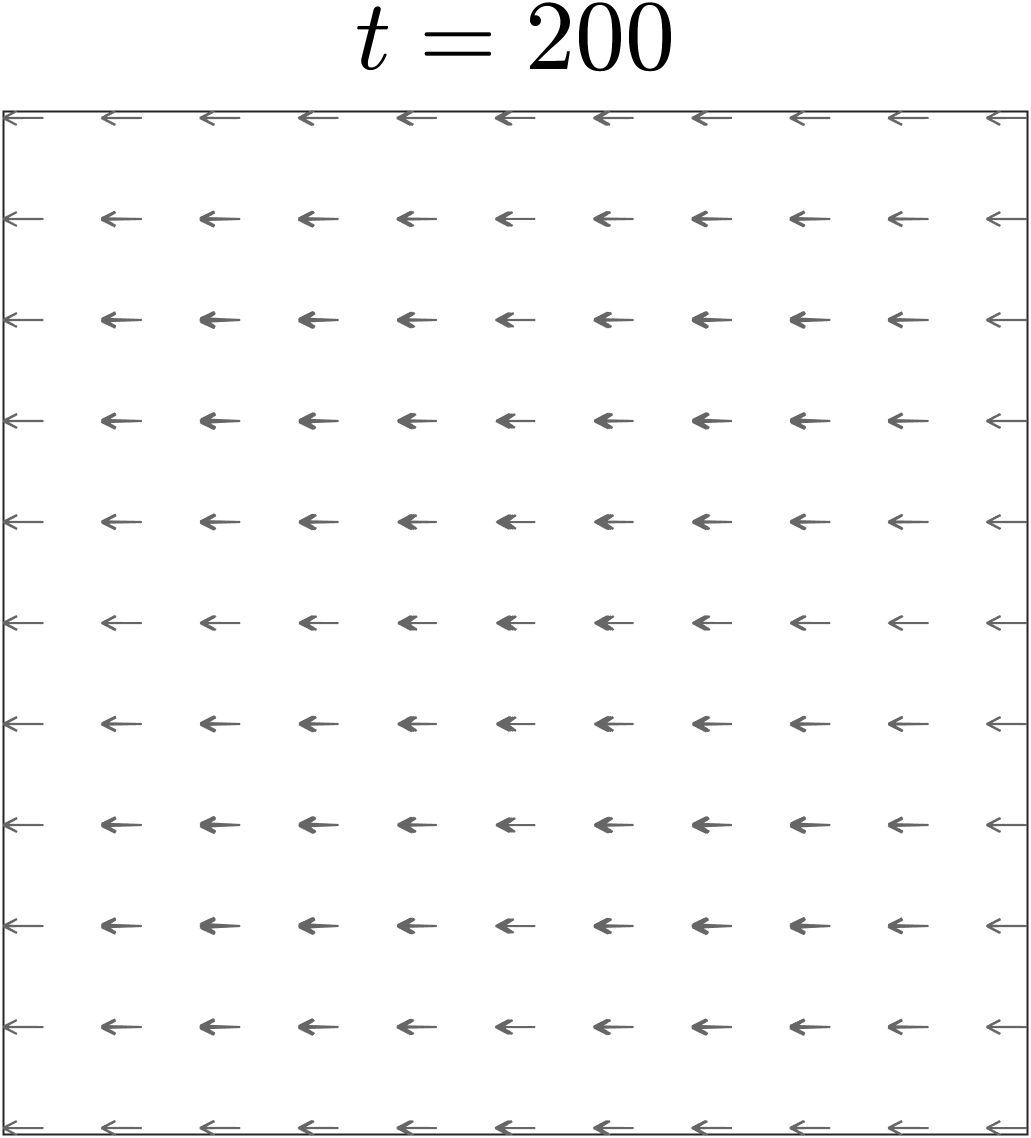}}
		\caption{A vertical view of the evolution of the matrix-valued field and interface at $t=0,10,20,22,35,37,40,50,200$. The initial field is given in \eqref{eq:4.6} with $\alpha(x,y,z)=4\pi xyz$ and $r=0.06$.}
  \label{fig:4.19}
    \end{center}
\end{figure}

 \begin{figure}
	\begin{center}
		\subfigure{\includegraphics[width=0.42\textwidth,
			height=45mm]{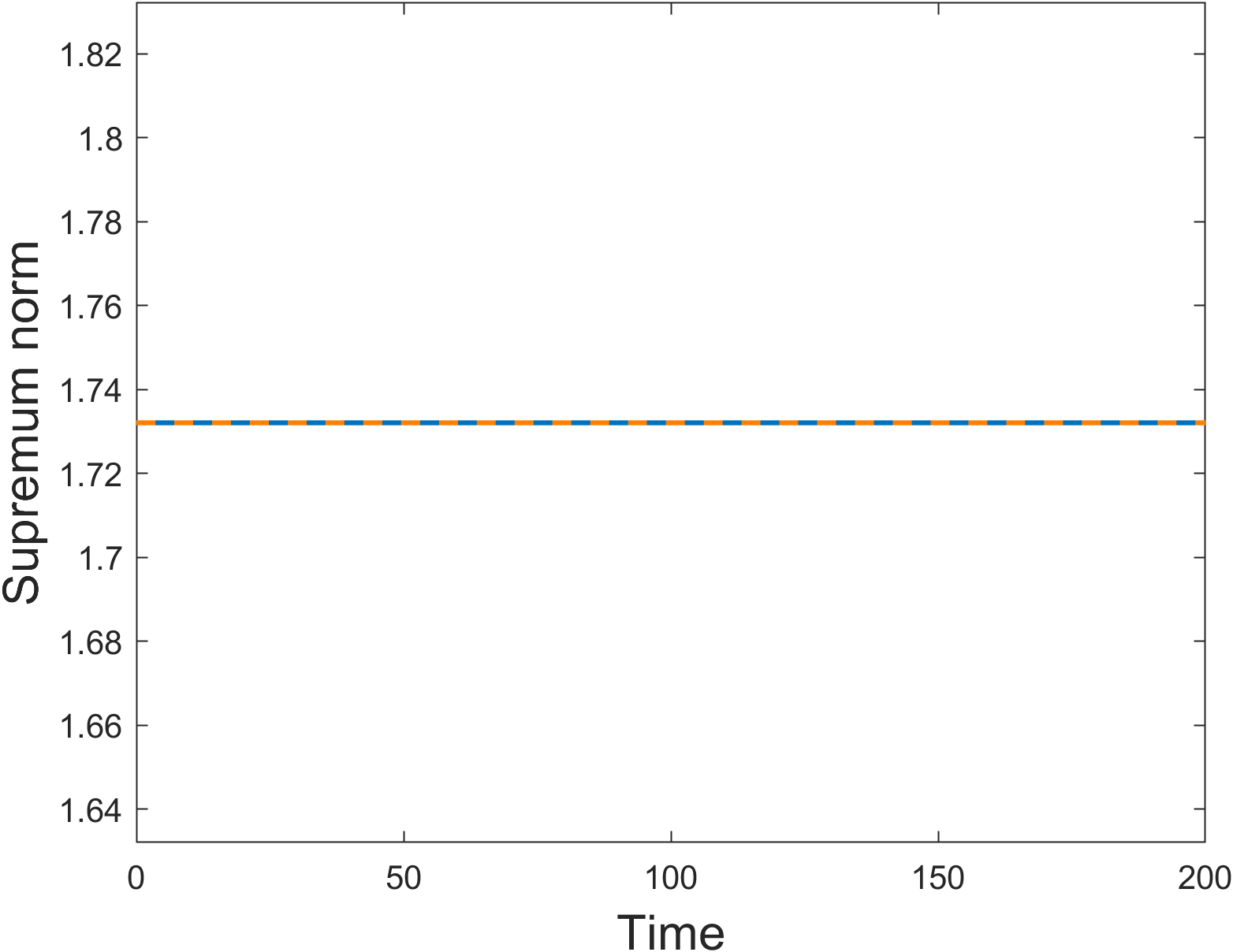}}\quad
		\subfigure{\includegraphics[width=0.42\textwidth,
			height=45mm]{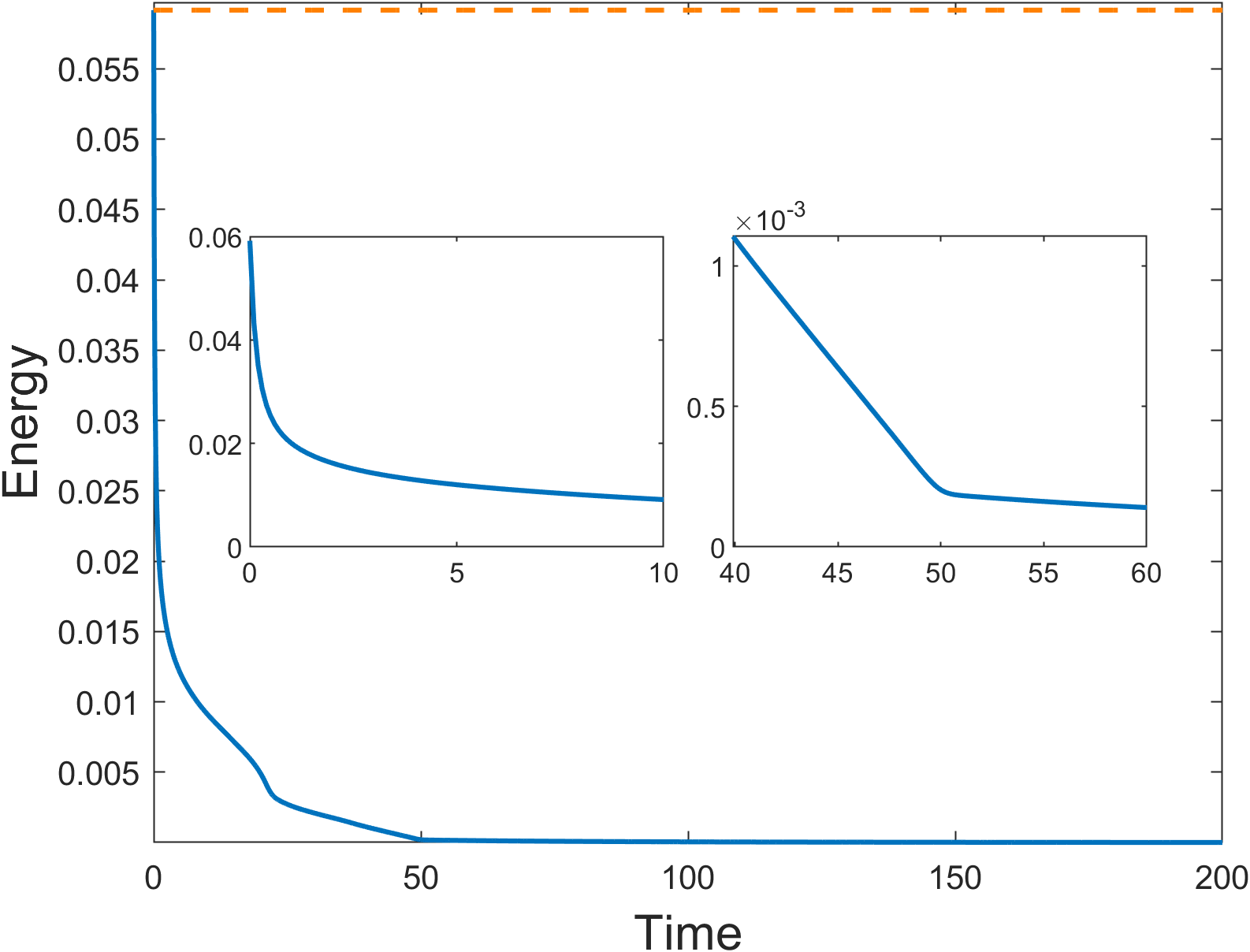}}
		\caption{Evolution of the supremum norm $\|\cdot\|_{\mathcal{X}}$ and energy with initial condition \eqref{eq:4.6} and $\alpha(x,y,z)=4\pi xyz$, $r=0.06$. The dashed line in the left figure is the maximum bound $\sqrt m$ while the dashed line in the right figure is the initial energy.}
  \label{fig:4.20}
	\end{center}
\end{figure}

\begin{figure}
	\begin{center}
		\subfigure{\includegraphics[width=0.31\textwidth,
			height=48mm]{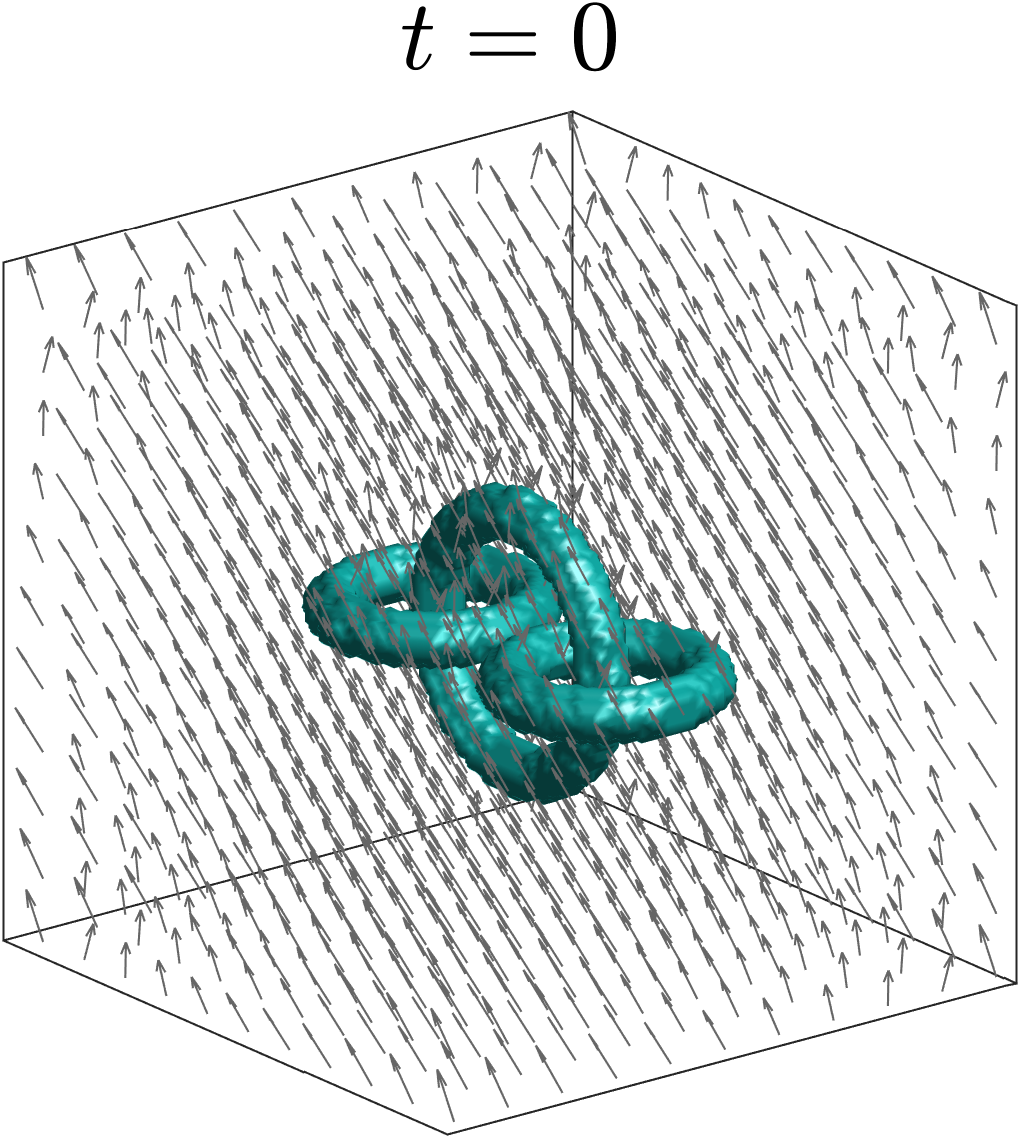}}\quad
		\subfigure{\includegraphics[width=0.31\textwidth,
			height=48mm]{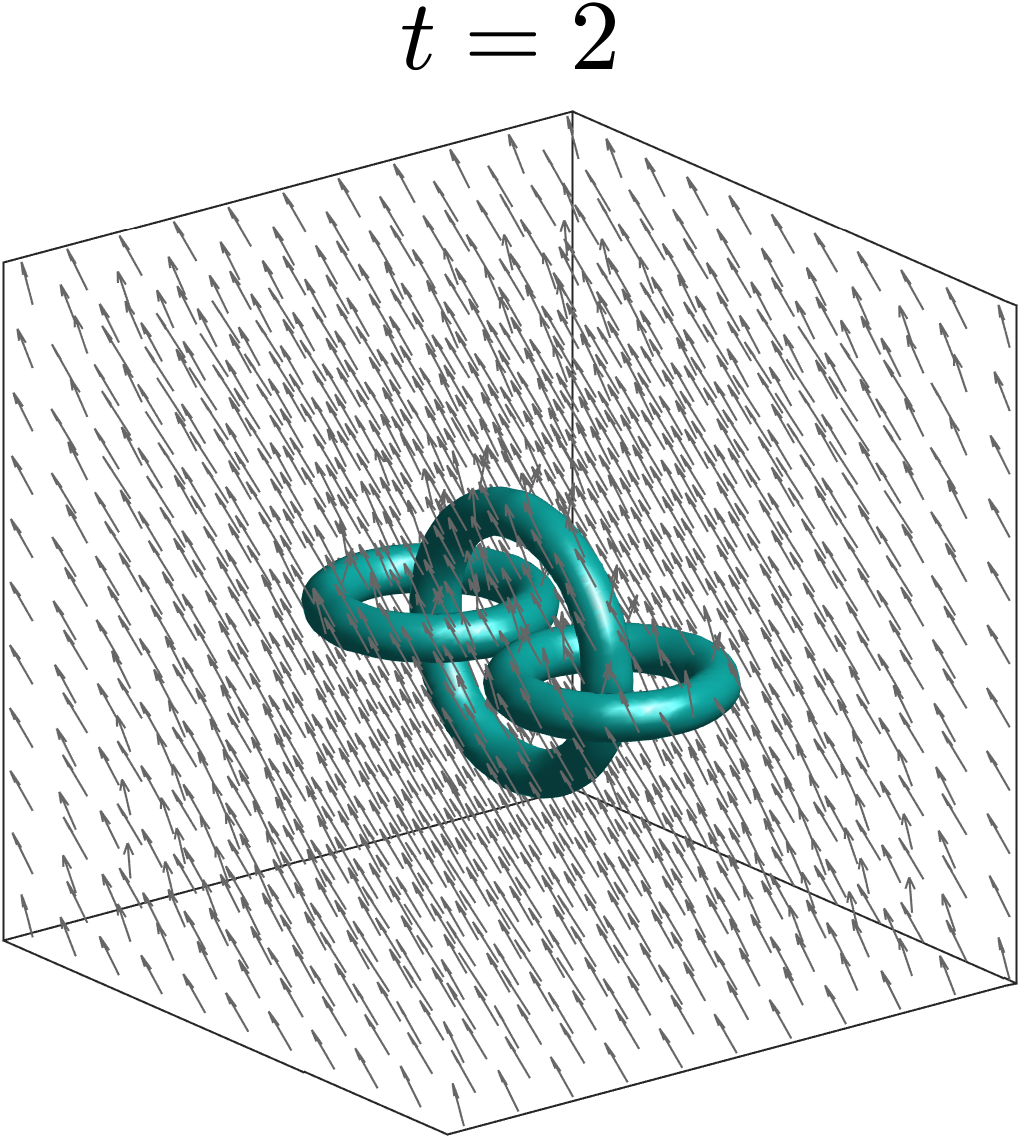}}\quad
		\subfigure{\includegraphics[width=0.31\textwidth,
			height=48mm]{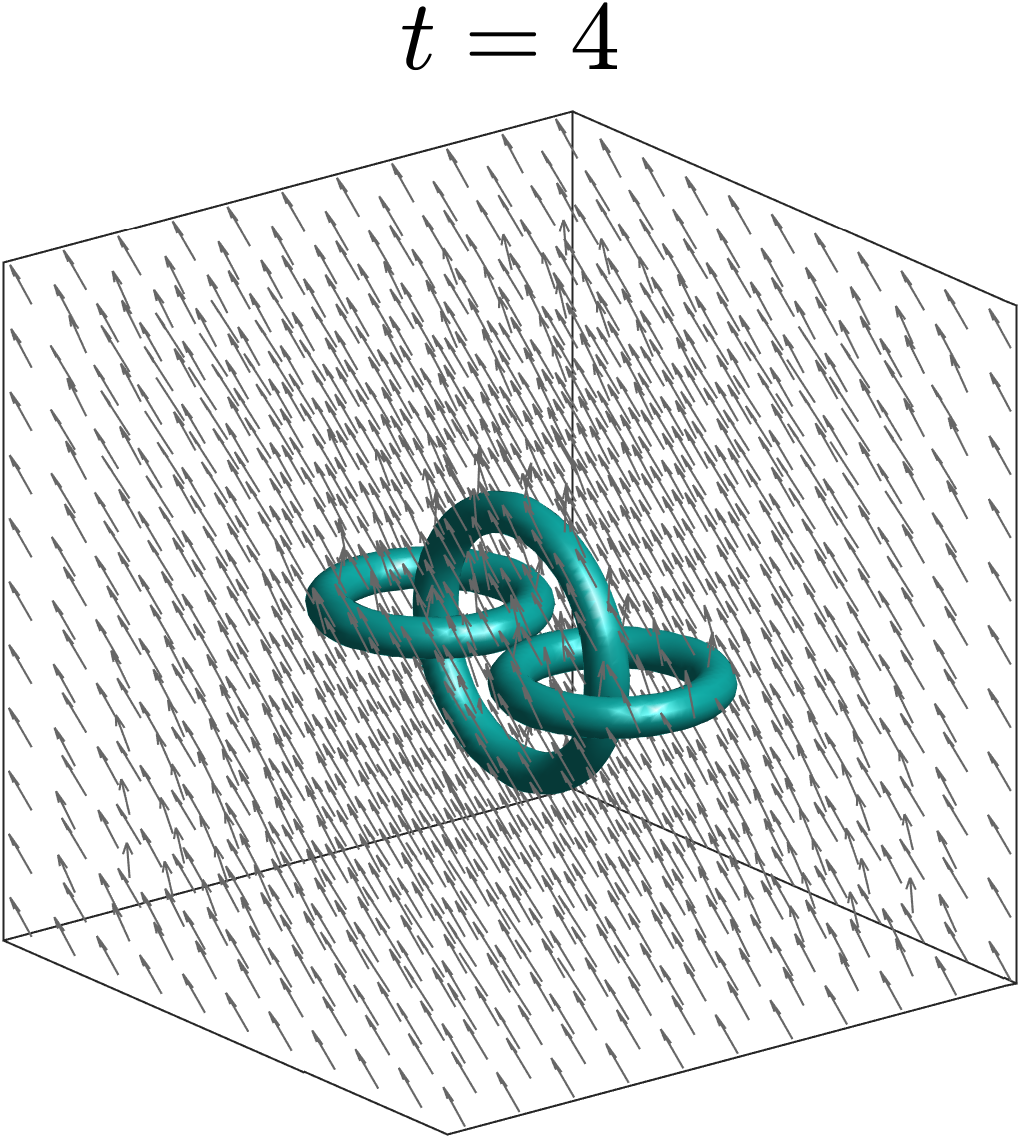}}\quad
		\subfigure{\includegraphics[width=0.31\textwidth,
			height=48mm]{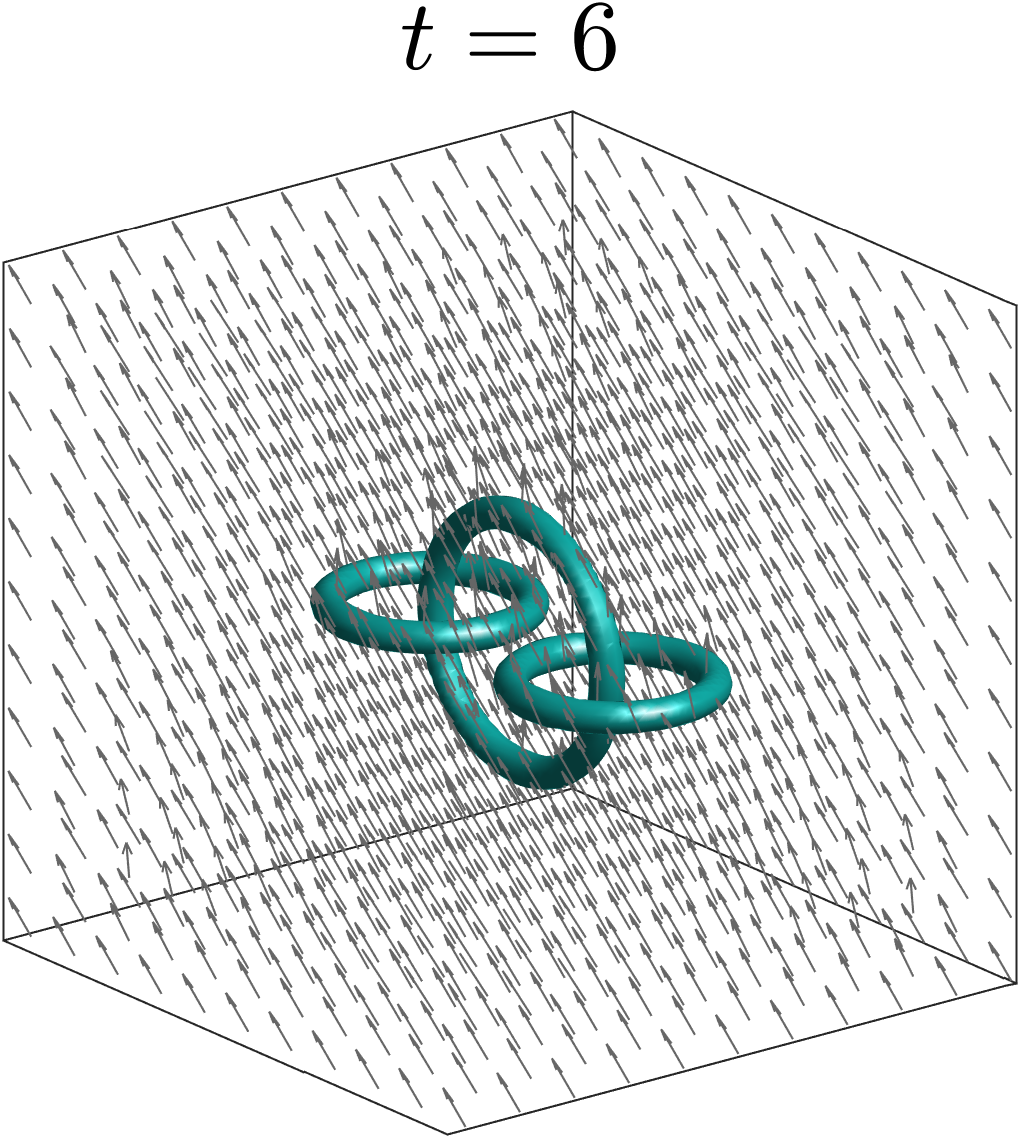}}\quad
		\subfigure{\includegraphics[width=0.31\textwidth,
			height=48mm]{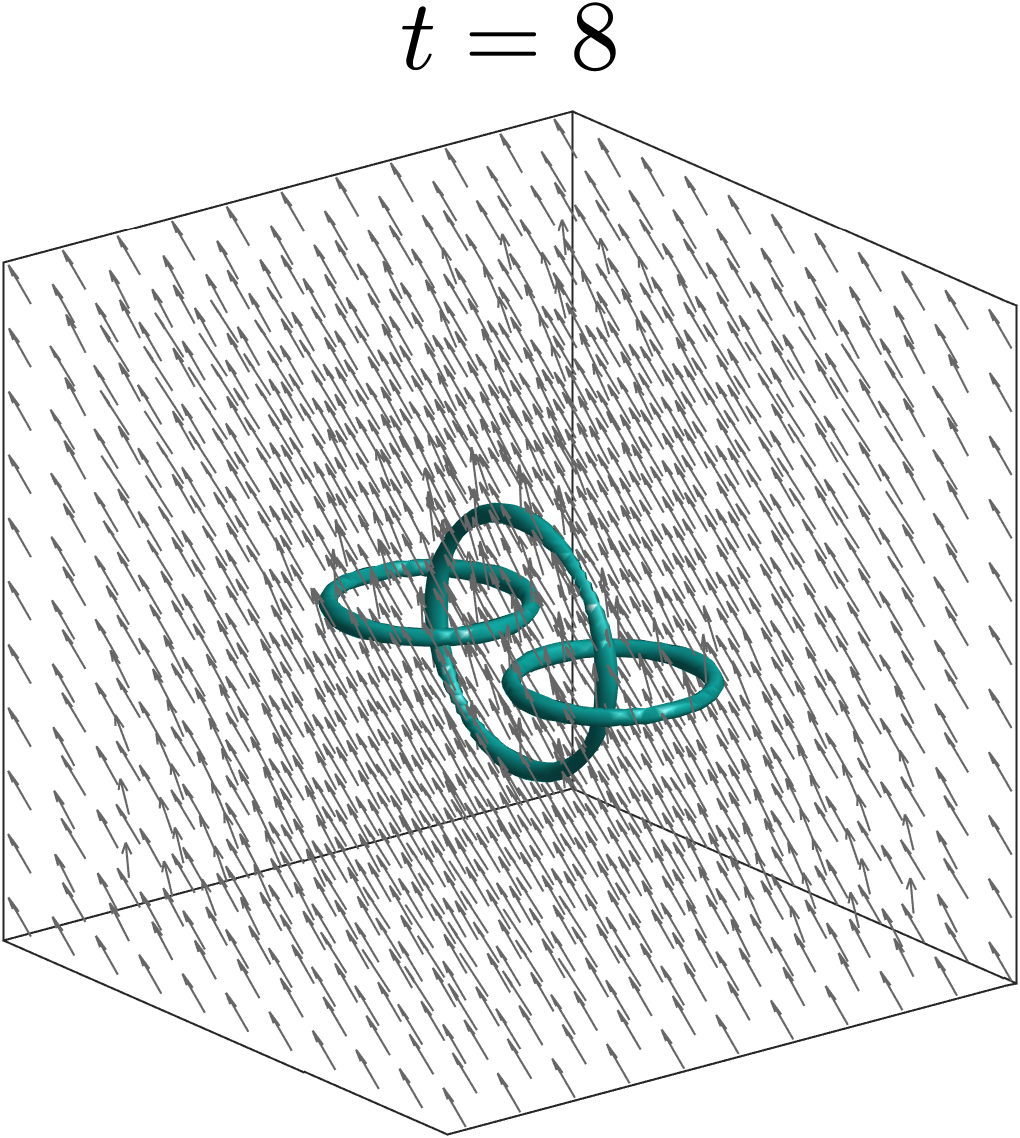}}\quad
		\subfigure{\includegraphics[width=0.31\textwidth,
			height=48mm]{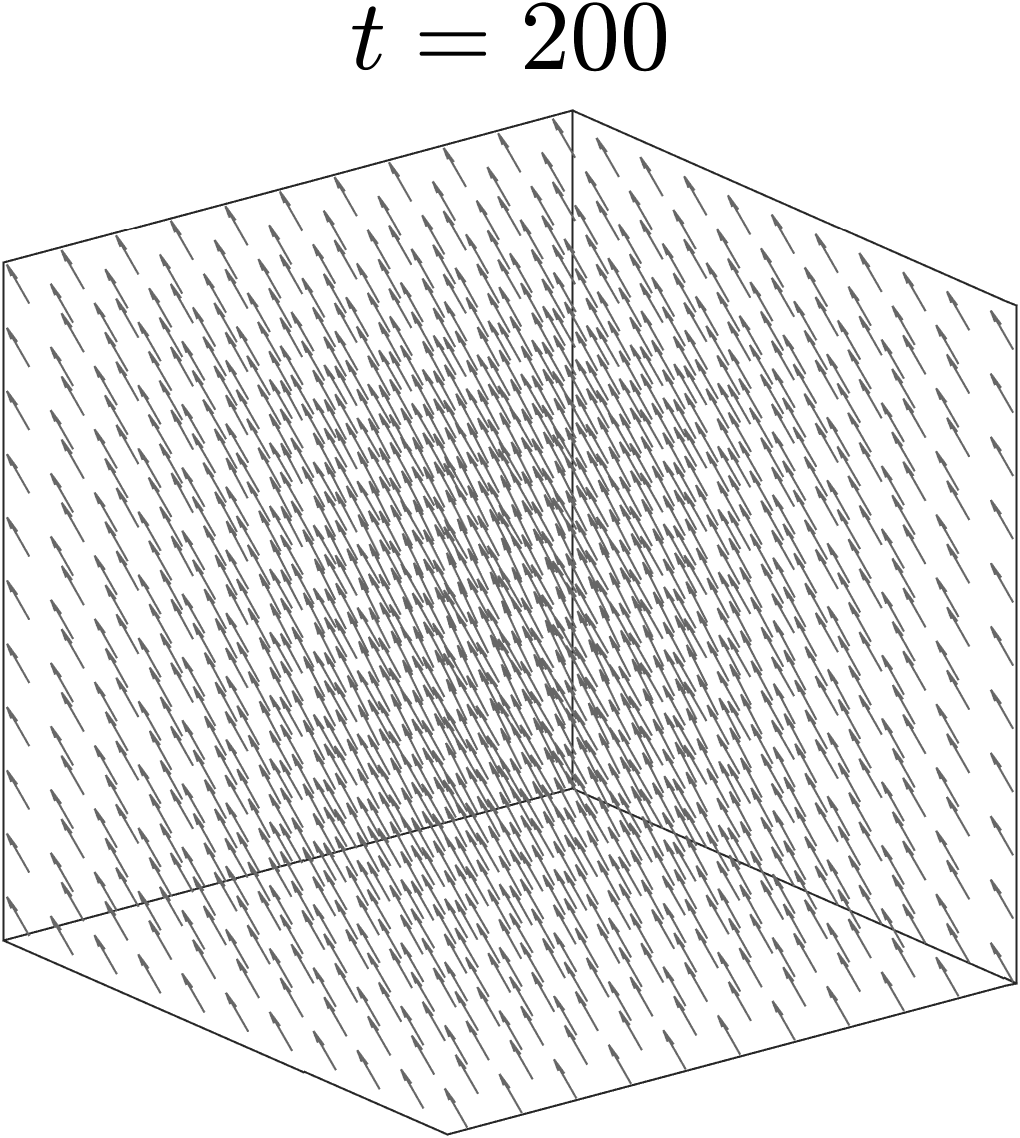}}
		\caption{Evolution of the matrix-valued field and interface at $t=0,2,4,6,8,200$. The initial field is given in \eqref{eq:4.6} with $\alpha(x,y,z)=4\pi xyz$ and $r=0.04$.}
  \label{fig:4.21}
	\end{center}
\end{figure}
 \begin{figure}
	\begin{center}
		\subfigure{\includegraphics[width=0.28\textwidth,
			height=40mm]{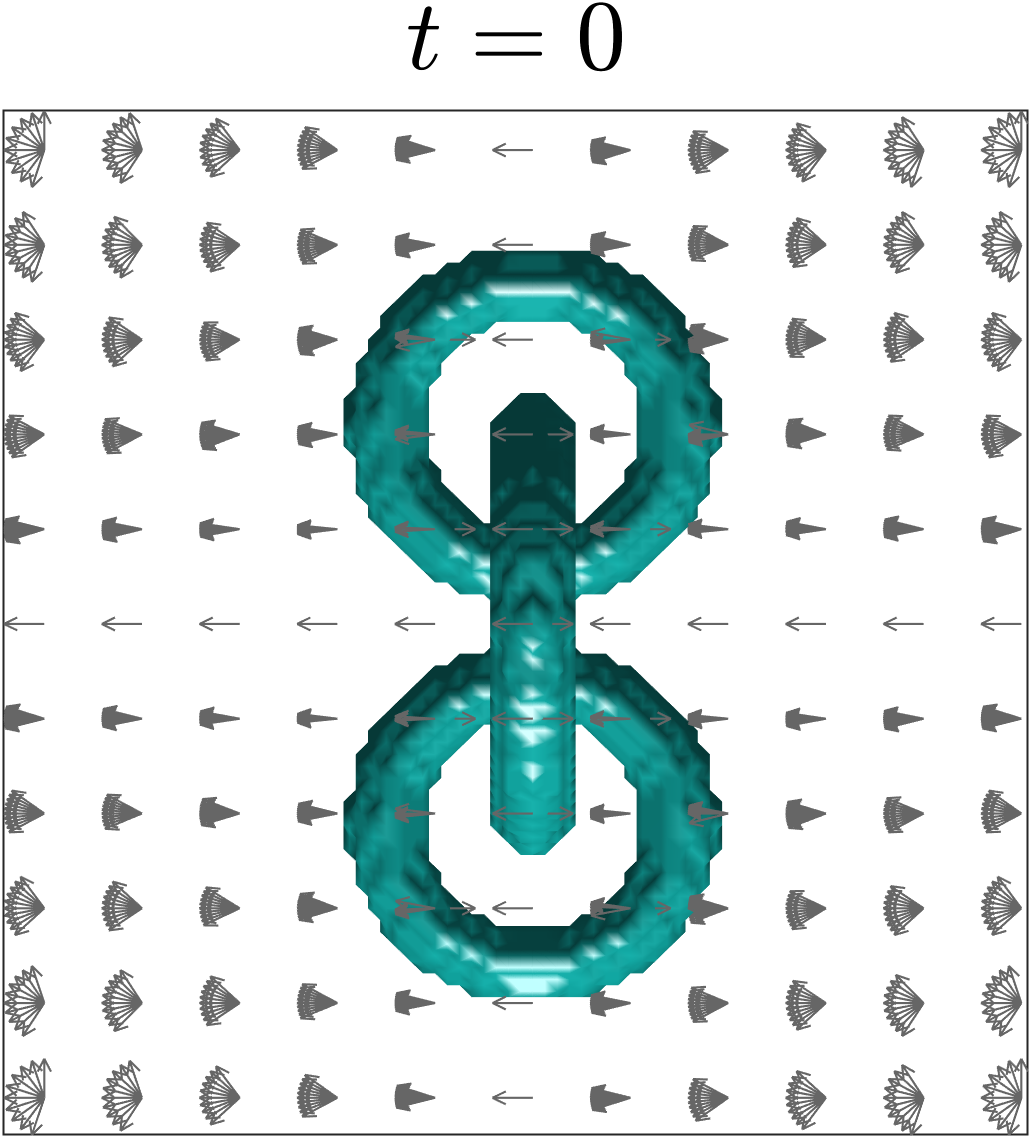}}\quad
		\subfigure{\includegraphics[width=0.28\textwidth,
			height=40mm]{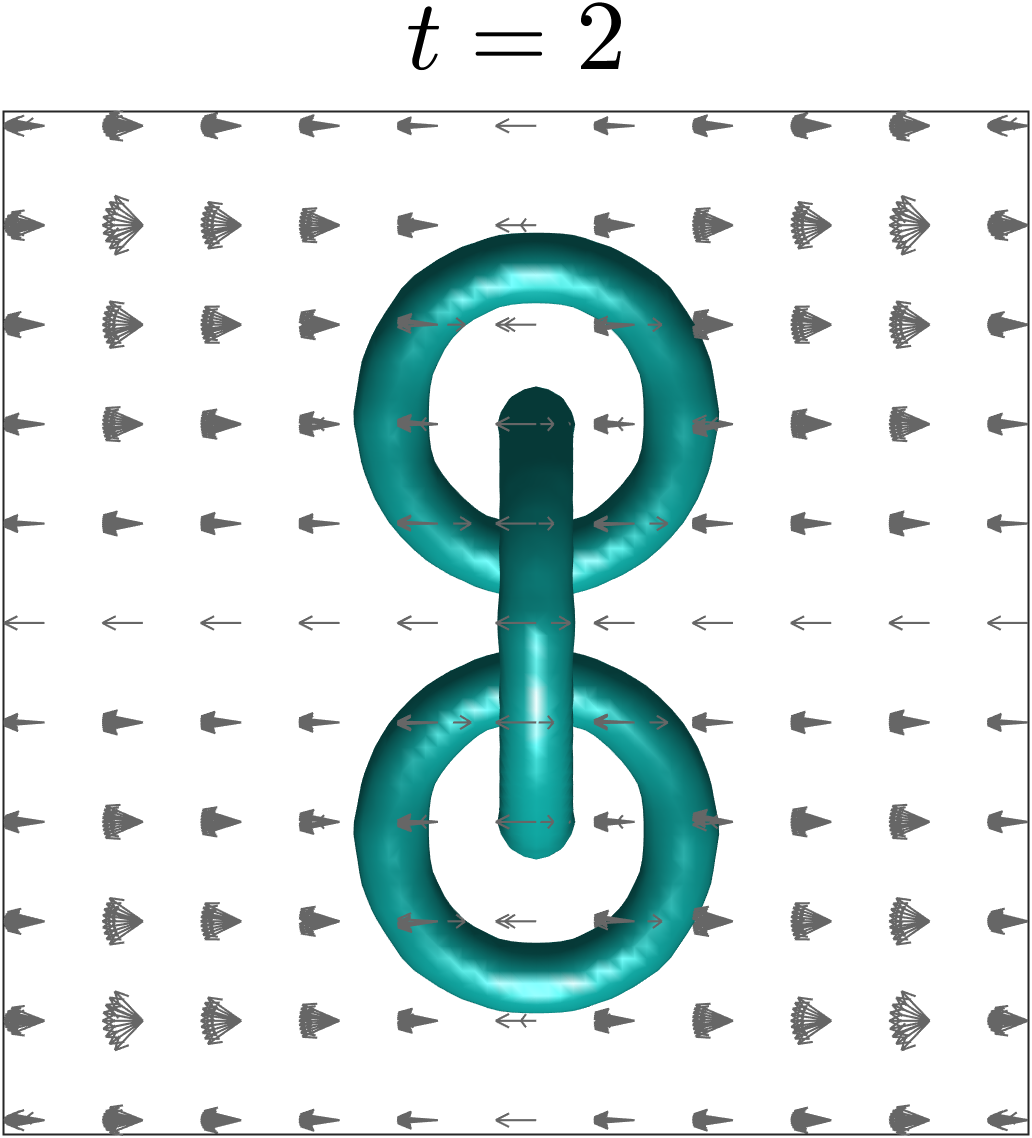}}\quad
    	\subfigure{\includegraphics[width=0.28\textwidth,
			height=40mm]{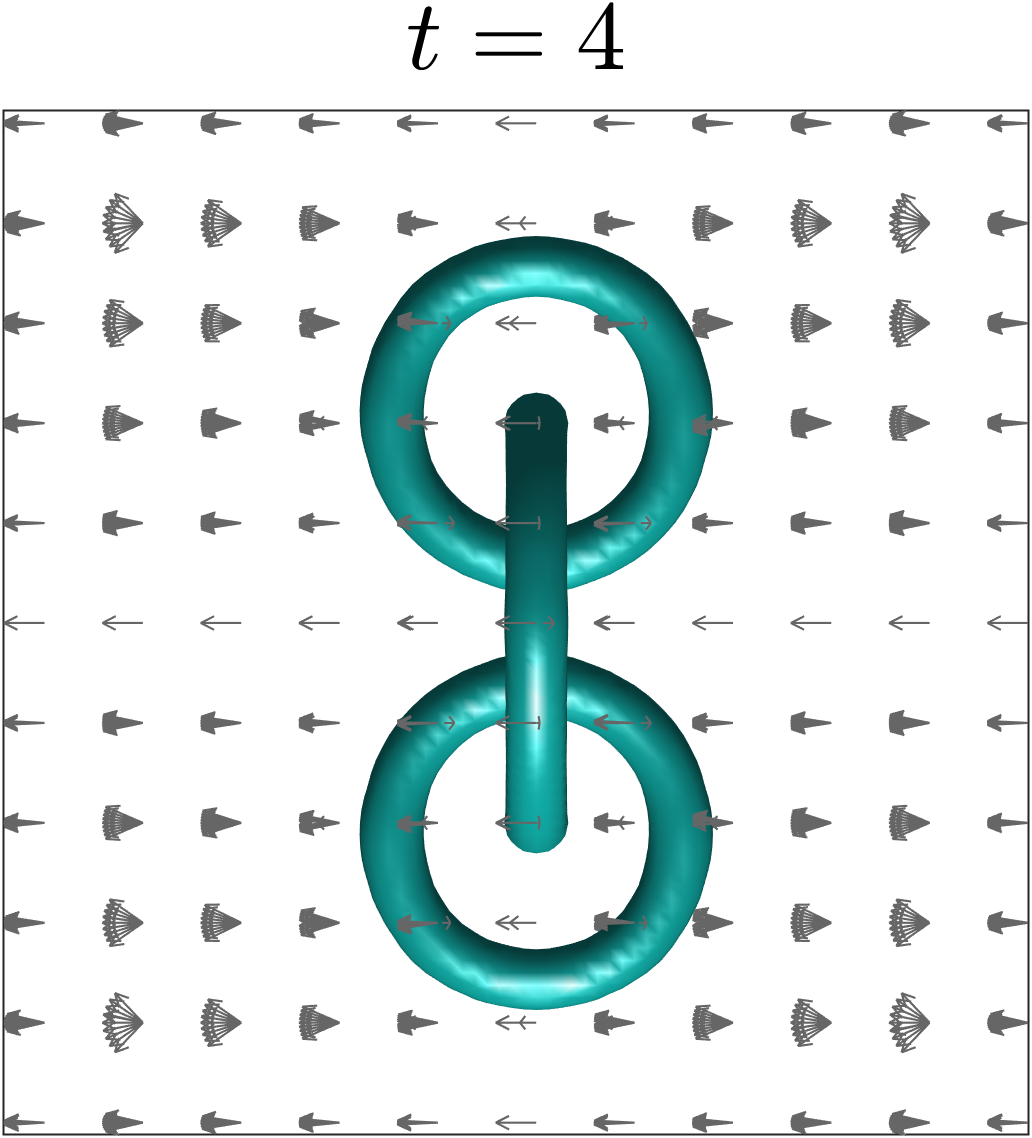}}\quad
		\subfigure{\includegraphics[width=0.28\textwidth,
			height=40mm]{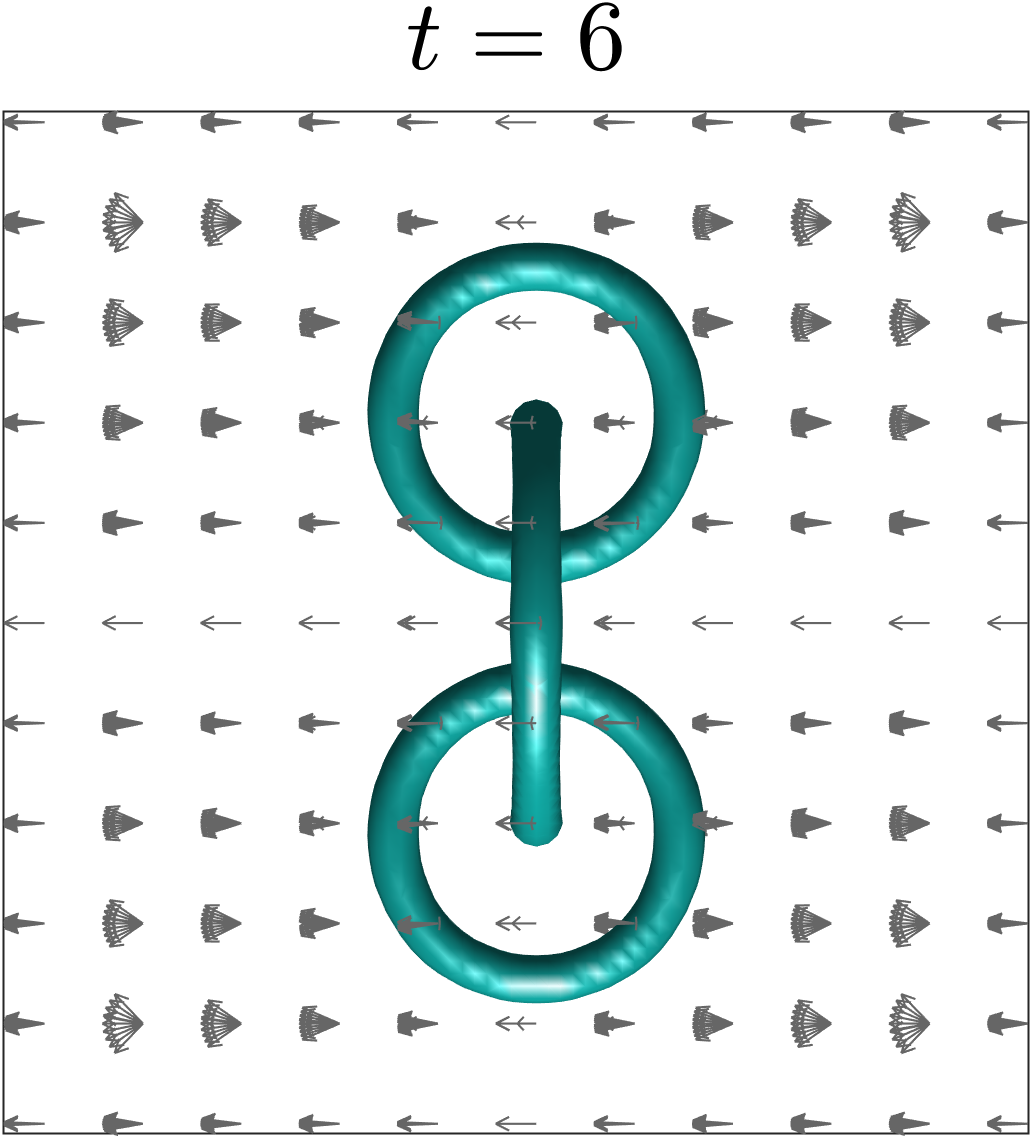}}\quad
    	\subfigure{\includegraphics[width=0.28\textwidth,
			height=40mm]{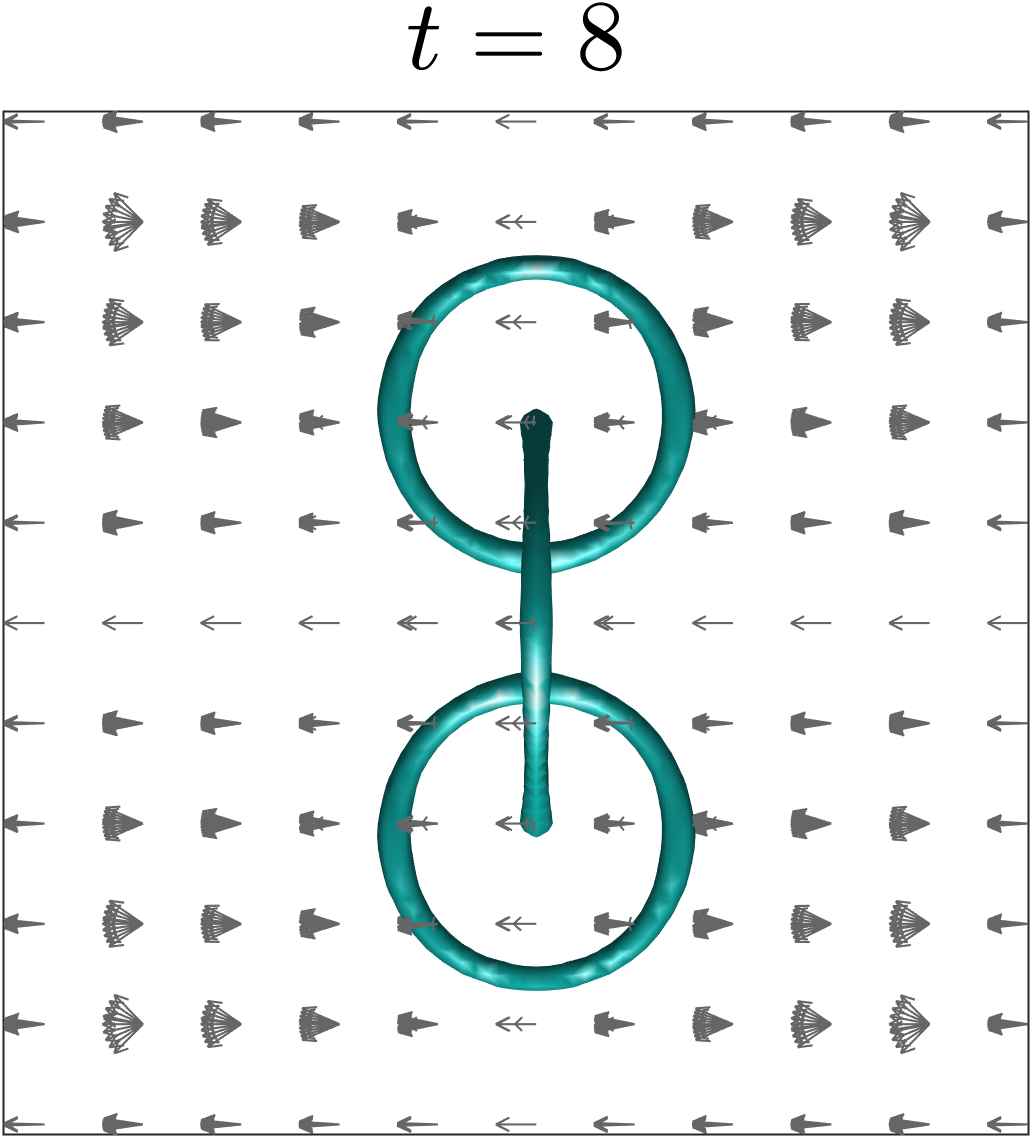}}\quad
		\subfigure{\includegraphics[width=0.28\textwidth,
			height=40mm]{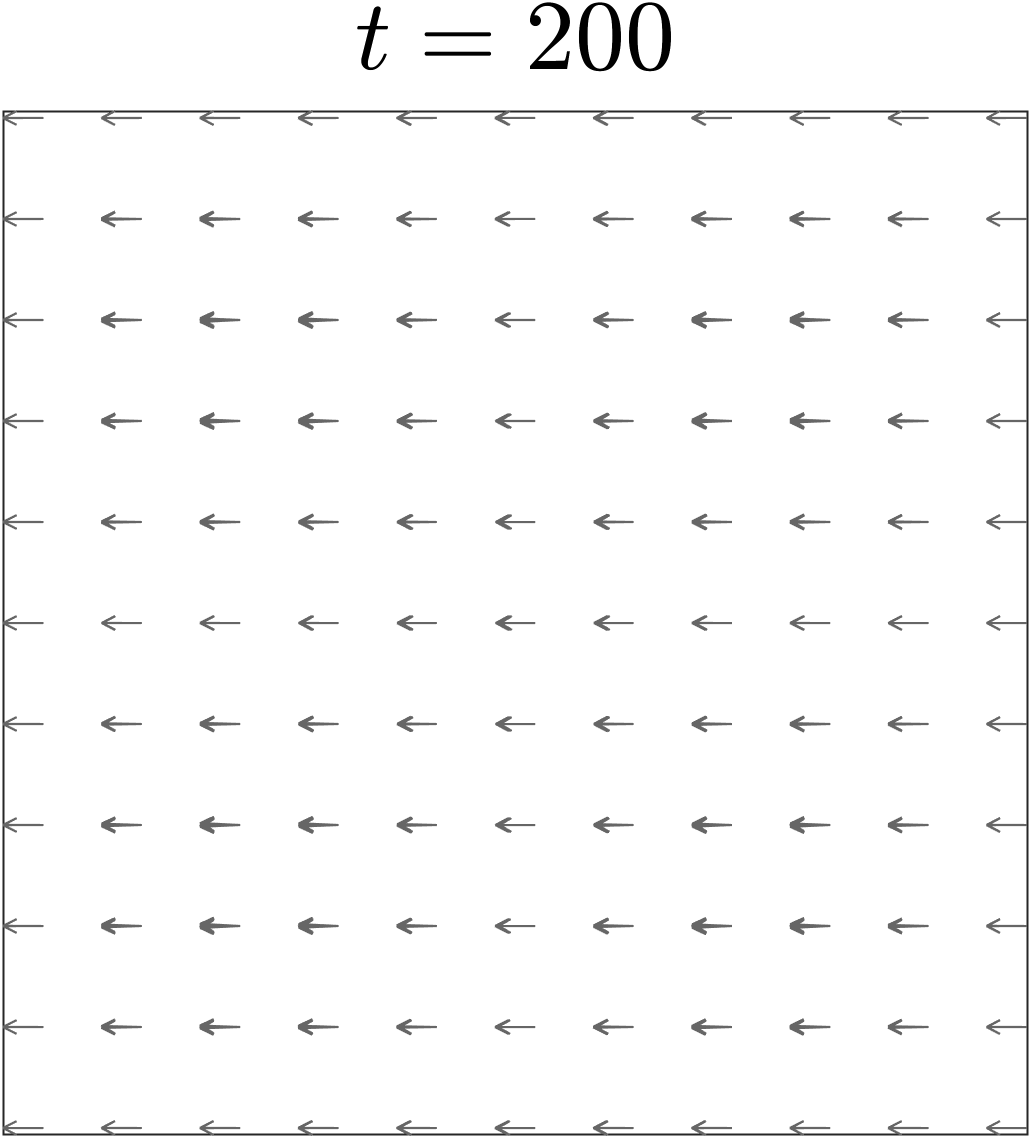}}
		\caption{A vertical view of the evolution of the matrix-valued field and interface at $t=0,2,4,6,8,200$. The initial field is given in \eqref{eq:4.6} with $\alpha(x,y,z)=4\pi xyz$ and $r=0.04$.}
  \label{fig:4.22}
	\end{center}
\end{figure}

\begin{figure}
	\begin{center}
		\subfigure{\includegraphics[width=0.42\textwidth,
			height=47mm]{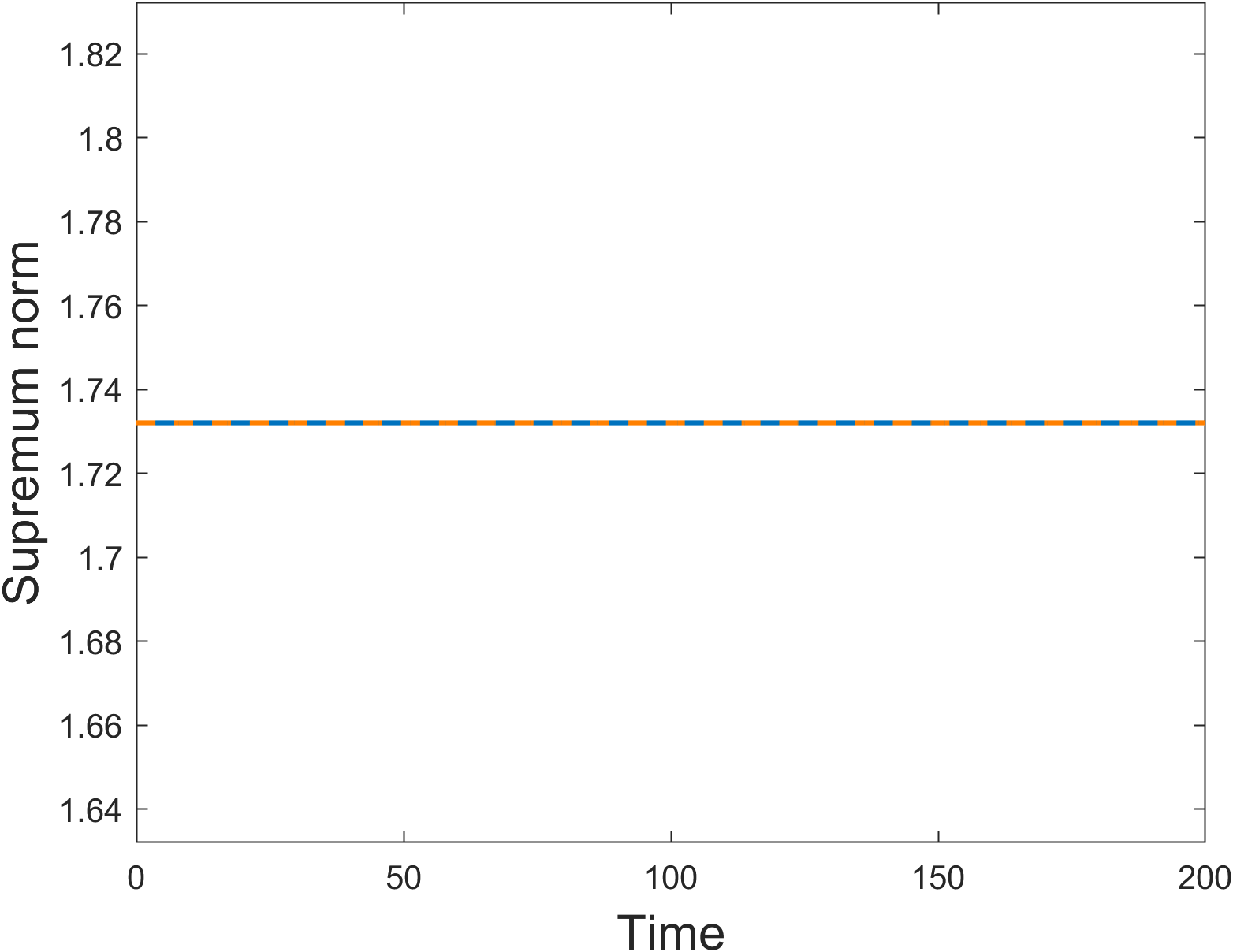}}\quad
		\subfigure{\includegraphics[width=0.42\textwidth,
			height=47mm]{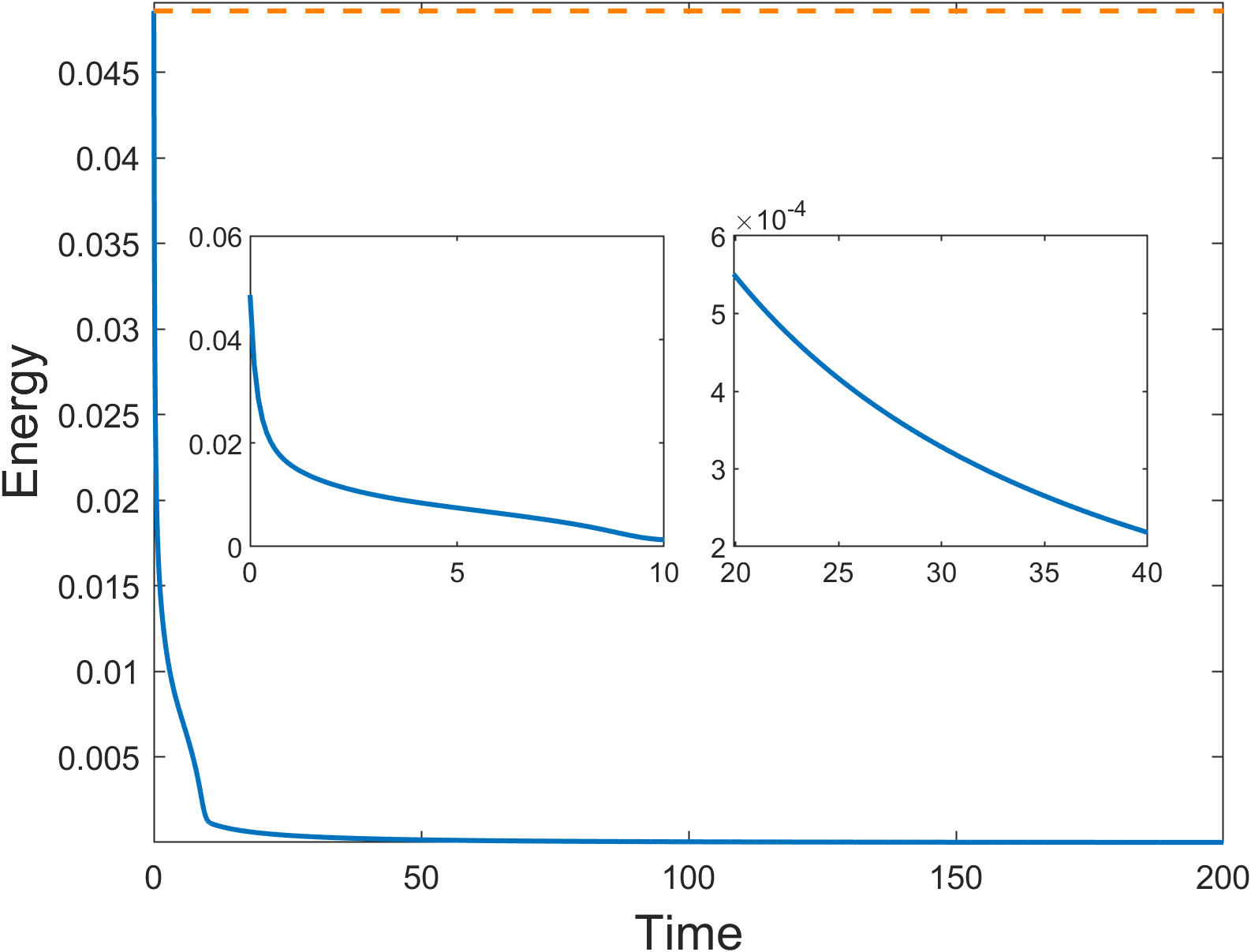}}
		\caption{Evolution of the supremum norm $\|\cdot\|_{\mathcal{X}}$ and energy with initial condition \eqref{eq:4.6} and $\alpha(x,y,z)=4\pi xyz$, $r=0.04$. The dashed line in the left figure is the maximum bound $\sqrt m$ while the dashed line in the right figure is the initial energy.}
        \label{fig:4.23}
	\end{center}
\end{figure}

\section{Conclusion}

In this work, we perform time discretization of the matrix-valued Allen--Cahn equation using first- and second-order ETD schemes. It is shown that both schemes can preserve unconditionally the MBP for any initial matrix-valued field $U\in \mathbb{R}^{m\times m}$ satisfying $\|U^0\|_F\leq\sqrt{m}$ for any $\mathbf x\in \overline \Omega$. This improves the MBP result in \cite{du2021maximum} requiring a symmetric initial matrix-valued field.
In addition, we prove that both the first and second order ETD schemes satisfy the energy dissipation law unconditionally.

In Theorem \ref{Conv1} and \ref{Conv2}, $U\in C^1([0,T];\mathcal{X})$ and $U\in C^2([0,T];\mathcal{X})$ are assumed respectively. We mention that 
if the initial data $U^0$ is sufficiently smooth, $\|U^0\|_{\mathcal X}\leq \sqrt m$, and $\varepsilon \leq 1$, it is expected that the first and second order derivatives in time of the exact solution $U$, $\partial_t U$ and $\partial^2_t U$, could be proved to be bounded independent of $\varepsilon$ in $\mathcal X$-norm. That is to say, the $C^1([0,T];\mathcal{X})$-norm and $C^2([0,T];\mathcal{X})$-norm are expected to be independent of $\varepsilon$.
Such $\varepsilon$-dependency issues will be discussed in our future work. 

\section*{Acknowledgements}

 C. Quan is supported by National Natural Science Foundation of China  (Grant No. 12271241), Guangdong Provincial Key Laboratory of Mathematical Foundations for Artificial Intelligence (2023B1212010001), Guangdong Basic and Applied Basic Research Foundation (Grant No. 2023B1515020030), and Shenzhen Science and Technology Innovation Program (Grant No. JCYJ20230807092402004, RCYX20210609104358076). 
D. Wang is partially supported by National Natural Science Foundation of China (Grant No. 12101524, 12422116), Guangdong Basic and Applied Basic Research Foundation (Grant No. 2023A1515012199), Shenzhen Science and Technology Innovation Program (Grant No. JCYJ20220530143803007, RCYX20221008092843046), Guangdong Provincial Key Laboratory of Mathematical Foundations for Artificial Intelligence (2023B1212010001), and Hetao Shenzhen-Hong Kong Science and Technology Innovation Cooperation Zone Project (No.HZQSWS-KCCYB-2024016).

\bibliographystyle{abbrvnat}
\bibliography{references}
\end{document}